%% file: main-sampling.tex
\documentclass[11pt]{article}
\usepackage[toc,page]{appendix}
\usepackage{bbm}
\usepackage{fullpage}
\usepackage[top=1in, bottom=1.25in, left=1in, right=1in]{geometry}
\usepackage{etaremune}
\usepackage{enumitem}
\usepackage{graphicx}
\usepackage[labelformat=empty]{caption}
\usepackage{subcaption}
\usepackage{tikz}
\usepackage{amssymb,amsmath,amsthm}
\usepackage{hyperref}
\usepackage[linesnumbered,lined,ruled]{algorithm2e}
\usepackage{comment}
\usepackage{mathabx}

\usetikzlibrary{decorations.pathreplacing,calligraphy}

\DeclareSymbolFont{rsfs}{U}{rsfs}{m}{n}
\DeclareSymbolFontAlphabet{\mathscrsfs}{rsfs}
\usepackage{mathrsfs}
\usepackage{url}
\usepackage{times}

\newif\ifonlyarxivversion
\onlyarxivversiontrue
\long\def\onlyarxiv#1{\ifonlyarxivversion#1\fi}
\long\def\onlyarxivalign#1{\ifonlyarxivversion&\rev{#1}\\\fi}

\hypersetup{
  colorlinks   = true, 
  urlcolor     = blue, 
  linkcolor    = blue, 
  citecolor   = red 
}

\newtheorem*{rep@theorem}{\rep@title}
\newcommand{\newreptheorem}[2]{%
\newenvironment{rep#1}[1]{%
 \def\rep@title{#2 \ref{##1}}%
 \begin{rep@theorem}}%
 {\end{rep@theorem}}}
\makeatother

\newtheorem{thm}{Theorem}[section]
\newtheorem{lem}[thm]{Lemma}

\newtheorem{ppn}[thm]{Proposition}
\newreptheorem{ppn}{Proposition}
\newtheorem{cor}[thm]{Corollary}

\newtheorem{fac}[thm]{Fact}

\theoremstyle{definition}

\newtheorem{rmk}[thm]{Remark}

\numberwithin{equation}{section}

\input{commands}

\author{
    Brice Huang\thanks{Department of Electrical Engineering and Computer Science, Massachusetts Institute of Technology}
    \and
    Andrea Montanari\thanks{Department of Statistics and Department of Mathematics, Stanford University}
    \and
    Huy Tuan Pham\thanks{Department of Mathematics, University of Chicago}
}

\title{Sampling from Spherical Spin Glasses in Total Variation via Algorithmic Stochastic Localization}

\date{\today}
\begin{document}

\maketitle

\begin{abstract}
    We consider the problem of algorithmically sampling from the Gibbs measure of a mixed $p$-spin spherical spin glass.
    We give a polynomial-time algorithm that samples from the Gibbs measure up to vanishing total variation error, for any model whose mixture satisfies
    \[
        \xi''(s) < \fr{1}{(1-s)^2}, \qquad \forall s\in [0,1).
    \]
     Our algorithm follows the algorithmic stochastic localization (SL) approach introduced in \cite{alaoui2022sampling} and in fact achieves a more difficult task. Namely, it (approximately) generates a sample path $(\by_t)_{t\ge 0}$ from the SL process. For large $T$, $\by_T/T$ is an approximate sample from the target  
     Gibbs measure, and hence we solve the original problem as a special case. 

     We provide evidence that the condition $\xi''(s) < (1-s)^{-2}$ is 
     sharp for the task of sampling SL sample paths. Furthermore,  
     we prove that it is optimal within a class of generalized SL processes. 
    Remarkably this threshold is close to the conjectured threshold for sampling from the original Gibbs measure. In particular, for the pure $p$-spin glasses it
    is within an absolute ($p$-independent) constant of the so-called shattering phase transition.
Earlier work on related models was suboptimal by a factor diverging polynomially in $p$.

    A key step of this approach is to estimate the mean of a sequence of tilted measures.
    We produce an improved estimator for this task by identifying a suitable correction to the TAP fixed point selected by approximate message passing (AMP).
    As a consequence, we improve the algorithm's guarantee over previous work, from normalized Wasserstein to total variation error. In particular, the new algorithm and analysis open the way to perform inference about one-dimensional projections of the measure.
\end{abstract}


\tableofcontents

\section{Introduction}

Let $\gamma_2,\gamma_3,\ldots \ge 0$ satisfy $\sum_{p\ge 2} 2^p \gamma_p^2 < \infty$.
The mixed $p$-spin glass Hamiltonian $H_N : \bbR^N \to \bbR$ is
\beq
    \label{eq:def-HN}
    H_N(\bsig) =
     \sum_{p\ge 2}
     \fr{\gamma_p}{N^{(p-1)/2}}
     \sum_{i_1,\ldots,i_p = 1}^N
     G_{i_1,\ldots,i_p}
     \sigma_{i_1} \cdots \sigma_{i_p},
     \qquad
     G_{i_1,\ldots,i_p} \stackrel{i.i.d.}{\sim} \cN(0,1).
\eeq
Define the mixture function $\xi(s) = \sum_{p\ge 2} \gamma_p^2 s^p$, so that $H_N$ is the Gaussian process with covariance
\[
    \EE H_N(\bsig^1) H_N(\bsig^2) = N\xi\big(\<\bsig^1,\bsig^2\>/N\big)\, .
\]
The Gibbs measure of this model is the probability measure over the sphere $S_N = \{\bx \in \bbR^N : \norm{\bx}_2^2 = N\}$ given by
\beq
    \label{eq:gibbs-measure}
    \mu_{H_N}(\de \bsig) = \fr{1}{Z_N} \exp(H_N(\bsig)) ~\mu_0(\de \bsig), \qquad
    Z_N = \int_{S_N} \exp(H_N(\bsig)) ~\mu_0(\de \bsig)\, .
\eeq
Here and below, $\mu_0$ denotes the uniform probability measure on $S_N$. We will denote by
$\bG= (G_{i_1,\ldots,i_p})_{p\ge2, i_{\ell} \le N}$ the vector of couplings that defines the Hamiltonian.
We consider the problem of efficiently sampling from this Gibbs measure.
For $\dist$ a distance on $\cP(\bbR^N)$ (the set of probability measures over $\bbR^N$), we seek a computationally efficient algorithm that generates $\bsig^\alg$ whose law $\mu^\alg$ satisfies $\dist(\mu^\alg, \mu_{H_N}) = o_N(1)$, with high probability over $H_N$.

The most classical approach to sampling is to construct a Markov chain that is reversible with respect to $\mu_{H_N}(\de \bsig)$, and
for which a single step can be implemented efficiently. In the present context, such a Markov chain can be obtained by discretizing Langevin dynamics,
see \cite{dalalyan2017theoretical,durmus2017nonasymptotic,dwivedi2019log}
and Section \ref{subsec:alg-sl}.

An alternative approach has emerged recently.
The basic idea is to generate (an approximation of) a sample path from the following Ito diffusion on $\bbR^N$:
\begin{align}\label{eq:sl-sde}
    \de \by_t = \bm(\by_t,t)\,\de t +\de\bB_t\, , \;\;\;\;\; \by_0=\bfzero\, ,
\end{align}
where $(\bB_t)_{t\ge 0}$ is a standard Brownian motion and $\bm(\by,t) = \EE[\bsig | t\bsig + \sqrt{t}\bg=\by]$ (conditioning on $\bG$ is implicit here),
with the conditional expectation being taken with respect to $(\bsig,\bg)\sim \mu_{H_N}\otimes\cN(\bzero,\bI_N)$.
The key remark (see Section \ref{subsec:sl}) is that $(\by_t)_{t\ge 0}$ thus defined has the same distribution as $(t\bsig+\bB_t')_{t\ge 0}$ (with $\bB'_t$ a different Brownian motion) and therefore $\by_t/t$ converges to a sample from the desired measure.
Of course, constructing an actual algorithm requires discretizing time and --- crucially ---
defining an efficient algorithm that approximates the conditional mean $\bm(\,\cdot\,,t)$ well enough.

This idea was introduced in \cite{alaoui2022sampling}, which developed it
as an algorithmic implementation of Eldan's 
stochastic localization (SL) technique for proving functional inequalities in high-dimensional probability
\cite{eldan2013thin,eldan2020taming,eldan2022analysis}.
Independently, the same process attracted considerable interest in machine learning, under the name of `denoising diffusions' \cite{sohl2015deep,ho2020denoising,song2021score}. 
In that context, the target distribution $\mu$ is unknown, but one
is given access to samples $\bsig_1$, \dots, $\bsig_n$. These are
used to learn an approximation of the drift $\bm(\,\cdot\,,t)$
and hence generate trajectories from Eq.~\eqref{eq:sl-sde}.
(We refer to \cite{montanari2023sampling} for further discussion of this connection.)

The analysis of \cite{alaoui2022sampling,celentano2024sudakov, alaoui2023sampling} establishes
that algorithmic SL  samples from the Gibbs measure of the
Sherrington-Kirkpatrick model on the (more difficult) cube $\Sigma_N = \{-1,1\}^N$, up to vanishing normalized Wasserstein error.
That is, with probability $1-o_N(1)$ over the Sherrington-Kirkpatrick Hamiltonian $H_N$, there is a coupling of $\mu_{H_N}$ and $\mu^\alg$ such that for $(\bsig,\bsig^\alg)$ drawn from this coupling,
\begin{align}
    \fr1N \EE_{\bsig,\bsig^\alg} \tnorm{\bsig - \bsig^\alg}_2^2 = o_N(1).
    \label{eq:NormalizedWasserstein}
\end{align}
\rev{In other words, the usual Wasserstein--$2$ distance is $o(\sqrt N)$; we use ``normalized Wasserstein error'' for the dimension-normalized quantity in \eqref{eq:NormalizedWasserstein}.}
Remarkably, this guarantee holds throughout the `replica symmetric'
phase of the Sherrington-Kirkpatrick model
\cite{talagrand2010mean}, i.e. for all $\beta<1$ 
\cite{celentano2024sudakov}. Conversely  \cite{alaoui2022sampling}
proves that no randomized algorithm whose output is a Lipschitz continuous 
function of the Hamiltonian $H_N$ (for each random seed) succeeds for $\beta>1$.
In particular, in this case, generating the whole SL path appears not to be 
fundamentally harder than sampling from the target measure $\mu_{H_N}$. 
On the other hand, the Hamiltonian of the Sherrington-Kirkpatrick model is quadratic,
and obtaining similar results for the general polynomial Hamiltonian 
of Eq.~\eqref{eq:def-HN} poses entirely new challenges \cite{alaoui2023sampling}.

Despite the intense research activity in this area, it is fair to say that many fundamental questions remain open from a mathematical viewpoint.
First, previous works on algorithmic SL produced algorithms that sampled from the target measure in the sense of vanishing normalized Wasserstein error.
It is important to ask whether this approach can be upgraded to sample with vanishing total variation error.
Second, generating an approximation of the whole SL trajectory $(\by_t)_{t\ge 0}$ is \emph{a priori} a more challenging task than sampling from the target distribution $\mu$.
One may thus ask whether there is a fundamental gap between measures $\mu(\de\bsig)$ that can be sampled using general polynomial-time algorithms and those that instead can be sampled via SL paths.

In this paper, we address both questions. We first affirmatively answer the first question. 
Our main result is an algorithm sampling with vanishing total variation error from the Gibbs measure of any spherical spin glass satisfying
\beq
    \label{eq:amp-works}
    \xi''(s) < \fr{1}{(1-s)^2}, \qquad \forall s\in [0,1).
\eeq
As we will soon discuss, the condition \eqref{eq:amp-works} presents a fundamental barrier to algorithms for generating sample paths from SL. The key challenge in implementing the algorithmic SL approach is
the construction of an efficient algorithm to approximate the mean of
the measure $\mu_{H_N}(\de \bsig)$, as well as its conditional mean given Gaussian observations $\by_t$.
The latter corresponds to the mean of an exponential tilt
$\mu_{H_N,\by}(\de \bsig)  \propto \exp(\<\by_t,\bsig\>)\mu_{H_N}(\de \bsig)$.
Approximating $\bm(\by)$ was achieved in \cite{alaoui2022sampling} by a variational approach that requires minimizing the so-called
Thouless-Anderson-Palmer (TAP) free energy \cite{thouless1977solution}. The same paper established
that the resulting estimate satisfies (with high probability) $\|\bm(\by)-\bm^{\sTAP}(\by)\|^2 = o(N)$.
(For the case of a measure supported on $S_N$, the function $\bm(\, \cdot\, )$
does not depend on $t$, and we will therefore omit this argument.)

Note that $\|\bm(\by)\|^2 = \Theta(N)$, and therefore
\cite{alaoui2022sampling} establishes the weakest non-trivial upper bound on
$\|\bm(\by)-\bm^{\sTAP}(\by)\|^2$.
This is the standard level of precision for estimates in spin glass theory
both in physics and mathematics \cite{mezard1987spin,talagrand2010mean}.
 
However, in order to obtain a sampling algorithm  with guarantees in total variation distance,
it is necessary to construct an efficient  estimator $\hbm(\by)$
satisfying $\|\bm(\by)-\hbm(\by)\|^2 = o(1)$. The construction and analysis of such an estimator is the main
problem solved in the present paper.

In fact we prove the following:
\begin{enumerate}
    \item The TAP estimator is significantly more accurate than what could be hoped from the analysis of \cite{alaoui2022sampling,alaoui2023sampling}. Namely, we prove that $\|\bm(\by)-\bm^{\sTAP}(\by)\|^2 = O(1)$.
    \item We design a correction $\bcorr(\by)$ to the TAP estimator that can be computed efficiently and such that, letting $\hbm(\by) = \bm^{\sTAP}(\by)+\bcorr(\by)$, we achieve the desired accuracy $\|\bm(\by)-\hbm(\by)\|^2 = o(1)$.
\end{enumerate}
To the best of our knowledge, neither of these results was known before, even at a heuristic level.

Turning to the second question on the general behavior of algorithms based on SL paths, 
we will show  that the general situation is 
richer than for the Sherrington-Kirkpatrick model. 
Indeed, it is reasonable to conjecture that a necessary condition for 
algorithmic SL to succeed is that the conditional law
of $\bsig$ given $\by_t$ remains replica symmetric for all $t$.
While one might naively expect that this conditional measure is simpler 
than the original one, statistical physics calculations suggest that this 
might not be the case \cite{montanari2007solving,ghio2023sampling}.
Namely, it can be the case that the original measure $\mu_{H_N}(\de\bsig)$
is replica symmetric and does not exhibit shattering 
(see  \cite{gamarnik2023shattering,alaoui2023shattering,arous2024shattering}  and below),
while the localized measure $\mu_{H_N}(\de\bsig|\by_t)$ exhibits
replica symmetry breaking for some $t$.
\rev{The assumption \eqref{eq:amp-works} can be physically understood as the criterion under which the localized measures $\mu_{H_N}(\de\bsig|\by_t)$ are replica symmetric for all $t\ge 0$, with high probability under the localization process.}

Within the spherical spin glass studied in this paper,
the conjectured threshold for efficient sampling is $\xi'(q)<q/(1-q)$
for all $q\in(0,1)$, and shattering is believed to take place beyond
this threshold. The conditional measure generated by the SL 
process at time $t$ is a spherical model with mixture $\xi_t(q) = \xi(q)+tq$
and the condition for absence of shattering for all $t$ turns out to be
that the equation $\xi'(q)+t= q/(1-q)$ has only one solution in $[0,1)$ for all
$t$. This is in turn equivalent to \eqref{eq:amp-works}, under which our sampling algorithm succeeds.

We further show that the condition \eqref{eq:amp-works} is a fundamental barrier for algorithmic SL.
We show that in the subset of the replica symmetric regime where \eqref{eq:amp-works} does not hold, a form of algorithmic SL with generalized side-channels does not succeed.
(We also expect this algorithm to continue to fail beyond the replica symmetric regime, see Remark~\ref{rmk:rs-artifact}.)
As the conjectured regime for efficient sampling is $\xi'(q)<q/(1-q)$ for all $q\in(0,1)$, which strictly contains the regime \eqref{eq:amp-works}, this suggests a separation between sampling by general polynomial-time algorithms and by SL paths.

\begin{rmk}
    The threshold \eqref{eq:amp-works} is remarkably close to the threshold for shattering in the measure $\mu_{H_N}$.
    For the special case of pure models $\xi(s) = \beta^2 s^p$, \eqref{eq:amp-works} holds for all $\beta < \beta_\sSL(p)$, where we defined the
    stochastic localization inverse temperature as
    \begin{align}\label{eq:SLPureModel}
        \beta_{\sSL}(p) := \fr12 \sqrt{\lt(\fr{p}{p-1}\rt)\lt(\fr{p}{p-2}\rt)^{p-2}}.
    \end{align}
    For large $p$ we have $\beta_{\sSL}(p) = e/2 + O(1/p)$.
    On the other hand, the conjectured no-shattering condition
    $\xi'(q)<q/(1-q)$ holds for all $\beta<\beta_{\ssh}(p)$
  where
\begin{align}\label{eq:ShatteringPureModel}
\beta_{\ssh}(p) = \sqrt{\frac{(p-1)^{p-1}}{p(p-2)^{p-2}}}\, .
\end{align}
For large $p$, $\beta_{\ssh}(p) = \sqrt{e}+O(1/p)$. 
In particular $\beta_{\ssh}(p)/\beta_{\sSL}(p)\le 2/\sqrt{e}\approx 1.213$
for all $p$. Further, the two thresholds coincide in mixed models
when $\xi''(0)$ is sufficiently large.
\end{rmk}

\rev{\begin{rmk}
    Ben Arous and Jagannath \cite{arous2024shattering} proved that, for the pure model $\xi(s) = \beta^2 s^p$, the Gibbs measure has \textit{metastable states} for all $\beta$ larger than an explicit $\beta_{\mathsf{BBM}} \asymp (\log p)^{-1/2}$, a phenomenon first predicted in
    \cite{barrat1996dynamics}. These are regions of the Gibbs measure with $o_N(1)$ mass but poor connectivity to the rest of the measure, from which the Langevin dynamics takes exponential time to escape. In particular, for the pure $p$-spin models, our algorithm succeeds in a range of $\beta$ where the Langevin dynamics with \textit{worst case} initialization provably fails to sample in polynomial time (though Langevin dynamics with \textit{random} initialization is expected to succeed, see below). To our knowledge, the $\beta$ at which metastable states emerge in mixed $p$-spin models is not known.
\end{rmk}}

\subsection{Further background and related work}

A substantial line of work in probability theory studies Langevin dynamics for the Gibbs measure \eqref{eq:gibbs-measure}.
This is defined as the following diffusion on $S_N$
\begin{align}
\de\bsig_t = \left( \Proj^{\perp}_{\bsig_t}\nabla H_N(\bsig_t)-\frac{N-1}{2N}\bsig_t\right)\de t
+\sqrt{2}\Proj^{\perp}_{\bsig_t}\de\bB_t\,,
\end{align}
where $\bB_t$ is a standard $N$-dimensional Brownian motion, and $\Proj^{\perp}_{\bsig_t}$
is the projector orthogonal to $\bsig_t$. Langevin dynamics is a Markov process reversible
for the measure  $\mu_{H_N}$ of Eq.~\eqref{eq:gibbs-measure}. 
As mentioned above, suitable discretizations
of Langevin dynamics can be used to sample from $\mu_{H_N}$
\cite{dalalyan2017theoretical,durmus2017nonasymptotic,dwivedi2019log}.

An asymptotically exact characterization of Langevin dynamics on short time horizons $t=O(1)$,
in the high-dimensional limit $N\to\infty$, is provided by the so-called Cugliandolo-Kurchan or `dynamical mean-field theory'
equations.
These were studied first in physics \cite{crisanti1993spherical,cugliandolo1993analytical} and
subsequently established rigorously in probability theory \cite{ben2006cugliandolo}.
Unfortunately, this approach does not give access to  mixing times. On top of that, the
Cugliandolo-Kurchan equations proved difficult to analyze rigorously except at sufficiently `high temperature'
(i.e. when $\xi(s) = \beta^2\xi_1(s)$, for a fixed $\xi_1$ and $\beta$ small enough) \cite{dembo2007limiting}.

Based on a postulated asymptotic form of the Cugliandolo-Kurchan equations, as well as
on thermodynamic calculations, physicists conjecture
a phase transition in the mixing time of Langevin dynamics,
when initialized uniformly at random \cite{crisanti1993spherical,cugliandolo1993analytical} .
Namely, they expect the mixing time to be polynomial in $N$
for 
\begin{align}
        \xi'(q) < \fr{q}{1-q}, \qquad \forall q\in (0,1)\, ,\label{eq:FirstShatteringEq}
\end{align}
and exponentially large in the opposite case, and
more precisely when $\sup_{q\in (0,1)}(1-q)\xi'(q)/q >1$.
This is commonly referred to as the `dynamical phase transition,' and corresponds to a phase transition in the
geometry of the Gibbs measure, known as the `shattering phase transition.'
In the homogeneous case $\xi(t)=\beta^2 t^p$,
the above formula implies that the
dynamical/shattering phase transition
takes place at $\beta=\beta_{\ssh}(p)$ given by
Eq.~\eqref{eq:ShatteringPureModel}.

We
also recall that a second phase transition (`condensation'
or `static' or `replica symmetry breaking') takes place
at a lower temperature
\begin{align}
 \beta_c^2(p)
    =
    \inf_{s\in [0,1]}
    \lt(\frac{1}{s^p}\log\lt(\frac{1}{1-s}\rt)-\frac{1}{s^{p-1}}\rt)\, .
    \label{eq:PspinExplicit}
\end{align}
This corresponds to a non-analyticity of the free energy,
and to the temperature at which the overlap stops concentrating
\cite{chen2013aizenman}. For large $p$, we have
$\beta_c(p) = \sqrt{\log p} (1+o_p(1))$.

Towards the goal of proving the dynamical phase transition phenomenon,
Ben Arous and Jagannath \cite{arous2024shattering} established
that --- for the homogeneous model ---
shattering takes place in a non-empty temperature interval, implying in particular $\beta_{\ssh}(p)<\beta_c(p)$ strictly.
An order-optimal bound was proven in \cite{alaoui2023shattering}:
$\beta_{\ssh}(p)\le C$ for a $p$-independent constant $C$.

A bolder version of the dynamical phase transition conjecture
postulates not only that Langevin dynamics is slow, but also that
sampling is fundamentally hard beyond the shattering
phase transition. Rigorous evidence was provided in
\cite{alaoui2023shattering}, which proves that `stable algorithms' fail to sample from $\mu_{H_N}$ under shattering.

In the positive direction Gheissari and Jagannath \cite{gheissari2019spectral}  proved that there exists
$\underline{\beta}(p)>0$ such that Langevin dynamics
mixes rapidly for $\beta<\underline{\beta}(p)$.
These authors also note that their proof technique extends
to mixed models.

A closely related model is the Ising version
of model \eqref{eq:gibbs-measure}, whereby the uniform measure
$\mu_0$ over the sphere $S_N$ is replaced by the uniform measure over the hypercube $\{+1,-1\}^N$. A
dynamical/shattering phase transition was conjectured in  that setting as well \cite{kirkpatrick1987p}, although at a different
temperature.  In this context, shattering for a non-empty interval of temperatures was proven in  \cite{gamarnik2023shattering}, while mixing of Glauber dynamics at high temperature was proven in
\cite{adhikari2022spectral,anari2023universality}. As for the spherical case, positive and negative
results are separated by a large gap, indeed diverging with $p$.

The algorithmic stochastic localization approach was applied to Ising mixed $p$-spin glasses in \cite{alaoui2023sampling}, which established the Wasserstein guarantee \eqref{eq:NormalizedWasserstein}.
\rev{A related algorithmic sampling question for the continuous random energy model was studied by Ho and Maillard \cite{ho2022efficient}.}

\subsection{Notations}

Throughout this paper, $\norm{\bsig}_N = \norm{\bsig}/ \sqrt{N}=\sqrt{\bsig^{\top}\bsig/N}$ is the norm corresponding to the inner product $\<\bsig_1,\bsig_2\>_N =\<\bsig_1,\bsig_2\>/N=\bsig_1^{\top}\bsig_2/N$.
There will be no confusion with the $\ell_p$ norm, which will not appear.
Given a matrix $\bA$, we denote by $\|\bA\|_{\mathrm{F}}$ its Frobenius norm.
\rev{We write $\|\bA\|_{\op}$ for its usual operator (spectral) norm.}
For $\bm \in \bbR^N$, measurable $I\subseteq \bbR$, and  $\rho>0$, we define
\begin{align*}
    \Band(\bm,I) & := \lt\{
        \bsig \in S_N : \<\bm,\bsig\>_N \in I
    \rt\}\, ,\\
    \Ball_N(\bm,\rho) &:= \lt\{
        \bx \in \bbR^N : \|\bx-\bm\|_N\le \rho
    \rt\}\, .
\end{align*}
We will occasionally abuse notation and write, for $q\in\bbR$,  $\Band(\bm,q)$
instead of  $\Band(\bm,\{q\})$.

We will often state that certain events occur with probability $1-e^{-cN}$.
When we do, $c>0$ is an unspecified constant, which may change from line to line and may depend on all parameters other than $N$. We use $\plim$ to denote limit in probability.

We write $\bG\sim\GOE(N)$ if $\bG$ is a symmetric matrix with independent centered Gaussian entries on
or above the diagonal with $G_{ii}\sim\cN(0,2/N)$ and $G_{ij}\sim\cN(0,1/N)$ for $i<j$.

Throughout the paper, the mixture $\xi$ is fixed and various constants can depend on $\xi$
but we will track this dependence. If $\iota$ is a small constant, we write $\iota'=o_{\iota}(1)$
if $|\iota'|\le h(\iota)$ where $h$ is a function independent of $N$, such that
$\lim_{\iota\to 0}h(\iota) = 0$.

\section{Main result}

\rev{In this section we describe the mean estimation algorithm (subsection 
\ref{subsec:alg-mean}), state our main result (subsection \ref{subsec:alg-sl}),
define the correction $\bDel(\bm)$ (subsection \ref{sec:CorrectionDef}).
Finally, subsection \ref{subsec:intuition-mean-alg} provides an intuitive
motivation for the mean estimation algorithm.}

Throughout, we assume the model $\xi$ satisfies \eqref{eq:amp-works}.

\subsection{Mean estimation of tilted measure}
\label{subsec:alg-mean}

We first describe the main subroutine of our algorithm, which estimates the mean of the following exponentially tilted version of $\mu_{H_N}$.
For $\by \in \bbR^N$, define
\beq
    \label{eq:tilted-msr}
    \mu_{H_N,\by}(\de \bsig) = \fr{1}{Z(\by)} \exp\lt\{
        H_N(\bsig) + \la \by, \bsig \ra
    \rt\} ~\mu_0(\de \bsig)\, .
\eeq
The tilt $\by$ will be generated by the outer loop of the algorithm described in Subsection~\ref{subsec:alg-sl}, which implements a discretized version of the stochastic localization process.
The outer loop also provides a time $t\ge 0$, which this subroutine will take as input.
The algorithm consists of three steps as outlined below. We defer the
description of the correction $\bcorr(\bm)$ to Section
\ref{sec:CorrectionDef}.
\begin{enumerate}[label=(\arabic*)]
    \item Let $\xi_t(s) = \xi(s) + ts$, and define the sequence $\{q_k : k\ge 0\}$ by $q_0=0$ and
    \beq
        \label{eq:def-q-seq}
        q_{k+1} = \fr{\xi'_t(q_k)}{1+\xi'_t(q_k)}.
    \eeq
    Starting from the initialization $\bm^{-1} = \bw^0 = \bzero$, run the approximate message passing (AMP) iteration
    \balnn
        \label{eq:main-amp}
        \bm^k &= (1-q_k) \bw^k, &
        \bw^{k+1} &= \nabla H_N(\bm^k) + \by - (1-q_k) \xi''(q_k) \bm^{k-1},
    \ealnn
    for $K_\sAMP$ iterations.
    Let $\bm^\sAMP = \bm^{K_{\AMP}}$.
    \item Define
    \beq
        \label{eq:theta}
        \theta(s) = \xi(1) - \xi(s) - (1-s)\xi'(s)
    \eeq
    and the TAP free energy
    \beq
        \label{eq:FTAP}
        \cF_\sTAP(\bm;\by) = H_N(\bm) + \la \by, \bm \ra + \fr{N}{2} \theta(\norm{\bm}_N^2) + \fr{N}{2} \log (1-\norm{\bm}_N^2).
    \eeq
    Starting from $\bm^\sAMP$, run gradient ascent on $\cF_\sTAP(\cdot;\by)$ for $K_{\GD}(N) := \lfloor K^*_\GD \log N\rfloor$ iterations, and let the resulting point be $\bm^\GD$.
    \item  Output $\bm^\alg : = \bm^\GD +\bcorr(\bm^{\GD})$, with
    $\bcorr(\bm)$ defined as in Section \ref{sec:CorrectionDef}.
\end{enumerate}




Pseudocode for the computation of $\bm^{\alg}$ is provided in Algorithm \ref{alg:mean}.

\begin{algorithm}
\DontPrintSemicolon 
\KwIn{$H_N$, $\by \in \bbR^N$, $t \ge 0$. Parameters: $K_\sAMP$, $K_\GD(N)$, $\eta > 0$}
$\bm^{-1} = \bw^0 = \bzero$,\\
For $k=0,\ldots,K_\sAMP$, run iteration \eqref{eq:main-amp}\\
 Let $\bu^0 = \bm^\sAMP = \bm^{K_\sAMP}$\\
\For{$k=0,\ldots,K_\GD(N)-1$} {
\rev{$\bu^{k+1} = \bu^k + \eta \nabla \cF_\sTAP (\bu^k;\by)$}
}
Let $\bm^\GD = \bu^{K_\GD(N)}$ \\
\Return{$\bm^\alg(H_N,\by,t) = \bm^\GD + \bcorr(\bm^{\GD})$}\;
\caption{{\sc Approximate mean computation}}\label{alg:mean}
\end{algorithm}

\subsection{Stochastic localization sampling}
\label{subsec:alg-sl}

We are now in position to describe the sampling algorithm, which uses Algorithm \ref{alg:mean}
as a subroutine.
The main idea is to truncate the diffusion process \eqref{eq:sl-sde}
to the interval $[0,T]$, and to replace it by its Euler discretization
(see Step \ref{eq:Euler} in Algorithm  \ref{alg:main} below).

We will prove that, for $T$ a sufficiently large constant,
the tilted measure of Eq.~\eqref{eq:tilted-msr}, with $\by=\by_T$, is well approximated
by a strongly log-concave measure.
As a consequence, we can sample from it in total variation using standard approaches such as the Metropolis-adjusted Langevin algorithm, or MALA (see \cite{chewi2021optimal} and references therein).
Formally, define
\balnn
    \label{eq:stereographic-proj-inv}
    \bsig_\by(\brho) &= \fr{\hby + \bU \brho}{\sqrt{1 + \norm{\brho}_N^2}}, &
    \hby &= \fr{\by}{\norm{\by}_N},
\ealnn
where $\bU \in \bbR^{N\times (N-1)}$ is an orthonormal basis of the orthogonal complement of $\by$, and
\beq
    \label{eq:def-Hproj}
    H_{N,\by}^{\proj}(\brho)
    = H_{N,\by}(\bsig_\by(\brho))
    - \fr{N}{2} \log(1 + \norm{\brho}_N^2).
\eeq
Note that $\bsig_\by$ is the inverse of the stereographic projection $\bT_\by$ from $S_N \cap \{\bsig: \la \bsig,\by \ra > 0\}$ to the affine plane $\{\hby + \bU \brho : \brho \in \bbR^{N-1}\}$.
We will see (Lemma~\ref{lem:push-forward-cap-law}) that the push-forward of $\mu_{H_N,\by}(\cdot | \la \bsig,\by \ra > 0)$ under $\bT_\by$ is precisely
\beq
    \label{eq:ProjMeasureNoApprox}
    \nu_{H_N,\by}^{\proj}(\de \brho)
    = \fr{1}{\hZ(\by)}
    \exp H_{N,\by}^{\proj}(\brho)
    ~\de \brho.
\eeq
\rev{Fix $\eps_0=0.1$ (any sufficiently small absolute constant would work)} and let $\varphi : [0,+\infty) \to [0,+\infty)$ be a twice continuously differentiable function satisfying $\varphi(x) = 0$ for $x\in [0,\eps_0]$ and
\balnn
    \label{eq:varphi-desiderata}
    \fr{1}{(1+x)^{3/2}} + \varphi'(x) &\ge \eps_0, &
    \fr{1-2x}{(1+x)^{5/2}} + \varphi'(x) + 2x\varphi''(x) &\ge \eps_0
\ealnn
for all $x\ge 0$.
(Existence of such a function is shown in Fact~\ref{fac:varphi-exists}.)
Define the following measure on $\bbR^{N-1}$:
\balnn
    \label{eq:ProjMeasure}
    \tnu_{H_N,\by}^{\proj}(\de \brho)
    &= \fr{1}{\tZ(\by)}
    \exp \tH_{N,\by}^{\proj}(\brho)
    ~\de \brho, &
    \tH_{N,\by}^{\proj}(\brho)
    = H_{N,\by}^{\proj}(\brho) - \fr{TN}{2} \varphi(\norm{\brho}_N^2).
\ealnn
We will show that for sufficiently large $T$, $\tnu_{H_N,\by}^{\proj}$ is strongly log-concave (Proposition~\ref{ppn:tHproj-concave}) and approximates $\nu_{H_N,\by}^{\proj}$ in total variation (Corollary~\ref{cor:tnu-nu-approx}).
Thus, we may sample from it using MALA, and produce samples from $\mu_{H_N,\by}$ by pushing forward through $\bsig_\by$.

\begin{algorithm}
\DontPrintSemicolon %
\KwIn{$H_N$.
Parameters: $K_\sAMP$, $K_\GD(N)$, $\eta$, $T>0$, where $T$ is a multiple of $N^{-4}$}
Set $\delta=N^{-4}$,
$L = T / \delta$\\
Set $\by^0 = \bzero$\\
\For{$\ell = 0, \ldots, L-1$} {
Let $\bm^\ell= \bm^\alg(H_N,\by^{\ell},\ell\delta)$ be the output of Algorithm~\ref{alg:mean} on input $(H_N,\by^\ell,\ell\delta,K_\sAMP,K_\GD(N),\eta)$\\
Draw $\bw^\ell \sim \cN(0,\bI_N)$ independent of everything else\\
 Set $\by^{\ell+1} = \by^\ell + \delta \bm^\ell + \sqrt{\delta} \bw^\ell$\label{eq:Euler}
}
Let $\tnu_{H_N,\by^L}^{\proj}$ be defined according to Eq.~\eqref{eq:ProjMeasure}\\
Use MALA to draw $\brho^{\sMALA}\sim \nu^{\sMALA}$, to accuracy $\TV(\nu^{\sMALA},\tnu_{H_N,\by^L}^{\proj})\le 1/N$\\
\Return{$\bsig_{\by^L}(\brho^{\sMALA})$}\;
\caption{{\sc Sampling}}\label{alg:main}
\end{algorithm}

\rev{We now state our main result. Its proof is presented in 
Section \ref{sec:proof-main-theorem}.}
\begin{thm}
    \label{thm:main}
    Suppose $\xi$ satisfies \eqref{eq:amp-works}.
    There exist constants $K_\sAMP,K_\GD^*,\eta,T$ depending on $\eps$ and $\xi$ such that
    when Algorithm \ref{alg:main} is run with parameters
     $K_\sAMP$, $K_{\GD}(N) = K_\GD^*\log N$, $\eta $, $T$, the following holds.
    With probability $1-o_N(1)$ over $H_N$, $\mu^\alg = \cL(\bsig^\alg)$ satisfies
    \[
        \TV(\mu^\alg,\mu_{H_N}) \le o_N(1).
    \]
    Further the complexity of the algorithm is upper bounded by $CN^4\, (N + \chi_{\nabla H}) \log N+
    \chi_{\mbox{\rm\tiny log-conc}}$, where
    $\chi_{\nabla H}$ is the complexity of evaluating $\nabla H_N(\bm)$ at a point $\bm$ with $\|\bm\|_N \le 1$,
    and $\chi_{\mbox{\rm\tiny log-conc}}$ is the complexity of sampling from a $1$-strongly log-concave
    measure in $N$ dimensions using {\rm MALA} to accuracy $1/N$ in total variation.
\end{thm}
\begin{rmk}
The main result of  \cite{chewi2021optimal}  implies that,  for a `warm start' initialization,
$\chi_{\mbox{\rm\tiny log-conc}}$ is of order $N^{3/2} \log N$. In the present case we do not have a good warm start,
and obtain $\chi_{\mbox{\rm\tiny log-conc}}\le C\cdot N^{5/2}$. We believe this bound is suboptimal, but
made no attempt at improving it.
\end{rmk}

\rev{\begin{rmk}
    Extending our algorithm to Ising models would require two new inputs.
    First, one would need a high-accuracy mean estimator in the Ising setting. 
    A TAP fixed point is still expected to approximate the mean well, and it may be possible to
    obtain an $o_N(1)$-accurate approximation by finding a suitable correction $\bcorr(\bm)$.
    However, our rigorous analysis of the mean estimator uses the
    rotational symmetry of the sphere and does not transfer to the hypercube.

    Second, after simulating stochastic localization up to a sufficiently large
    time $T=O(1)$, one needs a way to sample from the time-$T$ localized measure
    $\mu_{H_N}(\cdot\mid\by_T)$. On the sphere, we use the fact that this
    measure is well approximated by a strongly log-concave measure. On the
    hypercube, we expect Glauber dynamics to mix rapidly for this
    large-external-field target measure, but proving this rigorously remains
    open. The recent work \cite{bandeira2026mixing} may provide a route to such
    a result, although its main result formally concerns a disorder-independent
    external field rather than the field generated by the stochastic
    localization process.
\end{rmk}}

\subsection{The correction $\bcorr(\bm)$}
\label{sec:CorrectionDef}

We now describe the computation of the correction
$\bcorr(\bm)$.
Let $\Ts_{\bm}$ be the $(N-1)$-dimensional subspace orthogonal to $\bm$ and
define $H_N(\,\cdot\, ;\bm):\Ts_{\bm}\to \bbR$ via
$H_N(\bx;\bm) := H_N(\bm+ \bx)$.
We then define the tensors
\begin{align}
\bA^{(2)}(\bm) := \nabla_{\bx}^2  H_N(\bzero;\bm)\, , \;\;\;\;
\bA^{(3)}(\bm) := \nabla_{\bx}^3  H_N(\bzero;\bm)\, .\label{eq:A2A3-Def}
\end{align}
These should be interpreted as tensors $\bA^{(i)}(\bm)\in \Ts^{\otimes i}_{\bm}$.
Let  $\gamma_{*,N}(\bm)$ be the unique solution of
\begin{align}
\begin{cases}\Tr\big((\gamma_{*,N}\bI_{N-1}-\bA^{(2)}(\bm))^{-1}\big)= N \cdot\Big(1-\|\bm\|^2/N\Big)&,\\
\gamma_{*,N}>\lambda_{\max}(\bA^{(2)}(\bm))&\, .
\end{cases}\label{eq:GammaStarCorr}
\end{align}
Here $\bI_{N-1}$ denotes the identity matrix acting on $\Ts_\bm$, and the inverse is over quadratic forms on $\Ts_\bm$.

Then we define
\begin{align}
\corr_i(\bm)
&= \frac{1}{2}\<\bA^{(3)}(\bm),\bQ(\bm)\otimes \bQ(\bm)_{i,\cdot}\>=\frac{1}{2} \sum_{a,b,c = 1}^{N} A^{(3)}_{abc}(\bm)Q_{ia}(\bm)Q_{bc}(\bm)\, \label{eq:def-corr},\\
\bQ(\bm)&: = \big(\gamma_{*,N}(\bm)\bI_{N-1}-\bA^{(2)}(\bm)\big)^{-1}\, .\label{eq:def-Q}
\end{align}

It is useful to make two additional remarks about the evaluation of $\bcorr(\bm)$:
\begin{enumerate}
\item For any fixed $\bm$, $\bA^{(2)}(\bm) \ed \sqrt{\xi''(\tnorm{\bm}_N^2) \cdot \fr{N-1}{N}}\, \bW$, for $\bW\sim\GOE(N-1)$.
It turns out that, although $\bm^\sTAP$ is itself random, this nonetheless gives the correct
asymptotics for $\gamma_{*,N}(\bm^\sTAP)$.
Let $q_\ast = q_\ast(t)$ be the solution to $\fr{q_\ast}{1-q_\ast} = \xi'_t(q_\ast)$, existence and uniqueness of which are shown in Fact~\ref{fac:qt-unique}.
We will show (see Proposition~\ref{ppn:local-concavity-and-conditioning}) that typically $\tnorm{\bm^\sTAP}_N^2 = q_\ast + o_N(1)$, and (see Lemma~\ref{lem:mtap-looks-like-goe}) $\gamma_{\ast,N}(\bm^\sTAP) = \gamma_\ast + o_N(1)$, for $\gamma_\ast = (1-q_\ast)^{-1}+(1-q_\ast)\xi''(q_\ast)$.
For the computation of $\bcorr$, we can replace $\gamma_{*,N}$ by $\gamma_*$ with negligible error.
\item The tensors $\bA^{(2)}(\bm)$ and $\bA^{(3)}(\bm)$ can be written as explicit linear functions of the
couplings $\bg$, and hence can be computed efficiently without need to take any numerical derivative.
\end{enumerate}

\subsection{Motivation for the mean estimation algorithm}
\label{subsec:intuition-mean-alg}

\rev{We now give some intuitive motivation
for the structure of the mean estimation algorithm, 
Algorithm~\ref{alg:mean}.
As mentioned in the introduction, critical points of the 
TAP free energy are widely believed to yield an approximation
of the magnetization vector. In particular, in the replica symmetric
phase, only one such local maximum $\bm^{\sTAP}(\by,t)$ is believed to exist in 
a suitable restricted region.

We will indeed prove that such a TAP local maximum 
exists, is unique, and provides a rather precise approximation of the 
magnetization, namely 
$\|\bm^{\sTAP}(\by,t)-\bm(\by,t)\|=O(1)$.
However, approximating $\bm^{\sTAP}(\by,t)$ is algorithmically
non-trivial since the TAP free energy 
$\cF_{\sTAP}$ is non-concave.

We use a two-phase algorithm to overcome this issue.
In the first phase, AMP produces a vector
$\bm^{\sAMP}$ such that 
$\|\bm^{\sAMP}(\by,t)-\bm^{\sTAP}(\by,t)\|$ is small.
A general class of AMP algorithms for tensor data (the tensors of
couplings $(G_{i_1,\dots,i_p})$) was introduced and studied in
\cite{el2021optimization}, and the update of 
Eq.~\eqref{eq:main-amp} belongs to this class. 
The specific construction of the update in 
Eq.~\eqref{eq:main-amp} is dictated by the condition that $\bm_*$ is a fixed point of
Eq.~\eqref{eq:main-amp} if and only if $\nabla \cF_{\sTAP}(\bm_*)=\bzero$.
Similar iterations and convergence to TAP solutions were studied in
\cite{bolthausen2014iterative,fan2021tap,
alaoui2023sampling}.

The state evolution analysis of AMP only enables us to characterize
$O(1)$ iterations, and hence prove 
$\|\bm^{\sAMP}(\by,t)-\bm^{\sTAP}(\by,t)\|\le \eps\sqrt{N}$ 
with high probability for any fixed $\eps>0$.   To overcome this issue,
we prove that $\cF_{\sTAP}$ is strongly concave in a neighborhood of 
$\bm^{\sTAP}$ and hence the gradient ascent phase
in Algorithm \ref{alg:mean} converges exponentially fast:
$\|\bu^k-\bm^{\sTAP}(\by,t)\|\le c_0\sqrt{N}\, e^{-c_1k}$ for some 
constants $c_0,c_1$, with $c_1$ bounded away from zero, and all $k\ge 0$.

Finally, we are left with the task of explaining the 
correction $\bcorr(\bm)$ whose form is given in the last section, 
Section \ref{sec:CorrectionDef}. A key estimate will show that the Gibbs measure
concentrates exponentially in a small spherical band 
$\Band(\bm^{\sTAP},I(\iota))$, where
$I(\iota):=[q_\ast-\iota,q_\ast+\iota]$ and $\iota>0$ is a
sufficiently small constant; see
Proposition~\ref{ppn:additional-event-pre-conditioning}\ref{itm:crit-band}.
To describe a slice of this band, set
$q_{\bm}:=\|\bm^{\sTAP}\|_N^2$. The Gibbs measure restricted to the
$(N-2)$-dimensional sphere
$\langle\bsig,\bm^{\sTAP}\rangle_N=q_{\bm}+r$, $|r|\leq2\iota$,
can be parametrized as
\[
    \bsig=a\bm^{\sTAP}+b\bU\btau,\qquad
    a:=1+\frac{r}{q_{\bm}},\qquad
    b:=\sqrt{1-a^2q_{\bm}},
\]
where $\bU\in\bbR^{N\times(N-1)}$ has orthonormal columns spanning
$(\bm^{\sTAP})^\perp$ and
$\btau\in\bbR^{N-1}$ satisfies $\|\btau\|^2=N$.
On this sphere the tilted Hamiltonian becomes
\[
    \tH_r(\btau):=H_N(a\bm^{\sTAP}+b\bU\btau)
      +\langle\by,a\bm^{\sTAP}+b\bU\btau\rangle,
\]
which is again a spherical mixed $p$-spin Hamiltonian.
The correction $\bcorr(\bm^{\sTAP})$ approximates the
magnetization of the Hamiltonian $\tH_r(\,\cdot\,)$
(integrated over $r$).

At first sight this might seem as hard as the original problem.
The crucial remark is that the linear component
(`magnetic field') in $\tH_r$ vanishes at $r=0$
(because $\nabla \cF_{\sTAP}(\bm^{\sTAP})=\bzero$
implies $\nabla H_N(\bm^{\sTAP})+\by= c_* \bm^{\sTAP}$ for some scalar $c_*$),
and is of order 
$\|\bU^\top[\nabla H_N(a\bm^{\sTAP})+\by]\|
=O(|r|\sqrt{N})$
for $|r|\le 2\iota$. Neglecting the magnetic field, 
we use a second moment argument to prove that
we can replace all terms of degree larger than $3$ 
in $\tH_r$ by a constant.
We therefore have to compute the magnetization of a 
spherical spin glass Hamiltonian with only quadratic and 
cubic terms (in its replica-symmetric phase),
which we call $\tH^{(2)}$ and $\tH^{(3)}$. If the term
$\tH^{(3)}$ were absent, the magnetization would vanish by symmetry.
The correction $\bcorr(\bm^{\sTAP})$ is precisely
the first-order term in the Taylor expansion of
the magnetization around $\tH^{(3)}=0$.
This only requires taking expectations with respect to
the Gibbs measure for $\tH^{(2)}$, which can be done
 explicitly. It is far from obvious that this crude approximation
yields the accurate estimate we need, but we will prove
this is the case.
}

\subsection{Fundamental limits of algorithmic SL, replica symmetry breaking, and
shattering}

It is useful to compare condition \eqref{eq:amp-works} with the
condition for (absence of) shattering, and replica symmetry breaking:
\begin{itemize}
\item As mentioned above (cf. Eq.~\eqref{eq:FirstShatteringEq}), it is conjectured 
\cite{crisanti1993spherical,crisanti1995thouless,bouchaud1998out} that shattering is absent if and only if 
\begin{align}
\xi'(q) < \frac{q}{1-q}\, ,\;\;\; \forall q\in (0,1)\, .\label{eq:Shattering1}
\end{align}
\item  The tight condition for replica symmetry was identified in \cite[Proposition 2.3]{talagrand2006spherical}.
\beq
    \label{eq:RS-tal06}
    \xi(q) + q + \log(1-q) \le 0, \qquad \forall q\in [0,1)
\eeq
\end{itemize}
\rev{Condition~\eqref{eq:Shattering1} is implied by the condition under which
our algorithm succeeds, namely Eq.~\eqref{eq:amp-works}, by integrating once.
In turn, condition~\eqref{eq:RS-tal06} holds under
Eq.~\eqref{eq:Shattering1} by integrating once, and hence under
Eq.~\eqref{eq:amp-works}.}
In this section, we prove that the condition  \eqref{eq:amp-works}
is necessary not only for Algorithm \ref{alg:main} to succeed,
but indeed for a broader class of stochastic localization
schemes that we next introduce. This points at a fundamental gap
between such schemes and the possible computational limit for sampling, a fact that was suggested in \cite{ghio2023sampling} and, in a related context,
in \cite{montanari2007solving}.

By the key remark below \eqref{eq:sl-sde}, the process $\by_t$ generated by \eqref{eq:sl-sde} consists of observations of some $\bsig \sim \mu_{H_N}$ through a progressively less noisy Gaussian channel.
A natural generalization of this process outputs observations of $\bsig, \bsig^{\otimes 2}, \bsig^{\otimes 3}, \ldots$ through Gaussian channels of varying signal strengths, and can similarly be converted to a sampling algorithm.

Consider any $J\in \bbN$ and any continuously differentiable, coordinate-wise increasing $\tau : [0,+\infty) \to [0,+\infty)^J$, normalized to $\tnorm{\tau(t)}_1 = t$ for all $t \in [0,+\infty)$, and such that $\lim_{t\to\infty} \tau_j(t) = \infty$ for at least one odd $j \le J$.
For each $j \le J$, let $(\bB^j_t)_{t\ge 0}$ be a standard Brownian motion in $(\bbR^N)^{\otimes j}$.
Let $(\vby_t)_{t\ge 0} = (\by^1_t,\ldots,\by^J_t) \in \bbR^N \times \cdots \times (\bbR^N)^{\otimes J}$ be given by the Ito diffusion
\beq
    \label{eq:gen-sl-sde}
    \de \by^j_t = \tau'_j(t) \bm_j(\vby_t,t)~\de t + \tau'_j(t)^{1/2} ~\de \bB^j_t, \qquad \vby_0 = \bzero,
\eeq
where, with expectation over $\bsig \sim \mu_{H_N}$ and $\bG^j \sim \cN(0,\bI_N^{\otimes j})$,
\begin{align}
    \bm_j(\vby_t,t) = \bbE[\bsig^{\otimes j} | \tau_i(t) \bsig^{\otimes i} + \tau_i(t)^{1/2} \bG^i = \by^i_t, \forall 1\le i \le J].\label{eq:Mj_def}
\end{align}
The process \eqref{eq:sl-sde} corresponds to the case $J = 1$. As
in that case, a sampling algorithm can be constructed from Eq.~\eqref{eq:gen-sl-sde}
by discretizing time and approximating the \rev{calculation of} $\bm_j(\vby_t,t)$
(see Remark \ref{rmk:Generalized-SL} below).

For $\bA \in (\bbR^N)^{\otimes j}$ and $1\le \ell\le j$, let $\bA^{(\ell)}$ be the tensor obtained by \rev{cyclically rotating the coordinates by $\ell$ positions}, that is
\[
    \bA^{(\ell)}_{i_1,\ldots,i_j} = \bA_{i_{\ell+1},\ldots,i_j,i_1,\ldots,i_\ell}.
\]
Then, for $\bB \in (\bbR^N)^{\otimes j-1}$, let $(\bA, \bB)_\sym \in \bbR^N$ be the vector satisfying
\[
    \la \bv, (\bA, \bB)_\sym \ra
    = \sum_{\ell=1}^j \la \bB \otimes \bv, \bA^{(\ell)} \ra
\]
for all $\bv \in \bbR^N$.
Let
\beq
    \label{eq:def-cxi-t}
    \rev{\cxi_t(s) = \xi(s) + \sum_{j=1}^J \tau_j(t) s^j}
\eeq
and define the sequence $\{\cq_k : k\ge 0\}$ by $q_0 = 0$ and
\beq
    \label{eq:def-cq}
    \cq_{k+1} = \fr{\cxi'_t(q_k)}{1 + \cxi'_t(q_k)}
\eeq
Finally define an AMP iteration analogous to \eqref{eq:main-amp} by
\balnn
    \label{eq:gen-amp}
    \cbm^k &= (1 - \cq_j) \cbw^k, &
    \cbw^{k+1} &= \nabla H_N(\cbm^k) + \sum_{j=1}^J \fr{1}{N^{j-1}} ((\cbm^k)^{\otimes j-1}, \by^j_t)_\sym - (1-\cq_k)\cxi''(\cq_k) \cbm^{k-1}.
\ealnn
The next theorem is proved in Section~\ref{sec:sl-hardness}, under the following condition which is a strict form of \eqref{eq:RS-tal06}.
\beq
    \label{eq:replica-symmetry}
    \xi''(0) < 1, \qquad
    \xi(q) + q + \log(1-q) < 0, \quad \forall q\in (0,1).
\eeq
\rev{This is a strict version of \eqref{eq:RS-tal06} in the following sense: differentiating \eqref{eq:RS-tal06} twice at $q=0$ gives $\xi''(0) \le 1$. Thus \eqref{eq:replica-symmetry} rules out equality in \eqref{eq:RS-tal06} for all $q\in (0,1)$, and in its second derivative at $q=0$. This strengthened hypothesis is also called the strictly replica-symmetric condition, and is also used in e.g. \cite[Equation~(2.7)]{huang2023constructive}.}
%
\begin{thm}
    \label{thm:sl-hardness}
    Suppose that \eqref{eq:replica-symmetry} holds and
    that there exists $q\in [0,1)$ such that $\xi''(q) > \fr{1}{(1-q)^2}$.
    There exists a positive-measure set $\cI \subseteq [0,+\infty)$ such that for all $t\in \cI$ the following holds.
    There exists $1\le j\le J$ such that $\tau'_j(t) > 0$ and, for $\vby_t$ generated from \eqref{eq:gen-sl-sde},
    \[
        \lim_{k\to\infty} \liminf_{N\to\infty}
        \bbE \fr{1}{N^j}
        \norm{(\cbm^k)^{\otimes j} - \bm_j(\vby_t,t)}_2^2
        > 0.
    \]
\end{thm}

\begin{rmk}
\label{rmk:rs-artifact}
In this theorem we assume Eq.~\eqref{eq:replica-symmetry} to hold, but note that this
is an artifact of our proof technique. Indeed efficient sampling is believed to be impossible
beyond the threshold \eqref{eq:replica-symmetry}. Indeed \cite{alaoui2023shattering}
implies that `stable' algorithms fail under replica symmetry breaking.
\end{rmk}

\begin{rmk}\label{rmk:Generalized-SL}
    As alluded to above, we can define a natural analog of Algorithm~\ref{alg:mean} for this generalized setting, which computes an estimator $\cbm^\alg$ for $\bm_1(\vby_t,t)$.
    For some $K_\sAMP \in \bbN$, the point $\cbm^{K_\sAMP}$ is the result of the first phase of this algorithm.
    The output $\cbm^\alg$ of this algorithm satisfies $\tnorm{\cbm^{K_\sAMP} - \cbm^\alg}_N \to 0$ as $K_\sAMP \to \infty$; see Theorem~\ref{thm:mean} and Proposition~\ref{ppn:local-concavity-and-conditioning} below, which show this for Algorithm~\ref{alg:mean} when \eqref{eq:amp-works} holds.

    The analog of Algorithm~\ref{alg:main} simulates the SDE \eqref{eq:gen-sl-sde} via an Euler discretization, estimating each $\bm_j(\vby_t,t)$ with $(\cbm^\alg)^{\otimes j}$.
    Theorem~\ref{thm:sl-hardness} shows that for an interval of $t$ of positive measure, this algorithm fails for a tensor order $j$ relevant to the Euler discretization.
    \rev{The relevance of tensor orders $j>1$ is that
    $\tau'_j(t)\bm_j(\vby_t,t)$ is precisely the drift in the $j$th channel of
    \eqref{eq:gen-sl-sde}. Hence, whenever $\tau'_j(t)>0$, a nonvanishing error
    in approximating this moment prevents the AMP-based scheme from consistently
    approximating the drift of the generalized stochastic-localization process,
    even if its estimate of the first moment is accurate. In particular, adding
    higher-order observation channels does not circumvent the obstruction
    identified by the theorem.}
\end{rmk}

\section{Preliminaries}

In this section we provide further background.
The contents of Subsections~\ref{subsec:sl} and \ref{subsec:planted} are known and we often refer to \cite[Sections 3 and 4.1]{alaoui2022sampling} for proofs.
Subsection~\ref{subsec:conditioning-lemma} introduces a lemma about conditioning a Gaussian process on a random vector: this is a fairly standard but crucial technical tool.

\subsection{Stochastic localization}
\label{subsec:sl}

Fix a realization of $H_N$. The stochastic localization process is defined
by the SDE \eqref{eq:sl-sde}, which has a unique strong solution provided $\by \mapsto \bm(\by,t)$ is
Lipschitz continuous.
Note that, for $\mu_{H_N,\by_t}$ as in \eqref{eq:tilted-msr},
$\bm$ is the mean
\[
    \bm(\by,t) = \int \bsig ~\mu_{H_N,\by}(\de \bsig).
\]
Therefore Lipschitz continuity is implied
by $\sup_{\by}\|\Cov(\mu_{H_N,\by})\|_{\op}<\infty$, which always holds since
$\mu_{H_N,\by}$ is supported on a compact set.

As already mentioned in the introduction, we have the following facts
(see for instance \cite{alaoui2022information}).
\begin{ppn} \label{ppn:BasicSL}
Let  $(\by_t)_{t\ge 0}$ be the unique solution of the SDE \eqref{eq:sl-sde}.
Then there exists a standard Brownian motion $\bB'_t$ independent
of $\bsig\sim \mu_{H_N}$, such that, for all $t$, $\by_t=t\bsig+\bB'_t$.

Further,
$\bbE\Cov(\mu_{H_N,\by_t})\preceq \bI_N\,/\, t$.
In particular $\mu_{H_N,\by_t}\Rightarrow \delta_{\bsig}$ almost surely
as $t\to\infty$.
\end{ppn}

\subsection{Planted model and contiguity}
\label{subsec:planted}

Recall that $\mu_0$ denotes the uniform probability measure on $S_N$. 
Further, let $\sH_N$ be the space of Hamiltonians $H_N$
(i.e., continuous functions $H_N:S_N\to \bbR$ endowed with the
uniform convergence topology and the induced Borel sigma-algebra)
and $\mu_\nul \in \cP(\sH_N)$ be the law induced on $H_N$ by Eq.~\eqref{eq:def-HN}.
Define the planted measure $\mu_\pl \in \cP(S_N \times \sH_N)$ by
\[
    \mu_\pl(\de \bx, \de H_N)
    := \fr{1}{Z_\pl} \exp \big\{H_N(\bx)\big\} ~\de \mu_0(\bx) \de \mu_\nul(H_N).
\]
For $H_N \in \sH_N$, define the partition function
\[
    Z(H_N) := \int \exp \big\{H_N(\bsig)\big\} ~\mu_0(\de \bsig)\, .
\]
\rev{
The joint distribution of $(\bsig,H_N) \sim \mu_\pl$ can be described in
two equivalent ways. In the first one, we generate first $H_N$ and then
$\bsig$ conditional on $H_N$:
\balnn
    \label{eq:planted-HN-first}
    H_N&\sim \mu_\pl(\de H_N)
    = \fr{Z(H_N)}{\bbE Z(H_N)} \mu_{\nul}(\de H_N)\, ,
    \qquad \bsig\sim\mu_{H_N}\, .
\ealnn
In the second, we generate first $\bsig$ and then $H_N$:
\balnn
    \label{eq:planted-bsig-first}
    \bsig&\sim\mu_0\, , \qquad
    H_N\sim \mu_\pl(\de H_N\mid \bsig)
    \propto e^{H_N(\bsig)}\mu_{\nul}(\de H_N).
\ealnn
}
\begin{lem}[Proved in Section~\ref{sec:partition-fn-fluctuations}]
    \label{lem:partition-fn-fluctuations}
    Suppose $\xi$ satisfies \eqref{eq:replica-symmetry}.
    Let $W \sim \cN(-\fr12 \sigma^2, \sigma^2)$, where $\sigma^2 = -\fr12 \log(1-\xi''(0))$.
    As $N\to\infty$, for $H_N \sim \mu_\nul$, the Radon--Nikodym derivative
    of $\mu_{\pl}$ with respect to $\mu_{\nul}$ is
    \[
        \frac{\de\mu_{\pl}}{\de\mu_{\nul}}
        (H_N) = \fr{Z(H_N)}{\EE Z(H_N)} \stackrel{d}{\to} \exp(W).
    \]
\end{lem}
\rev{In the SK and spherical SK models, the analogous lognormal limit theorem was proved in the earlier works \cite{aizenman1987some,baik2016fluctuations}. However, for spherical mixed models with higher-degree interactions, proving this order-one lognormal limit throughout the (strictly) RS regime \eqref{eq:replica-symmetry} is more challenging, and involves a truncated second moment idea introduced in \cite{huang2023constructive} to control the fluctuations due to degrees $3$ and higher. See also \cite{bovier2002fluctuations,bovier2024fluctuations,subag2017geometry,benarous2020geometry,subag2023secondmoment} for related free energy fluctuation results.}
\begin{rmk}
    In most of this paper, we are interested in $\xi$ satisfying the condition \eqref{eq:amp-works}, which implies \eqref{eq:replica-symmetry} by integrating twice.
    However, the proof of Theorem~\ref{thm:sl-hardness} in Section~\ref{sec:sl-hardness} only assumes $\xi$ satisfies \eqref{eq:replica-symmetry}, so we state this lemma with the more general condition.
\end{rmk}
For any $T>0$, let $\bbP, \bbQ \in \cP(S_N\times\sH_N \times C([0,T], \bbR^N))$ be the laws of $(\bsig, H_N,(\by_t)_{t\in [0,T]})$,
generated as follows.
\begin{itemize}
    \item Under $\bbQ$:
    \begin{align}
    H_N \sim \mu_\nul,\;\;\;\; \bsig\sim\mu_{H_N},\;\;\;\;
    \by_t = t\bsig+\bB_t\, ,
    \end{align}
    for $\bB_t$ a standard Brownian motion independent of
    $\bsig$ and $H_N$. By Proposition \ref{ppn:BasicSL}, an equivalent description
    of this distribution is:  $H_N \sim \mu_\nul$,
    $(\by_t)_{t\ge 0}$ given by the SDE \eqref{eq:sl-sde}
    and $\bsig = \lim_{t\to\infty}\by_t/t$.
    \item Under $\bbP$:
    \begin{align}
    \rev{(\bsig,H_N)} \sim \mu_\pl,\;\;\;\;\;\;
    \by_t = t\bsig+\bB_t\, , \label{eq:planted-process}
    \end{align}
    for $\bB_t$ a standard Brownian motion independent of
    $\bsig$ and $H_N$. As before, we can equivalently generate first $H_N$, then $(\by_t)_{t\ge 0}$ given by the SDE \eqref{eq:sl-sde} and finally $\bsig$.
\end{itemize}
A short calculation shows that $H_N \sim \mu_\pl(\cdot | \bx)$
 is given by \rev{the function $\bsig\mapsto H_N(\bsig)$
 defined by}
\beq
    \label{eq:planted-H-law}
    H_N(\bsig) = N\xi(\<\bx,\bsig\>_N) + \tH_N(\bsig),
\eeq
where $\tH_N \sim \mu_\nul$.
The above definition has the following immediate consequence.
\begin{ppn}[{\cite[Proposition 4.2]{alaoui2022sampling}}]
    \label{ppn:planted-likelihood-ratio}
    For all $T\ge 0$,
    \[
        \fr{\de \bbP}{\de \bbQ}(\bsig, H_N,(\by_t)_{t\in [0,T]}) = \fr{Z(H_N)}{\EE Z(H_N)}.
    \]
\end{ppn}
As a consequence of Lemma~\ref{lem:partition-fn-fluctuations} and Proposition~\ref{ppn:planted-likelihood-ratio}, Le Cam's first lemma implies the following.
\begin{cor}
    \label{cor:contiguous}
    The measures $\bbP$ and $\bbQ$ are mutually contiguous.
    That is, for any sequence of events $\cE_N$, $\bbP(\cE_N) \to 0$ if and only if $\bbQ(\cE_N) \to 0$.
\end{cor}
Thus it suffices to analyze our algorithm under the planted distribution $\bbP$.



\subsection{Basic regularity estimate}


For a tensor $\bA \in (\bbR^N)^{\otimes k}$, define the operator norm
\[
    \norm{\bA}_{\op,N}
    = \fr{1}{N} \sup_{\tnorm{\bsig^1}_N,\ldots,\tnorm{\bsig^k}_N \le 1}
    |\la \bA, \bsig^1 \otimes \cdots \otimes \bsig^k \ra|.
\]
Notice that this normalization is different from the standard injective norm $\norm{\cdot}_{\inj}$ in that $\norm{\bA}_{\op,N} = N^{(k-2)/2} \norm{\bA}_{\inj}$.
\rev{In particular, for $k=2$, it agrees with the usual matrix operator
norm: $\norm{\bA}_{\op,N}=\norm{\bA}_{\op}$.}

\begin{ppn}[{\cite[Proposition 2.3]{huang2022tight}}]
    \label{ppn:gradients-bounded}
    There exists a sequence of constants $(C_k)_{k\ge 0}$ independent of $N$ for which the following holds.
    Define the event
    \[
        K_N := \lt\{
            \sup_{\norm{\bsig}_N \le 1}
            \tnorm{\nabla^k H_N(\bsig)}_{\op,N} \le C_k \;\;
            \forall k\ge 0
        \rt\}.
    \]
    Then $\PP(K_N) \ge 1-e^{-cN}$.
\end{ppn}

\subsection{Conditioning lemma}
\label{subsec:conditioning-lemma}


\begin{lem}
    \label{lem:conditioning-lemma-v2}
    Let $D \subseteq \bbR^N$ be an open set and $\cF : D \to \bbR$ be a (not necessarily centered) $C^2$ Gaussian process on a probability space $(\Omega,\Sigma,\bbP)$.
    Let $X$ be a random variable on $(\Omega,\Sigma)$ taking values in $[0,1]$, and $\bm_0$ be a random vector on the same space taking values in $\bbR^N$.
    For $\eps, c_\spec, c_{\op} > 0$ satisfying $\eps\le c_\spec^2/10c_{\op}$, define $U_{\bm_0} := \Ball_N(\bm_0,5\eps / c_\spec)$ and the events
    \baln
        \cG(\eps,c_\spec) &:=\Big\{
            \|\nabla \cF(\bm_0)\|_N \le \eps\, ,\;\;\;
            \nabla^2 \cF(\bm_0)\preceq -c_\spec\bI_n
        \Big\}\, ,\\
        \cH(c_{\op}) &:=\Big\{
            \sup_{\bm\in D} \|\nabla^2 \cF(\bm)\|_{\op,N}
            \le c_{\op},\;\;\;
            \sup_{\bm\in D} \|\nabla^3 \cF(\bm)\|_{\op,N}
            \le c_{\op}
        \Big\}\, ,\\
        \cE_{\cond} &:= \cG(\eps,c_\spec) \cap \cH(c_{\op}) \cap \{\norm{\bm_0}_N \le 1\} \cap \{U_{\bm_0} \subseteq D\}\,.
    \ealn
    Finally, assume $\bm \mapsto \EE \nabla \cF(\bm)$ is continuous and $\lambda_{\min}(\Cov(\nabla \cF(\bm)))$ is bounded away from $0$ uniformly over $\bm \in D$.
    Then, with $\varphi_{\nabla \cF(\bm)}$ the probability density of $\nabla \cF(\bm)$ w.r.t. Lebesgue measure on $\bbR^N$ and $\de^N$ denoting integration against this measure,
    \[
        \bbE(X \ind\{\cE_{\cond}\})
        = \int_D \EE \lt[
            |\det \nabla^2 \cF(\bm)|
            X\ind\{\cE_{\cond} \cap \{\bm\in U_{\bm_0}\}\}
            \big| \nabla\cF(\bm)=\bzero
        \rt]\,
        \varphi_{\nabla \cF(\bm)}(\bzero)~\de^N \bm.
    \]
\end{lem}
\begin{proof}
    On the event $\cE_{\cond}$, for all $\bm \in U_{\bm_0}$ we have
    \beq
        \label{eq:conditioning-lem-strong-concavity}
        \lambda_{\max}(\nabla^2\cF(\bm))
        \le \lambda_{\max}(\nabla^2\cF(\bm_0)) + c_{\op} \norm{\bm - \bm_0}_N
        \le -c_\spec + \fr{5\eps c_{\op}}{c_\spec} \le -\fr12 c_\spec.
    \eeq
    Since $\|\nabla \cF(\bm_0)\|_N \le \eps$, there is exactly one solution to $\nabla \cF(\bm_\ast) = \bzero$ in $U_{\bm_0}$, which is measurable on $(\Omega,\Sigma)$ and furthermore lies in $\Ball_N(\bm_0,4\eps/c_\spec)$.
    The strong concavity \eqref{eq:conditioning-lem-strong-concavity} implies that $\nabla \cF$ is injective on $U_{\bm_0}$ and its image contains a neighborhood of $\bzero$.
    By the area formula, for sufficiently small $\iota > 0$,
    \[
        1 = \fr{1}{|\Ball_N(\bzero,\iota)|} \int_{U_{\bm_0}} |\det \nabla^2 \cF(\bm)| \ind\{\norm{\nabla \cF(\bm)}_N \le \iota\} ~\de^N \bm.
    \]
    Multiplying by $X\ind\{\cE_{\cond}\}$, taking expectations and using Fubini's theorem yields
    \baln
        &\EE( X \ind\{\cE_{\cond}\}) \\
        &= \fr{1}{|\Ball_N(\bzero,\iota)|}
        \int_D
        \EE \lt[|\det \nabla^2 \cF(\bm)| X \ind\{\cE_{\cond} \cap \{\bm \in U_{\bm_0}\} \cap \{\norm{\nabla \cF(\bm)}_N \le \iota\}\}\rt]
        ~\de^N \bm \\
        &= \int_D
        \EE \lt[
            |\det \nabla^2 \cF(\bm)|
            X \ind\{\cE_{\cond} \cap \{\bm \in U_{\bm_0}\}\}
            \big|
            \norm{\nabla \cF(\bm)}_N \le \iota
        \rt]
        \fr{\PP(\norm{\nabla \cF(\bm)}_N \le \iota)}{|\Ball_N(\bzero,\iota)|}
        ~\de^N \bm.
    \ealn
    Note that on $\cE_{\cond}$, $|\det \nabla^2 \cF(\bm)| \le c_{\op}^N$.
    Since $\cE_{\cond}$ is contained in the event $\norm{\bm_0}_N \le 1$, $\{\bm \in U_{\bm_0}\}$ can only occur for $\bm$ in a bounded set.
    Since $\lambda_{\min}(\Cov(\nabla \cF(\bm)))$ is bounded away from $0$, $\varphi_{\nabla \cF(\bm)}$ is bounded, and thus so is $\PP(\norm{\nabla \cF(\bm)}_N \le \iota) / |\Ball_N(\bzero,\iota)|$.
    Therefore the integral in the last display is dominated by a bounded integrable function.
    Continuity of $\EE \nabla \cF(\bm)$ implies that $\varphi_{\nabla \cF(\bm)}(\bz)$ is continuous in $\bz$ in a neighborhood of $\bzero$.
    We take the $\iota \to 0$ limit of the last display by dominated convergence to conclude.
\end{proof}

\section{Proof of Theorem \ref{thm:main}}
\label{sec:proof-main-theorem}

\rev{
We can now state our key technical results, 
and show how they are used to prove our main theorem,
Theorem \ref{thm:main}.}

\subsection{Statement of main technical results and
proof of  Theorem \ref{thm:main}}

\rev{Note that Algorithm~\ref{alg:main} decomposes into two parts.
The first one (lines 1-7) implements an Euler
discretization of the stochastic localization process
of Eq.~\eqref{eq:sl-sde}.
The second one (lines 8-14) uses the MALA 
algorithm
to sample from the measure 
$\tnu_{H_N,\by_{L}}^{\proj}$ defined in
 Eq.~\eqref{eq:ProjMeasure}, which in turn
is a good approximation of the tilted measure
$\mu_{H_N,\by_{T}}$, $T=L\delta$, modulo the stereographic
projection.

Our first statement shows that the output
of the first part of the algorithm
is close in total variation to the stochastic
localization process at time $T=L\delta$.
\begin{ppn}
    \label{ppn:yT-close-in-TV}
    Let $(H_N,\by_{T})$ be sampled from the marginal of the planted distribution $\bbP$ (as defined in Eq.~\eqref{eq:planted-process}).
    Let $\by^L$ be generated as in Algorithm~\ref{alg:main}, run on input $H_N$.
    Then,
    \[
        \bbE_{H_N} \TV\lt(
            \cL(\by_{T} | H_N),
            \cL(\by^L | H_N)
        \rt) = o_N(1).
    \]
\end{ppn}
Establishing this proposition is the main technical challenge
addressed in the rest of the paper. We provide a
proof sketch in Section \ref{sec:sketch-yT-close-in-TV}.
A formal proof, using results established in other sections,
is given in Section \ref{sec:Proof-Euler-Closeness}.

Our second proposition shows that the
error incurred by using MALA 
on $\tnu_{H_N,\by_{T}}^{\proj}$ instead of 
 sampling from $\mu_{H_N,\by_{T}}$ is small.
\begin{ppn}
    \label{ppn:mala-accuracy}
    Let \rev{$(\bsig,H_N,\by_{T})$} be sampled from the marginal of the planted distribution $\bbP$.
    Let $\brho^\sMALA$ be the (random) output of {\rm MALA} run on $\tnu_{H_N,\by_{T}}^{\proj}$ (recall Eq.~\eqref{eq:ProjMeasure}) and $\hbsig = \bsig_{\by_{T}}(\brho^\sMALA)$ (recall Eq.~\eqref{eq:stereographic-proj-inv}).
    Then,
    \[
        \bbE_{H_N,\by_{T}} \TV\lt(
            \cL(\bsig | H_N,\by_{T}),
            \cL(\hbsig | H_N,\by_{T})
        \rt) = o_N(1).
    \]
\end{ppn}
This proposition follows from the 
fact that, for $T$ a large  constant, the exponential
tilt in Eq.~\eqref{eq:tilted-msr} dominates the 
Hamiltonian and hence the measure $\mu_{H_N,\by_{T}}$ in a neighborhood of radius 
$O(\sqrt{N/T})$ of $\by_{T}/\|\by_{T}\|_N$.
We will show that, as a consequence, the local approximation
given by $\tnu_{H_N,\by_{T}}^{\proj}$ is accurate and 
strongly log-concave. The latter property implies
that it can be sampled efficiently using MALA (or 
similar algorithms). We refer to Section
\ref{sec:MALA-Accuracy} for a complete proof.

\begin{proof}[Proof of Theorem~\ref{thm:main}]
    Let $\cK : \sH_N \times \bbR^N \to \bbR^N$ be the random map that, given input $(H_N,\by)$, generates $\brho^\sMALA$ by running MALA on $\nu_{H_N,\by}$ and outputs $\hbsig = \bsig_\by(\brho^\sMALA)$.
    Let \rev{$(\bsig,H_N,\by_{T})$} be sampled from the marginal of $\bbP$.
    Let $\bbP_{\alg,H_N}$ denote the law of the output of $\by^L$ generated by Algorithm~\ref{alg:main} on input $H_N$.
    Then,
    \baln
        \bbE_{H_N \sim \bbP} \TV(\mu_{H_N},\mu^\alg)
        &= \bbE_{H_N \sim \bbP} \TV\lt(
            \bbE_{\by_{T} \sim \bbP(\cdot | H_N)}
            \cL(\bsig | H_N, \by_{T}),
            \bbE_{\by^L \sim \bbP_{\alg,H_N}}
            \cL(\cK(H_N,\by^L))
        \rt) \\
        &\le \bbE_{H_N \sim \bbP} \TV\lt(
            \bbE_{\by_{T} \sim \bbP(\cdot | H_N)}
            \cL(\bsig | H_N, \by_{T}),
            \bbE_{\by_{T} \sim \bbP(\cdot | H_N)}
            \cL(\cK(H_N,\by_{T}))
        \rt) \\
        &\qquad + \bbE_{H_N \sim \bbP} \TV\lt(
            \bbE_{\by_{T} \sim \bbP(\cdot | H_N)}
            \cL(\cK(H_N,\by_{T})),
            \bbE_{\by^L \sim \bbP_{\alg,H_N}}
            \cL(\cK(H_N,\by^L))
        \rt) \\
        &\le \bbE_{(H_N,\by_{T}) \sim \bbP} \TV\lt(
            \cL(\bsig | H_N, \by_{T}),
            \cL(\cK(H_N,\by_{T}))
        \rt) \\
        &\qquad + \bbE_{H_N \sim \bbP} \TV\lt(
            \cL(\by_{T} | H_N),
            \cL(\by^L | H_N)
        \rt).
    \ealn
    \rev{The last inequality follows from the data processing inequality for total variation distance.}
    By Propositions~\ref{ppn:yT-close-in-TV} and \ref{ppn:mala-accuracy}, the final bound is $o_N(1)$.
    Thus, with probability $1-o_N(1)$ over $H_N \sim \bbP$, $\TV(\mu_{H_N},\mu^\alg) = o_N(1)$.
    By Corollary~\ref{cor:contiguous}, the same is true for $H_N \sim \bbQ$.
\end{proof}
}

\subsection{Proof sketch of Proposition \ref{ppn:yT-close-in-TV}}
\label{sec:sketch-yT-close-in-TV}

\rev{As mentioned above, the proof of 
Proposition \ref{ppn:yT-close-in-TV} constitutes the majority
of the technical work in this paper, with some important 
estimates carried out in the companion paper 
\cite{huang2026approximating}. 

At a high level, the proof comprises the following three steps.

\paragraph{1. Reduction to magnetization computation.}
We use Pinsker's inequality to bound the total variation
distance in Proposition \ref{ppn:yT-close-in-TV} 
by the KL divergence. By a standard application of 
Girsanov's theorem, the latter is bounded as 
(we use the shorthand 
$\hbm(\by,t)=\bm^{\alg}(H_N,\by,t)$):
\begin{align}
    \bbE_{H_N \sim \bbP}
    \KL(\cL(\by_{T} | H_N),\cL(\hby_{T} | H_N))
    \le
    \fr12
    \sum_{\ell = 0}^{L-1}
    \int_{\ell \delta}^{(\ell + 1)\delta}
    \bbE_{\bbP}
    \norm{\hbm(\by_{\ell \delta},\ell \delta) - \bm(\by_t,t)}^2
    ~\de t.
\end{align}
We refer to Lemma \ref{lem:KL-girsanov} for a formal proof,
and to \cite{chen2022sampling} for earlier 
uses of similar bounds.

We next control the difference 
$\bbE_{\bbP}\norm{\bm(\by_t,t)-
\bm(\by_{\ell\delta},t)}^2$,
as detailed in the proof of Lemma \ref{lem:girsanov-error}.
This reduces the problem
to showing that
$\bbE_{\bbP}\norm{\bm(\by_t,t)-\hbm(\by_t,t)}^2
=o_N(1)$ uniformly for $t\in[0,T]$,
which is the focus of the next two steps.
%

\paragraph{2. Properties of the TAP critical point.}
As mentioned above, our algorithm for
approximating the magnetization
(Algorithm~\ref{alg:mean}) attempts to find
a local maximum of the TAP free energy 
$\cF_\sTAP(\,\cdot\,)$
of Eq.~\eqref{eq:FTAP}. 
We will prove that indeed the TAP free energy has 
a unique local maximum $\bm^{\sTAP}$ in a region containing vectors
$\bm$ which approximately 
match the norm and inner product with the planted vector
$\bx$ of the true magnetization
(Proposition~\ref{ppn:local-concavity-and-conditioning}\ref{itm:crit-unique}). Further the AMP output
$\bm^{\sAMP}$ is close to $\bm^{\sTAP}$,
and the TAP free energy is strongly concave in a 
neighborhood of $\bm^{\sTAP}$
(Proposition~\ref{ppn:local-concavity-and-conditioning}\ref{itm:crit-unique}--\ref{itm:amp-near-crit}). As a consequence, the
GD portion of the algorithm will find a point $\bm^{\alg}$ 
close to $\bm^{\sTAP}$
(Proposition~\ref{ppn:additional-event-pre-conditioning}\ref{itm:gd-near-crit}).

\paragraph{3. Approximating the magnetization by TAP 
critical point plus correction.}
We next prove that the Gibbs measure is
exponentially concentrated in a spherical band of vectors approximately orthogonal
to $\bm^{\sTAP}$
(Proposition~\ref{ppn:additional-event-pre-conditioning}\ref{itm:crit-band}).
Hence, the restricted barycenter is exponentially close to the full
magnetization. 
We can view this band as a union of codimension $2$ slices with specified
overlaps with $\bm^{\sTAP}$ and $\bx$.
The Gibbs measure conditional on the slice exactly orthogonal to $\bm^{\sTAP}$
is well approximated by a mixed $p$-spin model with suitable mixture function 
and zero magnetic field.
On the other slices in this band, the corresponding mixed $p$-spin model has
a slightly changed mixture and small magnetic field.
A precise analysis of the magnetization on such models is
carried out in the companion paper \cite{huang2026approximating}
and its results are given in Section 
\ref{sec:LocalComputation}.

Proposition~\ref{ppn:local-barycenter} uses the estimates from
\cite{huang2026approximating} to prove that, conditional 
 on $\bm^{\sTAP}$ being a TAP critical point,
$\bm^{\sTAP}+\bcorr(\bm^{\sTAP})$ approximates the barycenter of the
restricted Gibbs measure.

 Finally, the conditioning estimate in
Proposition~\ref{ppn:local-concavity-and-conditioning}\ref{itm:conditioning-bound}
transfers the conditional bound of Proposition~\ref{ppn:local-barycenter}
to the random critical point, while the Lipschitz property of the
correction
(Proposition~\ref{ppn:additional-event-pre-conditioning}\ref{itm:LipCorr})
allows $\bm^{\sTAP}$ to be replaced by the GD output. 
Together these
estimates yield the magnetization approximation of
Theorem~\ref{thm:mean}. }

\section{Analysis of mean computation algorithm}\label{sec:analysis-mean-alg}

The next several sections are devoted to the analysis of Algorithm~\ref{alg:mean}.
We fix \rev{a large constant $T>0$} and consider $(\bx, H_N,(\by_t)_{t\ge 0})\in S_N\times\sH_N \times C([0,T], \bbR^N)$
distributed according to the planted law $\PP$ defined in Eq.~\eqref{eq:planted-process}.
\rev{The rest of this paper observes the following uniformity convention: all assertions about sufficiently large $N$, implicit constants, and all $o(\cdot)$ error terms hold uniformly over $t\in [0,T]$.}

\rev{Tecall that $\xi_t(q)=\xi(q)+tq$. Under $\PP$, define
\begin{align}
    H_{N,t}(\bsig)
    &:=H_N(\bsig)+\la\by_t,\bsig\ra,
    \label{eq:HNt-with-y}\\
    &=N\xi_t(\<\bx,\bsig\>_N)+\tH_{N,t}(\bsig),
    \label{eq:planted-rewrite}\\
    \tH_{N,t}(\bsig)
    &:=\tH_N(\bsig)+\la\bB_t,\bsig\ra.
    \label{eq:HtildeDef}
\end{align}
Here $\tH_N\sim\mu_\nul$, and $\tH_{N,t}$ is a spin glass with mixture $\xi_t$.
By Eq.~\eqref{eq:tilted-msr}, $\mu_t:=\mu_{H_N,\by_t}$ is the Gibbs measure
associated with $H_{N,t}$; let $\bm_t$ denote its mean.
We will use the two representations \eqref{eq:HNt-with-y} and
\eqref{eq:planted-rewrite} interchangeably.}
The main result of our analysis is the following.
(We refer to the end of this section for its proof.)
\begin{thm}
    \label{thm:mean}
    Under the condition \eqref{eq:amp-works}, there exist parameters $(K_\sAMP, K_\GD^*, \eta)$ \rev{depending only on $(\xi,T)$ and common to all $t\in[0,T]$} such that the point $\bm^\alg$ output by Algorithm~\ref{alg:mean} on input $(H_N,\by_t)$,
    with parameters $K_\sAMP$, $K_\GD(N)=K^*_{\GD}\log N$, $\eta$, satisfies
    \[
        \rev{\sup_{t\in[0,T]}} \EE \tnorm{\bm^\alg - \bm_t}_N^2 = o(N^{-1}).
    \]
\end{thm}
\begin{fac}
    \label{fac:qt-unique}
    \rev{For any $t\in[0,\infty)$, there is a unique solution
    $q_\ast=q_*(t)\in [0,1)$ to
    \beq
        \label{eq:def-qt}
        \xi'_t(q) = \fr{q}{1-q}.
    \eeq
    The function $q_*$ is continuous and strictly increasing, with $q_*(0) = 0$.
    As $t\downarrow0$,
    \beq
        \label{eq:qt-linear}
        q_*=\frac{t}{1-\xi''(0)}+O(t^2),
    \eeq
    where the implicit constants depend only on $\xi$.
    Moreover, the function
    \[
        g_t(q)=\xi_t(q)+q+\log(1-q)
    \]
    is uniquely maximized over $[0,1)$ at $q=q_\ast$.}
\end{fac}
\onlyarxiv{\rev{\begin{proof}
    The derivative
    \[
        g_t'(q)=\xi'_t(q)-\fr{q}{1-q}
    \]
    is strictly decreasing because
    \[
        g_t''(q)
        = \xi''(q) - \fr{1}{(1-q)^2}
        \stackrel{\eqref{eq:amp-works}}{<} 0,
    \]
    while $g_t'(0)=t\ge0$ and $\lim_{q\to1^-}g_t'(q)=-\infty$.
    Thus $g_t'$ has a unique zero $q_\ast\in[0,1)$ (with $q_\ast=0$ when
    $t=0$), and $g_t$ is uniquely maximized there.
    By \eqref{eq:amp-works}, the map $h(q) = q/(1-q)-\xi'(q)$ is continuous and strictly increasing from
    $[0,1)$ onto $[0,\infty)$.
    Equation~\eqref{eq:def-qt} identifies $q_*(t)$ with its inverse, so $q_*$ is continuous and strictly increasing.
    Taylor expanding $t=q_*/(1-q_*)-\xi'(q_*)$ at $q_*=0$ gives \eqref{eq:qt-linear}.
\end{proof}}}
Henceforth let $q_\ast$ denote this solution.
The first step of our analysis characterizes the limiting performance of the AMP iteration \eqref{eq:main-amp} on $(H_N,\by_t)$ generated from the planted process \eqref{eq:planted-process}.
Recall the TAP free energy $\cF_\sTAP$ introduced in \eqref{eq:FTAP}.
With the notation \eqref{eq:planted-rewrite}, we can write
\[
    \cF_\sTAP(\bm) = N \xi_t(\<\bx, \bm\>_N) + \tH_{N,t}(\bm) + \fr{N}{2} \theta(\norm{\bm}_N^2) + \fr{N}{2} \log(1-\norm{\bm}_N^2).
\]
\rev{Differentiating this identity gives
\beq
    \label{eq:nabla-cF-sTAP}
    \nabla \cF_\sTAP(\bm)
    = \nabla\tH_{N,t}(\bm)
    +\xi'_t(\<\bx,\bm\>_N)\bx
    -\lt(\fr{1}{1-\tnorm{\bm}_N^2}
    +(1-\tnorm{\bm}_N^2)\xi''(\tnorm{\bm}_N^2)\rt)\bm.
\eeq}
\begin{ppn}
    \label{ppn:amp-performance}
    For any $\iota>0$, there exists $k_0 \in \bbN$, \rev{depending only on $(\xi,T,\iota)$}, such that for any fixed $k\ge k_0$,
    \rev{for each fixed $t\in[0,T]$, the following conclusions hold on an
    event (which may depend on $t$) whose complement has probability at most
    $e^{-cN}$, where the same $k_0$ and a constant $c>0$ independent of $t$
    are used throughout.}
    The AMP iterate $\bm^k$ satisfies
    \beq
        \label{eq:amp-limiting-overlap}
        |\<\bx,\bm^k\>_N - q_\ast|,
        |\<\bm^k,\bm^k\>_N - q_\ast|
        \le \iota
    \eeq
    and
    \beq
        \label{eq:amp-limiting-gradient}
        \norm{\nabla \cF_\sTAP(\bm^k)}_N,
        \norm{
            \nabla \tH_{N,t}(\bm^k)
            + \xi'_t(q_\ast) \bx
            - \lt((1-q_\ast)\xi''(q_\ast) + \fr{1}{1-q_\ast}\rt) \bm^k
        }_N \le \iota.
    \eeq
    Moreover,\rev{
    \begin{align}
        \label{eq:amp-dominate-gibbs}
        \mu_t(\Band(\bm^k,I) \cap \Band(\bx,I)) \ge 1-e^{-cN}\, ,\\
        I& =I(\iota) := [q_\ast-\iota,q_\ast+\iota]\, .\label{eq:defI-iota}
    \end{align}}
\end{ppn}
The proof of this proposition is presented in Section \ref{sec:Proof-amp-performance}.
For $\iota > 0$, define \rev{the $\bx$-dependent set}
\begin{align}
    \rev{\cS_\iota=\cS_\iota(\bx)} := \lt\{
        \bm \in \bbR^N :
        |\<\bm,\bx\>_N - q_\ast|,
        |\<\bm,\bm\>_N - q_\ast| \le \iota
    \rt\}.\label{eq:SiotaDef}
\end{align}
\rev{We also define
\begin{align}\label{eq:local-domain}
 \mathcal D_\iota(t,\bx)=\left\{\bm\in\cS_\iota(\bx):
 c_0q_*\le\|\bm\|_N^2\le C_0q_*,\quad
 \frac{|\langle\bx,\bm\rangle_N|}{\|\bm\|_N}\le\beta_0
 \text{ if }t<t_0\right\},
\end{align}
where $c_0,C_0$ are the constants in \eqref{eq:tap-norm-shell}, and
$t_0,\beta_0>0$ are sufficiently small constants. Lemma~\ref{lem:tap-local-domain}
below shows that the TAP point belongs to this domain with probability
$1-e^{-cN}$, uniformly for $t\in(0,T]$. We include this event in $\cE_0$.} 
\rev{Unqualified probabilities below are taken under the joint planted law
$\PP$. The two displayed estimates involving $\mathcal D_\iota(t,\bx)$ are instead
conditional on a fixed $\bx$ and hold for every $\bx \in S_N$.
By rotational invariance, the value of $\bx$ is immaterial;
for deterministic $\bm$, the conditional disorder laws relevant to these
bounds depend on $(\bm,\bx)$ only through $\<\bm,\bx\>_N$ and $\<\bm,\bm\>_N$.}
\begin{ppn}
    \label{ppn:local-concavity-and-conditioning}
    \rev{Fix any sufficiently small $\iota>0$ depending only on $(\xi,T)$, and let $\delta > 0$.
    There exist $C^{\spec}_{\max}>C^{\spec}_{\min}>0$, $K_{\sAMP,0},L,c,c_{\op}>0$, depending only on $(\xi,T,\iota)$, and $C_\delta > 0$ depending only on $(\xi,T,\iota,\delta)$, which are common to all $t\in[0,T]$, such that the following holds.
    For each fixed $t\in[0,T]$ and $K_\sAMP \ge K_{\sAMP,0}$, there is an event $\cE_0$ with probability $1-e^{-cN}$ on which:}
    \begin{enumerate}[label=(\alph*)]
        \item \label{itm:gradients-bounded} The event $K_N$ from Proposition~\ref{ppn:gradients-bounded} holds.
        \item \label{itm:crit-unique} $\cF_\sTAP$ has a unique critical point $\bm^\sTAP$ in \rev{$\cS_\iota(\bx)$},
        which \rev{lies in $\mathcal D_\iota(t,\bx)$ and} further satisfies
        \beq
            \label{eq:mtap-well-conditioned}
            \spec(\nabla^2 \cF_\sTAP(\bm^\sTAP)) \subseteq [-C^{\spec}_{\max},-C^{\spec}_{\min}].
        \eeq
        \item \label{itm:amp-near-crit} We have \rev{$\bm^\sAMP \in \cS_{\iota/2}(\bx)$} and $\tnorm{\bm^\sAMP - \bm^\sTAP}_N \le \iota/2$.
    \end{enumerate}
    Note that under \ref{itm:gradients-bounded}, \rev{we may choose $c_{\op}$ large enough} that $\norm{\nabla^2 \cF_\sTAP(\bm)}_{\op,N}, \norm{\nabla^3 \cF_\sTAP(\bm)}_{\op,N} \le c_{\op}$ uniformly over \rev{$\bm \in \cS_\iota(\bx)$}, for all sufficiently small $\iota > 0$.
    Let
    \beq
        \label{eq:def-eps-local-concavity}
        \eps = \min\lt(
            \rev{\fr{\iota C^\spec_{\min}}{100}},
            \fr{(C^\spec_{\min})^2}{40c_{\op}}
        \rt).
    \eeq
    Let $\cE = \cE_0 \cap \{\norm{\nabla \cF_\sTAP(\bm^\sAMP)}_N \le \eps\}$.
    (For $K_\sAMP$ large enough, this holds with probability $1-e^{-cN}$ by Proposition~\ref{ppn:amp-performance}.)
    We further have:
    \begin{enumerate}[resume,label=(\alph*)]
        \item \label{itm:conditioning-bound}
        \rev{For all $\bx \in S_N$ and any random variable $X$ with $0\le X\le 1$ almost surely,}
        \[
            \rev{\EE [X \ind\{\cE\}\mid\bx]}
            \le C_\delta
            \rev{\sup_{\bm \in \mathcal D_\iota(t,\bx)}}
            \EE \lt[
                X^{1+\delta} \ind\{\cE\}
                \big | \rev{\bx,}\nabla \cF_\sTAP(\bm) = \bzero
            \rt]^{1/(1+\delta)}.
        \]
    \end{enumerate}
\end{ppn}
\begin{ppn}
    \label{ppn:additional-event-pre-conditioning}
    \rev{Fix any sufficiently small $\iota>0$ depending only on $(\xi,T)$.
    There exist sufficiently small $\eta>0$, sufficiently large $K_\GD^*>0$, and $L,c,c_{\rm loc}>0$, depending only on $(\xi,T,\iota)$ and common to all $t\in(0,T]$, such that the following holds.
    For each fixed $t\in(0,T]$, in addition to the conclusions of Proposition~\ref{ppn:local-concavity-and-conditioning}, the following properties hold with probability $1-e^{-cN}$:}
    \begin{enumerate}[label=(\alph*)]
        \item \label{itm:crit-band} For $I = I(\iota)$ 
        of Eq.~\eqref{eq:defI-iota}, we have
        \[
            \mu_t(\Band(\bm^\sTAP,I) \cap \Band(\bx, I)) \ge 1-e^{-cN}.
        \]
        \item \label{itm:gd-near-crit} We have \rev{$\tnorm{\bm^\GD - \bm^\sTAP}_N \le \sqrt{q_*}\,N^{-10}$}.
        \item \label{itm:LipCorr} \rev{For any $\bm_1,\bm_2\in \Ball_N(\bm^{\sTAP},c_{\rm loc}\sqrt{q_*})$,
        we have $\|\bcorr(\bm_1)-\bcorr(\bm_2)\|_N\le \fr{L}{\sqrt{q_*}} \|\bm_1-\bm_2\|_N$.}
    \end{enumerate}
\end{ppn}
The proofs of the last two propositions are given in Section \ref{sec:proof-local-concavity-and-conditioning}.

For $\iota > 0$, define the truncated magnetization
\[
    \tbm_\iota(\bm) = \fr{
        \int_{\Band(\bm,I(\iota)) \cap \Band(\bx,I(\iota))}
        \bsig \exp(H_{N,t}(\bsig)) ~\mu_0(\de \bsig)
    }{
        \int_{\Band(\bm,I(\iota)) \cap \Band(\bx,I(\iota))}
        \exp(H_{N,t}(\bsig)) ~\mu_0(\de \bsig)
    }.
\]
\begin{ppn}
    \label{ppn:local-barycenter}
    Let $\bcorr(\,\cdot\,)$ be defined as in Section \ref{sec:CorrectionDef}.
    \rev{There exist sufficiently small $\iota,\delta>0$, depending only on $(\xi,T)$ and common to all $t\in(0,T]$, such that, for all $\bx\in S_N$,}
    \[
        \rev{\sup_{t\in(0,T]}\sup_{\bm \in \mathcal{D}_\iota(t,\bx)}}\EE \lt[
            \tnorm{\bm + \bcorr(\bm) - \tbm_{2\iota}(\bm)}_N^{2+\delta}
            \big | \rev{\bx,}\nabla \cF_\sTAP(\bm) = \bzero
        \rt]
        \le N^{-(1+\delta)}.
    \]
\end{ppn}
The proof of this proposition is given in Section \ref{sec:LocalComputation}.

\begin{proof}[Proof of Theorem~\ref{thm:mean}]
    \rev{First choose $\iota,\delta>0$ sufficiently small that the uniform conclusions of Propositions~\ref{ppn:local-concavity-and-conditioning}, \ref{ppn:additional-event-pre-conditioning}, and \ref{ppn:local-barycenter} hold for every $t\in(0,T]$. With this fixed $\iota$, choose $K_\sAMP$ once so that the uniform conclusions of Propositions~\ref{ppn:amp-performance} and \ref{ppn:local-concavity-and-conditioning} hold for every $t\in(0,T]$. Finally, choose $\eta$ sufficiently small and $K_\GD^*$ sufficiently large, once for all $t\in(0,T]$, as in Proposition~\ref{ppn:additional-event-pre-conditioning}. We keep these five parameters fixed below.}
    Let $\cE_1$ be the intersection of $\cE$ from Proposition~\ref{ppn:local-concavity-and-conditioning} and the event in Proposition~\ref{ppn:additional-event-pre-conditioning}.
    On $\cE_1$, the point $\bm^\sTAP$ is well-defined and we can write
    \baln
        \bm^\alg - \bm_t
        &= \bm^{\GD} + \bcorr(\bm^{\GD}) - \bm_t \\
        &= (\bm^{\GD} - \bm^{\sTAP})
        + (\bcorr(\bm^\GD) - \bcorr(\bm^\sTAP))
        + (\tbm_{2\iota}(\bm^\sTAP) - \bm_t) \\
        &\qquad + (\bm^\sTAP + \bcorr(\bm^\sTAP) - \tbm_{2\iota}(\bm^\sTAP))\,,
    \ealn
    whence
    \baln
        \tnorm{\bm^\alg - \bm_t}_N^2
        &\le 4 \tnorm{\bm^{\GD} - \bm^{\sTAP}}_N^2
        + 4 \tnorm{\bcorr(\bm^\GD) - \bcorr(\bm^\sTAP)}_N^2
        + 4 \tnorm{\tbm_{2\iota}(\bm^\sTAP) - \bm_t}_N^2 \\
        &\qquad + 4 \tnorm{\bm^\sTAP + \bcorr(\bm^\sTAP) - \tbm_{2\iota}(\bm^\sTAP)}_N^2
    \ealn
    The following also holds on $\cE_1$.
    \rev{For all sufficiently large $N$, Proposition~\ref{ppn:additional-event-pre-conditioning}\ref{itm:gd-near-crit} gives}
    \[
        \rev{
        \tnorm{\bm^{\GD} - \bm^{\sTAP}}_N
        \le \sqrt{q_*}\,N^{-10}
        \le c_{\rm loc}\sqrt{q_*},}
    \]
    \rev{so Proposition~\ref{ppn:additional-event-pre-conditioning}\ref{itm:LipCorr}
    applies.  Consequently,}
    \rev{
    \baln
        \tnorm{\bm^{\GD} - \bm^{\sTAP}}_N^2
        &\le q_* N^{-20} \le N^{-20}, \\
        \tnorm{\bcorr(\bm^\GD) - \bcorr(\bm^\sTAP)}_N
        &\le \frac{L}{\sqrt{q_*}}
        \tnorm{\bm^{\GD} - \bm^{\sTAP}}_N
        \le LN^{-10}.
    \ealn
    }
    By \rev{Proposition~\ref{ppn:additional-event-pre-conditioning}\ref{itm:crit-band}}, the complement of $\Band(\bm^\sTAP,I) \cap \Band(\bx,I)$ accounts for an $e^{-cN}$ fraction of the Gibbs measure.
    \rev{Because $\|\bsig\|_N=1$ for every $\bsig\in S_N$,} this implies
    \[
        \tnorm{\tbm_{2\iota}(\bm^\sTAP) - \bm_t}_N^2 \le e^{-cN}.
    \]
    Therefore, on $\cE_1$, for all sufficiently large $N$
    \[
        \tnorm{\bm^\alg - \bm_t}_N^2
        \le CN^{-20} + 4 \tnorm{\bm^\sTAP + \bcorr(\bm^\sTAP) - \tbm_{2\iota}(\bm^\sTAP)}_N^2.
    \]
    \rev{Set
    $\Delta(\bm):=\tnorm{\bm+\bcorr(\bm)-\tbm_{2\iota}(\bm)}_N$.
    Thus, by the tower property,
    \begin{align*}
        \EE [\tnorm{\bm^\alg - \bm_t}_N^2]
        &=\EE_{\bx\sim\mu_0}\lt[
            \EE[\tnorm{\bm^\alg - \bm_t}_N^2\mid\bx]
        \rt] \\
        &\le \EE_{\bx\sim\mu_0}\lt[
            \PP(\cE_1^c\mid\bx)
            +\EE[\tnorm{\bm^\alg - \bm_t}_N^2\ind\{\cE_1\}\mid\bx]
        \rt] \\
        &\le CN^{-20}+4\EE_{\bx\sim\mu_0}\lt[
            \EE\lt[
                \Delta(\bm^\sTAP)^2\ind\{\cE_1\}
                \mathrel{\Big|}\bx
            \rt]
        \rt] \\
        &\le CN^{-20}+4C_{\delta/2}\EE_{\bx\sim\mu_0}\lt[
            \sup_{\bm\in\mathcal D_\iota(t,\bx)}
            \EE\lt[
                \Delta(\bm)^{2+\delta}\ind\{\cE_1\}
                \mathrel{\Big|}\bx,\nabla\cF_\sTAP(\bm)=\bzero
            \rt]^{1/(1+\delta/2)}
        \rt] \\
        &\le CN^{-20}+4C_{\delta/2}N^{-(1+\delta)/(1+\delta/2)}
        =o(N^{-1}).
    \end{align*}
    In the penultimate inequality, we applied
    Proposition~\ref{ppn:local-concavity-and-conditioning}\ref{itm:conditioning-bound}
    conditionally on $\bx$, using $\cE_1\subseteq\cE$ and the fact that on
    $\cE_1$, conditional on $\bx$ and $\nabla\cF_\sTAP(\bm)=\bzero$, uniqueness gives
    $\bm^\sTAP=\bm\in\mathcal D_\iota(t,\bx)$ almost surely. The final inequality follows from
    Proposition~\ref{ppn:local-barycenter}, which holds for all $\bx \in S_N$.
    Since the four proposition estimates and the parameter choices above are
    uniform in $t\in(0,T]$, this yields
    \[
        \sup_{t\in(0,T]}\EE \tnorm{\bm^\alg-\bm_t}_N^2=o(N^{-1}).
    \]
    }


\end{proof}

\section{Analysis of AMP iteration: proof of Proposition~\ref{ppn:amp-performance}}
\label{sec:Proof-amp-performance}

Throughout this section, we fix $t\in [0,T]$.
We observe the uniformity convention stated in Section~\ref{sec:analysis-mean-alg}, so that all implicit constants, $o(\cdot)$ error terms, and convergence statements are uniform in $t\in [0,T]$.

\subsection{State evolution limit}

We first prove \eqref{eq:amp-limiting-overlap} and \eqref{eq:amp-limiting-gradient} using the state evolution result of \cite{bolthausen2014iterative,bayati2011dynamics,javanmard2013state}.
Using the change of notation \eqref{eq:planted-rewrite}, we can rewrite the AMP iteration \eqref{eq:main-amp} as $\bm^{-1} = \bw^0 = \bzero$,
\balnn
    \label{eq:amp-planted-model}
    \bm^k &= (1-q_k) \bw^k, \\
    \notag
    \bw^{k+1} &= \nabla H_{N,t}(\bm^k) - (1-q_k) \xi''(q_k) \bm^{k-1} \\
    \notag
    &= \nabla \tH_{N,t}(\bm^k) + \xi'_t(\<\bx,\bm^k\>_N) \bx - (1-q_k) \xi''(q_k) \bm^{k-1}.
\ealnn
 Here and below, the sequence $(q_k)_{k\ge 0}$ is defined as per Eq.~\eqref{eq:def-q-seq}.

Set $\gamma_0 = \Sigma_{0,i} = \Sigma_{i,0} = 0$ for all $i\ge 0$, and define the following recurrence.
Sample $X \sim \cN(0,1)$ and, for $k\ge 0$,
\[
    (G_1,\ldots,G_k) \sim \cN(0, \Sigma_{\le k}), \qquad
    W_i = G_i + \gamma_i X.
\]
Then, let
\begin{align}
    \gamma_{k+1} &= \xi'_t((1-q_k) \gamma_k) \\
    \Sigma_{k+1,j+1} &= \xi'_t\lt(
        (1-q_k)(1-q_j)
        \EE [W_kW_j]
    \rt).
\end{align}

The following proposition is an immediate consequence of \cite[Proposition 3.1]{el2021optimization},
which generalizes to the tensor case \cite[Theorem 1]{bayati2011dynamics}.
\begin{ppn}
    \label{ppn:state-evolution}
    For any $k\ge 0$, the empirical distribution of the AMP iterates' coordinates converges in $W_2$ in probability:
    \[
        \fr1N \sum_{i=1}^N \delta_{x_i,w_i^1,\ldots,w_i^k} \overset{W_2}{\to} \cL(X,W_1,\ldots,W_k).
    \]
    (In words, the left-hand side is the probability distribution on $\bbR^{k+1}$ that puts mass $1/N$ on
    each point $(x_i,w_i^1,\ldots,w_i^k)$, for $i\in [N]$.)
\end{ppn}
\rev{For each fixed $k$, the finite collections $(q_j)_{j\le k}$, $(\gamma_j)_{j\le k}$, $(\Sigma_{i,j})_{i,j\le k}$, and the Onsager coefficients depend continuously on $t\in[0,T]$ and range over compact sets. Inspecting the proof of AMP state evolution shows that the convergence in Proposition~\ref{ppn:state-evolution} is uniform in $t \in [0,T]$.}
\begin{lem}\label{lemma:gamma_vs_q}
    For all $k,j\ge 0$, we have $\Sigma_{k,j} = \gamma_{k \wedge j} = \fr{q_{k\wedge j}}{1-q_{k\wedge j}}$.
\end{lem}
\onlyarxiv{\rev{\begin{proof}
    We first prove by induction that $\gamma_k = \fr{q_k}{1-q_k}$.
    For $k=0$ this is clear, and then by induction
    \[
        \gamma_{k+1} = \xi'_t(q_k) = \fr{q_{k+1}}{1-q_{k+1}}.
    \]
    Similarly, by induction
    \[
        (1-q_k)(1-q_j) \EE[W_kW_j]
        = (1-q_k)(1-q_j) \lt(\Sigma_{k,j} + \gamma_k \gamma_j\rt)
        = (1-q_{k\vee j})q_{k\wedge j} + q_kq_j
        = q_{k\wedge j},
    \]
    and thus
    \[
        \Sigma_{k+1,j+1} = \xi'_t(q_{k\wedge j}) = \fr{q_{k\wedge j + 1}}{1 - q_{k\wedge j + 1}}.
    \]
\end{proof}}}
\begin{lem}
    \label{lem:scalar-amp-recursion-converge}
    As $k\to\infty$, we have $q_k \to q_\ast$.
\end{lem}
\onlyarxiv{\rev{\begin{proof}
    Let $f_t(q)=\xi'_t(q)/(1+\xi'_t(q))$.  Since $f_t$ is increasing, $f_t(0)\ge0$, and $f_t(q_\ast)=q_\ast$, induction gives $0\le q_k\le q_\ast$ and $q_{k+1}\ge q_k$.  By Eqs.~\eqref{eq:amp-works} and \eqref{eq:def-qt}, if $0\le q\le q_\ast=q_\ast(t)$, then
    \[
        \xi'_t(q)\ge \fr{q}{1-q},
        \qquad
        f_t'(q)=\fr{\xi''(q)}{(1+\xi'_t(q))^2}
        \le (1-q)^2\xi''(q).
    \]
    By continuity and monotonicity of $q_\ast$, its range over $[0,T]$ is the compact interval $[0,q_\ast(T)]\subset[0,1)$.  Hence, by condition~\eqref{eq:amp-works},
    \[
        \rho:=\max_{q\in[0,q_\ast(T)]}(1-q)^2\xi''(q)<1.
    \]
    The mean value theorem now yields
    \[
        0\le q_\ast-q_k\le \rho^k q_\ast\le \rho^k q_\ast(T),
    \]
    which proves the claimed uniform convergence.
\end{proof}}}
\begin{ppn}
    \label{ppn:amp-overlaps}
    With probability $1-e^{-cN}$, \eqref{eq:amp-limiting-overlap} and \eqref{eq:amp-limiting-gradient} hold \rev{for any fixed} $k \ge k_0$.
\end{ppn}
\begin{proof}
    Let $\simeq$ denote equality up to an additive error $o_{P,N}(1)$ (a term vanishing in probability as $N\to\infty$).
    By Proposition~\ref{ppn:state-evolution},
    \beq
        \label{eq:se-limiting-overlap-xm}
        \<\bx, \bm^k\>_N
        = (1-q_k) \<\bx, \bw^k\>_N
        \simeq (1-q_k) \gamma_k
        = q_k.
    \eeq
    Moreover,
    \beq
        \label{eq:se-limiting-overlap-mm}
        \<\bm^k,\bm^k\>_N
        = (1-q_k)^2 \<\bw^k,\bw^k\>_N
        \simeq (1-q_k)^2 \lt(\Sigma_{k,k} + \gamma_k^2\rt)
        = q_k.
    \eeq
   \rev{By the contraction estimate in the proof of Lemma~\ref{lem:scalar-amp-recursion-converge}, we may choose $k_0$ depending only on $(\xi,T,\iota)$ such that $|q_k-q_*|\le\iota/3$ for every $t\in[0,T]$ and $k\ge k_0$.  Hence \eqref{eq:amp-limiting-overlap} holds with high probability with this common choice of $k_0$.}
    Rearranging the AMP iteration gives
    \begin{align}
        \nabla \tH_{N,t}(\bm^k)
        &= -\xi'_t(\<\bx,\bm^k\>_N) \bx + \bw^{k+1}
        + (1-q_k) \xi''(q_k) \bm^{k-1} \nonumber\\
        &= -\xi'_t(\<\bx,\bm^k\>_N) \bx + \fr{1}{1-q_{k+1}} \bm^{k+1}
        + (1-q_k) \xi''(q_k) \bm^{k-1},\label{eq:RearrangeIter}
    \end{align}
     By Proposition \ref{ppn:state-evolution},
    Lemma \ref{lemma:gamma_vs_q}, and Lemma \ref{lem:scalar-amp-recursion-converge},
    \rev{the state-evolution identities
    \begin{align*}
        \EE\big[((1-q_{k+1})W_{k+1}-(1-q_k)W_k)^2\big]
        &=q_{k+1}-q_k,\\
        \EE[(W_{k+1}-W_k)^2]
        &=(\gamma_{k+1}-\gamma_k)+(\gamma_{k+1}-\gamma_k)^2
    \end{align*}
    show that the limiting squared increments tend to zero uniformly in $t$.  Indeed, $q_k\to q_*$ uniformly and $q\mapsto q/(1-q)$ is Lipschitz on $[0,q_*(T)]$.  Thus}
    \begin{align}
    \lim_{k\to\infty}\plim_{N\to\infty}\|\bm^{k+1}-\bm^k\| &= 0\, ,\\
     \lim_{k\to\infty}\plim_{N\to\infty}\|\bw^{k+1}-\bw^k\| &= 0\, .
     \end{align}
    \rev{The two limits are uniform in $t$. Since $q_k\le q_*(T)<1$, the coefficients in Eq.~\eqref{eq:RearrangeIter} also converge uniformly in $t$.  Therefore, Eq.~\eqref{eq:RearrangeIter} gives, uniformly in $t$,}
    \[
         \lim_{k\to\infty}\plim_{N\to\infty}\norm{
        \nabla\tH_{N,t}(\bm^k)
            +\xi'_t(q_\ast) \bx \rev{-} \lt(\fr{1}{1-q_\ast} + (1-q_\ast) \xi''(q_\ast) \rt) \bm^k
        }_N =0 \, .
    \]
    \rev{Using Eq.~\eqref{eq:nabla-cF-sTAP}, Eqs.~\eqref{eq:se-limiting-overlap-xm} and \eqref{eq:se-limiting-overlap-mm} further imply}
    \[
        \lim_{k\to\infty}\plim_{N\to\infty}
        \tnorm{\nabla \cF_\sTAP(\bm^k)}_N =0 \, .
    \]
    \rev{This limit is uniform in $t$ as well. Enlarging the same $k_0$ if necessary, still depending only on $(\xi,T,\iota)$, therefore makes \eqref{eq:amp-limiting-gradient} hold with high probability.}

    To improve these assertions to $1-e^{-cN}$ probability, note that by \cite[Section 8]{huang2022tight}, the AMP iterate $\bm^k$ is, on an event $\cE_\Lip$ with probability $1-e^{-cN}$, an $O(1)$-Lipschitz function of the disorder Gaussians in $\tH_{N,t}$.
    \rev{For each fixed $k$, the relevant Lipschitz constants and failure probabilities are bounded depending only on $(\xi,T,k,\iota)$ and are common to all $t\in[0,T]$.}
    By Kirszbraun's extension theorem, there is a measurable, $O(1)$-Lipschitz function $\tbm^k$ of the disorder which agrees with $\bm^k$ on $\cE_\Lip$.
    Thus $\<\bx,\tbm^k\>_N$ and $\<\tbm^k,\tbm^k\>_N$ are $O(N^{-1/2})$-Lipschitz in the disorder.
    By Gaussian concentration of measure,
    \[
        |\<\bx,\tbm^k\>_N - \EE \<\bx,\tbm^k\>_N|,
        |\<\tbm^k,\tbm^k\>_N - \EE \<\tbm^k,\tbm^k\>_N| \le \iota/3
    \]
    with probability $1-e^{-cN}$.
    Since $\bm^k = \tbm^k$ on $\cE_\Lip$, \eqref{eq:amp-limiting-overlap} holds with probability $1-e^{-cN}$.

    By Proposition~\ref{ppn:gradients-bounded}, $\bm \mapsto \nabla \tH_{N,t}(\bm)$ is also $O(1)$-Lipschitz over $\norm{\bm}_N \le 1$ with probability $1-e^{-cN}$.
    A similar argument shows that \eqref{eq:amp-limiting-gradient} holds with probability $1-e^{-cN}$.
\end{proof}

\subsection{Overlap with AMP iterates}

The following proposition constitutes the first half of the proof of Eq.~\eqref{eq:amp-dominate-gibbs}.
\begin{ppn}
    \label{ppn:amp-dominate-gibbs-1}
    Let $\iota > 0$, and let $I = I(\iota)$ be as in Eq.~\eqref{eq:defI-iota}.
    With probability $1-e^{-cN}$, \rev{for any fixed $k\ge k_0$, where $k_0$ is a sufficiently large constant depending only on $(\xi,T,\iota)$},
    \[
        \mu_t(\Band(\bm^k,I)) \ge 1-e^{-cN}.
    \]
\end{ppn}
To prove Proposition~\ref{ppn:amp-dominate-gibbs-1}, we will combine Lemma~\ref{lem:band-recursion-one-step} below, which identifies a band on which the Gibbs measure $\mu_t$ concentrates, with a self-reduction argument.
\rev{We return to the representations in Eqs.~\eqref{eq:HNt-with-y}, \eqref{eq:planted-H-law}, and \eqref{eq:planted-process}: at fixed $t$, we write $\bB_t=\sqrt{t}\,\bg$, where $\bg\sim\cN(0,\bI_N)$.}
Let $\la \cdot \ra$ denote the average with respect to $\bsig \sim \mu_t$.
The following fact is a restatement of Bayes' theorem: sampling $\bx$ and then $\by_t$
is equivalent to sampling $\by_t$ and then $\bx$ from the posterior. In the context of statistical
physics, this is known as `Nishimori's property.'
\begin{fac}
    \label{fac:nishimori}
    For any bounded measurable $f$, $\EE f(\bx,\by_t) = \EE \la f(\bsig,\by_t) \ra$.
\end{fac}

\begin{lem}
    \label{lem:band-recursion-one-step}
    Let $\iota > 0$ be arbitrary.
    With probability $1-e^{-cN}$,
    \[
        \lt|\norm{\by_t}_N^2 - t^2 - t\rt| \le \iota, \quad
        |\<\bx,\by_t\>_N - t| \le \iota, \quad
        \rev{\lt|\fr{\<\bx,\by_t\>_N}{\norm{\by_t}_N}
        -\sqrt{\fr{t}{1+t}}\right| \le \iota,}
    \]
    and
    \[
        \mu_t(\Band(\by_t, [t-\iota,t+\iota])) \ge 1-e^{-cN}.
    \]
\end{lem}
\begin{proof}
    Clearly $\norm{\by_t}_N^2 \simeq t^2 + t$ and $\<\bx,\by_t\>_N \simeq t$.
    The first three conclusions follow by standard concentration arguments \rev{(which further show $c$ is uniform in $t\in [0,T]$)}.
    By Fact~\ref{fac:nishimori},
    \[
        \EE \lt\la \ind\lt\{\<\bsig, \by_t\>_N \not \in [t-\iota,t+\iota]\rt\} \rt\ra
        = \PP \lt(\<\bx, \by_t\>_N \not \in [t-\iota,t+\iota] \rt) \le e^{-cN}.
    \]
    By Markov's inequality,
    \[
        \PP \lt\{
            \lt\la \ind\lt\{\<\bsig, \by_t\>_N \not \in [t-\iota,t+\iota]\rt\} \rt\ra
            \ge e^{-cN/2}
        \rt\} \le e^{-cN/2}.
    \]
    This implies the final conclusion after adjusting $c$.
\end{proof}
We next introduce a self-reduction property of models obtained by restriction to a certain band.
Define
\[
    U = \lt\{
        \bsig \in \bbR^N : \<\bsig, \by_t\>_N = 0
    \rt\}.
\]

Recall that $(q_k)_{k\ge 0}$ is defined by Eq.~\eqref{eq:def-q-seq}, and in particular
$q_1= t/(1+t)$.
Let $\hby_t = \by_t / \norm{\by_t}_N$ and $r = \sqrt{q_1}$.
Consider the Hamiltonian on $\brho \in U$ defined by
\[
    \hH (\brho) = H_N(r\hby_t + \sqrt{1-r^2} \brho) - H_N(r\hby_t).
\]
Further define
\[
    \xi_{(1)}(s) = \xi(q_1 + (1-q_1)s) - \xi(q_1).
\]
Let $r_1 = \<\bx, \hby_t\>_N$ and define $\bx^\perp \in U$ by $\bx = r_1 \hby_t + \sqrt{1-r_1^2} \bx^\perp$.
Note that conditionally on $(\by_t,r_1)$, $\bx^\perp$ is a uniformly  random vector in $U \cap S_N$.
Also define the Hamiltonian
\[
    \hH'(\brho) = N\xi_{(1)}(\<\bx^\perp,\brho\>_N) + \tH'(\brho),
\]
where $\tH'$ is a Gaussian process on $U$ with covariance
\[
    \EE \tH'(\brho^1) \tH'(\brho^2) = N\xi_{(1)}(\<\brho^1,\brho^2\>_N).
\]
Note that $\hH'$ is of the form \eqref{eq:planted-rewrite}, with one fewer dimension and $\xi_{(1)}$ in place of $\xi_t$.
\begin{ppn}[Self-reduction]
    \label{ppn:self-similarity}
    There exists a constant $C$ such that the following holds.
    Let $\iota > 0$.
    Let $S$ be the $(\by_t,r_1)$-measurable event
    \beq
        \label{eq:def-self-similarity-event}
        \lt|\norm{\by_t}_N - \sqrt{t(1+t)}\rt|,
        |\<\bx,\hby_t\>_N - \sqrt{q_1}| \le \iota.
    \eeq
    Then $\PP(S) \ge 1-e^{-cN}$ and for any $(\by_t,r_1) \in S$ the following holds.
    There is a coupling $\cC$ of $\cL(\hH | \by_t, r_1)$ and $\cL(\hH')$ such that almost surely,
    \beq
        \label{eq:self-similar-approx}
        \begin{aligned}
       & \fr1N \sup_{\brho \in U \cap S_N} |\hH(\brho) - \hH'(\brho)|\le C\iota\, ,\\
        & \sup_{\brho \in U \cap S_N} \norm{\nabla_U \hH(\brho) - \nabla_U \hH'(\brho)}_N
        \le C\iota.
        \end{aligned}
    \eeq
\end{ppn}
\begin{proof}
    \rev{Lemma~\ref{lem:band-recursion-one-step} and $q_1 = t/(1+t)$ prove $\PP(S) \ge 1-e^{-cN}$, after adjusting $\iota$.}
    Now suppose $(\by_t,r_1) \in S$.
    We have $\hH(\brho) = \hH_1(\brho) + \hH_2(\brho)$, where
    \baln
        \hH_1(\brho)
        &= N \lt\{\xi \lt(R\lt(r\hby_t + \sqrt{1-r^2} \brho, r_1\hby_t + \sqrt{1-r_1^2} \bx^\perp\rt)\rt)
        - \xi \lt(R\lt(r\hby_t, r_1\hby_t + \sqrt{1-r_1^2} \bx^\perp\rt)\rt)\rt\}, \\
        \hH_2(\brho) &= \lt\{\tH_N \lt(r\hby_t + \sqrt{1-r^2} \brho\rt) - \tH_N (r\hby_t)\rt\}.
    \ealn
    The first summand simplifies as
    \[
        \hH_1(\brho)
        = N \lt\{
            \xi\lt(rr_1 + \sqrt{(1-r^2)(1-r_1^2)} \<\brho,\bx^\perp\>_N \rt)
            - \xi(rr_1)
        \rt\}
        = N \xi_{(1)}(\<\brho,\bx^\perp\>_N)
        + N\cdot O(\iota).
    \]
    The second summand is a Gaussian process on $U$ with covariance
    \[
        \EE \hH_2(\brho^1) \hH_2(\brho^2)
        = N\lt(
            \xi(r^2 + (1-r^2) \<\brho^1,\brho^2\>_N) - \xi(r^2)
        \rt)
        = N\xi_{(1)}(\<\brho^1,\brho^2\>_N).
    \]
    Thus we can couple $\hH_2$ and $\tH'$ so that $\hH_2 = \tH'$ almost surely.
    \rev{Moreover, $q_1\in[0,T/(1+T)]$, so the reduced mixtures $\xi_{(1)}$
    range over a compact family as $t$ ranges over $[0,T]$.  Thus the constants
    in the preceding Taylor comparison and coupling estimate can be chosen
    depending only on $(\xi,T)$.}
\end{proof}
Define $\hq_0 = 0$ and, similarly to \eqref{eq:def-q-seq},
\[
    \hq_{k+1} = \fr{\xi'_{(1)}(\hq_k)}{1 + \xi'_{(1)}(\hq_k)}.
\]
\begin{lem}
    \label{lem:q-to-hq}
    For all $k\ge 0$, we have $q_1 + (1-q_1) \hq_k = q_{k+1}$.
\end{lem}
\onlyarxiv{\rev{\begin{proof}
    We induct on $k$.
    The base case $k=0$ is trivial.
    Since $q_1 = \fr{t}{1+t}$, the inductive step follows from
    \baln
        q_1 + (1-q_1) \hq_{k+1}
        &= q_1 + (1-q_1) \fr{\xi'_{(1)}(\hq_k)}{1 + \xi'_{(1)}(\hq_k)}
        = 1 - (1-q_1) \lt(1 - \fr{\xi'_{(1)}(\hq_k)}{1 + \xi'_{(1)}(\hq_k)} \rt)\\
        &= 1 - \fr{1-q_1}{1 + (1-q_1) \xi'(q_{k+1})}
        = 1 - \fr{1}{1+t+\xi'(q_{k+1})} \\
        &= \fr{\xi'_t(q_{k+1})}{1+\xi'_t(q_{k+1})}
        = q_{k+2}.
    \ealn
\end{proof}}}
Define the AMP iteration, analogous to \eqref{eq:amp-planted-model}, on the reduced model $\hH'$, by $\hbm^{-1} = \hbw^0 = \bzero$ and
\[
    \hbm^k = (1-\hq_k) \hbw^k, \qquad
    \hbw^{k+1} = \nabla_U \hH'(\hbm^k) - (1-\hq_k)\xi''_{(1)}(\hq_k)\hbm^{k-1}.
\]
Note that $\hbm^k,\hbw^k \in U$.
\begin{ppn}[Self-reduction of AMP iterates]
    \label{ppn:self-similarity-amp}
    Let $\iota > 0$.
    Suppose $(\by_t, r_1) \in S$ for $S$ as in Proposition~\ref{ppn:self-similarity}, and couple $\cL(\hH | \by_t, r_1)$ and $\hH'$ as in that proposition.
    Then (conditionally on $\by_t, r_1$) with probability $1-e^{-cN}$, for all $1\le k\le O(1)$,
    \beq
        \label{eq:self-similarity-amp}
        \tnorm{\bm^k - \tbm^k}_N \le O(\iota), \qquad \text{where} \qquad
        \tbm^{k+1} = \sqrt{q_1} \hby_t + \sqrt{1-q_1} \hbm^k.
    \eeq
\end{ppn}
\begin{proof}
    We induct on the claim that \eqref{eq:self-similarity-amp} holds for all $1\le k\le K$.
    \rev{For each fixed $K$, compactness of the reduced-mixture family identified above makes the finite-depth state-evolution constants, the AMP-disorder Lipschitz constants, the Hamiltonian-gradient Lipschitz constants, and their failure exponents common to all $t\in[0,T]$.  Accordingly, all $O(\iota)$ bounds in this induction are uniform in $t$.}
    First, we have
    \beq
        \label{eq:bm1-calculation}
        \bm^1 = (1-q_1) \by_t = \fr{\by_t}{1+t}, \qquad
        \tbm^1 = \sqrt{q_1} \hby_t = \sqrt{\fr{t}{1+t}} \hby_t.
    \eeq
    For $(\by_t, r_1) \in S$, we have $|\norm{\by_t}_N - \sqrt{t(1+t)}| \le \iota$, and thus
    \[
        \tnorm{\bm^1 - \tbm^1}_N
        = \bigg|\fr{\sqrt{\norm{\by_t}_N}}{1+t} - \sqrt{\fr{t}{1+t}}\bigg|
        \le \fr{\iota}{1+t}.
    \]
    This proves the base case $K=1$.
    Suppose \eqref{eq:self-similarity-amp} holds for $1\le k\le K$.
    By Proposition~\ref{ppn:state-evolution}, for all $1\le j,k \le K+1$,
    \[
        \<\bm^j,\bm^k\>_N \to_p q_{j\wedge k}, \qquad
        \<\hbm^j,\hbm^k\>_N \to_p \hq_{j\wedge k},
    \]
    and thus, by Lemma~\ref{lem:q-to-hq},
    \[
        \<\tbm^j,\tbm^k\>_N \to_p q_1 + (1-q_1) \hq_{(j-1)\wedge (k-1)} = q_k.
    \]
    Because AMP iterates are Lipschitz in the disorder (see the proof of Proposition~\ref{ppn:amp-overlaps}), on an event with probability $1-e^{-cN}$,
    \beq
        \label{eq:state-evolution-event}
        \<\bm^j,\bm^{k}\>_N, \<\tbm^j,\tbm^{k}\>_N \in [q_{j\wedge k} - \iota, q_{j\wedge k} + \iota]
    \eeq
    for all $1\le j,k \le K+1$.
    Since $\bm^1$ is a multiple of $\by_t = \nabla H_{N,t}(\bzero)$,
    \[
        \bm^{K+1} \in \spn(\bm^1,\ldots,\bm^K,\nabla H_{N,t}(\bm^K))
        = \spn(\bm^1,\ldots,\bm^K,\nabla_U H_N(\bm^K)).
    \]
    As
    \[
        \hbm^K \in \spn(\hbm^1,\ldots,\hbm^{K-1},\nabla_U \hH'(\hbm^{K-1})),
    \]
    we have
    \[
        \tbm^{K+1} \in \spn(\tbm^1,\ldots,\tbm^K,\nabla_U \hH'(\hbm^{K-1})).
    \]
    Note that $\sqrt{1-q_1} \nabla_U H_N(\tbm^K) = \nabla_U \hH(\hbm^{K-1})$.
    Thus, on the intersection of an event where $\nabla H_N$ is $O(1)$-Lipschitz and the event in Proposition~\ref{ppn:self-similarity}, each of which has probability $1-e^{-cN}$,
    \baln
        \norm{\sqrt{1-q_1} \nabla_U H_N(\bm^K) - \nabla_U \hH'(\hbm^{K-1})}_N
        &\le \sqrt{1-q_1} \norm{\nabla_U H_N(\bm^K) - \nabla_U H_N(\tbm^K)}_N \\
        &+ \norm{\nabla_U \hH(\hbm^{K-1}) - \nabla_U \hH'(\hbm^{K-1})}_N
        = O(\iota).
    \ealn
    This and \eqref{eq:state-evolution-event} imply $\tnorm{\bm^{K+1}-\tbm^{K+1}}_N = O(\iota)$, completing the induction.
\end{proof}
\begin{ppn}
    \label{ppn:band-recursion-dominate-gibbs}
    For all $\iota > 0$ and $k \ge 1$ fixed, the following holds.
    Let
    \[
        V_k(\iota) = \lt\{
            \bsig \in S_N:
            |\<\bsig, \bm^j\>_N - q_j| \le \iota, \quad \forall 1\le j\le k
        \rt\}.
    \]
    Then, with probability $1-e^{-cN}$,
    \[
        \mu_t(V_k(\iota)) \ge 1 - e^{-cN}.
    \]
\end{ppn}
\begin{proof}
    We induct on $k$.
    \rev{For each fixed terminal depth, every iterated reduced mixture encountered below has the form
    \[
        s\mapsto \xi(q_j+(1-q_j)s)-\xi(q_j)
    \]
    for some $q_j\in[0,q_*(T)]$.  These mixtures form a compact family, so all coupling, Lipschitz, finite-depth state-evolution, and concentration constants in the induction can be chosen commonly for $t\in[0,T]$.}
    By Lemma~\ref{lem:band-recursion-one-step}, with probability $1-e^{-cN}$,
    \beq
        \label{eq:amp-dominate-gibbs-induction}
        \mu_t(\Band(\by_t,[t-\iota,t+\iota])) \ge 1-e^{-cN}.
    \eeq
    As calculated in \eqref{eq:bm1-calculation}, $\bm^1 = \by_t/(1+t)$, so $\bsig \in \Band(\by_t, [t-\iota,t+\iota])$ if and only if
    \[
        \<\bsig, \bm^1\>_N = \fr{t}{1+t} + O(\iota) = q_1 + O(\iota).
    \]
    \rev{The norm and inner-product concentration estimates used here have a common exponential rate for $t\in[0,T]$.}
    This proves the base case $k=1$ after adjusting $\iota$ by a constant factor.

    For the inductive step, \rev{choose $\iota_1$ depending only on $(\xi,T,k,\iota)$ and sufficiently small for all the comparisons and common failure exponents below.
    By the third estimate of Lemma~\ref{lem:band-recursion-one-step} and Nishimori's property,
    \[
        \EE\lt\la
            \ind\lt\{
                |\<\bsig,\hby_t\>_N-\sqrt{q_1}|>\iota_1
            \rt\}
        \rt\ra
        =\PP\lt(|r_1-\sqrt{q_1}|>\iota_1\rt)
        \le e^{-cN},
    \]
    with $c>0$ common to all $t\in[0,T]$. By Markov's inequality,
    \begin{equation}\label{e:normalized-gibbs-band}
        \bbP\lt(\lt\la
            \ind\lt\{
                |\<\bsig,\hby_t\>_N-\sqrt{q_1}|>\iota_1
            \rt\}
        \rt\ra \ge e^{-cN}\rt)
        \le e^{-cN}
    \end{equation}
    after adjusting $c$.}
    Let $S_1$ be the event \eqref{eq:def-self-similarity-event} with right-hand side $\iota_1$.
    By Proposition~\ref{ppn:self-similarity}, $(\by_t,r_1) \in S_1$ with probability $1-e^{-cN}$.
    \rev{We work on the intersection of this event with the complement of the event in \eqref{e:normalized-gibbs-band}, and condition on $(\by_t,r_1)$ when applying the self-reduction coupling. On this intersection,}
    \[
        \mu_t(\Band(\hby_t,[\sqrt{q_1}-\iota_1,\sqrt{q_1}+\iota_1])) \ge 1-e^{-cN}.
    \]
    For $r_2 \in [\sqrt{q_1} - \iota_1, \sqrt{q_1} + \iota_1]$, let $\hmu_t^{r_2}$ be the Gibbs measure on $U \cap S_N$ given by
    \[
        \hmu_t^{r_2}
        = Q_{\#} \mu_t(\cdot | \<\bsig,\hby_t\>_N = r_2), \qquad \text{where} \qquad
        Q(\bsig) = \fr{P_{\hby_t}^\perp(\bsig)}{\tnorm{P_{\hby_t}^\perp(\bsig)}_N}.
    \]
    Note that $\hmu_t^{\sqrt{q_1}}$ is the Gibbs measure on $U \cap S_N$ corresponding to Hamiltonian $\hH$.
    Couple $\hH$ and $\hH'$ as in Proposition~\ref{ppn:self-similarity}, and let $\hmu'_t$ be the Gibbs measure on $U \cap S_N$ corresponding to Hamiltonian $\hH'$.

    By the inductive hypothesis \rev{applied to Hamiltonian $\hH'$ and mixture $\xi_{(1)}$}, with probability $1-e^{-cN}$, $\hmu'_t(\hV_k(\iota)) \ge 1-e^{-cN}$, where
    \[
        \hV_k(\iota) = \lt\{
            \brho \in U \cap S_N:
            |\<\brho, \hbm^j\>_N - \hq_j| \le \iota, \quad \forall 1\le j\le k
        \rt\}.
    \]
    By Proposition~\ref{ppn:self-similarity},
    \[
        \fr1N \sup_{\brho \in U \cap S_N}
        |\hH(\brho) - \hH'(\brho)| \le \iota_1.
    \]
    For $\iota_1$ small enough in $\iota$, this implies
    \[
        \hmu_t^{\sqrt{q_1}}(\hV_k(2\iota)) \ge 1-e^{-cN}.
    \]
    By Lipschitz continuity of $H_{N,t}$, for $\iota_1$ small enough in $\iota$, we have
    \[
        \hmu_t^{r_2}(\hV_k(3\iota)) \ge 1-e^{-cN}, \qquad
        \forall r_2 \in [\sqrt{q_1} - \iota_1, \sqrt{q_1} + \iota_1].
    \]
    This implies $\mu_t(\tV_{k+1}(4\iota)) \ge 1-e^{-cN}$, where
    \[
        \tV_k(\iota)
        = \lt\{
            \bsig \in S_N :
            |\<\bsig, \tbm^j\>_N - q_j| \le \iota, \quad \forall 1\le j\le k
        \rt\}.
    \]
    However, by Proposition~\ref{ppn:self-similarity-amp}, with probability $1-e^{-cN}$, $\tnorm{\bm^j - \tbm^j}_N \le \iota$ for all $1\le j\le k+1$.
    On this event, $\tV_{k+1}(4\iota) \subseteq V_k(5\iota)$.
    Thus $\mu_t(V_k(5\iota)) \ge 1-e^{-cN}$.
    \rev{Since the induction has fixed finite depth, taking the minimum of its finitely many uniform exponents $c$ and adjusting $\iota_1$ preserves a common exponent for all $t\in[0,T]$.}
    This completes the induction, upon adjusting $\iota$.
\end{proof}

\begin{proof}[Proof of Proposition~\ref{ppn:amp-dominate-gibbs-1}]
    \rev{Fix any $k\ge k_0$, where the common choice $k_0=k_0(\xi,T,\iota)$ is supplied by Proposition~\ref{ppn:amp-overlaps}.}
    Let
    \[
        V^+_k(\iota) = \lt\{
            \bsig \in S_N:
            |\<\bsig, \bm^k\>_N - q_k| \le \iota
        \rt\},
    \]
    so clearly $V^+_k(\iota) \supseteq V_k(\iota)$.
    \rev{By the contraction estimate in the proof of Lemma~\ref{lem:scalar-amp-recursion-converge}, we have $|q_k-q_\ast|\le\iota$ for this fixed $k$.}
    Thus (recalling the definition \eqref{eq:defI-iota})
    \[
        \Band(\bm^k,I(2\iota)) \supseteq V^+_k(\iota).
    \]
    By Proposition~\ref{ppn:band-recursion-dominate-gibbs}, with probability $1-e^{-cN}$,
    \[
        \mu_t(\Band(\bm^k,I(2\iota))) \ge 1-e^{-cN}.
    \]
    \rev{The preceding induction supplies an exponent $c$ common to all $t\in[0,T]$ for this fixed $k$, and the tolerance adjustment can likewise be made once for the whole interval.}
    The result follows by adjusting $\iota$.
\end{proof}

\subsection{Overlap with planted signal}

The following proposition completes the proof of \eqref{eq:amp-dominate-gibbs}.
\begin{ppn}
    \label{ppn:amp-dominate-gibbs-2}
    Let $\iota > 0$, and let $I = I(\iota)$ be as in Eq.~\eqref{eq:defI-iota}.
    With probability $1-e^{-cN}$,
    \[
        \mu_t(\Band(\bx,I)) \ge 1-e^{-cN}.
    \]
\end{ppn}
We will use the following replica-symmetric upper bound on the free energy.
Let $\hH_N$ be the Hamiltonian of a spherical spin glass with mixture $\hxi$, which may contain a degree-$1$ term
(i.e., possibly $\hxi'(0) > 0$).

Define the partition function
\beq
    \label{eq:def-hZ}
    \hZ_N = \int_{S_N} \exp \big\{\hH_N(\bsig)\big\} ~\mu_0(\de \bsig).
\eeq
\begin{ppn}
    \label{ppn:parisi-rs-ub}
    For any $u\in [0,1)$, we have
    \beq
        \label{eq:parisi-rs-ub}
        \plim_{N\to\infty} \fr1N \log \hZ_N
        \le
        \fr12 \lt(
            \hxi(1) - \hxi(u) + \fr{u}{1-u} + \log(1-u)
        \rt).
    \eeq
    Furthermore, equality holds if
    \beq
        \label{eq:parisi-extremality-1}
        g(s) = \int_0^s \lt(\hxi'(r) - \fr{r}{(1-u)^2}\rt) ~\de r
    \eeq
    is maximized over $s\in [0,u]$ at $s=u$, and $\hxi_u(s) = \hxi(u + (1-u)s) - \hxi(u) - (1-u)\hxi'(u)s$ satisfies
    \beq
        \label{eq:parisi-extremality-2}
        \hxi_u(s) + s + \log(1-s) \le 0
    \eeq
    for all $s\in [0,1)$.
\end{ppn}
\begin{proof}
    The bound \eqref{eq:parisi-rs-ub} is the spherical Parisi formula \cite[Theorem 1.1]{talagrand2006spherical} with order parameter $\delta_u$.
    The equality condition follows from the extremality condition \cite[Proposition 2.1]{talagrand2006spherical}.
\end{proof}
Let $H_{N,t}$ be as in \eqref{eq:planted-rewrite}.
Let $\psi_N$ denote the probability density of $z_1$, where $\bz$ is a sample from the \rev{Haar probability measure} on the unit sphere $\bbS^{N-1}$.
It is known that
\beq
    \label{eq:psi-Nq-density}
    \psi_N(q) = \fr{1}{Z_{N,\psi}} (1-q^2)^{(N-3)/2}, \qquad q\in [-1,1]
\eeq
for some normalizing constant $Z_{N,\psi}$.
For $q\in [-1,1]$, define
\beq
    \label{eq:def-Z-q}
    Z(q) = \int_{\Band(\bx,q)} \exp \big\{H_{N,t}(\bsig)\big\} ~\de \mu_{(q)}(\bsig)\,,
\eeq
where $\mu_{(q)}$ is the uniform measure on $\Band(\bx,q)$, normalized to $\mu_{(q)}(\Band(\bx,q)) = \psi_N(q)$.
Note that
\[
    \int_{-1}^1 Z(q) ~\de q = \int_{S_N} \exp \big\{H_{N,t}(\bsig)\big\} ~\de \mu_0(\bsig).
\]
%
%
\begin{ppn}
    \label{ppn:band-model-rs-ub}
    For any fixed $q\in (-1,1)$,
    \beq
        \label{eq:band-model-rs-ub}
        \plim_{N\to\infty} \fr1N \log Z(q) \le \fr12 \Big(\xi_t(1) + \xi_t(|q|) + |q| + \log(1-|q|)\Big)
        \rev{+ \Big(\xi_t(q)-\xi_t(|q|)\Big)}.
    \eeq
    Equality holds for $q = q_\ast$.
\end{ppn}
\begin{proof}
    Consider first $q \in [0,1]$.
    On $\Band(\bx,q)$, if we write $\bsig = q \bx + \sqrt{1-q^2} \brho$, where $\la \bx, \brho \ra = 0$, then the random part
    \[
        \hH_{N,q}(\brho) := \tH_{N,t}(\bsig)- \tH_{N,t}(q\bx)  = \tH_{N,t}(q \bx + \sqrt{1-q^2} \brho)
    \]
    is a spin glass with one fewer dimension and mixture $\xi$ replaced by
    \[
        \hxi(s) = \xi_t(q^2 + (1-q^2)s) - \xi_t(q^2).
    \]
    Then,
    \beq
        \label{eq:band-model-fe}
        \plim_{N\to\infty} \fr1N \log Z(q)
        = \xi_t(q) + \fr12 \log(1-q^2)
        + \plim_{N\to\infty} \fr1N \log \hZ_{N,q},
    \eeq
    where $\hZ_{N,q}$ is the free energy of the spin glass with Hamiltonian $\hH_{N,q}$.
    By Proposition~\ref{ppn:parisi-rs-ub} with $u = \fr{q}{1+q}$,
    \beq
        \label{eq:band-model-rs-ub-1}
        \plim_{N\to\infty}
        \fr1N \log \hZ_{N,q}
        \le \fr12 \lt(
            \xi_t(1) - \xi_t(q) + q - \log(1+q)
        \rt).
    \eeq
    Combining with \eqref{eq:band-model-fe} proves \eqref{eq:band-model-rs-ub}.
    \rev{For $q<0$, \eqref{eq:band-model-fe} still holds, and $\hxi$ depends on $q$ only through $q^2$.
    Applying \eqref{eq:band-model-rs-ub-1} with $|q|$ in place of $q$ proves \eqref{eq:band-model-rs-ub}.}

    To show that equality holds in \eqref{eq:band-model-rs-ub} for $q=q_\ast$, we will verify that \eqref{eq:band-model-rs-ub-1} holds with equality.
    Let $u_\ast = \fr{q_\ast}{1+q_\ast}$. \rev{Here and below, primes on $\hxi$ denote derivatives of the base function $\hxi$, not derivatives with respect to $u$ of the family $\hxi_u(s)$.} Then
    \[
        \rev{\hxi'(u_\ast)}
        = (1-q_\ast^2)\xi'_t(q_\ast)
        \stackrel{\eqref{eq:def-qt}}{=} q_\ast(1+q_\ast)
        = \fr{u_\ast}{(1-u_\ast)^2},
    \]
    while
    \[
        \rev{\hxi''(u_\ast)}
        = (1-q_\ast^2)^2 \xi''(q_\ast)
        \stackrel{\eqref{eq:amp-works}}{<} (1+q_\ast)^2
        = \fr{1}{(1-u_\ast)^2}.
    \]
    Thus, for $g$ in \eqref{eq:parisi-extremality-1}, $g'(u_\ast) = 0$ and $g''(u_\ast) < 0$.
    However, over $s\in [0,u_\ast]$,
    \[
        \rev{g'(s) = \hxi'(s) - \fr{s}{(1-u_\ast)^2}}
    \]
    is convex because $\hxi'$ is convex.
    \rev{Thus $g''$ is nondecreasing, and hence $g''(s) \le g''(u_\ast)<0$ for all $s\in [0,u_\ast]$. This implies $g'(s) \ge g'(u_\ast)=0$ throughout this interval.}
    It follows that $g(s)$ is maximized over $s\in [0,u_\ast]$ at $u_\ast$, verifying \eqref{eq:parisi-extremality-1}.
    Since
    \baln
        \hxi_{u_\ast}(s)
        &= \xi_t(q_\ast + (1-q_\ast)s) - \xi_t(q_\ast) - (1-q_\ast)\xi'_t(q_\ast)s \\
        &\stackrel{\eqref{eq:def-qt}}{=}
        \xi_t(q_\ast + (1-q_\ast)s) - \xi_t(q_\ast) - q_\ast s,
    \ealn
    we have
    \baln
        \hxi_{u_\ast}(s) + s + \log(1-s)
        &= \Big\{
            \xi_t(q_\ast + (1-q_\ast)s) + (q_\ast + (1-q_\ast)s) + \log\lt[1 - (q_\ast + (1-q_\ast)s)\rt]
        \Big\} \\
        &\quad - \Big\{
            \xi_t(q_\ast) + q_\ast + \log(1-q_\ast)
        \Big\}
        \le 0,
    \ealn
    where the final inequality is by \rev{Fact~\ref{fac:qt-unique}}.
    This verifies \eqref{eq:parisi-extremality-2} and completes the proof.
\end{proof}

\begin{proof}[Proof of Proposition~\ref{ppn:amp-dominate-gibbs-2}]
    Fix $\iota>0$ arbitrarily (independent of $N$).
    \rev{We will choose $\upsilon$ sufficiently small, depending only on $(\xi,T,\iota)$ and common to all $t\in[0,T]$.}
    Let
    \[
        q^+_k = q_\ast + \iota + k\upsilon, \qquad
        q^-_k = q_\ast - \iota - k\upsilon,
    \]
    and let $k^+$ (resp. $k^-$) be the largest integer such that $q^+_{k^+} \le 1$ (resp. $q^-_{k^-} \ge -1$).
    Let
    \[
        \rev{J = \{q^-_{k^-},\ldots,q^-_0,q^+_0,\ldots,q^+_{k^+}\}.}
    \]
    Define
    \[
        \rev{h(q) = \fr12 (\xi_t(1) + \xi_t(|q|) + |q| + \log(1-|q|)) +\xi_t(q)-\xi_t(|q|)}
    \]
    to be the right-hand side of \eqref{eq:band-model-rs-ub}.
    \rev{Let $g_t$ be as in Fact~\ref{fac:qt-unique}.  Then
    \[
        h(q_\ast)-h(q)=
        \begin{cases}
            \fr12\big(g_t(q_\ast)-g_t(q)\big), & q\ge0,\\
            \fr12\big(g_t(q_\ast)-g_t(|q|)\big)
            +\xi_t(|q|)-\xi_t(q), & q<0.
        \end{cases}
    \]
    This is nonnegative, and zero only at $q = q_\ast$.
    Moreover, $q_\ast$ is continuous in $t$, $h$ is continuous in $(t,q)$ away from the endpoints, and $h(q)\to-\infty$ uniformly in $t\in[0,T]$ as $|q|\uparrow 1$.
    Compactness therefore gives a constant $a_\iota>0$, depending only on $(\xi,T,\iota)$ and common to all $t\in[0,T]$, such that
    \beq
        \label{eq:uniform-planted-overlap-gap}
        h(q_\ast)-h(q)\ge a_\iota,
    \eeq
    for all $q$ such that  $|q-q_\ast| \ge \iota$.}
    Consider the event \rev{$\cE$}:
    \begin{itemize}
        \item $K_N$ from Proposition~\ref{ppn:gradients-bounded} holds,
        \item $\fr1N \log Z(q_\ast) \ge h(q_\ast) - \upsilon$,
        \item $\fr1N \log Z(q) \le h(q) + \upsilon$ for all $q\in J$.
    \end{itemize}
    By standard concentration properties of $Z(q)$, and a union bound over these $q$,
    \rev{there is a constant $c>0$, common to all $t\in[0,T]$, such that
    \beq
        \label{eq:amp-dominate-gibbs-2-event}
        \PP(\cE)\ge 1-e^{-cN}.
    \eeq}
    \rev{For any Borel set $A\subseteq S_N$, define the restricted partition function
    \[
        Z(A) := \int_A \exp\{H_{N,t}(\bsig)\}\,\mu_0(\de\bsig).
    \]
    Then let
    \baln
        Z_0 &= Z\big(\Band(\bx,I(\upsilon))\big), \\
        Z^+_k &= Z\big(\Band(\bx,[q^+_k,q^+_k+\upsilon])\big), \\
        Z^-_k &= Z\big(\Band(\bx,[q^-_k-\upsilon,q^-_k])\big).
    \ealn}
    Since $K_N$ holds, $H_{N,t}(\bsig)$ is $O(1)$-Lipschitz, and thus
    \[
        Z_0 \ge Z(q_\ast) e^{-o_\upsilon(1)N}, \qquad
        Z^+_k \le Z(q^+_k) e^{o_\upsilon(1)N}, \qquad
        Z^-_k \le Z(q^-_k) e^{o_\upsilon(1)N}\,.
    \]
Here and below, $o_{\upsilon}(1)$ denotes a term independent of $N$ that vanishes as $\upsilon\to 0$.
    So (recalling the definition \eqref{eq:defI-iota})
    \[
        \fr1N \log \int_{S_N}
        \ind\{\<\bsig,\bx\>_N \in I(\iota)\}
        \exp H_{N,t}(\bsig) \de \mu_0(\bsig)
        \ge \fr1N \log Z_0
        \ge h(q_\ast) - o_\upsilon(1)
    \]
    while
    \baln
        \fr1N \log \int_{S_N}
        \ind\{\<\bsig,\bx\>_N \not\in I(\iota)\}
        &\le \fr1N \log \lt(\sum_{k=0}^{k^+} Z^+_k + \sum_{k=0}^{k^-} Z^-_k\rt) \\
        &\le \max_{q\in J} h(q) + o_\upsilon(1).
    \ealn
    \rev{Choose $\upsilon$ small enough that $\upsilon<\iota$ and the total of the preceding $o_\upsilon(1)$ errors is at most $a_\iota/2$.
    Since every $q\in J$ satisfies $|q-q_\ast|\ge\iota$, \eqref{eq:uniform-planted-overlap-gap} shows that, on $\cE$,
    \[
        Z\big(S_N\setminus\Band(\bx,I(\iota))\big)
        \le e^{-a_\iota N/2} Z\big(\Band(\bx,I(\iota))\big).
    \]
    Combining this bound with \eqref{eq:amp-dominate-gibbs-2-event}, and decreasing $c$ if necessary, yields
    \[
        \mu_t(\Band(\bx,I(\iota)))\ge1-e^{-cN}
    \]
    with the same $c>0$ available for every $t\in[0,T]$.}
\end{proof}

\begin{proof}[Proof of Proposition~\ref{ppn:amp-performance}]
    This follows from Propositions~\ref{ppn:amp-overlaps}, \ref{ppn:amp-dominate-gibbs-1}, and \ref{ppn:amp-dominate-gibbs-2}.
\end{proof}

\section{Description of TAP fixed point: proof of Propositions \ref{ppn:local-concavity-and-conditioning}--\ref{ppn:additional-event-pre-conditioning}}
\label{sec:proof-local-concavity-and-conditioning}

We first introduce a proposition on the existence and uniqueness of TAP fixed points, whose proof comprises most of the proof of Proposition \ref{ppn:local-concavity-and-conditioning}.
We say that $\bm$ is an $\iota$-approximate critical point of $\cF_\sTAP$ if $\tnorm{\nabla \cF_\sTAP(\bm)}_N \le \iota$.
\begin{ppn}
    \label{ppn:tap-existence-uniqueness}
    \rev{Fix any sufficiently small $\iota>0$ depending only on $(\xi,T)$, and common to all $t\in [0,T]$.
    There exist $C^\spec_{\max} > C^\spec_{\min} > 0$ and $c>0$, depending only on $(\xi,T,\iota)$, such that the following holds with probability $1-e^{-cN}$.}
    \begin{enumerate}[label=(\alph*)]
        \item \label{itm:tap-existence-uniqueness} $\cF_\sTAP$ has a unique critical point $\bm^\sTAP$ in $\cS_\iota$, which further satisfies \eqref{eq:mtap-well-conditioned}.
        \item \label{itm:tap-approx-crits} There exists $\iota' = o_\iota(1)$ \rev{(depending only on $\iota$, and also common to all $t\in [0,T]$)} such that any $\iota$-approximate critical point $\bm \in \cS_\iota$ of $\cF_\sTAP$ satisfies $\tnorm{\bm - \bm^\sTAP}_N \le \iota'$.
    \end{enumerate}
\end{ppn}
\rev{The rest of this section is organized as follows.
The proof of Proposition~\ref{ppn:tap-existence-uniqueness} is split between $t\in [t_0,T]$, treated in Subsections~\ref{subsec:tap-existence-uniqueness}--\ref{subsec:approx-stationarity-local-concavity}, and $t\in [0,t_0)$, treated in Subsection~\ref{subsec:tap-existence-uniqueness-small-t}.
The parameter $t_0$ is defined in the latter subsection, in Lemma~\ref{lem:tap-small-t-fixed-point}.
We set $\iota$ small depending on $\xi,T,t_0$ and another parameter $r$ (defined in Proposition~\ref{ppn:tap-fixed-ball-concavity}); in particular both $t_0,r$ depend on only $(\xi,T)$, so that $\iota$ ultimately depends on only $(\xi,T)$.
Subsections~\ref{subsec:proof-local-concavity-and-conditioning}--\ref{subsec:goe-determinants} complete the proof of Proposition~\ref{ppn:local-concavity-and-conditioning}.
Subsection~\ref{sec:LipCorrection} proves Proposition~\ref{ppn:additional-event-pre-conditioning}.

Throughout this section, we fix $t\in [0,T]$.
We observe the following uniformity convention, extending the one stated in Section~\ref{sec:analysis-mean-alg}.
In Subsections~\ref{subsec:tap-existence-uniqueness}--\ref{subsec:approx-stationarity-local-concavity}, all implicit constants and $o(\cdot)$ error terms are uniform in $t\in [t_0,T]$ for any fixed $t_0 \in (0,T)$, while in Subsection~\ref{subsec:tap-existence-uniqueness-small-t} they are uniform in $t\in [0,t_0)$ for $t_0$ defined in Lemma~\ref{lem:tap-small-t-fixed-point} therein. 
In Subsections~\ref{subsec:proof-local-concavity-and-conditioning}--\ref{sec:LipCorrection} they are uniform over all $t\in (0,T]$.}

\subsection{Existence and uniqueness of TAP fixed point: $t \in [t_0,T]$}
\label{subsec:tap-existence-uniqueness}

\rev{From here to the end of Subsection~\ref{subsec:approx-stationarity-local-concavity}, we assume $t\in [t_0,T]$, and will prove Proposition~\ref{ppn:tap-existence-uniqueness} for $t$ in this range.
We take $\iota>0$ sufficiently small depending on $(\xi,T,t_0)$. All normalizations and denominators in this branch are uniformly well-conditioned for $t\in[t_0,T]$.}
The proof of this proposition depends on an understanding of the landscape of $\tH_{N,t}$ restricted to $\cS_0$, given in Proposition~\ref{ppn:trivialization} below (recall that $\tH_{N,t}$ is the centered version of the
Hamiltonian $H_{N,t}$, cf. Eqs.~\eqref{eq:planted-rewrite} and \eqref{eq:HtildeDef}).
Note that $\cS_0$ is an affine transformation of the sphere $S_{N-2}$; we will view it as a Riemannian manifold.
We first recall notions of Riemannian gradient and Hessian.
For $\bm \in \cS_0$, let
\[
    \bm^\perp = \fr{\bm - q_\ast \bx}{\sqrt{q_\ast(1-q_\ast)}},
\]
so that $\<\bx,\bm^\perp\>_N = 0$ and $\tnorm{\bm^\perp}_N = 1$.
The Riemannian gradient and radial derivative of $\tH_{N,t}$ are
\[
    \nabla_\sp \tH_{N,t}(\bm) = P_{\spn(\bm,\bx)}^\perp \nabla \tH_{N,t}(\bm), \qquad
    \partial_\rd \tH_{N,t}(\bm) = \la \bm^\perp, \nabla \tH_{N,t}(\bm) \ra / \sqrt{N}.
\]
In the calculations below, it will be convenient to work with the following rescaled radial derivative, whose typical maximum is $O(1)$:
\[
    \tpartial_\rd \tH_{N,t}(\bm)
    = \partial_\rd \tH_{N,t}(\bm) / \sqrt{N}
    = \<\bm^\perp, \nabla \tH_{N,t}(\bm)\>_N.
\]
As above, we say $\bm \in \cS_0$ is a \textbf{Riemannian critical point} of $\tH_{N,t}$ if $\nabla_\sp \tH_{N,t}(\bm) = \bzero$, and an \textbf{$\iota$-approximate Riemannian critical point} if $\tnorm{\nabla_\sp \tH_{N,t}(\bm)}_N \le \iota$.
Further define the tangential and Riemannian Hessians (these will be used in the next subsection)
\baln
    \nabla^2_\tn \tH_{N,t}(\bm)
    &= P_{\spn(\bm,\bx)}^\perp
    \nabla^2 \tH_{N,t} (\bm)
    P_{\spn(\bm,\bx)}^\perp, \\
    \nabla^2_\sp \tH_{N,t}(\bm)
    &= \nabla^2_\tn \tH_{N,t}(\bm)
    - \fr{\tpartial_\rd \tH_{N,t}(\bm)}{\sqrt{q_\ast (1-q_\ast)}}
    P_{\spn(\bm,\bx)}^\perp.
\ealn
\begin{ppn}
    \label{ppn:trivialization}
    There exist $C^\spec_{\max} > C^\spec_{\min} > 0$ such that for any $\iota > 0$, the following holds with probability $1-e^{-cN}$.
    \begin{enumerate}[label=(\alph*)]
        \item \label{itm:triv} $\tH_{N,t}$ has exactly two Riemannian critical points $\bm_{\pm}$ on $\cS_0$, and their (rescaled) radial derivatives satisfy
        \beq
            \label{eq:radial-deriv-estimate}
            \lt|
                \tpartial_\rd \tH_{N,t}(\bm_{\pm})
                \mp \sqrt{\fr{q_\ast}{1-q_\ast}} \lt(1 + (1-q_\ast)^2 \xi''(q_\ast)\rt)
            \rt| \le \iota.
        \eeq
        Moreover, there exists $\iota' = o_\iota(1)$ such that all $\iota$-approximate Riemannian critical points $\bm$ on $\cS_0$ satisfy $\tnorm{\bm - \bm_\pm}_N \le \iota'$ for some choice of sign $\pm$.
        \item \label{itm:triv-approx-crit} The point $\bm_+$ is an $\iota$-approximate critical point of $\cF_\sTAP$
        (i.e. $\tnorm{\nabla \cF_\sTAP(\bm)}_N \le \iota$).
        \item \label{itm:triv-local-max} The point $\bm_+$ satisfies
        \[
            \spec(\nabla^2 \cF_\sTAP(\bm_+)) \subseteq [-C^\spec_{\max},-C^\spec_{\min}].
        \]
    \end{enumerate}
\end{ppn}
We will prove this proposition in Subsection~\ref{subsec:riemannian-crits}.
We first show Proposition~\ref{ppn:tap-existence-uniqueness} given Proposition~\ref{ppn:trivialization}.
\begin{lem}
    \label{lem:select-crit}
    For sufficiently small $\iota > 0$, with probability $1-e^{-cN}$, $\cF_\sTAP$ has a unique critical point $\bm$ in the region $\norm{\bm - \bm_+}_N \le \iota$, which further satisfies \eqref{eq:mtap-well-conditioned}.
\end{lem}
\begin{proof}
    Throughout this proof, assume the event $K_N$ from Proposition~\ref{ppn:gradients-bounded} holds, which occurs with probability $1-e^{-cN}$.
    By Proposition~\ref{ppn:trivialization}\ref{itm:triv-local-max}, with probability $1-e^{-cN}$, $\bm_+$ is well-defined and
    \[
        \spec(\nabla^2 \cF_\sTAP(\bm_+)) \subseteq [-C^\spec_{\max},-C^\spec_{\min}].
    \]
    \rev{On $K_N$, both extremal eigenvalues of $\nabla^2\cF_\sTAP(\bm)$ are uniformly $O(1)$-Lipschitz on a fixed neighborhood of $\cS_0$.
    Thus, after decreasing the uniform choice of $\iota$ if necessary,}
    \beq
        \label{eq:concavity-near-mplus}
        \spec(\nabla^2 \cF_\sTAP(\bm)) \subseteq \lt[-2C^\spec_{\max},-\fr{1}{2} C^\spec_{\min}\rt] \qquad
        \forall \norm{\bm - \bm_+}_N \le \iota.
    \eeq
    \rev{Fix $\upsilon<C^\spec_{\min}\iota/4$. This choice is uniform in $t\in[t_0,T]$.}
    By Proposition~\ref{ppn:trivialization}\ref{itm:triv-approx-crit}, with probability $1-e^{-cN}$, $\norm{\nabla \cF_\sTAP(\bm_+)}_N \le \upsilon$.
    Combined with \eqref{eq:concavity-near-mplus}, this implies $\cF_\sTAP$ has a unique critical point $\bm$ in the region $\norm{\bm - \bm_+}_N \le \iota$.
    By \eqref{eq:concavity-near-mplus}, this critical point also satisfies \eqref{eq:mtap-well-conditioned}, upon adjusting the constants $C^\spec_{\min},C^\spec_{\max}$.
    \rev{The choices above and $c$ are common to $t\in[t_0,T]$.}
\end{proof}

\begin{lem}
    \label{lem:crit-unique}
    For any sufficiently small $\iota > 0$, there exists $\iota' = o_\iota(1)$ such that with probability $1-e^{-cN}$, all $\iota$-approximate critical points $\bm \in \cS_\iota$ of $\cF_\sTAP$ satisfy $\norm{\bm - \bm_+}_N \le \iota'$.
\end{lem}
\begin{proof}
    Suppose $K_N$ holds.
    \rev{In this proof, $O_{t_0}(x)$ denotes a quantity bounded by $C|x|$ uniformly in $t\in[t_0,T]$, where $C$ depends on only $\xi, T, t_0$.
    Let $\bm \in \cS_\iota$ be an $\iota$-approximate critical point of $\cF_\sTAP$, and let $\tbm$ be the nearest point in $\cS_0$ to $\bm$, so that $\norm{\bm - \tbm}_N \le O_{t_0}(\iota)$.
    On $K_N$, the map $\bm \mapsto \nabla \cF_\sTAP(\bm)$ is uniformly $O_{t_0}(1)$-Lipschitz. Thus $\tbm$ is an $O_{t_0}(\iota)$-approximate critical point of $\cF_\sTAP$, i.e.}
    \beq
        \label{eq:tbm-approx-crit}
        \norm{
            \nabla \tH_{N,t}(\tbm)
            + \xi'_t(q_\ast) \bx
            - \lt((1-q_\ast)\xi''(q_\ast) + \fr{1}{1-q_\ast}\rt)\tbm
        }_N \le \rev{O_{t_0}(\iota)}.
    \eeq
    \rev{Thus $\tnorm{\nabla_\sp \tH_{N,t}(\tbm)}_N \le O_{t_0}(\iota)$.
    By Proposition~\ref{ppn:trivialization}\ref{itm:triv}, there exists a uniform $\iota'=o_\iota(1)$ such that on an event with probability $1-e^{-cN}$, $\tnorm{\tbm - \bm_\pm}_N \le \iota'/2$ for some choice of sign $\pm$.}
    We now show the sign must be $+$.
    By \eqref{eq:tbm-approx-crit},
    \balnn
        \notag
        \tpartial_\rd \tH_{N,t}(\tbm)
        &= \fr{1}{\sqrt{q_\ast(1-q_\ast)}}
        R\lt(
            \tbm - q_\ast \bx,
            - \xi'_t(q_\ast) \bx
            + \lt((1-q_\ast)\xi''(q_\ast) + \fr{1}{1-q_\ast}\rt)\tbm
        \rt)
        + \rev{O_{t_0}(\iota)} \\
        \label{eq:radial-deriv-of-TAP}
        &= \sqrt{\fr{q_\ast}{1-q_\ast}} \lt(1 + (1-q_\ast)^2 \xi''(q_\ast)\rt)
        + \rev{O_{t_0}(\iota)}.
    \ealnn
    \rev{By Fact~\ref{fac:qt-unique}, the leading term in \eqref{eq:radial-deriv-of-TAP} is bounded below uniformly on $[t_0,T]$ by $\sqrt{q_*(t_0)/(1-q_*(t_0))}>0$.}
    If we had $\tnorm{\tbm - \bm_-}_N \le \iota'/2$, then Eq.~\eqref{eq:radial-deriv-estimate} and Lipschitzness of $\bm \mapsto \nabla \cF_\sTAP(\bm)$ would imply
    \[
        \tpartial_\rd \tH_{N,t}(\tbm)
        = - \sqrt{\fr{q_\ast}{1-q_\ast}} \lt(1 + (1-q_\ast)^2 \xi''(q_\ast)\rt)
        + \rev{O_{t_0}(\iota+\iota')},
    \]
    \rev{which contradicts \eqref{eq:radial-deriv-of-TAP} for sufficiently small $\iota$, uniformly in $t\in[t_0,T]$.}
    Thus $\tnorm{\tbm - \bm_+}_N \le \iota'/2$.
    \rev{The bound $\norm{\bm - \tbm}_N \le O_{t_0}(\iota)$ implies the conclusion after enlarging $\iota'=o_\iota(1)$.
    The resulting $\iota'$ and $c$ are uniform over $t\in[t_0,T]$.}
\end{proof}

\begin{proof}[Proof of Proposition~\ref{ppn:tap-existence-uniqueness}, for $t\in \lbrack t_0,T \rbrack$]
    \rev{Fix sufficiently small constants $\delta,\kappa>0$, uniformly in $t\in[t_0,T]$, so that $\Ball_N(\bm_+,\kappa\iota)\subseteq\cS_\iota$ for all sufficiently small $\iota$.
    Applying Lemma~\ref{lem:select-crit} with $\delta$ and $\kappa\iota$ gives the same critical point $\bm^\sTAP\in\cS_\iota$, unique in $\Ball_N(\bm_+,\delta)$ and satisfying \eqref{eq:mtap-well-conditioned}.

    Let $\iota'=o_\iota(1)$ be given by Lemma~\ref{lem:crit-unique}, which we increase if necessary so that $\iota'\ge\iota$.
    For sufficiently small $\iota$, we have $\iota'<\delta$, so every critical point in $\cS_\iota$ equals $\bm^\sTAP$.
    This proves part~\ref{itm:tap-existence-uniqueness}.}
    Furthermore, for $\iota$-approximate critical points $\bm \in \cS_\iota$,
    \[
        \tnorm{\bm - \bm^\sTAP}_N
        \le \tnorm{\bm - \bm_+}_N
        + \tnorm{\bm^\sTAP - \bm_+}_N
        \le 2\iota'.
    \]
    This proves part~\ref{itm:tap-approx-crits} upon adjusting $\iota'$.
\end{proof}

\subsection{Characterization of Riemannian critical points: proof of Proposition~\ref{ppn:trivialization}}
\label{subsec:riemannian-crits}

The proof builds on a sequence of recent results on \textbf{topological trivialization} in
spherical spin glasses \cite{fyodorov2014topology, fyodorov2015high, belius2022triviality, huang2023strong}.
\begin{proof}[Proof of Proposition~\ref{ppn:trivialization}\ref{itm:triv}]
    For $\bm \in \cS_0$, we may write $\bm = q_\ast\bx + \sqrt{q_\ast(1-q_\ast)} \btau$, where $\<\bx,\btau\>_N = 0$ and $\norm{\btau}_N = 1$.
    Let
    \[
        \hH(\btau)
        = \tH_{N,t}(q_\ast \bx + \sqrt{q_\ast(1-q_\ast)} \btau)
        - \tH_{N,t}(q_\ast \bx).
    \]
    This is a spin glass (in one fewer dimension) with mixture
    \beq
        \label{eq:triv-equivalent-mixture}
        \txi(s) = \xi_t(q_\ast^2 + q_\ast(1-q_\ast)s) - \xi_t(q_\ast^2).
    \eeq
    Note that
    \baln
        \txi'(1) &= q_\ast(1-q_\ast) \xi'_t(q_\ast) \stackrel{\eqref{eq:def-qt}}{=} q_\ast^2, &
        \txi''(1) &= q_\ast^2(1-q_\ast)^2 \xi''(q_\ast) \stackrel{\eqref{eq:amp-works}}{<} q_\ast^2.
    \ealn
    Thus $\txi'(1) > \txi''(1)$, which is the condition for topological trivialization identified in \cite[Equation 64]{fyodorov2015high}; see also \cite[Theorem 1.1]{belius2022triviality}.
    \rev{For $t\in[t_0,T]$, the mixtures $\txi$ form a compact family, and
    \beq\label{e:compact-family-uniformity}
        \txi'(1)-\txi''(1)
        =q_\ast^2\big[1-(1-q_\ast)^2\xi''(q_\ast)\big]
    \eeq
    is bounded below by a positive constant. The estimates in \cite[Theorem~1.6]{huang2023strong} therefore hold with the same $c>0$ on this family.}
    Thus, with high probability, $\hH$ has exactly two critical points $\btau_\pm$, which have radial derivative
    \[
        \tpartial_\rd \hH(\btau_{\pm})
        = \pm \lt(
            \sqrt{\txi'(1)} + \fr{\txi''(1)}{\sqrt{\txi'(1)}}
        \rt) + O(\iota)
        = \pm q_\ast\lt(
            1 + (1-q_\ast)^2 \rev{\xi''(q_\ast)}
        \rt) + O(\iota).
    \]
    By \cite[Theorem 1.6]{huang2023strong}, this actually holds with probability $1-e^{-cN}$.
    On this event, $\tH_{N,t}$ has exactly two Riemannian critical points $\bm_\pm$ on $\cS_0$, which have radial derivative
    \[
        \tpartial_\rd \tH_{N,t}(\bm_\pm)
        = \fr{1}{\sqrt{q_\ast(1-q_\ast)}} \cdot \tpartial_\rd \hH(\btau^\pm)
        = \pm \sqrt{\fr{q_\ast}{1-q_\ast}} \lt(
            1 + (1-q_\ast)^2 \rev{\xi''(q_\ast)}
        \rt) + O(\iota).
    \]
    The estimate \eqref{eq:radial-deriv-estimate} holds by adjusting $\iota$.
    The claim about approximate critical points also follows from \cite[Theorem 1.6]{huang2023strong}, which shows that all approximate critical points are close to exact critical points.
\end{proof}
We will prove parts~\ref{itm:triv-approx-crit} and \ref{itm:triv-local-max} by slightly modifying the calculation in \cite{fyodorov2015high,belius2022triviality}.
This calculation is based on the Kac--Rice formula, which we now recall.
Let $\Crt$ denote the set of Riemannian critical points of $\tH_{N,t}$ on $\cS_0$ and $\mu_{\cS_0}$ denote the $(N-2)$-dimensional Hausdorff measure on $\cS_0$.
The Kac--Rice Formula \cite{rice1944mathematical,kac1948average} (see \cite{adler2007random} for a textbook treatment), applied to $\nabla \tH_{N,t}$ on the Riemannian manifold $\cS_0$, states that for any (random) measurable set $\cT \subseteq \cS_0$,
\beq
    \label{eq:kac-rice}
    \EE |\Crt \cap \cT|
    = \int_{\cS_0} \EE\lt[
        |\det \nabla^2_\sp \tH_{N,t}(\bm)|
        \ind\{\bm \in \cT\}
        \big | \nabla_\sp \tH_{N,t}(\bm) = \bzero
    \rt]
    \varphi_{\nabla_\sp \tH_{N,t}(\bm)}(\bzero) ~\de \mu_{\cS_0}(\bm).
\eeq
Here $\varphi_X$ denotes the probability density of the random variable $X$, and $\nabla^2_\sp \tH_{N,t}(\bm)$ is understood as an $(N-2)\times (N-2)$ matrix.
The following fact is standard; see, e.g., \cite[Lemma 1]{auffinger2013complexity}.
\begin{fac}
    \label{fac:derivative-laws}
    For any $\bm \in \cS_0$, the random variables $\partial_\rd \tH_{N,t}(\bm)$, $\nabla_\sp \tH_{N,t}(\bm)$, and $\nabla^2_\tn \tH_{N,t}(\bm)$ are independent and Gaussian.
    Moreover, with $\bG \sim \GOE(N-2)$, we have
    \[
        \nabla^2_\tn \tH_{N,t}(\bm)
        \stackrel{d}{=}
        \sqrt{\xi''(q_\ast) \fr{N-2}{N}} \bG.
    \]
    \rev{Consequently, for any $r\in\bbR$, conditional on
    $\nabla_\sp \tH_{N,t}(\bm)=\bzero$ and
    $\tpartial_\rd \tH_{N,t}(\bm)=r$,
    \beq
        \label{eq:conditional-riemannian-hessian-law}
        \begin{aligned}
        \nabla^2_\sp \tH_{N,t}(\bm)
        &\stackrel{d}{=}
        \sqrt{\xi''(q_\ast) \fr{N-2}{N}} \bG
        - \fr{r}{\sqrt{q_\ast (1-q_\ast)}} \bI \\
        &= \sqrt{\xi''(q_\ast) \fr{N-2}{N}} \lt(
            \bG - \sqrt{\fr{N}{N-2}}
            \fr{r}{\sqrt{q_\ast (1-q_\ast) \xi''(q_\ast)}} \bI
        \rt).
        \end{aligned}
    \eeq}
\end{fac}
We defer the proof of the following lemma to Subsection~\ref{subsec:goe-determinants}.
\begin{lem}
    \label{lem:goe-concentration}
    Let $\bG \sim \GOE(N)$.
    \rev{For every $\kappa>0$ and $p\ge 1$, there exists $C_{\kappa,p}>0$, depending only on $(\kappa,p)$, such that uniformly for $r\ge 2+\kappa$,}
    \[
        \rev{\EE \lt[|\det (r\bI - \bG)|^p \rt]^{1/p}
        \le
        C_{\kappa,p} \EE \lt[|\det (r\bI - \bG)| \rt].}
    \]
\end{lem}
\begin{ppn}
    \label{ppn:crt-equal-2}
    We have $\EE |\Crt| = 2 + o_N(1)$, \rev{uniformly for $t\in[t_0,T]$}.
\end{ppn}
\begin{proof}
    As shown in the proof of Proposition~\ref{ppn:trivialization}\ref{itm:triv} above, after reparametrizing $\cS_0$ to a sphere of radius $\sqrt{N}$, the restriction of $\tH_{N,t}$ to $\cS_0$ is a spherical spin glass in one fewer dimension with mixture $\txi$ given in \eqref{eq:triv-equivalent-mixture}, which satisfies $\txi'(1) > \txi''(1)$.
    \rev{The claim follows uniformly over the compact family above from \cite[Equation~64]{fyodorov2015high} or \cite[Theorem~1.2]{belius2022triviality}.}
\end{proof}
We will use \eqref{eq:kac-rice} through the following lemma.
Let
\begin{align}
    r_\ast = \sqrt{\fr{q_\ast}{1-q_\ast}} \lt(1+(1-q_\ast)^2\xi''(q_\ast)\rt).\label{eq:R-star-def}
\end{align}
\begin{lem}
    \label{lem:conditioning-cs}
    \rev{Let $\iota > 0$ be sufficiently small uniformly in $t\in[t_0,T]$}, $I_\iota = [r_\ast - \iota, r_\ast + \iota]$, and
    \[
        \cT_\iota = \lt\{
            \bm \in \cS_0 :
            \tpartial_\rd \tH_{N,t}(\bm) \in I_\iota
        \rt\}.
    \]
    \rev{There exists $C>0$, independent of $\iota$ and $t\in[t_0,T]$, such that} for any measurable $\cT \subseteq \cT_\iota$,
    \[
        \EE |\Crt \cap \cT|
        \le C
        \sup_{\bm \in \cS_0}
        \sup_{r\in I_\iota}
        \PP\lt[
            \bm \in \cT
            \big | \nabla_\sp \tH_{N,t}(\bm) = \bzero,
            \tpartial_\rd \tH_{N,t}(\bm) = r
        \rt]^{1/2}.
    \]
\end{lem}
\begin{proof}
    By Fact~\ref{fac:derivative-laws}, $\tpartial_\rd \tH_{N,t}(\bm)$ is independent of $\nabla_\sp \tH_{N,t}(\bm)$.
    Explicitly integrating $\tpartial_\rd \tH_{N,t}(\bm)$ in \eqref{eq:kac-rice} gives
    \baln
        \EE |\Crt \cap \cT|
        &=
        \int_{\cS_0} \int_{I_\iota} \EE\lt[
            |\det \nabla^2_\sp \tH_{N,t}(\bm)|
            \ind\{\bm \in \cT\}
            \big | \nabla_\sp \tH_{N,t}(\bm) = \bzero,
            \tpartial_\rd \tH_{N,t}(\bm) = r
        \rt] \\
        &\qquad \times
        \varphi_{\tpartial_\rd \tH_{N,t}(\bm)}(r)
        \varphi_{\nabla_\sp \tH_{N,t}(\bm)}(\bzero)
        ~\de r ~\de \mu_{\cS_0}(\bm).
    \ealn
    By Cauchy-Schwarz,
    \baln
        &\EE\lt[
            |\det \nabla^2_\sp \tH_{N,t}(\bm)|
            \ind\{\bm \in \cT\}
            \big | \nabla_\sp \tH_{N,t}(\bm) = \bzero,
            \tpartial_\rd \tH_{N,t}(\bm) = r
        \rt] \\
        &\le
        \EE\lt[
            |\det \nabla^2_\sp \tH_{N,t}(\bm)|^2
            \big | \nabla_\sp \tH_{N,t}(\bm) = \bzero,
            \tpartial_\rd \tH_{N,t}(\bm) = r
        \rt]^{1/2} \\
        &\qquad \times \PP\lt[
            \bm \in \cT
            \big | \nabla_\sp \tH_{N,t}(\bm) = \bzero,
            \tpartial_\rd \tH_{N,t}(\bm) = r
        \rt]^{1/2}.
    \ealn
    \rev{The conditional law of $\nabla^2_\sp \tH_{N,t}(\bm)$ given
    $\nabla_\sp \tH_{N,t}(\bm)=\bzero$ and
    $\tpartial_\rd \tH_{N,t}(\bm)=r$ is stated in
    Eq.~\eqref{eq:conditional-riemannian-hessian-law}.}
    In light of \eqref{eq:amp-works},
    \[
        \fr{r_\ast}{\sqrt{q_\ast (1-q_\ast) \xi''(q_\ast)}}
        = (1-q_\ast) \xi''(q_\ast)^{1/2} + \fr{1}{(1-q_\ast) \xi''(q_\ast)^{1/2}} > 2,
    \]
    \rev{Due to the uniformity noted in \eqref{e:compact-family-uniformity}, the diagonal shift in \eqref{eq:conditional-riemannian-hessian-law}
    \[
        \fr{r}{\sqrt{q_\ast (1-q_\ast) \xi''(q_\ast)}}
    \]
    lies in a compact subset of $(2,+\infty)$, uniformly for $t\in[t_0,T]$ and $r\in I_\iota$.
    Lemma~\ref{lem:goe-concentration} therefore gives, with a uniform $C>0$,}
    \balnn
        \notag
        &\EE\lt[
            |\det \nabla^2_\sp \tH_{N,t}(\bm)|^2
            \big | \nabla_\sp \tH_{N,t}(\bm) = \bzero,
            \tpartial_\rd \tH_{N,t}(\bm) = r
        \rt]^{1/2} \\
        \notag
        &= \EE\lt[
            \lt|\det
                \sqrt{\xi''(q_\ast) \fr{N-2}{N}} \lt(
                    \bG - \sqrt{\fr{N}{N-2}} \fr{r}{\sqrt{q_\ast (1-q_\ast) \xi''(q_\ast)}} \bI
                \rt)
            \rt|^2
        \rt]^{1/2} \\
        \notag
        &\le C
        \EE\lt[
            \lt|\det
                \sqrt{\xi''(q_\ast) \fr{N-2}{N}} \lt(
                    \bG - \sqrt{\fr{N}{N-2}} \fr{r}{\sqrt{q_\ast (1-q_\ast) \xi''(q_\ast)}} \bI
                \rt)
            \rt|
        \rt] \\
        \label{eq:goe-concentration-application}
        &= C
        \EE\lt[
            |\det \nabla^2_\sp \tH_{N,t}(\bm)|
            \big | \nabla_\sp \tH_{N,t}(\bm) = \bzero,
            \tpartial_\rd \tH_{N,t}(\bm) = r
        \rt].
    \ealnn
    Combining, we find
    \baln
        \EE |\Crt \cap \cT|
        &\le
        C \sup_{\bm \in \cS_0} \sup_{r\in I_\iota}
        \PP\lt[
            \bm \in \cT
            \big | \nabla_\sp \tH_{N,t}(\bm) = \bzero,
            \tpartial_\rd \tH_{N,t}(\bm) = r
        \rt]^{1/2} \\
        &\times
        \int_{\cS_0} \int_{I_\iota} \EE\lt[
            |\det \nabla^2_\sp \tH_{N,t}(\bm)|
            \big | \nabla_\sp \tH_{N,t}(\bm) = \bzero,
            \tpartial_\rd \tH_{N,t}(\bm) = r
        \rt] \\
        &\qquad \times
        \varphi_{\tpartial_\rd \tH_{N,t}(\bm)}(r)
        \varphi_{\nabla_\sp \tH_{N,t}(\bm)}(\bzero)
        ~\de r ~\de \mu_{\cS_0}(\bm).
    \ealn
    By the Kac--Rice formula, the last integral is the expected number of Riemannian critical points $\bm$ of $\tH_{N,t}$ with radial derivative $\tpartial_\rd \tH_{N,t}(\bm) \in I_\iota$.
    \rev{By Proposition~\ref{ppn:crt-equal-2}, this is at most $2+o_N(1)$ uniformly for $t\in[t_0,T]$.}
\end{proof}

\begin{ppn}
    \label{ppn:conditional-tap-gradient-hessian}
    There exist $C^\spec_{\max} > C^\spec_{\min} > 0$ such that for all sufficiently small $\iota > 0$, there exists $\iota' =h(\iota) =o_\iota(1)$ such that the following holds.
    For any $\bm \in \cS_0$ define the events
    \[
        E_1(\bm,\iota') := \lt\{
            \tnorm{\nabla \cF_\sTAP(\bm)}_N \le \iota'
        \rt\}, \qquad
        E_2(\bm) := \lt\{
            \spec(\nabla^2 \cF_\sTAP(\bm)) \subseteq [-C^\spec_{\max}, -C^\spec_{\min}]
        \rt\}.
    \]
    Then,
    \[
        \inf_{r\in I_\iota}
        \PP\lt[
            E_1(\bm,\iota') \cap E_2(\bm) \big|
            \nabla_\sp \tH_{N,t}(\bm) = \bzero,
            \tpartial_\rd \tH_{N,t}(\bm) = r
        \rt]
        \ge 1-e^{-cN}.
    \]
    Here the constant $c$ is uniform over $\bm \in \cS_0$.
\end{ppn}
We prove this proposition in the next subsection.
Assuming it, we first complete the proof of Proposition~\ref{ppn:trivialization}.
\begin{proof}[Proof of Proposition~\ref{ppn:trivialization}\ref{itm:triv-approx-crit}\ref{itm:triv-local-max}]
    \rev{Let $h,C^\spec_{\min},C^\spec_{\max}$ be the uniform choices from Proposition~\ref{ppn:conditional-tap-gradient-hessian}, and choose $\upsilon>0$, independently of $t\in[t_0,T]$, so that $\max(\upsilon,h(\upsilon))\le\iota$.}
    Let $\cT \subseteq \cS_0$ be the set of points $\bm$ such that
    \begin{itemize}
        \item $\tpartial_\rd \tH_{N,t}(\bm) \in I_\upsilon$, and
        \item $E_1(\bm,\iota) \cap E_2(\bm)$ does not hold.
    \end{itemize}
    Thus $\cT \subseteq \cT_\upsilon$.
    By Lemma~\ref{lem:conditioning-cs} and Proposition~\ref{ppn:conditional-tap-gradient-hessian} (with $\upsilon$ for $\iota$),
    \baln
        \EE |\Crt \cap \cT|
        &\le C
        \sup_{\bm \in \cS_0}
        \sup_{r\in I_\upsilon}
        \PP\lt[
            (E_1(\bm,\iota) \cap E_2(\bm))^c
            \big | \nabla_\sp \tH_{N,t}(\bm) = \bzero,
            \tpartial_\rd \tH_{N,t}(\bm) = r
        \rt]^{1/2} \\
        &\le e^{-cN}.
    \ealn
    \rev{Thus, with probability $1-e^{-cN}$, every critical point with radial derivative in $I_\upsilon$ satisfies $E_1(\bm,\iota)\cap E_2(\bm)$.
    Intersecting this event with that of Proposition~\ref{ppn:trivialization}\ref{itm:triv}, applied with $\upsilon$, gives $E_1(\bm_+,\iota)\cap E_2(\bm_+)$ with the same uniform conclusion.
    Together with part~\ref{itm:triv}, we infer that the constants, $c$, and $o_\iota(1)$ are uniform over $t\in[t_0,T]$.}
\end{proof}

\subsection{Approximate stationarity and local concavity of $\cF_\sTAP$: proof of Proposition \ref{ppn:conditional-tap-gradient-hessian}}
\label{subsec:approx-stationarity-local-concavity}

\begin{lem}
    \label{lem:x-grad-H-conditional-law}
    Let $\bm \in \cS_0$ and $r \in I_\iota$.
    Conditional on $\nabla_\sp \tH_{N,t}(\bm) = \bzero$ and $\tpartial_\rd \tH_{N,t}(\bm) = r$, $\<\bx, \nabla \tH_{N,t}(\bm))$ is Gaussian with mean $q_\ast (1-q_\ast) \xi''(q_\ast) + O(\iota)$ and variance $O(N^{-1})$.
\end{lem}
\begin{proof}
    All the random variables considered are jointly Gaussian, so it suffices to compute the conditional mean and variance.
    \rev{Since $q_\ast\in[q_\ast(t_0),q_\ast(T)]\Subset(0,1)$, all coefficients, denominators, and implicit constants below are uniform in $t\in[t_0,T]$.}
    A short linear-algebraic calculation shows
    \[
        \EE \lt[
            \<\bx, \nabla \tH_{N,t}(\bm)\>_N
            \big| \nabla_\sp \tH_{N,t}(\bm), \tpartial_\rd \tH_{N,t}(\bm)
        \rt]
        = \fr{q_\ast^{3/2} (1-q_\ast)^{1/2} \xi''(q_\ast)}{\xi'_t(q_\ast) + q_\ast (1-q_\ast) \xi''(q_\ast)}
        \tpartial_\rd \tH_{N,t}(\bm).
    \]
    Thus
    \baln
        &\EE \lt[
            \<\bx, \nabla \tH_{N,t}(\bm)\>_N
            \big| \nabla_\sp \tH_{N,t}(\bm) = \bzero, \tpartial_\rd \tH_{N,t}(\bm) = r
        \rt] \\
        &= \fr{q_\ast^{3/2} (1-q_\ast)^{1/2} \xi''(q_\ast)}{\xi'_t(q_\ast) + q_\ast (1-q_\ast) \xi''(q_\ast)}
        r_\ast + O(\iota) \\
        \onlyarxivalign{\stackrel{\eqref{eq:def-qt},\eqref{eq:R-star-def}}{=} \fr{q_\ast^{3/2} (1-q_\ast)^{1/2} \xi''(q_\ast)}{\fr{q_\ast}{1-q_\ast} + q_\ast (1-q_\ast) \xi''(q_\ast)}
        \cdot \sqrt{\fr{q_\ast}{1-q_\ast}} \lt(1+(1-q_\ast)^2\xi''(q_\ast)\rt)
        + O(\iota)}
        &= q_\ast (1-q_\ast) \xi''(q_\ast) + O(\iota).
    \ealn
    Before any conditioning, $\<\bx, \nabla \tH_{N,t}(\bm)\>_N$ is Gaussian with variance $O(N^{-1})$, and conditioning only reduces variance.
\end{proof}

\begin{ppn}
    \label{ppn:conditional-tap-gradient}
    Let $\bm \in \cS_0$ and $r \in I_\iota$.
    Conditional on $\nabla_\sp \tH_{N,t}(\bm) = \bzero$ and $\tpartial_\rd \tH_{N,t}(\bm) = r$, $E_1(\bm,\iota')$ holds with probability $1-e^{-cN}$, for some $\iota' = o_\iota(1)$.
\end{ppn}
\begin{proof}
    By Lemma~\ref{lem:x-grad-H-conditional-law}, with conditional probability $1-e^{-cN}$,
    \[
         |\<\bx, \nabla \tH_{N,t}(\bm)\>_N- q_\ast (1-q_\ast) \xi''(q_\ast)| \le O(\iota).
    \]
    Suppose this event holds.
    Since $\nabla_\sp \tH_{N,t}(\bm) = \bzero$,
    \baln
        \nabla \tH_{N,t}(\bm)
        &= \tpartial_\rd \tH_{N,t}(\bm) \fr{\bm - q_* \bx}{\sqrt{q_\ast (1-q_\ast)}}
        + \<\bx, \nabla \tH_{N,t}(\bm)\> \bx \\
        &= \sqrt{\fr{q_\ast}{1-q_\ast}} \lt(1+(1-q_\ast)^2\xi''(q_\ast)\rt) \fr{\bm - q_* \bx}{\sqrt{q_\ast (1-q_\ast)}}
        + q_\ast (1-q_\ast) \xi''(q_\ast) \bx
        + O(\iota) \bx + O(\iota) \bm \\
        &= -\xi'_t(q_\ast) \bx + \lt(\fr{1}{1-q_\ast} + (1-q_*)\xi''(q_\ast) \rt) \bm
        + O(\iota) \bx + O(\iota) \bm.
    \ealn
    \rev{Since $\bm\in\cS_0$, combining the preceding display with Eq.~\eqref{eq:nabla-cF-sTAP} gives $\tnorm{\nabla \cF_\sTAP(\bm)}_N\le C\iota$ uniformly in $t\in[t_0,T]$. Thus one may take $\iota'=C\iota$, with the same $c>0$ for all $t\in[t_0,T]$.}
\end{proof}
The next lemma is a linear-algebraic calculation of the law of $\nabla^2 \tH_{N,t}(\bm)$ conditional on $\nabla \tH_{N,t}(\bm)$, now as a Hessian in $\bbR^N$ rather than a Riemannian Hessian in $\cS_0$.
While $\bm \in \cS_0$ for the proofs in the current subsection, we will not assume this for use in Fact~\ref{fac:Km-form} below.
\begin{lem}
    \label{lem:nabla2-H-conditional-law}
    Let $\bm \in \bbR^N$ with $\norm{\bm}_N^2 = q_\bm < 1$.
    Conditional on $\nabla \tH_{N,t}(\bm) = \bz$, we have
    \[
        \nabla^2 \tH_{N,t}(\bm) \stackrel{d}{=}
        \fr{\xi''(q_\bm)}{\xi'_t(q_\bm)}
        \cdot \fr{\bm \bz^\top + \bz \bm^\top}{N}
        + \fr{\<\bm,\bz\>_N}{\xi'_t(q_\bm) + q_\bm \xi''(q_\bm)}
        \lt(\xi^{(3)}(q_\bm) - \fr{2\xi''(q_\bm)^2}{\xi'_t(q_\bm)}\rt)
        \fr{\bm\bm^\top}{N}
        + \bM,
    \]
    where $\bM$ is the following symmetric random matrix.
    Let $(\be_1,\ldots,\be_N)$ be an orthonormal basis of $\bbR^N$ with $\be_1 = \bm / \norm{\bm}_2$, and, to reduce notation, let $\bM(i,j) = \la \bM \be_i, \be_j \ra$.
    Then the random variables $\{\bM(i,j) : 1\le i\le j\le N\}$ are independent centered Gaussians with variance
    \beq
        \label{eq:def-bM}
        \EE \bM(i,j)^2
        = N^{-1} \times
        \begin{cases}
            (\text{irrelevant $O(1)$}) & 1=i=j \\
            \xi''(q_\bm) + q_\bm \xi^{(3)}(q_\bm) - \fr{q_\bm\xi''(q_\bm)^2}{\xi'_t(q_\bm)} & 1=i<j \\
            2 \xi''(q_\bm) & 1<i=j \\
            \xi''(q_\bm) & 1<i<j
        \end{cases}
    \eeq
\end{lem}
\begin{rmk}
    \label{rmk:variance-positive-by-cs}
    The covariance calculation in the proof of Lemma~\ref{lem:nabla2-H-conditional-law} implies $\xi''(q_\bm) + q_\bm\xi^{(3)}(q_\bm) - \fr{q_\bm\xi''(q_\bm)^2}{\xi'_t(q_\bm)} \ge 0$, but this can also be seen directly by Cauchy-Schwarz:
    \baln
        \lt(\xi''(q_\bm) + q_\bm\xi^{(3)}(q_\bm)\rt) \xi'_t(q_\bm)
        &\ge q_\bm\lt(\sum_{p\ge 2} p(p-1)^2 \gamma_p^2 (q_\bm)^{p-2} \rt)
        \lt(\sum_{p\ge 2} p\gamma_p^2 (q_\bm)^{p-2}\rt) \\
        &\ge q_\bm\lt(\sum_{p\ge 2} p(p-1) \gamma_p^2 (q_\bm)^{p-2} \rt)^2
        = q_\bm\xi''(q_\bm)^2.
    \ealn
\end{rmk}
\begin{proof}
    It suffices to compute the conditional mean and covariance.
    Let $\bu^1,\bu^2 \in \bbS^{N-1}$.
    Then
    \[
        \EE \lt[
            \la \nabla^2 \tH_{N,t}(\bm) \bu^1, \bu^2 \ra
            \big| \nabla \tH_{N,t}(\bm)
        \rt] = \la \bv, \nabla \tH_{N,t} (\bm) \ra
    \]
    for $\bv = \bv(\bu^1,\bu^2,\bm)$ such that for all $\bw \in \bbR^N$,
    \[
        \la \nabla^2 \tH_{N,t}(\bm) \bu^1, \bu^2 \ra
        - \la \bv, \nabla \tH_{N,t} (\bm) \ra
        \indep
        \la \bw, \nabla \tH_{N,t} (\bm) \ra.
    \]
    We calculate that
    \baln
        \EE \la \nabla^2 \tH_{N,t}(\bm) \bu^1, \bu^2 \ra
        \la \bw, \nabla \tH_{N,t}(\bm) \ra
        &= N^{-1} \lt(
            \la \bu^1, \bw \ra \la \bu^2, \bm \ra
            + \la \bu^1, \bm \ra \la \bu^2, \bw \ra
        \rt) \xi''(q_\bm) \\
        &+ N^{-2} \la \bu^1, \bm \ra \la \bu^2, \bm \ra \la \bw, \bm \ra \xi^{(3)}(q_\bm), \\
        \EE \la \bv, \nabla \tH_{N,t}(\bm) \ra \la \bw, \nabla \tH_{N,t}(\bm) \ra
        &= \la \bv, \bw \ra \xi'_t(q_\bm)
        + N^{-1} \la \bv, \bm \ra \la \bw, \bm \ra \xi''(q_\bm).
    \ealn
    Thus, $\bv$ must satisfy
    \baln
        &N^{-1} \lt(
            \la \bu^2, \bm \ra \bu^1
            + \la \bu^1, \bm \ra \bu^2
        \rt) \xi''(q_\bm)
        + N^{-2} \la \bu^1, \bm \ra \la \bu^2, \bm \ra \xi^{(3)}(q_\bm) \bm \\
        &=
        \xi'_t(q_\bm) \bv
        + N^{-1} \la \bv, \bm \ra \xi''(q_\bm) \bm.
    \ealn
    This has a solution $\bv = a_1 \bu^1 + a_2 \bu^2 + a_3 \bm$, where
    \baln
        a_1 &= \fr{\xi''(q_\bm)}{N\xi'_t(q_\bm)} \la \bu^2, \bm \ra, \qquad
        a_2 = \fr{\xi''(q_\bm)}{N\xi'_t(q_\bm)} \la \bu^1, \bm \ra, \\
        a_3 &= \fr{\la \bu^1, \bm \ra \la \bu^2, \bm \ra}{N^2(\xi'_t(q_\bm) + q_\bm \xi''(q_\bm))} \lt(\xi^{(3)}(q_\bm) - \fr{2\xi''(q_\bm)^2}{\xi'_t(q_\bm)}\rt).
    \ealn
    Thus
    \[
        \EE \lt[
            \la \nabla^2 \tH_{N,t}(\bm) \bu^1, \bu^2 \ra
            \big| \nabla \tH_{N,t}(\bm)
        \rt]
        = a_1 \la \bu^1, \nabla \tH_{N,t}(\bm) \ra
        + a_2 \la \bu^2, \nabla \tH_{N,t}(\bm) \ra
        + a_3 \la \bm, \nabla \tH_{N,t}(\bm) \ra,
    \]
    which implies
    \baln
        \EE \lt[
            \nabla^2 \tH_{N,t}(\bm)
            \big| \nabla \tH_{N,t}(\bm)
        \rt]
        &=
        \fr{\xi''(q_\bm)}{N\xi'_t(q_\bm)}
        (\bm \nabla \tH_{N,t}(\bm)^\top + \nabla \tH_{N,t}(\bm) \bm^\top) \\
        &+ \fr{\<\bm,\nabla \tH_{N,t}(\bm)\>_N}{N(\xi'_t(q_\bm) + q_\bm \xi''(q_\bm))}
        \lt(\xi^{(3)}(q_\bm) - \fr{2\xi''(q_\bm)^2}{\xi'_t(q_\bm)}\rt)
        \bm\bm^\top,
    \ealn
    as desired.
    The conditionally random part of $\nabla^2 \tH_{N,t}$ is thus
    \[
        \bM =
        \nabla^2 \tH_{N,t}(\bm) -
        \EE \lt[
            \nabla^2 \tH_{N,t}(\bm)
            \big| \nabla \tH_{N,t}(\bm)
        \rt].
    \]
    Direct evaluation of covariances $\EE \bM(i_1,j_1) \bM(i_2,j_2)$ gives the covariance structure \eqref{eq:def-bM}.
    The calculation is greatly simplified by the fact that $\la \be_i, \bm \ra = 0$ for all $i\neq 1$, which implies, e.g., that $\bM(i,j) = \la \nabla^2 \tH_{N,t}(\bm) \be_i, \be_j \ra$ for all $i,j \neq 1$.
\end{proof}
\begin{cor}
    \label{cor:nabla2-F-conditional-law}
    Let $\iota > 0$ be sufficiently small.
    Let $\bm, \bz \in \bbR^N$ with $\norm{\bm}_N^2 = q_\bm$ and $\<\bm, \bx\>_N = q_\bx$, such that $|q_\bm-q_\ast|, |q_\bx-q_\ast|, \tnorm{\bz}_N \le \iota$.
    Conditional on $\nabla \cF_\sTAP(\bm) = \bz$,
    \balnn
        \notag
        \nabla^2 \cF_\sTAP(\bm)
        &\stackrel{d}{=}
        \lt(\fr{2+q_\ast}{q_\ast} \xi''(q_\ast) - (1-q_\ast) \xi^{(3)}(q_\ast) - \fr{2}{(1-q_\ast)^2} \rt) \fr{\bm\bm^\top}{N}
        + \xi''(q_\ast) \fr{(\bx - \bm)(\bx - \bm)^\top}{N} \\
        \label{eq:conditional-nabla2-F}
        &- \lt((1-q_\bm) \xi''(q_\bm) + \fr{1}{1-q_\bm} \rt) \bI
        + \bM
        + \bDelM.
    \ealnn
    Here, $\bM$ is as in \eqref{eq:def-bM}, and $\bDelM$ is an $(\bx,\bm,\bz)$-measurable symmetric matrix satisfying $\tnorm{\bDelM}_{\op} \le o_\iota(1)$, whose kernel contains $\spn(\bx,\bm,\bz)^\perp$.
\end{cor}
\begin{proof}
    In the calculations below, $\bDelM$ is an error term satisfying the above, which may change from line to line.
    Conditioning on $\nabla \cF_\sTAP(\bm) = \bz$ is equivalent to conditioning on
    \[
        \nabla \tH_{N,t}(\bm) = \tbz \equiv \bz - \xi'_t(q_\bx) \bx + \lt((1-q_\bm)\xi''(q_\bm) + \fr{1}{1-q_\bm}\rt) \bm.
    \]
    By Lemma~\ref{lem:nabla2-H-conditional-law},
    \baln
        \nabla^2 \tH_{N,t}(\bm)
        &\stackrel{d}{=}
        \fr{\xi''(q_\bm)}{\xi'_t(q_\bm)}
        \cdot \fr{\bm \tbz^\top + \tbz \bm^\top}{N}
        + \fr{\<\bm,\tbz\>_N}{\xi'_t(q_\bm) + q_\bm \xi''(q_\bm)}
        \lt(\xi^{(3)}(q_\bm) - \fr{2\xi''(q_\bm)^2}{\xi'_t(q_\bm)}\rt)
        \fr{\bm\bm^\top}{N}
        + \bM \\
        &= - \xi''(q_\ast) \fr{\bm \bx^\top + \bx \bm^\top}{N}
        + \fr{2\xi''(q_\ast)}{\xi'_t(q_\ast)}
        \lt((1-q_\ast)\xi''(q_\ast) + \fr{1}{1-q_\ast}\rt)
        \fr{\bm\bm^\top}{N} \\
        &+ \fr{q_\ast (-\xi'_t(q_\ast) + (1-q_\ast) \xi''(q_\ast) + \fr{1}{1-q_\ast})}{\xi'_t(q_\ast) + q_\ast \xi''(q_\ast)}
        \lt(\xi^{(3)}(q_\ast) - \fr{2\xi''(q_\ast)^2}{\xi'_t(q_\ast)}\rt) \fr{\bm\bm^\top}{N}
        + \bM
        + \bDelM \\
        &\stackrel{\eqref{eq:def-qt}}{=} - \xi''(q_\ast) \fr{\bm \bx^\top + \bx \bm^\top}{N}
        + \lt(\fr{2\xi''(q_\ast)}{q_\ast} + (1-q_\ast) \xi^{(3)}(q_\ast) \rt) \fr{\bm\bm^\top}{N}
        + \bM
        + \bDelM.
    \ealn
    Then
    \baln
        \nabla^2 \cF_\sTAP(\bm)
        &= \nabla^2 \tH_{N,t}(\bm)
        + \xi''(q_\ast) \fr{\bx \bx^\top}{N}
        - \lt( (1-q_\ast)\xi^{(3)}(q_\ast) - \xi''(q_\ast) + \fr{1}{(1-q_\ast)^2} \rt)
        \fr{2\bm\bm^\top}{N} \\
        &- \lt((1-q_\bm)\xi''(q_\bm) + \fr{1}{1-q_\bm} \rt) \bI
        + \bDelM.
    \ealn
    Combining gives the conclusion.
    \rev{The coefficients above range over a compact set for $t\in[t_0,T]$, so the $o_\iota(1)$ bound is uniform in $t$.}
\end{proof}
\begin{lem}
    \label{lem:gaussian-mtx-coupling}
    Let $\iota > 0$ be sufficiently small and $|q_\bm-q_\ast| \le \iota$.
    Fix an orthonormal basis $\be_1,\ldots,\be_N$ of $\bbR^N$ as discussed before \eqref{eq:def-bM}.
    Let $\bM$ be as in \eqref{eq:def-bM}.
    Let $\bM_\ast$ be sampled from the same law, except with $q_\bm$ replaced by $q_\ast$, and with
    \[
        \bM_\ast(i,j) = 0, \qquad \forall i,j \in \{1,2\}.
    \]
    There is a coupling of $\bM, \bM_\ast$ such that with probability $1-e^{-cN}$, $\tnorm{\bM - \bM_\ast}_{\op} \le o_\iota(1)$.
\end{lem}
\begin{proof}
    Let $\bM'$ be the matrix with $\bM'(i,j) = 0$ for all $i,j\in \{1,2\}$, and otherwise $\bM'(i,j) = \bM(i,j)$.
    Since the $\bM(i,j)$ have variance $O(N^{-1})$, with probability $1-e^{-cN}$, $\tnorm{\bM - \bM'}_{\op} \le \iota$.

    \rev{The variance coefficients in \eqref{eq:def-bM} are uniformly Lipschitz for $t\in [t_0, T]$.
    Hence the following $O(\iota)$ and $o_\iota(1)$ bounds, and the same $c>0$, are uniform in $t$.}
    For all $(i,j) \not \in \{1,2\}^2$, $|\EE \bM(i,j)^2 - \EE \bM_\ast(i,j)^2| \le O(\iota)/N$.
    We couple $\bM$ and $\bM_\ast$ as follows.
    If $\EE \bM(i,j)^2 \le \EE \bM_\ast(i,j)^2$, we first sample $\bM(i,j)$ from its law, and then sample
    \[
        \bM_\ast(i,j) = \bM(i,j) + \upsilon_{i,j} g_{i,j},
    \]
    for $g_{i,j} \sim \cN(0,1/N)$ and suitable $\upsilon_{i,j} = O(\iota^{1/2})$.
    If $\EE \bM(i,j)^2 \ge \EE \bM_\ast(i,j)^2$, we follow a similar procedure, sampling $\bM_\ast(i,j)$ first.
    Let $\bE = \bM' - \bM_\ast$.
    Then
    \[
        \bE(i,j) = (\eps_{i,j} \upsilon_{i,j} g_{i,j})_{i,j \in [N]},
    \]
    for some (deterministic) signs $\eps_{i,j} \in \{\pm 1\}$.
    Let $\upsilon = \max_{i,j} (\upsilon_{i,j})$.
    There exists a random symmetric Gaussian matrix $\bE'$, independent of $\bE$, such that $\bE + \bE' =_d \upsilon \bG$, where $\bG \sim \GOE(N)$.
    Define
    \[
        \cK = \lt\{\bA \in \bbR^{N\times N}~\text{symmetric} : \tnorm{\bA}_{\op} \le 3\upsilon \rt\}.
    \]
    Note that
    \[
        \PP(\bE + \bE' \not\in \cK) = \PP(\tnorm{\bG}_{\op} > 3) \le e^{-cN},
    \]
    while by convexity of $\cK$ and symmetry of $\bE'$,
    \[
        \PP(\bE + \bE' \not\in \cK | \bE \not\in \cK) \ge \fr12.
    \]
    It follows that $\PP(\bE \not\in \cK) \le 2e^{-cN}$, concluding the proof.
\end{proof}
\begin{lem}
    \label{lem:spherical-SK-ground-state}
    Let $\bG \sim \GOE(N)$ and $\bg \sim \cN(0, \bI_N/N)$.
    For any $a,b \in \bbR$, $\iota > 0$,
    \[
        \lt|
            \sup_{\bv \in \bbS^{N-1}}
            \lt\{a\la \bG \bv, \bv \ra + 2b \la \bg, \bv \ra\rt\}
            - 2\sqrt{a^2+b^2}
        \rt| \le \iota
    \]
    with probability $1-e^{-cN}$.
\end{lem}
\begin{proof}
    By \cite[Proposition 1.1]{chen2017parisi},
    \[
        \plim_{N\to\infty}
        \sup_{\norm{\bv}_2 = 1}
        a\la \bG \bv, \bv \ra + 2b \la \bg, \bv \ra
        = 2\sqrt{a^2+b^2}.
    \]
    For each fixed $\bv \in \bbS^{N-1}$, $a\la \bG \bv, \bv \ra + 2b \la \bg, \bv \ra$ has variance $O(N^{-1})$.
    The result follows from Borell-TIS.
    \rev{The conclusion is uniform when $(a,b)$ ranges over a fixed compact set, by a finite net.}
\end{proof}

\begin{ppn}
    \label{ppn:conditional-tap-hessian}
    Let $\nabla^2 \cF_\sTAP(\bm)$ be as in Eq.~\eqref{eq:conditional-nabla2-F}.
    There exist $C^\spec_{\max} > C^\spec_{\min} > 0$ such that for sufficiently small $\iota > 0$,
    \[
        \spec(\nabla^2 \cF_\sTAP(\bm)) \subseteq [-C^\spec_{\max}, -C^\spec_{\min}]
    \]
    with probability $1-e^{-cN}$.
\end{ppn}
\begin{proof}
    Let $\tbm = \bm / \norm{\bm}_2$, $\tbx = P_\bm^\perp \bx / \tnorm{P_\bm^\perp \bx}_2$.
    Throughout this proof, we will denote by $\bDelM$, $\bDelM_1$,  $\bDelM_2$, and so on
     error terms with the same meaning as in Corollary~\ref{cor:nabla2-F-conditional-law},
     namely  $(\bx,\bm,\bz)$-measurable symmetric matrices satisfying
     $\tnorm{\bDelM}_{\op} \le o_\iota(1)$, whose kernel contains $\spn(\bx,\bm,\bz)^\perp$.
     \rev{All covariance parameters and deterministic coefficients below range over compact sets as $t\in[t_0,T]$. 
     Hence all error terms and operator-norm bounds below are uniform in $t$.}
     In particular
    \[
        q_\ast \tbm\tbm^\top - \fr{\bm\bm^\top}{N}=:\bDelM_1,\;\;
        (1-q_\ast) \tbx\tbx^\top - \fr{(\bx - \bm)(\bx - \bm)^\top}{N} =: \bDelM_2.
    \]
    Let $\be_1,\ldots,\be_N$ be an orthonormal basis of $\bbR^N$ with $\be_1 = \tbm$, $\be_2 = \tbx$.
    Let $\bM_\ast$ be defined in Lemma~\ref{lem:gaussian-mtx-coupling}, coupled to $\bM$ so that $\tnorm{\bM-\bM_\ast}_{\op} \le o_\iota(1)$ with probability $1-e^{-cN}$.
    Taking $\iota$ small, it suffices to show
    \beq
        \label{eq:spectral-bound-goal}
        -C^\spec_{\max} \bI \preceq \bA \preceq -C^\spec_{\min} \bI,
    \eeq
    for
    \baln
        \bA &=
        \lt((2+q_\ast) \xi''(q_\ast) - q_\ast(1-q_\ast) \xi^{(3)}(q_\ast) - \fr{2q_\ast}{(1-q_\ast)^2} \rt) \tbm\tbm^\top
        + (1-q_\ast) \xi''(q_\ast) \tbx\tbx^\top \\
        & - \lt((1-q_\ast) \xi''(q_\ast) + \fr{1}{1-q_\ast} \rt) \bI
        + \bM_\ast.
    \ealn
    \rev{Comparing $\bM_\ast$ to a fixed multiple of a GOE, as in Lemma~\ref{lem:gaussian-mtx-coupling}, gives $\tnorm{\bM_\ast}_{\op}=O(1)$ with probability $1-e^{-cN}$, uniformly in $t$.}
    On this event, all terms in $\bA$ have bounded operator norm, and thus $-C^\spec_{\max} \bI \preceq \bA$.
    For the upper bound in \eqref{eq:spectral-bound-goal}, let
    \[
        \psi = \xi''(q_\ast) + q_\ast \xi^{(3)}(q_\ast) - \fr{q_\ast \xi''(q_\ast)^2}{\xi'_t(q_\ast)},
    \]
    which (recall Remark~\ref{rmk:variance-positive-by-cs}) is nonnegative.
    Then
    \baln
        \bA
        &= \lt(
            -(1-q_\ast) \psi
            + q_\ast \lt(\xi''(q_\ast) - \fr{1}{(1-q_\ast)^2}\rt)
            - (1-q_\ast)^2 \lt(\xi''(q_\ast) - \fr{1}{(1-q_\ast)^2}\rt)^2
        \rt) \be_1\be_1^\top \\
        &- \fr{1}{1-q_\ast} \be_2\be_2^\top
        - \lt((1-q_\ast)\xi''(q_\ast) + \fr{1}{1-q_\ast}\rt) \sum_{i=3}^N \be_i\be_i^\top
        + \bM_\ast
    \ealn
    \rev{By \eqref{eq:amp-works} and compactness, there exists a uniform $c_0>0$ such that}
    \baln
        \bA &\preceq \bA' - c_0 (\be_1\be_1^\top + \be_2\be_2^\top), \qquad \text{where} \\
        \bA' &=
        - (1-q_\ast) \psi \be_1\be_1^\top
        - (1-q_\ast) \xi''(q_\ast) \be_2\be_2^\top
        - \lt((1-q_\ast)\xi''(q_\ast) + \fr{1}{1-q_\ast}\rt) \sum_{i=3}^N \be_i\be_i^\top
        + \bM_\ast \\
        &=
        - (1-q_\ast) \psi \be_1\be_1^\top
        - (1-q_\ast) \xi''(q_\ast) \sum_{i=2}^N \be_i\be_i^\top
        - \fr{1}{1-q_\ast} \sum_{i=3}^N \be_i\be_i^\top
        + \bM_\ast.
    \ealn
    By \eqref{eq:def-bM}, (with $\bM_\ast(i,j)$ having the same meaning as above)
    \baln
        \lt(\bM_\ast(1,i) : 3\le i\le N\rt)
        &\stackrel{d}{=} \sqrt{\psi} \bg^1, &
        \bg^1 &\sim \cN(0,\bI_{N-2}/N), \\
        \lt(\bM_\ast(2,i) : 3\le i\le N\rt)
        &\stackrel{d}{=} \sqrt{\xi''(q_\ast)} \bg^2, &
        \bg^2 &\sim \cN(0,\bI_{N-2}/N), \\
        \lt(\bM_\ast(i,j) : 3\le i,j\le N\rt)
        &\stackrel{d}{=} \sqrt{\xi''(q_\ast) \cdot \fr{N-2}{N}} \bG &
        \bG &\sim \GOE(N-2),
    \ealn
    and $\bg^1,\bg^2,\bG$ are independent.
    Fix $a,b$ with $a^2+b^2 \le 1$ and consider temporarily the restricted set
    \[
        \bbS^{N-1}_{a,b} = \lt\{\bv \in \bbS^{N-1} : \la \bv, \be_1 \ra = a, \la \bv, \be_2 \ra = b\rt\}.
    \]
    For any $\bv \in \bbS^{N-1}_{a,b}$ we can write
    \[
        \bv = a \be_1 + b \be_2 + \sqrt{1-a^2-b^2} \bw,
    \]
    where $\bw \in \bbS^{N-1}_{0,0}$.
    Because we defined $\bM_\ast(i,j)=0$ for all $i,j \in \{1,2\}$,
    \baln
        \la \bM_\ast \bv, \bv \ra
        &=
        2a \sqrt{(1-a^2-b^2) \psi} \la \bg^1, \bw \ra
        + 2b \sqrt{(1-a^2-b^2) \xi''(q_\ast)} \la \bg^2, \bw \ra \\
        &+ (1-a^2-b^2) \sqrt{\xi''(q_\ast) \cdot \fr{N-2}{N}} \la \bG \bw, \bw \ra.
    \ealn
    By Lemma~\ref{lem:spherical-SK-ground-state}, with probability $1-e^{-cN}$,
    \[
        \bigg|\sup_{\bv \in \bbS^{N-1}_{a,b}} \la \bM_\ast \bv, \bv \ra - 2\sqrt{f(a,b)}\bigg| \le \iota,
    \]
    where
    \baln
        f(a,b)
        &= (1-a^2-b^2)^2 \xi''(q_\ast)
        + a^2 (1-a^2-b^2) \psi
        + b^2 (1-a^2-b^2) \xi''(q_\ast) \\
        &= (1-a^2-b^2)\lt((1-a^2)\xi''(q_\ast) + a^2 \psi\rt).
    \ealn
    On this event, for all $\bv \in \bbS^{N-1}_{a,b}$,
    \baln
        \la \bA \bv, \bv \ra
        &\le
        - (1-q_\ast) \psi a^2
        - (1-q_\ast) \xi''(q_\ast) (1-a^2)
        - \fr{1-a^2-b^2}{1-q_\ast}
        + 2\sqrt{f(a,b)}
        - c_0 (a^2+b^2)
        + \iota \\
        &=
        -\lt(
            \sqrt{(1-q_\ast) \lt((1-a^2)\xi''(q_\ast) + a^2 \psi\rt)}
            - \sqrt{\fr{1-a^2-b^2}{1-q_\ast}}
        \rt)^2
        - c_0 (a^2+b^2)
        + \iota.
    \ealn
    \rev{At $a=b=0$, the first term is bounded above by a negative constant uniformly in $t$, by \eqref{eq:amp-works}.
    Compactness in $(q_\ast,a,b)$ therefore gives a uniform $c_1>0$ such that, for all $a^2+b^2\le1$,}
    \[
        -\lt(
            \sqrt{(1-q_\ast) \lt((1-a^2)\xi''(q_\ast) + a^2 \psi\rt)}
            - \sqrt{\fr{1-a^2-b^2}{1-q_\ast}}
        \rt)^2
        - c_0 (a^2+b^2)
        \le -c_1.
    \]
    We have thus shown that, for fixed $a,b$, with probability $1-e^{-cN}$,
    \beq
        \label{eq:subsphere-concavity-bd}
        \sup_{\bv \in \bbS^{N-1}_{a,b}} \la \bA \bv, \bv \ra
        \le -c_1 + \iota.
    \eeq
    \rev{Recall that $\tnorm{\bA}_{\op}=O(1)$ with probability $1-e^{-cN}$, uniformly in $t$. Thus the map}
    \[
        (a,b) \mapsto \sup_{\bv \in \bbS^{N-1}_{a,b}} \la \bA \bv, \bv \ra
    \]
    \rev{is uniformly $1/2$-H\"older.
    By a union bound, with probability $1-e^{-cN}$ uniformly in $t$, \eqref{eq:subsphere-concavity-bd} holds for all $(a,b)$ in a fixed $\iota^2$-net of $a^2+b^2 \le 1$.}
    On this event,
    \[
        \sup_{\bv \in \bbS^{N-1}} \la \bA \bv, \bv \ra
        \le -c_1 + O(\iota).
    \]
    \rev{Taking $C^\spec_{\min}=c_1/2$ and $\iota$ uniformly small completes the proof with uniform $C^\spec_{\min},C^\spec_{\max}$ and $c$.}
\end{proof}
\begin{proof}[Proof of Proposition~\ref{ppn:conditional-tap-gradient-hessian}]
    \rev{Let $\iota'=h(\iota)$ be the uniform choice from Proposition~\ref{ppn:conditional-tap-gradient}.}
    By this proposition, for any $\bm \in \cS_0$, $r\in I_\iota$,
    \[
        \PP\lt(
            E_1(\bm,\iota')^c
            \big|
            \nabla_\sp \tH_{N,t}(\bm) = \bzero,
            \tpartial_\rd \tH_{N,t}(\bm) = r
        \rt)
        \le e^{-cN}.
    \]
    Since $\tnorm{\nabla \cF_\sTAP(\bm)}_N \le \iota'$ on $E_1(\bm,\iota')$, and $\nabla_\sp \tH_{N,t}(\bm), \tpartial_\rd \tH_{N,t}(\bm)$ are $\nabla \cF_\sTAP(\bm)$-measurable,
    \[
        \PP\lt(
            E_1(\bm,\iota') \cap E_2(\bm)^c
            \big|
            \nabla_\sp \tH_{N,t}(\bm) = \bzero,
            \tpartial_\rd \tH_{N,t}(\bm) = r
        \rt)
        \le
        \sup_{\tnorm{\bz}_N \le \iota'}
        \PP \lt(
            E_2(\bm)^c
            \big| \nabla \cF_\sTAP(\bm) = \bz
        \rt).
    \]
    By Corollary~\ref{cor:nabla2-F-conditional-law} and Proposition~\ref{ppn:conditional-tap-hessian}, this last probability is $\le e^{-cN}$.
    This completes the proof.
\end{proof}

\subsection{Existence and uniqueness of TAP fixed point at small $t$}
\label{subsec:tap-existence-uniqueness-small-t}

\rev{In this subsection we set a sufficiently small $t_0 \in (0,T)$ and prove Proposition~\ref{ppn:tap-existence-uniqueness} for $t\in [0,t_0)$.
\begin{ppn}
    \label{ppn:tap-fixed-ball-concavity}
    There exist constants $r,C^{\min}_\ssc,C^{\max}_\ssc,c>0$, depending only on $\xi$, such that with probability $1-e^{-cN}$, it holds simultaneously for every $t\in [0,T]$ and $\bm\in\Ball_N(\bzero,r)$ that
    \beq
        \label{eq:tap-fixed-ball-concavity}
        -C^{\max}_\ssc\bI
        \preceq \nabla^2\cF_\sTAP(\bm;\by_t)
        \preceq -C^{\min}_\ssc\bI.
    \eeq
\end{ppn}
\begin{proof}
    Condition on $\bx \in S_N$, and let $\widehat\bx = \bx / \sqrt{N}$.
    A direct calculation shows
    \beq
        \label{eq:tap-hessian-at-origin}
        \nabla^2\cF_\sTAP(\bzero;\by_t)
        \stackrel{d}{=}
        \sqrt{\xi''(0)}\bW + \xi''(0) \widehat\bx{\widehat\bx}^\top - (1+\xi''(0))\bI,
    \eeq
    where $\bW \sim \GOE(N)$.
    Recall from \eqref{eq:amp-works} that $\xi''(0) < 1$.
    Thus the random matrix $\bW + \sqrt{\xi''(0)} \widehat\bx{\widehat\bx}^\top$ has a subcritical BBP spike, and its largest eigenvalue limits in probability to $2$.
    By standard concentration properties of the largest eigenvalue, for $\eta = (1 - \sqrt{\xi''(0)})^2 / (2\sqrt{\xi''(0)})$,
    \[
        \bbP\lt(
            \lambda_{\max}(\bW + \sqrt{\xi''(0)} \widehat\bx{\widehat\bx}^\top)
            \le 2 + \eta
        \rt) \ge 1 - e^{-cN}
    \]
    for some $c$ depending only on $\xi''(0)$.
    On this event,
    \[
        \lambda_{\max}\lt(\nabla^2\cF_\sTAP(\bzero;\by_t)\rt)
        \le 2\sqrt{\xi''(0)} + \sqrt{\xi''(0)}\eta - (1+\xi''(0))
        = -\sqrt{\xi''(0)}\eta.
    \]
    We also have $\lambda_{\min}(\bW) \ge -3$ with probability $1-e^{-cN}$, and on this event
    \[
        \lambda_{\min}\lt(\nabla^2\cF_\sTAP(\bzero;\by_t)\rt)
        \ge -\lt(1+\xi''(0)+3\sqrt{\xi''(0)}\rt).
    \]
    This proves \eqref{eq:tap-fixed-ball-concavity} at $\bm = \bzero$.
    On the event $K_N$ from Proposition~\ref{ppn:gradients-bounded}, the third derivative bound therein implies \eqref{eq:tap-fixed-ball-concavity} for every $\bm\in\Ball_N(\bzero,r)$, for suitable $r$.
    Finally note that all the events above are shared by all $t\in [0,T]$, since changing $\by_t$ changes only the linear term in $\cF_\sTAP(\cdot;\by_t)$.
\end{proof}

\begin{lem}
    \label{lem:tap-small-t-fixed-point}
    There exist $t_0\in(0,T)$ and $c,C>0$, depending only on $(\xi,T)$, such that for every $t\in[0,t_0)$, with probability $1-e^{-cN}$, $\cF_\sTAP(\cdot;\by_t)$ has a unique critical point $\bm^\sTAP$ in $\Ball_N(\bzero,r)$, which further lies in $\Ball_N(\bzero,r/4)$.
    Moreover,
    \beq
        \label{eq:small-t-tap-norm}
        c\max(\sqrt{t},\sqrt{q_*(t)})
        \le \|\bm^\sTAP\|_N
        \le C\min(\sqrt{t},\sqrt{q_*(t)}).
    \eeq
    and $\cS_0 \subseteq \Ball_N(\bzero,r/2)$ for all $t\in [0, t_0)$.
\end{lem}
\begin{proof}
    Condition on $\bx$.
    Since $\by_t =_d t\bx + \sqrt{t} \bg$ for $\bg \sim \cN(0,\bI)$, there are $c,C>0$, common to $t\in(0,T]$, such that with probability $1-e^{-cN}$,
    \[
        c\sqrt{t}
        \le \|\by_t\|_N
        \le C\sqrt{t}.
    \]
    If we set $t_0$ small enough that $C\sqrt{t_0} < \fr14 C^{\min}_\ssc r$, the strong concavity from Proposition~\ref{ppn:tap-fixed-ball-concavity} ensures that $\cF_\sTAP(\cdot;\by_t)$ has a unique critical point $\bm^\sTAP$ in $\Ball_N(\bzero,r)$, which satisfies
    \[
        \fr{\|\by_t\|_N}{C^{\max}_\ssc}
        \le \|\bm^\sTAP\|_N
        \le \fr{\|\by_t\|_N}{C^{\min}_\ssc}.
    \]
    In particular the upper bound implies $\bm^\sTAP \in \Ball_N(\bzero,r/4)$.
    The last two displays also imply
    \[
        c\sqrt{t}
        \le \|\bm^\sTAP\|_N
        \le C\sqrt{t}
    \]
    after adjusting $c,C$ (which still depend only on $\xi,T$).
    After shrinking $t_0$ if necessary, \eqref{eq:qt-linear} gives the estimates involving $q_*(t)$ in \eqref{eq:small-t-tap-norm} and $\cS_0 \subseteq \Ball_N(\bzero,r/2)$ for all $t\in [0,t_0)$.
\end{proof}
Having fixed $t_0$, we may take $\iota>0$ sufficiently small that $\cS_\iota\subseteq\Ball_N(\bzero,2r/3)$ for all $t\in[0,t_0)$.

\begin{proof}[Proof of Proposition~\ref{ppn:tap-existence-uniqueness}, for $t\in \lbrack 0,t_0)$]
    Fix $\upsilon>0$ sufficiently small depending on $\iota,r,C^{\min}_\ssc$, and then fix $k$ sufficiently large in Proposition~\ref{ppn:amp-performance}.
    By \eqref{eq:amp-limiting-overlap} and \eqref{eq:amp-limiting-gradient}, with probability $1-e^{-cN}$,
    \[
        \bm^k\in\cS_\upsilon\subseteq\cS_\iota\subseteq\Ball_N(\bzero,r)
    \]
    and $\|\nabla\cF_\sTAP(\bm^k)\|_N\le\upsilon$.
    On the events in Proposition~\ref{ppn:tap-fixed-ball-concavity} and Lemma~\ref{lem:tap-small-t-fixed-point}, \eqref{eq:tap-fixed-ball-concavity} gives $\|\bm^k-\bm^\sTAP\|_N \le \upsilon / C^{\min}_\ssc$.
    Consequently,
    \begin{align*}
        |\<\bx,\bm^\sTAP\>_N-q_*|
        &\le \upsilon+\frac{\upsilon}{C^{\min}_\ssc}, &
        |\|\bm^\sTAP\|_N^2-q_*|
        &\le \upsilon+\frac{2r\upsilon}{C^{\min}_\ssc}.
    \end{align*}
    Thus $\bm^\sTAP\in\cS_\iota$ after possibly decreasing $\upsilon$.
    Lemma~\ref{lem:tap-small-t-fixed-point} gives uniqueness in $\cS_\iota$, and \eqref{eq:tap-fixed-ball-concavity} gives \eqref{eq:mtap-well-conditioned}, after decreasing $C^\spec_{\min}$ and increasing $C^\spec_{\max}$ if necessary.
    This proves part~\ref{itm:tap-existence-uniqueness}.
    If $\bm\in\cS_\iota$ is an $\iota$-approximate critical point, the same estimate gives
    \[
        \|\bm-\bm^\sTAP\|_N
        \le \frac{\|\nabla\cF_\sTAP(\bm)\|_N}{C^{\min}_\ssc}
        \le \frac{\iota}{C^{\min}_\ssc}.
    \]
    Hence part~\ref{itm:tap-approx-crits} holds with $\iota'=\iota/C^{\min}_\ssc$.
\end{proof}

We will use the following lemma in the proof of Proposition~\ref{ppn:additional-event-pre-conditioning}\ref{itm:gd-near-crit}.
\begin{lem}
    \label{lem:tap-small-t-amp-bound}
    For every fixed $K_\sAMP$ sufficiently large, there exist $c,C>0$, depending only on $(\xi,T,\iota,K_\sAMP)$, such that for every $t\in[0,t_0)$, with probability $1-e^{-cN}$,
    \[
        \|\bm^\sAMP\|_N\le C\sqrt{q_*(t)},
        \qquad
        \|\bm^\sAMP-\bm^\sTAP\|_N\le C\sqrt{q_*(t)}.
    \]
\end{lem}
\begin{proof}
    Choose $\tau\in(0,t_0)$ so small that \eqref{eq:qt-linear} gives $q_*(t)\asymp t$ for $0\le t<\tau$ and the induction below remains in $\Ball_N(\bzero,1)$.
    For $0\le t<\tau$, the proof of Lemma~\ref{lem:scalar-amp-recursion-converge} and \eqref{eq:qt-linear} give
    \[
        0\le q_k\le q_*(t)\le Ct,
    \]
    while $\|\by_t\|_N\le C\sqrt{t}$ with probability $1-e^{-cN}$, as in the proof of Lemma~\ref{lem:tap-small-t-fixed-point}.
    On $K_N$, the derivative bounds of Proposition~\ref{ppn:gradients-bounded} give by induction in \eqref{eq:main-amp}
    \[
        \max_{0\le k\le K_\sAMP}
        \big(\|\bm^k\|_N+\|\bw^k\|_N\big)
        \le C\sqrt{t}.
    \]
    Thus $\|\bm^\sAMP\|_N\le C\sqrt{q_*(t)}$ for $t<\tau$.
    For $\tau\le t<t_0$, \eqref{eq:amp-limiting-overlap} gives
    \[
        \|\bm^\sAMP\|_N^2
        \le q_*(t)+\iota
        \le \left(1+\frac{\iota}{q_*(\tau)}\right)q_*(t)
    \]
    with probability $1-e^{-cN}$.
    The second conclusion follows from \eqref{eq:small-t-tap-norm}.
\end{proof}

}

\subsection{Proof of conditioning bound}
\label{subsec:proof-local-concavity-and-conditioning}

Propositions~\ref{ppn:gradients-bounded} and \ref{ppn:tap-existence-uniqueness}, together with Lemma~\ref{lem:tap-local-domain} below, imply parts \ref{itm:gradients-bounded} and \ref{itm:crit-unique} of Proposition~\ref{ppn:local-concavity-and-conditioning}.
\rev{Decrease $\iota$, if necessary, so that $\iota<q_*(t_0)/2$.
For $t<t_0$, include the event from Lemma~\ref{lem:tap-small-t-fixed-point} among those defining $\cE_0$ (which we recall is an event on which \ref{itm:gradients-bounded}--\ref{itm:amp-near-crit} from Proposition~\ref{ppn:local-concavity-and-conditioning} hold).
By \eqref{eq:small-t-tap-norm} and $\bm^\sTAP\in\cS_\iota$ for $t\ge t_0$, there are constants $0<c_0<1<C_0$ such that on $\cE_0$,
\beq
    \label{eq:tap-norm-shell}
    c_0q_*\le\|\bm^\sTAP\|_N^2\le C_0q_*.
\eeq}

\rev{
\begin{lem}\label{lem:tap-local-domain}
For every sufficiently small $\beta_0>0$, we may take $t_0$ small
enough that, uniformly for $t\in(0,T]$, the TAP point belongs to
$\mathcal D_\iota(t,\bx)$ with probability $1-e^{-cN}$, conditional on any
fixed $\bx$.
\end{lem}
\begin{proof}
For $t\ge t_0$, the claim follows from \eqref{eq:tap-norm-shell} and $\bm^\sTAP\in\cS_\iota$.
For $t<t_0$, let
$\bB=-\nabla^2\cF_\sTAP(\bzero)$ and let
$\by_t=t\bx+\sqrt t\,\bg$. Conditional on $\bx$, $\bB$ is
independent of $\bg\sim\cN(0,\bI_N)$. On the fixed-ball concavity
and derivative events, Taylor's formula at zero gives
\[
 \bm^\sTAP=\bB^{-1}(t\bx+\sqrt t\,\bg)+\mathbf r,
 \qquad \|\mathbf r\|_N\le C\|\bm^\sTAP\|_N^2.
\]
Since $\|\bB^{-1}\|_{\op}\le C$ and
$\|\bm^\sTAP\|_N\asymp\sqrt t$, it follows that
\[
 \frac{|\langle\bx,\bm^\sTAP\rangle_N|}{\|\bm^\sTAP\|_N}
 \le C\sqrt t+C|\langle\bx,\bB^{-1}\bg\rangle_N|.
\]
Conditional on $\bB$, the last Gaussian has variance at most $C/N$.
Taking $t_0$ small proves the claim. We may intersect this event with
the previous definition of $\cE_0$ without changing its probability bound.
\end{proof}}

We now prove the remainder of this proposition.

\begin{proof}[Proof of Proposition~\ref{ppn:local-concavity-and-conditioning}\ref{itm:amp-near-crit}]
    \rev{Choose $\upsilon\le \iota/2$ so that the $o_\upsilon(1)$ error term in Proposition~\ref{ppn:tap-existence-uniqueness}\ref{itm:tap-approx-crits} is at most $\iota/2$, uniformly in $t\in[0,T]$.
    By \eqref{eq:amp-limiting-overlap} and \eqref{eq:amp-limiting-gradient}, for sufficiently large $K_\sAMP$, with probability $1-e^{-cN}$,
    \[
        \bm^\sAMP\in\cS_\upsilon\subseteq\cS_{\iota/2},
        \qquad
        \|\nabla\cF_\sTAP(\bm^\sAMP)\|_N\le\upsilon.
    \]
    Intersecting with the event in Proposition~\ref{ppn:tap-existence-uniqueness}\ref{itm:tap-approx-crits} gives $\|\bm^\sAMP-\bm^\sTAP\|_N\le\iota/2$, with conditional probability $1-e^{-cN}$ given any fixed $\bx$, where $c$ is uniform in $t\in[0,T]$.}
\end{proof}

We now turn to the proof of part~\ref{itm:conditioning-bound}.
Define
\[
    K(\bm) = P_{\spn(\bx,\bm)}^\perp \nabla^2 \cF_\sTAP(\bm) P_{\spn(\bx,\bm)}^\perp.
\]
We will treat this as an $(N-2)\times (N-2)$ matrix, after a suitable change of coordinates.
The following fact is a consequence of Corollary~\ref{cor:nabla2-F-conditional-law}.
\begin{fac}
    \label{fac:Km-form}
    Let $\tnorm{\bm}_N^2 = q_\bm < 1$.
    Conditional on \rev{$\bx$ and} $\nabla \cF_\sTAP(\bm) = \bzero$,
    \[
        K(\bm) \stackrel{d}{=} - \lt((1-q_\bm) \xi''(q_\bm) + \fr{1}{1-q_\bm} \rt) \bI + \sqrt{\xi''(q_\bm) \fr{N-2}{N}} \bG,
        \qquad \bG \sim \GOE(N-2).
    \]
\end{fac}
The next fact is verified by direct calculation.
\begin{fac}
    \label{fac:nabla-tH-law}
    For any $\bm$, $\nabla \tH_{N,t}(\bm)$ is Gaussian, with variance $\xi'_t(q_\bm) + q_\bm \xi''(q_\bm)$ in the direction of $\bm$ and $\xi'_t(q_\bm)$ in all directions orthogonal to $\bm$.
\end{fac}
We will need the following technical lemma, which we prove in Subsection~\ref{subsec:goe-determinants}.
\begin{lem}
    \label{lem:kacrice-fullspace-computation}
    \rev{Set $q_\bm=\|\bm\|_N^2$, and let $c_0,C_0$ be as in \eqref{eq:tap-norm-shell}.
    For every sufficiently small $\iota>0$, there is $C=C(\xi,T,\iota)>0$ such that, uniformly for $t\in[0,T]$ and $\bx\in S_N$,}
    \[
        \int_{\rev{\substack{\bm\in\cS_\iota\\ c_0q_*\le q_\bm\le C_0q_*}}}
        \EE\lt[
            |\det K(\bm)|
            \big|
            \rev{\bx,}\nabla \cF_\sTAP(\bm) = \bzero
        \rt]
        \varphi_{\nabla \cF_\sTAP(\bm)\rev{\mid\bx}}(\bzero)
        ~\de^N (\bm)
        \le C.
    \]
\end{lem}
\begin{proof}[Proof of Proposition~\ref{ppn:local-concavity-and-conditioning}\ref{itm:conditioning-bound}]
    \rev{Fix $\bx \in S_N$ and condition on it throughout the proof.}
    We will apply Lemma~\ref{lem:conditioning-lemma-v2} with $\cF_\sTAP$ for $\cF$, $\bm^\sAMP$ for $\bm_0$, the interior of $\cS_\iota$ for $D$, and $C^\spec_{\min}/2$ for $c_\spec$.
    Note that \eqref{eq:def-eps-local-concavity} implies $\eps \le c_{\spec}^2/10c_{\op}$.
    We next verify that the event $\cE$ is contained in the event $\cE_{\cond}$ defined in Lemma~\ref{lem:conditioning-lemma-v2}.
    Suppose $\cE$ holds.
    Then the event $\cH(c_{\op})$ holds by part \ref{itm:gradients-bounded}.
    $\tnorm{\nabla \cF_\sTAP(\bm^\sAMP)}_N \le \eps$ by definition, and by parts \ref{itm:gradients-bounded}, \ref{itm:crit-unique}, and \ref{itm:amp-near-crit},
    \[
        \lambda_{\max} (\nabla^2 \cF_\sTAP(\bm^\sAMP))
        \le \lambda_{\max} (\nabla^2 \cF_\sTAP(\bm^\sTAP))
        + c_{\op} \norm{\bm^\sAMP - \bm^\sTAP}_{\op,N}
        \le -C^{\spec}_{\min} + c_{\op} \iota / 2
        \le -C^{\spec}_{\min} / 2
    \]
    for small enough $\iota$.
    Thus $\cG(\eps,c_\spec)$ holds.
    We have $\norm{\bm^\sAMP}_N \le 1$ because $\bm^\sAMP \in \cS_{\iota/2}$, by part \ref{itm:amp-near-crit}.
    \rev{Also, \eqref{eq:def-eps-local-concavity} and $\bm^\sAMP\in\cS_{\iota/2}$ imply
    $U=\Ball_N(\bm^\sAMP,5\eps/c_\spec)\subseteq\cS_\iota$.}

    \rev{By Eq.~\eqref{eq:nabla-cF-sTAP}, $\EE[\nabla\cF_\sTAP(\bm)\mid\bx]$ is continuous in $\bm$, while for each fixed $t>0$, Fact~\ref{fac:nabla-tH-law} gives $\Cov(\nabla\cF_\sTAP(\bm)\mid\bx)\succeq t\bI$ on $\cS_\iota$.
    Thus Lemma~\ref{lem:conditioning-lemma-v2} applies.
    The constant below is uniform in $t$ by Lemma~\ref{lem:kacrice-fullspace-computation}.}
    \rev{Conditionally on $\nabla\cF_\sTAP(\bm)=\bzero$, the event $\cE\cap\{\bm\in U\}$ implies that $\bm$ is a critical point in $\cS_\iota$.
    Uniqueness on $\cE$ then gives $\bm=\bm^\sTAP\in\mathcal D_\iota(t,\bx)$, so the integrand below vanishes outside this domain.}
    By this lemma,
    \[
        \rev{\begin{aligned}
        \EE\lt[X\ind\{\cE\}\rev{\mid\bx}\rt]
        &\le
        \int_{\rev{\bm\in\mathcal D_\iota(t,\bx)}}
        \EE\lt[
            |\det \nabla^2 \cF_\sTAP(\bm)|
            X \ind\{\rev{\cE \cap \{\bm\in U\}}\}
            \big|
            \rev{\bx,}\nabla \cF_\sTAP(\bm) = \bzero
        \rt] \\
        &\qquad{}\times
        \varphi_{\nabla \cF_\sTAP(\bm)\rev{\mid\bx}}(\bzero)
        ~\de^N (\bm).
        \end{aligned}}
    \]
    By H\"older's inequality,
    \baln
        &\EE\lt[
            |\det \nabla^2 \cF_\sTAP(\bm)|
            X \ind\{\rev{\cE \cap \{\bm\in U\}}\}
            \big|
            \rev{\bx,}\nabla \cF_\sTAP(\bm) = \bzero
        \rt] \\
        &\le
        \EE\lt[
            |\det \nabla^2 \cF_\sTAP(\bm)|^{1+\delta^{-1}}
            \ind\{\rev{\cE \cap \{\bm\in U\}}\}
            \big|
            \rev{\bx,}\nabla \cF_\sTAP(\bm) = \bzero
        \rt]^{\delta/(1+\delta)} \\
        &\qquad \times
        \EE\lt[
            X^{1+\delta}
            \ind\{\cE\}
            \big|
            \rev{\bx,}\nabla \cF_\sTAP(\bm) = \bzero
        \rt]^{1/(1+\delta)}.
    \ealn
    \rev{On $\cE\cap\{\bm\in U\}$, Eq.~\eqref{eq:conditioning-lem-strong-concavity} and $\cH(c_{\op})$ imply the eigenvalues of $\nabla^2\cF_\sTAP(\bm)$ lie in $[-c_{\op},-C^\spec_{\min}/4]$.
    Interlacing therefore gives}
    \[
        \rev{|\det \nabla^2 \cF_\sTAP(\bm)| \le c_{\op}^2 |\det K(\bm)|.}
    \]
    Thus,
    \baln
        &\EE\lt[
            |\det \nabla^2 \cF_\sTAP(\bm)|^{1+\delta^{-1}}
            \ind\{\rev{\cE \cap \{\bm\in U\}}\}
            \big|
            \rev{\bx,}\nabla \cF_\sTAP(\bm) = \bzero
        \rt]^{\delta/(1+\delta)} \\
        &\le
        \rev{c_{\op}^2}
        \EE\lt[
            |\det K(\bm)|^{1+\delta^{-1}}
            \big|
            \rev{\bx,}\nabla \cF_\sTAP(\bm) = \bzero
        \rt]^{\delta/(1+\delta)} \\
        &\le
        \rev{c_{\op}^2 C'_\delta}
        \EE\lt[
            |\det K(\bm)|
            \big|
            \rev{\bx,}\nabla \cF_\sTAP(\bm) = \bzero
        \rt].
    \ealn
    \rev{Indeed, for $q_\bm\in[0,q_*(T)+\iota]$, Eq.~\eqref{eq:amp-works} gives, for some $\kappa>0$,
    \[
        r_\bm:=(1-q_\bm)\sqrt{\xi''(q_\bm)}
        +\frac{1}{(1-q_\bm)\sqrt{\xi''(q_\bm)}}\ge 2+\kappa.
    \]
    Since Fact~\ref{fac:Km-form} gives
    \[
        K(\bm)=\sqrt{\xi''(q_\bm)\fr{N-2}{N}}
        \lt(\bG-\sqrt{\fr N{N-2}}\,r_\bm\bI\rt),
    \]
    Lemma~\ref{lem:goe-concentration} with $p=1+\delta^{-1}$ gives the last estimate, with $C'_\delta$ uniform in $t$.}
    Combining,
    \baln
        \EE\lt[X\ind\{\cE\}\rev{\mid\bx}\rt]
        &\le
        \rev{c_{\op}^2 C'_\delta}
        \int_{\rev{\bm\in\mathcal D_\iota(t,\bx)}}
        \EE\lt[
            |\det K(\bm)|
            \big|
            \rev{\bx,}\nabla \cF_\sTAP(\bm) = \bzero
        \rt]
        \varphi_{\nabla \cF_\sTAP(\bm)\rev{\mid\bx}}(\bzero)
        ~\de^N (\bm) \\
        &\qquad \times
        \sup_{\bm \in \mathcal D_\iota(t,\bx)}
        \EE\lt[
            X^{1+\delta}
            \ind\{\cE\}
            \big|
            \rev{\bx,}\nabla \cF_\sTAP(\bm) = \bzero
        \rt]^{1/(1+\delta)}.
    \ealn
    \rev{Lemma~\ref{lem:kacrice-fullspace-computation} bounds this integral by a uniform $C>0$, so the result follows with $C_\delta=c_{\op}^2C'_\delta C$.}
\end{proof}

\subsection{Determinant concentration and estimate of Kac--Rice integral}
\label{subsec:goe-determinants}

In this subsection, we provide the deferred proofs of Lemmas~\ref{lem:goe-concentration} and \ref{lem:kacrice-fullspace-computation}.
These are the final ingredients for the proof of Proposition~\ref{ppn:local-concavity-and-conditioning}.
\begin{proof}[Proof of Lemma~\ref{lem:goe-concentration}]
    \rev{Fix $\kappa>0$ and $p\ge1$, and let
    \[
        \cE_\kappa=\{\|\bG\|_{\op}\le 2+\kappa/2\},
        \qquad
        f_r(x)=\log\max(|r-x|,\kappa/2).
    \]
    Uniformly for $r\ge2+\kappa$, the function $f_r$ is $2/\kappa$-Lipschitz and $\PP(\cE_\kappa^c)\le e^{-c_\kappa N}$. By \cite[Theorem 1.1(b)]{guionnet2000concentration},
    \[
        \PP\lt(\big|\Tr f_r(\bG)-\EE\Tr f_r(\bG)\big|\ge s\rt)
        \le 2\exp(-\kappa^2s^2/32).
    \]
    Since $|\det(r\bI-\bG)|\le e^{\Tr f_r(\bG)}$, with equality on $\cE_\kappa$, the same concentration estimate and Cauchy--Schwarz give
    \baln
        \EE|\det(r\bI-\bG)|^p
        &\le C_{\kappa,p}e^{p\EE\Tr f_r(\bG)},\\
        \EE|\det(r\bI-\bG)|
        &\ge (1-C_\kappa e^{-c_\kappa N})e^{\EE\Tr f_r(\bG)}. \qedhere
    \ealn}
\end{proof}
\begin{lem}
    \label{lem:goe-det-precise}
    Let $\bG \sim \GOE(N)$.
    \rev{For every $\kappa>0$, there exists $C_\kappa>0$ such that, uniformly for $r\ge 2+\kappa$,}
    \[
        \rev{\EE [|\det (r\bI - \bG)|] \le C_\kappa \exp(N\Phi(r)),}
    \]
    where
    \[
        \Phi(r) = \fr14 r^2 - \fr12 - \fr14 r \sqrt{r^2-4} + \log \fr{r + \sqrt{r^2-4}}{2}.
    \]
\end{lem}
\begin{proof}
    This follows from \cite[Lemmas 2.1 and 2.2(i)]{belius2022triviality} with $N+1$ for $N$ and $\sqrt{\fr{N}{2(N+1)}} r$ for $x$.
    Note that the matrix $\GOE_{N-1}(N^{-1})$ therein is defined with typical spectral radius $\sqrt{2}$, while our $\GOE(N)$ has spectral radius $2$.
    \rev{The cited proof gives the bound uniformly for $r\ge 2+\kappa$.}
\end{proof}

\begin{proof}[Proof of Lemma~\ref{lem:kacrice-fullspace-computation}]
    \rev{Fix $\bx \in S_N$.}
    Throughout this proof, $C = C(\xi,T,\iota)>0$ is a constant, uniform over $\bm \in \cS_\iota$, which may change from line to line.
    Let $\|\bm\|^2_N = q_\bm$ and $\<\bx,\bm\>_N= q_\bx$, so that $q_\bm,q_\bx \in I(\iota)$.
    By Fact~\ref{fac:Km-form}, for $\bG \sim \GOE(N-2)$,
    \baln
        \EE \lt[
            |\det K(\bm)|
            \big|
            \rev{\bx,}\nabla \cF_\sTAP(\bm) = \bzero
        \rt]
        &= \lt(\xi''(q_\bm) \cdot \fr{N-2}{N} \rt)^{(N-2)/2}
        \EE \lt[ \lt|
            \det \lt( \sqrt{\fr{N}{N-2}} r_\bm  \bI - \bG \rt)
        \rt|\rt] \\
        \text{where} \quad
        r_\bm &=
        (1-q_\bm) \xi''(q_\bm)^{1/2} + \fr{1}{(1-q_\bm) \xi''(q_\bm)^{1/2}}.
    \ealn
    \rev{By \eqref{eq:amp-works}, $r_\bm\ge2+\kappa$ uniformly for $q_\bm\in I(\iota)$, for some $\kappa>0$.
    Lemma~\ref{lem:goe-det-precise} gives
    \baln
        \EE \lt[
            |\det K(\bm)|
            \big|
            \bx,\nabla \cF_\sTAP(\bm) = \bzero
        \rt]
        &\le C\lt(\xi''(q_\bm)\fr {N-2}N\rt)^{(N-2)/2}
        \exp\lt((N-2)\Phi\lt(\sqrt{\fr N{N-2}}\,r_\bm\rt)\rt) \\
        &\le C \xi''(q_\bm)^{N/2} \exp(N\Phi(r_\bm))
        = C \exp(Nf_1(q_\bm)),
    \ealn}
    where $f_1(q_\bm)=\fr12\log\xi''(q_\bm)+\Phi(r_\bm)$.
    On the other hand, by \eqref{eq:nabla-cF-sTAP},
    \[
        \nabla \cF_\sTAP(\bm)
        = \nabla \tH_{N,t}(\bm) + \xi'_t(q_\bx) \lt(\bx - \fr{q_\bx}{q_\bm} \bm \rt)
        - \lt((1-q_\bm) \xi''(q_\bm) + \fr{1}{1-q_\bm} - \fr{q_\bx \xi'_t(q_\bx)}{q_\bm} \rt) \bm.
    \]
    Since $\bx - \fr{q_\bx}{q_\bm} \bm$ is orthogonal to $\bm$, Fact~\ref{fac:nabla-tH-law} yields
    \rev{
    \[
        \varphi_{\nabla \cF_\sTAP(\bm)\rev{\mid\bx}}(\bzero)
        =
        \lt(\frac{\xi'_t(q_\bm)}{\xi'_t(q_\bm)+q_\bm\xi''(q_\bm)}\rt)^{1/2}
        \exp(Nf_2(q_\bm,q_\bx))
        \le \exp(Nf_2(q_\bm,q_\bx)),
    \]
    where}
    \baln
        f_2(q_\bm,q_\bx)
        &= -\fr12 \log (2\pi \xi'_t(q_\bm)) - \fr{\xi'_t(q_\bx)^2}{2\xi'_t(q_\bm)}\lt(1 - \fr{q_\bx^2}{q_\bm}\rt) \\
        &- \fr{q_\bm}{2(\xi'_t(q_\bm) + q_\bm \xi''(q_\bm))} \lt((1-q_\bm) \xi''(q_\bm) + \fr{1}{1-q_\bm} - \fr{q_\bx \xi'_t(q_\bx)}{q_\bm} \rt)^2.
    \ealn
    Combining the above and applying a polar change of coordinates,
    \balnn
        \label{eq:fullspace-complexity-2d-integral}
        &\int_{
            \rev{\substack{
                \bm\in\cS_\iota\\
                c_0q_*\le q_\bm\le C_0q_*
            }}
        }
        \EE \lt[
            |\det K(\bm)|
            \big|
            \rev{\bx,}\nabla \cF_\sTAP(\bm) = \bzero
        \rt]
        \varphi_{\nabla \cF_\sTAP(\bm)\rev{\mid\bx}}(\bzero)
        ~\de^N (\bm) \\
        \nonumber
        &\le CN \int_{
            \rev{\substack{
                c_0q_*\le q_\bm\le C_0q_*,\ q_\bx^2<q_\bm \\
                |q_\bm-q_*|\le\iota,\ |q_\bx-q_*|\le\iota
            }}
        }
        \rev{(q_\bm-q_\bx^2)^{-3/2}}
        \exp\lt(N(f_1(q_\bm) + f_2(q_\bm,q_\bx) + f_3(q_\bm,q_\bx))\rt)
        ~\de q_\bx \de q_\bm,
    \ealnn
    where
    \[
        f_3(q_\bm,q_\bx) = \fr12 + \fr12 \log (2\pi (q_\bm - q_\bx^2)).
    \]
    Let
    \baln
        F(q_\bm,q_\bx) &= f_1(q_\bm) + f_2(q_\bm,q_\bx) + f_3(q_\bm,q_\bx) \\
        &= \fr12 + \fr12 \log \fr{q_\bm - q_\bx^2}{\xi'_t(q_\bm) (1-q_\bm)^2}
        + \fr12 (1-q_\bm)^2 \xi''(q_\bm)
        - \fr{\xi'_t(q_\bx)^2}{2\xi'_t(q_\bm)}\lt(1 - \fr{q_\bx^2}{q_\bm}\rt) \\
        &- \fr{q_\bm}{2(\xi'_t(q_\bm) + q_\bm \xi''(q_\bm))} \lt((1-q_\bm) \xi''(q_\bm) + \fr{1}{1-q_\bm} - \fr{q_\bx \xi'_t(q_\bx)}{q_\bm} \rt)^2.
    \ealn
    We note for use below that by direct differentiation using \eqref{eq:def-qt},
    \balnn\label{e:F-value-and-stationarity}
        F(q_*,q_*) &= 0, & 
        \nabla F(q_*,q_*) &= 0.
    \ealnn
    Let
    \[
        \Delta_0 = \xi''(q_\ast) - \fr{1}{(1-q_\ast)^2} \stackrel{\eqref{eq:amp-works}}{<} 0
    \]
    and
    \[
        \Delta_1 = \fr{(1-q_\ast)^3 \Delta_0^3 - q_\ast (1-q_\ast) \Delta_0^2}{2q_\ast^2 (1 + (1-q_\ast) \xi''(q_\ast))}, \qquad
        \Delta_2 = - \fr{(1-q_\ast)^2}{q_\ast} \Delta_0^2 + \Delta_0.
    \]
    Then
    \[
        \fr{\partial^2 F}{\partial q_\bm^2} (q_\ast,q_\ast) = \Delta_1 + \Delta_2, \qquad
        \fr{\partial^2 F}{\partial q_\bm \partial q_\bx} (q_\ast,q_\ast) = -\Delta_2, \qquad
        \fr{\partial^2 F}{\partial q_\bx^2} (q_\ast,q_\ast) = \Delta_2.
    \]
    It follows that
    \beq\label{e:F-hessian}
        \nabla^2 F(q_\ast,q_\ast) = \Delta_1 (1,0)^{\otimes 2} + \Delta_2 (1,-1)^{\otimes 2}.
    \eeq
    \rev{Let $t_0$ be a sufficiently small constant depending on only $(\xi,T)$.
    For the rest of the proof, we consider the cases $t\in [t_0,T]$ and $t\in [0,t_0)$ separately.

    Suppose first $t \in [t_0,T]$.
    Then $q_* = q_*(t)$ ranges over a compact subset of $(0,1)$, so if $\iota$ is sufficiently small, $(q_\bm-q_\bx^2)^{-3/2}$ is uniformly bounded for $q_\bm, q_\bx \in I(\iota)$.
    Both $\Delta_1$ and $\Delta_2$ are negative, and since $q_*$ ranges over a compact subset of $(0,1)$ in this case, they are uniformly negative.
    Thus \eqref{e:F-hessian} implies the existence of $c>0$ depending on only $(\xi,T,\iota)$ such that, uniformly for $t\in [t_0,T]$,}
    \[
        F(q_\bm,q_\bx)\le -c\lt((q_\bm-q_*)^2+(q_\bx-q_*)^2\rt).
    \]
    \rev{Together with \eqref{e:F-value-and-stationarity}, this implies the two-dimensional integral in \eqref{eq:fullspace-complexity-2d-integral} is $O(N^{-1})$, canceling the factor $N$ and finishing this case.

    We turn to the case $t<t_0$.
    Below, we will set $s$ sufficiently small depending on $\xi$, and then $t_0$ sufficiently small depending on $(\xi,T,s)$.
    We make the change of variables $q_\bm = q_* u$, $q_\bx = \sqrt{q_*} v$.
    Then the right-hand side of \eqref{eq:fullspace-complexity-2d-integral} is at most
    \beq\label{e:F-integral-rescaled}
        CN\int_\cR
        (u-v^2)^{-3/2}\exp\lt(NF(q_*u,\sqrt{q_*}v)\rt)
        \,\de u\,\de v,
    \eeq
    where
    \[
        \cR = \{(u,v) \in \bbR^2 : c_0 \le u \le C_0, v^2 < u\}.
    \]
    By \eqref{eq:qt-linear}, as $q_*\downarrow0$ the functions $F_{q_*}(u,v) = F(q_*u,\sqrt{q_*}v)$ converge in $C^2$ on compact subsets of $\cR$ to
    \baln
        F_0(u,v)
        &= \frac{1+\xi''(0)}{2}
        + \frac12\log\frac{u-v^2}{1-\xi''(0)+\xi''(0)u}
        - \frac{\xi''(0)^2v^2}{2(1-\xi''(0)+\xi''(0)u)}\lt(1-\frac{v^2}{u}\rt) \\
        &\qquad - \frac{u}{2(1-\xi''(0)+2\xi''(0)u)}
        \lt(1+\xi''(0)-\frac{\xi''(0)v^2}{u}\rt)^2.
    \ealn
    Recall $\xi''(0) < 1$ by \eqref{eq:amp-works}.
    Direct calculation shows that for all $u$, $\partial_v F_0(u,v)$ has the same sign as $-v$, while $\partial_u F_0(u,0)$ has the same sign as $1-u$.
    Thus $F_0$ has the unique maximum $F_0(1,0) = 0$.
    Direct calculation also shows that
    \[
        \nabla^2 F_0(1,0) = -(1-\xi''(0))^2 \mathsf{diag}\lt(\fr{1-\xi''(0)}{2(1+\xi''(0))}, 1\rt) \prec 0.
    \]
    By continuity of $\nabla^2 F_0$, there exists $c > 0$ and an open neighborhood $\cU$ of $(1,0)$, depending only on $\xi$, such that $\nabla F_0(u,v) \preceq -c$ for $(u,v)$ in an open neighborhood $\cU$ of $(1,0)$.

    We will separately bound the integral \eqref{e:F-integral-rescaled} in three regions:
    \baln
        \cR_1 &= \{(u,v) \in \cR : (u-1)^2 + v^2 \le s\}, \\
        \cR_2 &= \{(u,v) \in \cR : u - v^2 < s\}, \\
        \cR_3 &= \cR \setminus (\cR_1 \cup \cR_2).
    \ealn
    We set $s$ small enough that $\cR_1 \subseteq \cU$, and that $\cR_1$ and $\cR_2$ are disjoint.
    Then these sets form a partition of $\cR$.

    Since the functions $F_{q_*}$ converge in $C^2$ on $\cR_1$ to $F_0$, and $\cR_1 \subseteq \cU$, we may set $t_0$ sufficiently small that for all $t\in [0,t_0)$, $\nabla^2 F_{q_*} \preceq -c/2$ on $\cR_1$.
    By shrinking $t_0$ if necessary, we can also ensure $(1,\sqrt{q_*(t)}) \in \cR_1$ for all $t \in [0,t_0)$.

    Since $\cR_1$ is disjoint from $\cR_2$, $u-v^2$ is uniformly bounded away from $0$ over $\cR_1$.
    Recalling \eqref{e:F-value-and-stationarity}, this implies
    \[
        CN\int_{\cR_1}
        (u-v^2)^{-3/2}\exp\lt(NF(q_*u,\sqrt{q_*}v)\rt)
        \,\de u\,\de v
        \le O(1).
    \]
    Since $F_0$ has unique maximum $F_0(1,0) = 0$, there exists $c>0$ depending on only $\xi$ such that
    \[
        \max_{(u,v) \in \overline{\cR}_3} F_0(u,v) \le -c,
    \]
    where $\overline{\cR}_3$ is the closure of $\cR_3$.
    By setting $t_0$ sufficiently small, we can ensure that for all $t\in [0,t_0)$,
    \[
        \max_{(u,v) \in \overline{\cR}_3} F_{q_*}(u,v) \le -c/2.
    \]
    Since $u-v^2$ is also uniformly bounded away from $0$ over $\cR_3$, this yields
    \[
        CN\int_{\cR_3}
        (u-v^2)^{-3/2}\exp\lt(NF(q_*u,\sqrt{q_*}v)\rt)
        \,\de u\,\de v
        \le O(1) Ne^{-cN/2}.
    \]
    Finally, inspecting the definition of $F$ shows that there exists $C' > 0$ depending on only $\xi$ such that
    \[
        F_{q_*}(u,v) - \fr12 \log (u-v^2) \le C'.
    \]
    Thus
    \[
        CN\int_{\cR_2}
        (u-v^2)^{-3/2}\exp\lt(NF(q_*u,\sqrt{q_*}v)\rt)
        \,\de u\,\de v
        \le CNe^{C'N} \int_{\cR_2} (u-v^2)^{(N-3)/2}\,\de u\,\de v
        \le O(1) (se^{2C'})^{(N-3)/2}\,.
    \]
    This is exponentially small if we set $s$ sufficiently small.}
\end{proof}

\subsection{Algorithmic guarantees and Lipschitz continuity of correction}
\label{sec:LipCorrection}

In this subsection, we prove Proposition~\ref{ppn:additional-event-pre-conditioning}.

\begin{proof}[Proof of Proposition~\ref{ppn:additional-event-pre-conditioning}\ref{itm:crit-band}]
    By \eqref{eq:amp-dominate-gibbs} and the definition \eqref{eq:defI-iota},
    \[
        \mu_t(\Band(\bm^\sAMP,I(\iota/2)) \rev{\cap \Band(\bx,I(\iota/2))}) \ge 1-e^{-cN}.
    \]
    Since $\tnorm{\bm^\sAMP - \bm^\sTAP}_N \le \iota/2$, we have
    \[
        \rev{\Band(\bm^\sTAP,I(\iota)) \cap \Band(\bx,I(\iota))}
        \supseteq
        \rev{\Band(\bm^\sAMP,I(\iota/2)) \cap \Band(\bx,I(\iota/2))}.
    \]
\end{proof}

\begin{proof}[Proof of Proposition~\ref{ppn:additional-event-pre-conditioning}\ref{itm:gd-near-crit}]
    \rev{On $K_N$, both extremal eigenvalues of $\nabla^2\cF_\sTAP(\bm)$ are uniformly Lipschitz in a fixed neighborhood of $\bm^\sTAP$.
    By \eqref{eq:mtap-well-conditioned}, after decreasing $\iota$ if necessary,}
    \[
        \spec(\nabla^2 \cF_\sTAP(\bm)) \subseteq \lt[-2C^\spec_{\max}, -\fr12 C^\spec_{\min}\rt],
        \qquad \forall \tnorm{\bm - \bm^\sTAP}_N \le \iota.
    \]
    Thus $\nabla^2 \cF_\sTAP$ is strongly concave and well-conditioned in the convex region $\tnorm{\bm - \bm^\sTAP}_N \le \iota$.
    It is classical (see, e.g., \cite{nesterov2003introductory}) that for suitable $\eta > 0$, gradient \rev{ascent}
    \[
        \rev{\bu^{k+1} = \bu^k + \eta \nabla \cF_\sTAP(\bu^k)}
    \]
    initialized from $\bu^0$ in this region satisfies
    \[
        \tnorm{\bu^k - \bm^\sTAP}_N
        \le (1-\eps)^k \tnorm{\bu^0 - \bm^\sTAP}_N
    \]
    \rev{for some $\eps > 0$. By Proposition~\ref{ppn:local-concavity-and-conditioning}\ref{itm:amp-near-crit}, $\bu^0 = \bm^\sAMP$ lies in this region.
    For $t<t_0$, Lemma~\ref{lem:tap-small-t-amp-bound} gives $\tnorm{\bm^\sAMP - \bm^\sTAP}_N \le C\sqrt{q_*}$.
    For $t\ge t_0$, the same bound follows from $\tnorm{\bm^\sAMP - \bm^\sTAP}_N \le \iota/2$ and $q_*\ge q_*(t_0)$.
    The contraction keeps every iterate in the region $\tnorm{\bm - \bm^\sTAP}_N \le \iota$.}
    Recalling $\bm^\GD = \bu^{K_\GD(N)}$ and $K_\GD(N) = \lfloor K^*_\GD \log N\rfloor$, we conclude
    \[
        \tnorm{\bm^\GD - \bm^\sTAP}_N
        \rev{\le C\sqrt{q_*}(1-\eps)^{K_\GD(N)}
        \le \sqrt{q_*}\,N^{-10}}
    \]
    for suitably large $K^*_\GD$.
    This implies part~\ref{itm:gd-near-crit}.
\end{proof}
We now turn to the proof of part \ref{itm:LipCorr}. 
Recall from \eqref{eq:GammaStarCorr} that $\bI_{N-1}$
denotes the identity operator on $\Ts_\bm$; we sometimes 
write this as $\bI_{N-1}^\bm$ to emphasize the 
dependence on $\bm$.
\rev{We also recall the definitions of the tensors 
$\bA^{(2)}$ and $\bA^{(3)}$ from Eq.~\eqref{eq:A2A3-Def}.}
\begin{lem}
    \label{lem:mtap-looks-like-goe}
    Let $\gamma_\ast = (1-q_\ast)^{-1} + (1-q_\ast) \xi''(q_\ast)$.
    Let $\iota, \bm^\sTAP \in \cS_\iota$ be as in Proposition~\ref{ppn:local-concavity-and-conditioning}.
    \rev{There exists $\iota' = o_\iota(1)$ such that, conditional on any $\bx \in S_N$, with probability $1-e^{-cN}$ the point $\bm^\sTAP$ is defined and satisfies}
    \beq
        \label{eq:event-bA2-spec}
        \spec ~\bA^{(2)}(\bm^\sTAP) \subseteq \rev{[-2\sqrt{\xi''(q_\ast)}-\iota',2\sqrt{\xi''(q_\ast)}+\iota']}
    \eeq
    and
    \beq
        \label{eq:event-bA2-trinv}
        \lt|\fr1N \Tr \lt((\gamma_\ast \bI_{N-1} - \bA^{(2)}(\bm^\sTAP))^{-1}\rt) - (1-q_\ast) \rt| \le \iota'.
    \eeq
\end{lem}
\begin{proof}
    \rev{Fix $\bx \in S_N$.}
    Let $\cE_\spec$ be the event that both \eqref{eq:event-bA2-spec} and \eqref{eq:event-bA2-trinv} hold, and let $\cE$ be as in Proposition~\ref{ppn:local-concavity-and-conditioning}.
    By Proposition~\ref{ppn:local-concavity-and-conditioning}\ref{itm:conditioning-bound} with $\delta = 1/2$,
    \baln
        \PP(\cE_\spec^c\rev{\mid\bx})
        &\le \PP(\cE^c\rev{\mid\bx})
        + \PP(\cE_\spec^c \cap \cE\rev{\mid\bx}) \\
        &\le e^{-cN}
        + C_{1/2}
        \sup_{\bm \in \mathcal D_\iota(t,\bx)}
        \PP \lt[
            \cE_\spec^c \cap \cE \big| \rev{\bx,}\nabla \cF_\sTAP(\bm) = \bzero
        \rt]^{1/2}.
    \ealn
    We will show that this probability is $e^{-cN}$, uniformly in $\bm \in \mathcal D_\iota(t,\bx)$.
    Note that on $\cE$, we have deterministically $\bm^\sTAP = \bm$.

    \rev{Let $q_\bm = \norm{\bm}_N^2$. Since $\bm\in\mathcal D_\iota(t,\bx)\subseteq\cS_\iota$, we have $|q_\bm-q_\ast|\le\iota$.}
    One checks analogously to Fact~\ref{fac:derivative-laws} that conditional on \rev{$\bx$ and} $\nabla \cF_\sTAP(\bm) = \bzero$, we have $\bA^{(2)}(\bm) =_d \sqrt{\xi''(q_\bm) \fr{N-1}{N}} \bG$, $\bG \sim \GOE(N-1)$.
    \rev{The GOE edge bound and uniform continuity of $\sqrt{\xi''}$ give \eqref{eq:event-bA2-spec}, with $\iota'=o_\iota(1)$, with probability $1-e^{-cN}$ under this conditional law.}
    Note that by \eqref{eq:amp-works},
    \beq
        \label{eq:gamma-ast-bounded-from-bulk}
        \gamma_\ast - 2 \sqrt{\xi''(q_\ast)}
        = \fr{1}{1-q_\ast} \lt(1 - (1-q_\ast) \xi''(q_\ast)^{1/2}\rt)^2 > 0.
    \eeq
    \rev{The right-hand side has a positive minimum for $q_\ast\in[0,q_*(T)]$, by \eqref{eq:amp-works} and compactness. Thus, for small enough $\iota$, \eqref{eq:event-bA2-spec} leaves a common gap between $\gamma_\ast$ and $\spec~\bA^{(2)}(\bm)$.}
    Recall the semicircle measure
    \beq
        \label{eq:def-semicircle}
        \rho_\smc(\lambda) = \fr{1}{2\pi} \sqrt{4-\lambda^2} ~\de \lambda.
    \eeq
    Applying \cite[Theorem 1.1(b)]{guionnet2000concentration} as in the proof of Lemma~\ref{lem:goe-concentration} shows that with probability $1-e^{-cN}$ \rev{under the same conditional law},
    \[
        \lt|\fr1N \Tr\lt((\gamma_\ast \bI_{N-1} - \bA^{(2)}(\bm))^{-1} \rt)
        - \int \fr{\rho_\smc(\de \lambda)}{\gamma_\ast - \sqrt{\xi''(q_\bm)} \lambda}\rt|
        \le \iota.
    \]
    \rev{Since the denominator is uniformly positive on the support of $\rho_\smc$, the integral is uniformly continuous in $(q_\ast,q_\bm)$ and equals $1-q_\ast$ when $q_\bm=q_\ast$.}
    Thus, for suitable $\iota'$, \eqref{eq:event-bA2-trinv} \rev{holds with probability $1-e^{-cN}$ conditional on $\bx$ and $\nabla \cF_\sTAP(\bm)=\bzero$}, as desired.
\end{proof}
\begin{lem}
    \label{lem:bA2-bA3-lipschitz}
    Suppose the event $K_N$ in Proposition~\ref{ppn:gradients-bounded} holds.
    \rev{There exists $L>0$ such that, for $j=2,3$ and all nonzero
    $\bm_1,\bm_2$ of norm at most $1$,
    \[
        \tnorm{\bA^{(j)}(\bm_1)-\bA^{(j)}(\bm_2)}_{\op,N}
        \le
        \frac{L}{\min(\tnorm{\bm_1}_N,\tnorm{\bm_2}_N)}
        \tnorm{\bm_1-\bm_2}_N,
    \]
    where the tensors are viewed in the ambient space.}
\end{lem}
\onlyarxiv{\rev{\begin{proof}
    Let $P_\bm^\perp$ be orthogonal projection onto $\Ts_\bm$.  Then
    \[
        \bA^{(j)}(\bm)
        =(P_\bm^\perp)^{\otimes j}\nabla^jH_N(\bm),
        \qquad
        \tnorm{P_{\bm_1}^\perp - P_{\bm_2}^\perp}_{\op,N}
        \le
        \frac{C\tnorm{\bm_1-\bm_2}_N}
        {\min(\tnorm{\bm_1}_N,\tnorm{\bm_2}_N)}.
    \]
    The derivative bounds on $K_N$ give the assertion for $j=2$.
    The case $j=3$ is identical, with a projection in each index.
\end{proof}}}

\begin{lem}
    Recall the definition of the matrix $\bQ(\bm)$
    from Eq.~\eqref{eq:def-Q}.
    \rev{There exists $L>0$ such that, for all $\bx \in S_N$, the following holds with conditional probability $1-e^{-cN}$ given $\bx$:} \rev{for all $\bm_1,\bm_2 \in \Ball_N(\bm^\sTAP,c_{\rm loc}\sqrt{q_*})$}
    (treating $\bQ(\bm_i)$ as a matrix in $\bbR^{N\times N}$)
    \baln
        \tnorm{\bQ(\bm_1)}_{\op,N} &\le L\,, \\
        \rev{|\gamma_{*,N}(\bm_1) - \gamma_{*,N}(\bm_2)|}
        &\rev{\le \frac{L}{\sqrt{q_*}}\tnorm{\bm_1 - \bm_2}_N}\,, \\
        \tnorm{\bQ(\bm_1) - \bQ(\bm_2)}_{\op,N}
        &\le \rev{\frac{L}{\sqrt{q_*}}}\norm{\bm_1 - \bm_2}_N\,.
    \ealn
\end{lem}
\begin{proof}
    Suppose $K_N$ holds and both \eqref{eq:event-bA2-spec} and \eqref{eq:event-bA2-trinv} from Lemma~\ref{lem:mtap-looks-like-goe} hold.
    \rev{By \eqref{eq:tap-norm-shell}, after reducing $c_{\rm loc}$,
    \[
        c\sqrt{q_*}\le\tnorm{\bm}_N\le C\sqrt{q_*},
        \qquad
        \bm\in\Ball_N(\bm^\sTAP,c_{\rm loc}\sqrt{q_*}).
    \]
    Lemmas~\ref{lem:mtap-looks-like-goe} and \ref{lem:bA2-bA3-lipschitz} then show, after reducing $c_{\rm loc}$ again, that $\spec~\bA^{(2)}(\bm)$ stays a fixed distance below $\gamma_*$ and that the solution $\gamma_{*,N}(\bm)$ of \eqref{eq:GammaStarCorr} remains in a fixed neighborhood of $\gamma_*$.
    In particular, as operators on $\Ts_\bm$,
    \[
        c\bI_{N-1}^{\bm}\preceq\bQ(\bm)
        \preceq L\bI_{N-1}^{\bm}.
    \]
    For $\bm_1,\bm_2\in\Ball_N(\bm^\sTAP,c_{\rm loc}\sqrt{q_*})$, choose an orthogonal rotation $\bR$ mapping $\Ts_{\bm_1}$ to $\Ts_{\bm_2}$ such that
    \[
        \tnorm{\bR-\bI_N}_{\op,N}
        \le\frac{C}{\sqrt{q_*}}\tnorm{\bm_1-\bm_2}_N.
    \]
    The displayed bound on $\bR-\bI_N$, the uniform bound $\tnorm{\bA^{(2)}(\bm_2)}_{\op,N}\le C$ on $K_N$, and Lemma~\ref{lem:bA2-bA3-lipschitz} give
    \[
        \tnorm{\bA^{(2)}(\bm_1)
        -\bR^{-1}\bA^{(2)}(\bm_2)\bR}_{\op,N}
        \le\frac{L}{\sqrt{q_*}}\tnorm{\bm_1-\bm_2}_N.
    \]
    Write $\gamma_i=\gamma_{*,N}(\bm_i)$ and $\bQ_i=\bQ(\bm_i)$. Equation~\eqref{eq:GammaStarCorr} gives
    \[
        \frac1N\Tr\bQ_1=1-\tnorm{\bm_1}_N^2,
        \qquad
        \frac1N\Tr(\bR^{-1}\bQ_2\bR)=1-\tnorm{\bm_2}_N^2.
    \]
    On $\Ts_{\bm_1}$, the resolvent identity reads
    \beq\label{e:resolvent-identity}
        \bQ_1 - \bR^{-1}\bQ_2\bR
        = \bQ_1 \big((\gamma_2-\gamma_1)\bI_{N-1}^{\bm_1} +\bA^{(2)}(\bm_1)
        - \bR^{-1}\bA^{(2)}(\bm_2)\bR\big) \bR^{-1}\bQ_2\bR.
    \eeq
    Taking traces and rearranging gives
    \begin{align*}
        (\gamma_1-\gamma_2) \frac1N
        \Tr(\bQ_1\bR^{-1}\bQ_2\bR)
        &=\tnorm{\bm_1}_N^2-\tnorm{\bm_2}_N^2 \\
        &\quad+\frac1N\Tr\big(
            \bQ_1(\bA^{(2)}(\bm_1)
            -\bR^{-1}\bA^{(2)}(\bm_2)\bR)
            \bR^{-1}\bQ_2\bR
        \big).
    \end{align*}
    The trace multiplying $\gamma_1-\gamma_2$ is bounded below by a positive constant, while the right-hand side is at most $Lq_*^{-1/2}\tnorm{\bm_1-\bm_2}_N$ in absolute value.
    This proves the stated bound on $\gamma_{*,N}$.
    Taking operator norms in \eqref{e:resolvent-identity} and using the preceding bounds gives
    \[
        \tnorm{\bQ_1-\bR^{-1}\bQ_2\bR}_{\op,N}
        \le\frac{L}{\sqrt{q_*}}\tnorm{\bm_1-\bm_2}_N.
    \]
    Moreover,
    \[
        \tnorm{\bR^{-1}\bQ_2\bR-\bQ_2}_{\op,N}
        \le 2\tnorm{\bQ_2}_{\op,N}\tnorm{\bR-\bI_N}_{\op,N}
        \le\frac{L}{\sqrt{q_*}}\tnorm{\bm_1-\bm_2}_N.
    \]
    This proves the stated bound on $\bQ(\bm_1)-\bQ(\bm_2)$.}
\end{proof}

\begin{proof}[Proof of Proposition~\ref{ppn:additional-event-pre-conditioning}\ref{itm:LipCorr}]
    \rev{By multilinearity of the contraction defining $\bcorr$, the
    bounds on $K_N$, the preceding lemma and Lemma~\ref{lem:bA2-bA3-lipschitz} imply that,
    after adjusting $L$, for every $\norm{\bv}_2=1$,
    \[
        |\la \bcorr(\bm_1)-\bcorr(\bm_2),\bv\ra|
        \le L\sqrt{\fr{N}{q_*}}\norm{\bm_1-\bm_2}_N.
    \]
    Taking the supremum over $\bv$ gives
    \[
        \norm{\bcorr(\bm_1)-\bcorr(\bm_2)}_2
        \le L\sqrt{\fr{N}{q_*}}\norm{\bm_1-\bm_2}_N,
        \qquad
        \norm{\bcorr(\bm_1)-\bcorr(\bm_2)}_N
        \le \fr{L}{\sqrt{q_*}}\norm{\bm_1-\bm_2}_N.
    \]}
\end{proof}

%
%
\section{Sharp local magnetization estimates and proof of 
Proposition \ref{ppn:local-barycenter}}
\label{sec:LocalComputation}

\rev{The proof of Proposition~\ref{ppn:local-barycenter} makes use
of two technical results on the magnetization of spherical 
models with a small field, which are proven in the companion paper
\cite{huang2026approximating}. For the reader's convenience, we
restate these estimates in Section \ref{sec:FirstLocal}, and subsequently 
use them to prove Proposition \ref{ppn:local-barycenter} in Section 
\ref{sec:SecondLocal}.}

\subsection{Magnetization estimates}
\label{sec:FirstLocal}

\begingroup\color{darkblue}
Consider a spherical mixed spin model
\[
 H(\bsig)=\langle \bu,\bsig\rangle+\sum_{p\geq 2}H_p(\bsig),
 \qquad
 \EE H_{\geq2}(\bsig^1)H_{\geq2}(\bsig^2)
 =N\xi(\langle\bsig^1,\bsig^2\rangle_N),
\]
whose mixture satisfies
\[
 \xi''(0)<1,\qquad
 \xi(q)+q+\log(1-q)<0\quad\text{for every }q\in(0,1).
\]
Write the quadratic component as
$H_2(\bsig)=\langle\bsig,\bW^{(2)}\bsig\rangle$, diagonalize
$\bW^{(2)}$ with eigenvalues $(\Lambda_i)_{i\leq N}$, and let
$\bW^{(3)}$ denote the symmetric tensor representing $H_3$ in this
basis.  \rev{For a symmetric matrix $\bA$, define
\[
 G(\gamma;\bA,\bu):=\gamma-\frac1{2N}\log\det(\gamma\bI-\bA)
 +\frac1{4N}\langle\bu,(\gamma\bI-\bA)^{-1}\bu\rangle.
\]
For $\bA=\bLambda$, write $G(\gamma)=G(\gamma;\bLambda,\bu)$ and
$\gamma_*=\arg\min_{\gamma>\Lambda_1}G(\gamma)$.}

\begin{thm}
\label{lemma:LocalEffective1}
There are constants $c_0,c>0$ such that, if
$\|\bu\|\leq N^{c_0}$, then for all $i\leq N$, with probability at least $1-N^{-c}$, 
\[
 \int\sigma_i\,\mu(\de\bsig)
 =(1+O(N^{-c}))\left[
 \frac{u_i}{2(\gamma_*-\Lambda_i)}
 +\rev{\frac{3}{4(\gamma_*-\Lambda_i)}}
 \sum_{j=1}^N\frac{W^{(3)}_{ijj}}{\gamma_*-\Lambda_j}
 +O(N^{-1/2-c})\right].
\]
\end{thm}

\begin{thm}
\label{lem:correction-band}
For every $\alpha\geq2$ there are constants $c_0,c>0$ such that, if
$\|\bu\|\leq N^{c_0}$, the explicit approximation $\hbm$ defined by
the bracketed expression in Theorem~\ref{lemma:LocalEffective1}
satisfies
\[
 \EE\left[\left\|\int\bsig\,\mu(\de\bsig)-\hbm\right\|_N^\alpha\right]
 =O(N^{-\alpha/2-c}).
\]
\end{thm}

\endgroup

\subsection{Integrating over bands: Proof of Proposition \ref{ppn:local-barycenter}}
\label{sec:SecondLocal}

Using the results in the previous section, we will complete the proof of Proposition \ref{ppn:local-barycenter}. We will assume the setup of Proposition \ref{ppn:local-barycenter}. We sample $\bx \sim \mu_\unif$, $\by = t\bx + \bB_t$, and $\tH(\,\cdot\, ) \sim \mu_{\nul}$ (the Gaussian process with covariance  $\EE \tH(\bsig^1) \tH(\bsig^2) = N\xi(\<\bsig^1,\bsig^2\>_N)$)
with $\bx, \bB, \tH$ independent. We define the tilted disorder $H(\bsig) = \tH(\bsig) + \<\by,\bsig\> + N\xi(\<\bx,\bsig\>_N)$.
Thus $(\bx,H,\by)\sim \bbP$, where $\bbP$ is the planted distribution obtained by
tilting the null model by its Gibbs weight. (For simplicity, we suppress
the dependence of $H$ and $\by$ on $t$.) In this section, we estimate the
mean of the Gibbs measure given by $H$.

\rev{Throughout this subsection, take
$\bm\in\mathcal D_\iota(t,\bx)$ as defined in Eq.~\eqref{eq:local-domain}. }
Recall that
\[
    \cF_\sTAP(\bm) = N \xi(\<\bx, \bm\>_N) + \tH(\bm) + \<\by, \bm\> + \fr{N}{2} \theta(\norm{\bm}_N^2) + \fr{N}{2} \log(1-\norm{\bm}_N^2),
\]
where $\theta(s) = \xi(1)-\xi(s)-(1-s)\xi'(s)$.

Let $\bm\in\mathbb{R}^{N}$ and $q=\|\bm\|_N^{2}$.
\rev{For $\iota>0$, write
\[
    \Band_*(\iota):=\Band(\bm,I(\iota))\cap\Band(\bx,I(\iota)),
\]
where $I(\iota)$ is defined in Eq.~\eqref{eq:defI-iota}.}
The following lemma follows from standard calculations. 
\begin{lem}\label{lem:cond-law}
The distribution of $\tH(\bsig)$ given $\nabla{\cal F}_{\TAP}(\bm)=0,\by,\bx$
is a Gaussian process with
\begin{align}\label{eq:mean-cond}
&N^{-1}\mathbb{E}[\tH(\bsig)\mid\nabla{\cal F}_{\TAP}(\bm)=0,\by,\bx]\nonumber\\
&=\frac{\xi'(\langle \bm,\bsig\rangle_N)\langle \bz,\bsig\rangle_N}{\xi'(q)}-\frac{\xi''(q)\langle \bm,\bz\rangle_N}{\xi'(q)\zeta(q)}\xi'(\langle \bm,\bsig\rangle_N)\langle \bm,\bsig\rangle_N,
\end{align}
with $\zeta(q)=\xi'(q)+q\xi''(q)$ and $\bz=-\by-\xi'(\langle \bx,\bm\rangle_N)\bx+(1-q)\xi''(q)\bm+\frac{\bm}{1-q}$,
and covariance
\begin{align}
 & N^{-1}\Cov[\tH(\bsig^{1}),\tH(\bsig^{2})\mid\nabla{\cal F}_{\TAP}(\bm)=0,\by,\bx]\nonumber\\
 & =\xi(\langle\bsig^{1},\bsig^{2}\rangle_N)-\frac{\xi'(\langle \bm,\bsig^{1}\rangle_N)\xi'(\langle \bm,\bsig^{2}\rangle_N)}{\xi'(q)}\langle\bsig^{1},\bsig^{2}\rangle_N\nonumber \\
 &\qquad \qquad +\frac{\xi''(q)\xi'(\langle \bm,\bsig^{1}\rangle_N)\xi'(\langle \bm,\bsig^{2}\rangle_N)\langle \bm,\bsig^{1}\rangle_N\langle \bm,\bsig^{2}\rangle_N}{\xi'(q)\zeta(q)}. \label{eq:cov-cond}
\end{align}
\end{lem}

\rev{Since $\tH$ is a Gaussian polynomial on $\bbR^N$, Eqs.~\eqref{eq:mean-cond} and \eqref{eq:cov-cond} also extend to $\bsig,\bsig^1,\bsig^2\in\bbR^N$ with $\|\bsig\|_N, \|\bsig^1\|_N, \|\bsig^2\|_N\le 1$. Let $b_*=\langle\bx,\bm\rangle_N$ and
\[
 K_q=\xi'(q)I+\frac{\xi''(q)}N\bm\bm^\top,\qquad
 \mathbf h=\left((1-q)\xi''(q)+\frac1{1-q}\right)\bm-\xi'(b_*)\bx.
\]
We can also compute 
\begin{align}\label{eq:conditional-y}
 \bar\by:=\EE[\by\mid\nabla\cF_\TAP(\bm)=0,\bx]
 &=t\bx+t(K_q+tI)^{-1}(\mathbf h-t\bx),\nonumber\\
 \Cov(\by\mid\nabla\cF_\TAP(\bm)=0,\bx)
 &=tK_q(K_q+tI)^{-1}.
\end{align}}

Let $\bsig^{\perp} = \Proj_{\{\bx,\bm\}^\perp}(\bsig)$ be the projection of $\bsig$ onto $\{\bx,\bm\}^\perp$, and similarly define $\by^{\perp}$, $\bz^{\perp}$. Define the band
\begin{align}
D_{N}(a,b):=\Big\{\bsig\in S_N: \;\;
\langle \bsig, \bm\rangle_N = aq \mbox{ and }\langle \bsig, \bx\rangle_N =  b\Big\}\, ,
\end{align}
and let $r(a,b) = \|\bsig - \bsig^{\perp}\|_N^2$ for $\bsig\in D_N(a,b)$.

Throughout the rest of the section, we will condition on the event $\nabla \cF_\TAP(\bm)=0$, and on $\by-\by^{\perp}$ and $\bx$. Conditional on $\nabla \cF_\TAP(\bm)=0,\by-\by^{\perp},\bx$, we can write
\[
N^{-1}H(\bsig) = \xi(b) + \frac{\xi'(aq)\langle \bz,\bsig\rangle_N}{\xi'(q)}-\frac{\xi''(q)\xi'(aq)aq\langle \bm,\bz\rangle_N}{\xi'(q)\zeta(q)} + N^{-1}\hH(\bsig^\perp)+\langle \by^\perp,\bsig^\perp\rangle_N + \langle \by - \by^{\perp},\bsig - \bsig^{\perp}\rangle_N,
\]
where $\hH$ is a centered Gaussian process with covariance
\begin{align*}
&N^{-1}\mathrm{Cov}(\hH(\bsig^{\perp,1}), \hH(\bsig^{\perp,2})) \\
&= \xi\lt(r(a,b) + \langle \bsig^{\perp,1},\bsig^{\perp,2}\rangle_N\rt) - \frac{\xi'(aq)^2}{\xi'(q)}\langle  \bsig^{\perp,1},\bsig^{\perp,2}\rangle_N - \frac{\xi'(aq)^2r(a,b)}{\xi'(q)} + \frac{\xi'(aq)^2 \xi''(q) (aq)^2}{\zeta(q)\xi'(q)}.
\end{align*}

\rev{Let $n=N-2$, $\lambda_N=N/n$, $s(a,b)=\sqrt{\lambda_N(1-r(a,b))}$, and
$\tilde\bsig=\bsig^\perp/s(a,b)\in S_n\subset V^\perp$, where
$V=\operatorname{span}(\bm,\bx)$. Let
\begin{align}\label{eq:band-coarea}
 J_N=\frac{N-2}{2\pi}\frac{\sqrt q}{\sqrt{1-b_*^2/q}},
 \quad b_*=\langle\bx,\bm\rangle_N.
\end{align}
} 
We can then write
\begin{align}\label{eq:(a,b)-reduction}
&\int_{D_N(a,b)} e^{H(\bsig)}\mu^{a,b}_0(\de\bsig) \nonumber \\
&= J_N\exp\left(N\left[\xi(b) + \frac{\xi'(aq)\langle \bz+\by,\bsig\rangle_N}{\xi'(q)}-\frac{\xi''(q)\xi'(aq)aq\langle \bm,\bz\rangle_N}{\xi'(q)\zeta(q)} + \lt(1-\frac{\xi'(aq)}{\xi'(q)}\rt)\langle \by - \by^{\perp},\bsig - \bsig^{\perp}\rangle_N \right]\right) \nonumber\\
&\qquad \int_{S_{N-2}} \exp\left(N\lt(1-\frac{\xi'(aq)}{\xi'(q)}\rt)s(a,b)\langle \by^\perp, \tilde{\bsig}\rangle_N + \tilde{\thH}(\tilde{\bsig}) + \frac{N-4}{2} \log(1-r(a,b))\right)  \mu_0(\de\tilde{\bsig}) \nonumber\\
&= J_N\exp\left( N\Gamma_{N}(\by,\bm;a,b) + \frac{N-4}{2}\log(1-r(a,b))\right)
\int_{S_{N-2}} e^{N^{1/2}g_{a,b}+{\thH}(\tilde{\bsig})} \mu_0(\de\tilde{\bsig}),
\end{align}
where $\mu^{a,b}_0$ is the measure induced on $D_N(a,b)$ by $\mu_0$
and $\Gamma_N$ is defined by
\begin{align}
    \Gamma_{N}(\by,\bm;a,b) &: =
    \xi(b) + \frac{\xi'(aq)\langle \bz+\by,\bsig\rangle_N}{\xi'(q)}-\frac{\xi''(q)\xi'(aq)aq\langle \bm,\bz\rangle_N}{\xi'(q)\zeta(q)} + \lt(1-\frac{\xi'(aq)}{\xi'(q)}\rt)\langle \by - \by^{\perp},\bsig - \bsig^{\perp}\rangle_N \, , 
\end{align}
and $\thH$ is a Hamiltonian on $\tilde{\bsig}$ with mixture 
$\tilde\xi(z) = \sum_{k\ge 1}\tilde\xi_kz^k$ given by
\begin{align}
\tilde{\xi}_1 &= (1-r(a,b))\left(\xi'(r(a,b)) -\frac{\xi'(aq)^2}{\xi'(q)} + \lt(1-\frac{\xi'(aq)}{\xi'(q)}\rt)^2v\right) =: \tilde{\gamma}_1^2,\label{eq:MixtureAB1}\\
\tilde{\xi}_2 &= \frac{1}{2} \xi''(r(a,b)) (1-r(a,b))^2 =: \tilde{\gamma}_2^2 ,\\
\tilde{\xi}_{p} &=  \frac{1}{p!}\xi^{(p)}(r(a,b))(1-r(a,b))^p , \;\;\,\, p\ge 3\, .
\label{eq:MixtureAB3}
\end{align}
\rev{Here $\by^\perp$ is a centered Gaussian, independent of $\by-\by^\perp$ and of the centered conditional process in Lemma~\ref{lem:cond-law}, with covariance $vI$ on $\{\bx,\bm\}^\perp$, where
\begin{align}\label{eq:conditional-v}
 v=\frac{t\xi'(q)}{t+\xi'(q)}.
\end{align}
}
Finally, $g_{a,b}$ is a Gaussian independent of $\thH$ with standard deviation $\tilde{\gamma}_0$ given by 
\begin{align}
 \tilde{\gamma}_0^2 &:= \xi(r(a,b)) - \frac{\xi'(aq)^2r(a,b)}{\xi'(q)} + \frac{\xi'(aq)^2 \xi''(q) (aq)^2}{\zeta(q)\xi'(q)} \, .
 \end{align}

\rev{
More precisely, for $b_* := \<\bx,\bm\>_N$ and
\begin{align}\label{eq:def-mab}
    \bm(a,b) = \frac{(aq-b_*b)\bm + q(b-ab_*)\bx}{q-b_*^2},
\end{align}
we have
\begin{align}\label{eq:def-gab}
    g_{a,b}
    = N^{-1/2}\lt(\tH(\bm(a,b)) - \EE[\tH(\bm(a,b))\mid\nabla\cF_\TAP(\bm)=\bfzero,\by,\bx]\rt).
\end{align}
For use in Lemma \ref{lem:lip-shift} below, we compute the covariance structure of $g_{a,b}$. 

\begin{lem}\label{lem:shift-cov}
Conditional on $\nabla\cF_\TAP(\bm)=\bfzero$, $\by-\by^\perp$, $\bx$, the
family $(g_{a,b})_{a,b}$ is a centered Gaussian process, independent of
$\by^\perp$ and of $(\thH^{a',b'})_{a',b'}$ (where $\thH^{a,b}$ denotes the Hamiltonian $\thH$ on $D_N(a,b)$), with covariance
\begin{align}\label{eq:shift-kernel}
    K\lt((a,b),(a',b')\rt) := \EE\lt[g_{a,b}g_{a',b'}\rt]
    = \xi(\pi) - \frac{\xi'(aq)\xi'(a'q)}{\xi'(q)}\pi
    + \frac{\xi''(q)\xi'(aq)\xi'(a'q)aa'q^2}{\xi'(q)\zeta(q)},
\end{align}
where
\begin{align}\label{eq:def-pi}
    \pi = \pi\lt((a,b),(a',b')\rt) := \<\bm(a,b),\bm(a',b')\>_N
    = aa'q + \frac{(b-ab_*)(b'-a'b_*)}{1-b_*^2/q}.
\end{align}
In particular $\pi((a,b),(a,b)) = r(a,b)$ and $K((a,b),(a,b)) =
\tilde\gamma_0^2$. Moreover, the following claims hold.
\begin{enumerate}
\item[(a)] On $\nabla \cF_\TAP(\bm)=0$, we have
$\nabla_{(a,b)} g_{a,b}|_{(a,b)=(1,b_*)} = \bfzero$ almost surely, and
\begin{align}\label{eq:shift-kernel-grad}
    \lt.\nabla_{(a,b)} K\lt((a,b),(a',b')\rt)\rt|_{(a,b)=(1,b_*)} = \bfzero
    \qquad \text{for all } (a',b').
\end{align}
\item[(b)] Let $\kappa := \sqrt q$,
$\beta := b_*/\sqrt q\in(-1,1)$, $w = (w_1,w_2) := (\sqrt q(a-1), b-b_*)$ and
$w' = (w_1',w_2') := (\sqrt q(a'-1), b'-b_*)$. Then $aq = \kappa(\kappa+w_1)$,
$aa'q = (\kappa+w_1)(\kappa+w_1')$, and
\begin{align}
    \pi &= (\kappa+w_1)(\kappa+w_1')
    + \frac{(w_2-\beta w_1)(w_2'-\beta w_1')}{1-\beta^2}, \label{eq:pi-scaled}\\
    K &= \xi(\pi) - \frac{\xi'(\kappa(\kappa+w_1))\xi'(\kappa(\kappa+w_1'))}{\xi'(\kappa^2)}
    \lt[\pi - \frac{\kappa^2\xi''(\kappa^2)}{\zeta(\kappa^2)}(\kappa+w_1)(\kappa+w_1')\rt]. \label{eq:K-scaled}
\end{align}
For $t<t_0$, on $\mathcal D_\iota$ we have $|\beta|\le\beta_0<1$. Eq.~\eqref{eq:K-scaled} extends analytically to $\kappa=0$ on a fixed neighborhood $|w|,|w'|\le\eps_0$, with derivatives of every fixed order bounded uniformly. 
\item[(c)] Letting $g_w := g_{a,b}$, for $|w|\le\eps_0$ we have
\begin{align}\label{eq:shift-grad-var}
    \EE\lt[\|\nabla_w g_w\|^2\rt]
    = \sum_{i=1}^2 \lt.\partial_{w_i}\partial_{w_i'} K(w,w')\rt|_{w'=w}
    \le C|w|^2,
\end{align}
with $C$ uniform over $\mathcal D_\iota(t,\bx)$. 
\end{enumerate}
\end{lem}
\begin{proof}
Standard calculations (Lemma \ref{lem:cond-law}) immediately yield the formulas for $K$. For the independence, let $\bsig'\in D_N(a',b')$; since $\bm(a,b),\bm\in\spn(\bm,\bx)\perp\bsig'^{\perp}$, we have $\<\bm(a,b),\bsig'\>_N = \pi$ and $\<\bm,\bsig'\>_N = a'q$, so by Eq.~\eqref{eq:cov-cond} the covariance of $g_{a,b}$ and $\tH(\bsig')$ given $\nabla\cF_\TAP(\bm)=0,\by,\bx$ equals $N^{1/2}K((a,b),(a',b'))$, independently of $\bsig'^\perp$; hence $g_{a,b}$ is uncorrelated with all components of $\hH^{a',b'}$ of degree $\ge 1$, and it is also independent of $\by^\perp$ by Eq.~\eqref{eq:cov-cond}.

(a) Let $m(\bv) := \EE[\tH(\bv)|\nabla\cF_\TAP(\bm)=\bfzero,\by,\bx]$ be
given by Eq.~\eqref{eq:mean-cond}. Differentiating,
\begin{align*}
    \nabla m(\bv) &= \frac{\xi'(\<\bv,\bm\>_N)}{\xi'(q)}\bz \\
    &\qquad + \lt[\frac{\xi''(\<\bv,\bm\>_N)\<\bz,\bv\>_N}{\xi'(q)}
    - \frac{\xi''(q)\<\bm,\bz\>_N}{\xi'(q)\zeta(q)}
    \lt(\xi'(\<\bv,\bm\>_N)+\xi''(\<\bv,\bm\>_N)\<\bm,\bv\>_N\rt)\rt]\bm,
\end{align*}
whose bracket vanishes at $\bv = \bm$, so that $\nabla m(\bm) = \bz = \nabla\tH(\bm)$ on $\nabla \cF_\TAP(\bm)=0$. Hence, for $\oH(\bv) = \tH(\bv) - \EE[\tH(\bv)\mid\nabla \cF_\TAP(\bm)=0,\by,\bx]$, $\nabla \oH(\bm)=0$, which yields $\nabla g_{a,b}|_{(a,b)=(1,b_*)}=0$ and, taking the covariance with $g_{a',b'}$, Eq.~\eqref{eq:shift-kernel-grad}.

(b) The claimed identity and estimates follow immediately from Eq.~\eqref{eq:shift-kernel}.
Indeed, Eqs.~\eqref{eq:pi-scaled} and \eqref{eq:K-scaled} follow from
Eqs.~\eqref{eq:def-pi} and \eqref{eq:shift-kernel} upon substituting
$a = 1+w_1/\kappa$, $b = b_*+w_2$, $b_* = \beta\kappa$ (and similarly for
$(a',b')$). 

For regularity at $\kappa=0$, let $p_0=\min\{p\ge2:\gamma_p\ne0\}$ and write $\xi'(s)=s^{p_0-1}F(s)$, where $F(0)>0$. For $X=\kappa+w_1$, $X'=\kappa+w'_1$,
\begin{align*}
 \frac{\xi'(\kappa X)\xi'(\kappa X')}{\xi'(\kappa^2)}
 &=\frac{X^{p_0-1}(X')^{p_0-1}F(\kappa X)F(\kappa X')}{F(\kappa^2)}.
\end{align*}
Also $\kappa^2\xi''(\kappa^2)/\zeta(\kappa^2)$ extends analytically to $(p_0-1)/p_0$ at $\kappa=0$. Since $1-\beta^2$ is bounded away from zero, these identities give the claimed uniform derivative bounds.

(c) This follows immediately from the derivative bounds in (b) and $\nabla_{(a,b)}g_{a,b}|_{(a,b)=(1,b_*)}=0$. 
\end{proof}

The mixture satisfies
\[
 \tilde\xi_{\ge2}(z)
 =\xi(r+(1-r)z)-\xi(r)-\xi'(r)(1-r)z.
\]
By compactness and Eq.~\eqref{eq:amp-works}, there is $\kappa>0$ such that
$(1-u)^2\xi''(u)\le1-\kappa$ on $[0,1]$. Thus, uniformly in $r,z\in[0,1)$,
\begin{align}\label{eq:tilde-xi-satisfies-replica-symmetry}
 \tilde\xi_{\ge2}''(z)
 =(1-r)^2\xi''(r+(1-r)z)
 \le\frac{1-\kappa}{(1-z)^2}.
\end{align}
By integrating twice, we obtain a uniform replica-symmetry gap.
For fixed $1<R<2$, positivity of the coefficients also gives a uniform
analytic bound on $\tilde\xi_{\ge2}(R)$, since $r+(1-r)R\le R$.
When applying the dimension-$n$ companion estimates, the covariance
per coordinate is larger by $\lambda_N=1+O(N^{-1})$.
For all sufficiently large $N$ this still leaves a common margin
$\kappa/2$ and a uniform analytic bound, without changing our
definition of $\tilde\xi$.
The uniform-mixture version of the companion estimates
\cite{huang2026approximating} therefore applies, including
when the quadratic coefficient tends to zero. 
}
%
Similarly, note that
\begin{align}\label{eq:sum>=3}
 \tilde\xi_{\ge3}(1)
 =\xi(1)-\xi(r)-\xi'(r)(1-r)-\frac12\xi''(r)(1-r)^2.
\end{align}
Following Subsection~\ref{sec:FirstLocal}, we write the quadratic component of $\thH$ as $\<\bA^{(2)}, \tilde{\bsig}^{\otimes 2}\>$, where 
\rev{
\[
 \bA^{(2)}=\bA^{(2)}(a,b)\ed\alpha_{a,b}\bW,\qquad
 \bW\sim\GOE(n),\qquad \alpha_{a,b}:=\sqrt{\lambda_N/2}\,\tilde\gamma_2.
\]
}
Recall the definition of $G(\gamma) =G(\gamma;\bA,\bu)$ \rev{from
Subsection~\ref{sec:FirstLocal}}.
We take $\bu$ to be the external field \rev{$\bu=\sqrt{\lambda_N}\,\tilde\gamma_1\bg$, $\bg\sim\cN(0,I_n)$},
and $\bA=\bA^{(2)}$. Note that $\bu$ and $\bA^{(2)}$ depend on the parameters $a,b$. Let $\gamma_{a,b} = \arg \min_{z>z_*} G(z; \bA^{(2)}, \bu)$, $z_*:= \lambda_{\max}(\bA^{(2)})$.
\rev{By \cite[Lemma~4.2, Eqs.~(4.10)--(4.11), and Lemma~3.1]{huang2026approximating}}, when
\begin{align*}
\lambda_N\tilde\gamma_1^2\le \tfrac12 n^{2c_0-1},
\end{align*}
we have (with probability at least $1-N^{-c}$, conditional on $\nabla \cF_\TAP(\bm)=0,\by-\by^{\perp},\bx$) that
\begin{align}\label{eq:approx-Z}
&\int_{D_N(a,b)} e^{H(\bsig)}\mu^{a,b}_0(\de\bsig) \nonumber \\
&= (1+O(N^{-c}))J_N(2e)^{-(N-2)/2}\sqrt{\frac{2}{G''(\gamma_{a,b}; \bA^{(2)}, \bu)}} \nonumber\\
&\qquad \qquad \exp\bigg(N\bigg[N^{-1/2} g_{a,b} + \Gamma_N(\by,\bm;a,b) + \frac{N-4}{2N}\log(1-r(a,b))  + \rev{\frac{n}{N}}\min_{z>z_*} G(z; \bA^{(2)}, \bu) \nonumber\\
&\qquad \qquad \qquad \qquad + \frac{1}{2} \lt(\xi(1)-\xi(r(a,b))-\xi'(r(a,b))(1-r(a,b)) - \frac{1}{2} \xi''(r(a,b))(1-r(a,b))^2\rt)\bigg]\bigg),
\end{align}
where we have simplified using Eq.~(\ref{eq:sum>=3}).
\rev{By independence of $\bu$ and $\bA^{(2)}$, and the law
$\bA^{(2)}\ed\alpha_{a,b}\bW$ for $\bW\sim\GOE(n)$, the following holds with probability at least $1-Ce^{-N^{c(\eta)}}$, for every fixed $\eta>0$, provided
$z>2\alpha_{a,b}+\delta$ for some constant $\delta>0$}
\begin{equation}
G(z; \bA^{(2)}, \bu) = G_{a,b}(z) + O(N^{-1+\eta}),
\end{equation}
where
\begin{equation}\label{eq:Gab}
    G_{a,b}(z) := z-\frac12\left(\psi(z/\alpha_{a,b})+\log\alpha_{a,b}\right)
    +\frac{\lambda_N\tilde\gamma_1^2}{4\alpha_{a,b}}\phi(z/\alpha_{a,b}),
\end{equation}
and, for $x>2$,
\[
\phi(x) = \frac{1}{2}(x-\sqrt{x^2-4}), \qquad \psi(x) = \frac{1}{2}((x-\sqrt{x^2-4})/2)^2-\log((x-\sqrt{x^2-4})/2).
\]
Note that $\phi(x) = \int (x-u)^{-1} \mu_{\semic}(\de u)$ and $\psi(x) = \int \log(x-u) \mu_{\semic}(\de u)$ where $\mu_{\semic}$ is the semicircular law. Moreover, $\psi'(x)=\phi(x)$. \rev{At $\tilde\gamma_2=0$, use the continuous extension
\[
 G_{a,b}(z)=z-\frac12\int\log(z-\alpha_{a,b}x)\mu_{\semic}(\de x)
 +\frac{\lambda_N\tilde\gamma_1^2}{4}\int\frac{\mu_{\semic}(\de x)}{z-\alpha_{a,b}x}.
\]}

Thus,
\begin{align}
J_N^{-1}\int_{D_{N}(a,b)} e^{H(\bsig)}\mu^{a,b}_0(\de\bsig) &= \sqrt{\frac{2}{G''(\gamma_{a,b}; \bA^{(2)}, \bu)}}\exp\bigg(NE(a,b) + N^{1/2}  g_{a,b} + O(N^{\eta})\bigg),\label{eq:IntDNab}
\end{align}
where we define
\begin{align}
E(a,b) &:= -\frac{N-2}{2N} \ln(2e) + \frac{N-4}{2N}\log (1-r(a,b)) + \frac{n}{N}\min_{z>2\alpha_{a,b}} G_{a,b}(z) + \Gamma_N(\by,\bm;a,b) \nonumber\\
& \qquad + \frac{1}{2} \lt(\xi(1)-\xi(r(a,b))-\xi'(r(a,b))(1-r(a,b)) - \frac{1}{2} \xi''(r(a,b))(1-r(a,b))^2\rt).
\end{align}
Let $b_* = \<\bx,\bm\>_N$. Note that $r(1,b_*) = q$. Furthermore, we have
\begin{align}
    r(a,b) = a^2q + \frac{(b-a\<\bx,\bm\>_N)^2}{1-\<\bx,\bm\>_N^2/q}.
\end{align}
We will next verify several properties of $E(a,b)$, starting with the observation that $(a,b)=(1,b_*)$ is an \rev{approximate} stationary point of $E$.
\begin{lem}\label{lem:Grad-E}
\rev{We have $\nabla_w E|_{w=0}=O(N^{-1})$, uniformly in $t\in(0,T]$.} 
\end{lem}
\begin{proof}
\rev{Write $\nabla=\nabla_{a,b}$ and
$z_*(a,b)=\arg\min_{z>2\alpha_{a,b}}G_{a,b}(z)$.
At $(a,b)=(1,b_*)$, both $\tilde\gamma_1^2$ and
$\nabla\tilde\gamma_1^2$ vanish. Hence, putting
$\alpha_*:=\alpha_{1,b_*}$ and $z_*:=z_*(1,b_*)$,
\begin{equation}\label{eq:stationary}
 \partial_zG_{1,b_*}(z_*)=0
 \quad\Longleftrightarrow\quad
 \int\frac{\mu_{\semic}(\de x)}{z_*-\alpha_*x}=2,
 \qquad z_*=\frac12+\lambda_N\tilde\gamma_2(1,b_*)^2.
\end{equation}
The common replica-symmetry margin ensures that this saddle stays
uniformly above $2\alpha_*$. For $j\in\{a,b\}$,
\begin{equation}\label{eq:dz}
 \partial_jz_*(a,b)
 =-\frac{\partial_j\partial_zG_{a,b}(z_*(a,b))}
          {\partial_z^2G_{a,b}(z_*(a,b))}.
\end{equation}
The envelope identity and Eq.~\eqref{eq:stationary} therefore give
\begin{equation}\label{eq:G}
 \left.\nabla\min_zG_{a,b}(z)\right|_{(1,b_*)}
 =\left.\frac{\lambda_N}{2}\nabla\tilde\gamma_2^2\right|_{(1,b_*)},
 \qquad
 \left.\nabla\left(\frac nN\min_zG_{a,b}(z)\right)\right|_{(1,b_*)}
 =\left.\frac12\nabla\tilde\gamma_2^2\right|_{(1,b_*)}.
\end{equation}
These identities also hold at $\tilde\gamma_2=0$ by the integral
representation of $G_{a,b}$.}
Furthermore, \[\nabla \tilde{\gamma}_2^2 = -\xi''(r(a,b))(1-r(a,b))\nabla r(a,b) + \frac{1}{2}\xi'''(r(a,b))(1-r(a,b))^2 \nabla r(a,b).\]

We have
\begin{align}
&\nabla \left(\frac{N-4}{2N}\log (1-r(a,b)) +  \frac{1}{2} (\xi(1)-\xi(r(a,b))-\xi'(r(a,b))(1-r(a,b)) - \frac{1}{2} \xi''(r(a,b))(1-r(a,b))^2) \right)\nonumber\\
& = \frac{-(N-4)}{2N(1-r(a,b))} \nabla r(a,b) - \frac{1}{4} \xi'''(r(a,b))(1-r(a,b))^2 \nabla r(a,b), \label{eq:r}
\end{align}
and furthermore $\partial_a r(a,b)|_{(1,b_*)} = 2q$, $\partial_b r(a,b)|_{(1,b_*)} = 0$ and $r(1,b_*) = q$.

Recall
\[
\langle \bz+\by,\bsig\rangle_N = -\xi'(\<\bx,\bm\>_N)b+aq\lt((1-q)\xi''(q)+\frac{1}{1-q}\rt).
\]
Moreover,
\[
\<\by-\by^\perp, \bsig-\bsig^\perp\>_N = a\<\by,\bm\>_N + \frac{b-a\<\bx,\bm\>_N}{1-\<\bx,\bm\>_N^2/q} \lt(\<\by,\bx\>_N - \<\bx,\bm\>_N\<\by,\bm\>_N/q\rt).
\]
Hence, 
\begin{align*}
&\frac{\xi'(aq)\langle \by+\bz,\bsig\rangle_N}{\xi'(q)} - \frac{\xi''(q)\xi'(aq)aq\langle \bm,\bz\rangle_N}{\xi'(q)\zeta(q)} + \lt(1-\frac{\xi'(aq)}{\xi'(q)}\rt)\langle \by - \by^\perp, \bsig - \bsig^\perp\rangle_N\\
&= \frac{\xi'(aq)(-\xi'(\<\bx,\bm\>_N)b+aq((1-q)\xi''(q)+\frac{1}{1-q}))}{\xi'(q)}-\frac{\xi''(q)\xi'(aq)aq\<\bm,\bz\>_N}{\xi'(q)\zeta(q)} \\
&\qquad \qquad \qquad + \lt(1-\frac{\xi'(aq)}{\xi'(q)}\rt)\lt(a\<\by,\bm\>_N + \frac{b-a\<\bx,\bm\>_N}{1-\<\bx,\bm\>_N^2/q} \lt(\<\by,\bx\>_N - \<\bx,\bm\>_N\<\by,\bm\>_N/q\rt)\rt)\\
&= \frac{\xi'(aq)(-\xi'(\<\bx,\bm\>_N)b+aq((1-q)\xi''(q)+\frac{1}{1-q}))}{\xi'(q)} \\
&\qquad \qquad \qquad  + \lt(1-\frac{\xi'(aq)}{\xi'(q)}\rt)\frac{b-a\<\bx,\bm\>_N}{1-\<\bx,\bm\>_N^2/q} (\<\by,\bx\>_N-\<\bx,\bm\>_N\<\by,\bm\>_N/q)
\\&\qquad \qquad \qquad \qquad  + \frac{\xi''(q)\xi'(aq)aq\lt(\xi'(\<\bx,\bm\>_N)\<\bx,\bm\>_N - q\xi''(q)(1-q)-\frac{q}{1-q}\rt)}{\xi'(q)\zeta(q)}  \\
&\qquad \qquad \qquad \qquad +\<\by,\bm\>_N\lt(\frac{\xi''(q)\xi'(aq)aq}{\xi'(q)\zeta(q)} + a\lt(1-\frac{\xi'(aq)}{\xi'(q)}\rt)\rt).
\end{align*}
Note that
\[
\left.\partial_a\lt(\lt(1-\frac{\xi'(aq)}{\xi'(q)}\rt)\frac{b-a\<\bx,\bm\>_N}{1-\<\bx,\bm\>_N^2/q}\rt) \right|_{(1,b_*)} = 
\left.\partial_b\lt(\lt(1-\frac{\xi'(aq)}{\xi'(q)}\rt)\frac{b-a\<\bx,\bm\>_N}{1-\<\bx,\bm\>_N^2/q}\rt) \right|_{(1,b_*)} = 0,
\]
and $\partial_a (\xi'(aq)aq)|_{(1,q)} = q\zeta(q)$, so
\[
\partial_a\lt(\frac{\xi''(q)\xi'(aq)aq}{\xi'(q)\zeta(q)} + a\lt(1-\frac{\xi'(aq)}{\xi'(q)}\rt)\rt) |_{(1,b_*)} = \partial_b\lt(\frac{\xi''(q)\xi'(aq)aq}{\xi'(q)\zeta(q)} + a\lt(1-\frac{\xi'(aq)}{\xi'(q)}\rt)\rt) |_{(1,b_*)} = 0.
\]
Thus, we can compute
\begin{align}
& \partial_a \lt(\frac{\xi'(aq)\langle \by+\bz,\bsig\rangle_N}{\xi'(q)} - \frac{\xi''(q)\xi'(aq)aq\langle \bm,\bz\rangle_N}{\xi'(q)\zeta(q)}+ \lt(1-\frac{\xi'(aq)}{\xi'(q)}\rt)\langle \by - \by^\perp, \bsig - \bsig^\perp\rangle_N\rt)|_{(1,b_*)}\nonumber\\
&= q(1-q)\xi''(q) + \frac{q}{1-q}.\label{eq:sa}
\end{align}
Similarly,
\begin{align}\label{eq:sb}
&\partial_b \lt(\frac{\xi'(aq)\langle \by+\bz,\bsig\rangle_N}{\xi'(q)} - \frac{\xi''(q)\xi'(aq)aq\langle \bm,\bz\rangle_N}{\xi'(q)\zeta(q)}+ \lt(1-\frac{\xi'(aq)}{\xi'(q)}\rt)\langle \by - \by^\perp, \bsig - \bsig^\perp\rangle_N\rt)|_{(1,b_*)} \nonumber\\
&= -\xi'(\<\bx,\bm\>_N).
\end{align}


\rev{Combining Eqs.~\eqref{eq:G}, \eqref{eq:r}, \eqref{eq:sa}, and
\eqref{eq:sb}, including the derivative of $\xi(b)$ in $\Gamma_N$, gives
\[
 \nabla_wE(0)=\left(\frac{4\sqrt q}{N(1-q)},0\right)=O(N^{-1}).
\]
The bound is uniform for $t\in(0,T]$, since $q$ stays bounded away from $1$.}
\end{proof}

\begin{lem}\label{lem:anneal-cal}
    We have 
    \begin{align}\label{eq:anneal}
        \bbE\lt[\int_{D_N(a,b)} e^{H(\bsig)} \mu_0^{a,b}(\de\bsig)
    \Big| \nabla\cF_{\sTAP}(\bm) = \bfzero, \bx,\by-\by^{\perp},g_{a,b}\rt] &= 
    J_N\exp\Big\{N\hat{E}(a,b) +\sqrt{N} g_{a,b}\Big\},
    \end{align}
    where 
    \begin{align}
        \hat{E}(a,b) &:= \frac{N-4}{2N}\log (1-r(a,b)) + \lt(\frac{1}{2}\tilde{\gamma}_2^2 + \frac{1}{2}\tilde{\gamma}_1^2\rt) + \Gamma_N(\by,\bm;a,b) \nonumber \\
        &\qquad + \frac{1}{2} \lt(\xi(1)-\xi(r(a,b))-\xi'(r(a,b))(1-r(a,b)) - \frac{1}{2} \xi''(r(a,b))(1-r(a,b))^2\rt).
    \end{align}
    Furthermore, $E(a,b)$ is uniformly upper bounded by $\hat{E}(a,b)$, $E(1,b_*) = \hat{E}(1,b_*)$, and \rev{$\nabla_w E(0),\nabla_w\hat E(0)=O(N^{-1})$.}
\end{lem}
\begin{proof}
    \rev{The covariance convention gives
$\EE\exp(\thH(\tilde\bsig))=\exp(N\tilde\xi(1)/2)$ on $S_n$,
which proves Eq.~\eqref{eq:anneal}.
For the comparison, put $z_0=1/2+\lambda_N\tilde\gamma_2^2$.
The common margin gives $2\lambda_N\tilde\gamma_2^2<1$, so
$z_0>2\alpha_{a,b}$ and the semicircle identities yield
\[
 \int\frac{\mu_{\semic}(\de x)}{z_0-\alpha_{a,b}x}=2,
 \qquad
 G_{a,b}(z_0)=\frac12\log(2e)
       +\frac{\lambda_N}{2}(\tilde\gamma_2^2+\tilde\gamma_1^2).
\]
Consequently,
\[
 \frac nN\left(\min_{z>2\alpha_{a,b}}G_{a,b}(z)-\frac12\log(2e)\right)
 \le\frac12(\tilde\gamma_2^2+\tilde\gamma_1^2),
\]
which gives $E\le\hat E$. At $(1,b_*)$ the field vanishes, so $z_0$
is the minimizer and equality holds. Since $\nabla\tilde\gamma_1^2=0$
there, the same derivative calculation as in Lemma~\ref{lem:Grad-E} gives
\[
 \nabla_w\hat E(0)=\nabla_wE(0)
 =\left(\frac{4\sqrt q}{N(1-q)},0\right). \qedhere
\]}
\end{proof}

\rev{Write $\bar\Gamma_N(a,b)=\Gamma_N(\bar\by,\bm;a,b)$ and let $\overline{\hat E}$ be obtained from $\hat E$ by replacing $\Gamma_N$ by $\bar\Gamma_N$. Writing $f=\xi'(q)$, $c=\xi''(q)$ and $A=(1-q)c+(1-q)^{-1}$, we compute 
\begin{align}\label{eq:Gamma-bar}
 \bar\Gamma_N(a,b)
 &=\xi(b)+tb+\frac{A(\xi'(aq)+t)aq}{\zeta(q)+t}\nonumber\\
 &\qquad-\frac{(\xi'(aq)+t)(\xi'(b_*)+t)}{f+t}
       \left(b-\frac{cb_*aq}{\zeta(q)+t}\right).
\end{align}
Let
\[
 e_1=\frac{\bm}{\sqrt{Nq}},\qquad
 e_2=\frac{\bx-b_*\bm/q}{\sqrt{N(1-b_*^2/q)}},\qquad
 \eta_i=\frac{\langle\by-\bar\by,e_i\rangle}{\sqrt N}.
\]
Conditional on $\bx,\nabla\cF_\TAP(\bm)=0$, these are independent centered Gaussians with variances $v_i/N$, where
\[
 v_1=\frac{t\zeta(q)}{t+\zeta(q)},\qquad
 v_2=\frac{t\xi'(q)}{t+\xi'(q)}.
\]
For $X=\sqrt q+w_1$ and $Y=(w_2-\beta w_1)/\sqrt{1-\beta^2}$, a direct calculation gives
\begin{align}\label{eq:Gamma-fluctuation}
 \Gamma_N(\by,\bm;a,b)-\bar\Gamma_N(a,b)
 =\eta_1X\left(1-\frac{\xi'(\sqrt qX)}{\zeta(q)}\right)
  +\eta_2Y\left(1-\frac{\xi'(\sqrt qX)}{\xi'(q)}\right).
\end{align}
Both coefficient functions have zero gradient at $w=0$. Their products with $\sqrt{v_i}$ have uniformly bounded derivatives in $w$ of every fixed order: if $p_0$ is the smallest degree present in $\xi$, then $\xi'(q),\zeta(q)\asymp q^{p_0-1}$ and $\sqrt{v_i}\le Cq^{(p_0-1)/2}$, which cancel the factors in the denominators as in Lemma~\ref{lem:shift-cov}(b).
Fix a small $\epsilon_1\in(0,1/4)$ and let $\mathcal Y=\{\max_i\sqrt{N/v_i}|\eta_i|\le N^{\epsilon_1}\}$. On $\mathcal Y$, the difference in \eqref{eq:Gamma-fluctuation} has $C^3(w)$ norm at most $CN^{-1/2+\epsilon_1}$ on $|w|\le\eps_0$, and the same bound in absolute value on the whole band. In particular,
\[
 \nabla_w(\hat E-\overline{\hat E})(0)=0.
\]

When $b_*=q=q_*(t)$, we have $\xi'(q)+t=q/(1-q)$, $\bar\by=(t/q)\bm$, and
\[
 \bar\Gamma_N(a,b)=\xi(b)+(a-b)\xi'(aq)+ta.
\]
Ignoring the $O(N^{-1})$ dimensional correction, we can therefore simplify $\overline{\hat E}$ as
\begin{align}
 \tilde E(a,b)&:=\frac12\log(1-r(a,b))+\xi(b)+(a-b)\xi'(aq)+ta\nonumber\\
 &\qquad+\frac12\left(\xi(1)-\xi(r(a,b))
       +(1-r(a,b))\left[t-\frac{(\xi'(aq)+t)^2}{\xi'(q)+t}\right]\right).
\end{align}
We will work under the conditions 
\begin{align}\label{eq:standing-scaled}
 \bm\in\mathcal D_\iota(t,\bx),\qquad
 \kappa=\sqrt q,\quad \beta=b_*/\sqrt q,\quad
 w=(\sqrt q(a-1),b-b_*),\quad d_q(a,b)=(qa^2+b^2)^{1/2}.
\end{align}
For $t<t_0$, $\tau=t/q$ is bounded above and away from zero and $|\beta|\le\beta_0$. For $t\ge t_0$, $(q,b_*)$ is uniformly close to $(q_*(t),q_*(t))$ when $\iota$ is small.}

\rev{Let $\tilde{a} = qa$ and define $\uE(\tilde{a},b) = \tilde{E}(\tilde{a}/q,b)$. Define $d_q(a,b) = (qa^2+b^2)^{1/2}$. 
\begin{lem}\label{lem:Hess-E}
There exists $\eta>0$ uniform in $t$ such that the following holds. For $\eps>0$ sufficiently small, there is $\eta>0$ such 
that for all $(\tilde{a},b)$ satisfying $d_q(\tilde{a}/q-1, b-q) \le \eps$, we have
$\nabla^2\uE(\tilde{a},b)\preceq -\eta \mathrm{Diag}((1+q)/q, 1)$.
\end{lem}}
\begin{proof}
We have
\begin{align}\label{eq:bb}
\partial_b^2 \tilde{E}(a,b)|_{(1,q)}
&= 2\xi''(q) - \left(\frac{1}{2(1-r(a,b))} + \frac{1}{2}\xi''(r(a,b))(1-r(a,b))\right) \partial_b^2(r(a,b))|_{(1,q)} \nonumber\\
&= -\frac{1}{(1-q)^2} + \xi''(q),
\end{align}
and
\begin{align}\label{eq:ab}
\partial_{b}\partial_a \tilde{E}(a,b)|_{(1,q)}
&=  -2q\xi''(q) - \left(\frac{1}{2(1-r(a,b))} + \frac{1}{2}\xi''(r(a,b))(1-r(a,b))\right) \partial_{a,b}(r(a,b))|_{(1,q)} \nonumber\\
&= -q\xi''(q) + \frac{q}{(1-q)^2}.
\end{align}
Finally, we compute
\begin{align}
&\partial_a^2 \tilde{E}(a,b)|_{(1,q)}\nonumber \\ 
&= -q^3\xi'''(q) + q(2\xi''(q)+q\xi'''(q)) + \rev{\lt(q\xi''(q)+q^2(1-q)\xi'''(q)-q(1-q)^2\xi''(q)^2\rt)} \nonumber\\
&\qquad - \frac{q(2q+1)}{(1-q)^2} -2q^2(1-q)\xi'''(q) + q(2q-1)\xi''(q)\nonumber\\
&= \rev{- \frac{q(2q+1)}{(1-q)^2} + 2q(q+1)\xi''(q) - q(1-q)^2\xi''(q)^2.} \label{eq:aa} 
\end{align}

\rev{Note that 
\begin{align*}
    &\frac{q(2q+1)}{(1-q)^2} - 2q(q+1)\xi''(q) + q(1-q)^2\xi''(q)^2 \\
    &= q\left[(1-q)^2\left(\frac{1}{(1-q)^2}-\xi''(q)\right)^2  + 2q\left(\frac{1}{(1-q)^2}-\xi''(q)\right)\right]. 
\end{align*}
}
Consider a change of variable $\tilde{a} = aq$ and let $\uE (\tilde{a},b) = \tilde{E}(\tilde{a}/q,b)$. Combining (\ref{eq:bb}), (\ref{eq:ab}), and (\ref{eq:aa}), we find that, under the condition $\xi''(q)<\frac{1}{(1-q)^2}$, 
\rev{
for $Q = -\nabla^2\uE(\tilde{a},b)$, and a suitable $\eta>0$ uniform in $t$,
\begin{align*}
    (x, y)Q (x, y)^T &= \left(\frac{1}{(1-q)^2}-\xi''(q)\right)\left\{(x-y)^2 + \left[1-\frac{(1-q)^2}{q} \left(\xi''(q) - \frac{1}{(1-q)^2}\right)\right] x^2\right\}\\
    &\ge \eta\lt(\frac{1+q}{q}x^2+y^2\rt).
\end{align*}
\rev{In the last inequality, we use the assumption that $\xi''(q) < \frac{1}{(1-q)^2}$ uniformly in $q$.}
The uniform derivative bounds in $w$ give the stated neighborhood by continuity. }
\end{proof}

\rev{For $t\ge t_0$, continuity and Lemma~\ref{lem:Hess-E} give strict concavity for $(q,b_*)$ close to $(q_*(t),q_*(t))$. For small $t$, \eqref{eq:Gamma-bar} and the formula for $\tilde\gamma_1^2$ give a uniform $C^3$ extension of $\overline{\hat E}$ in $(\kappa,\beta,\tau,w)$, since $f+t$ and $\zeta(q)+t$ factor as $\kappa^2$ times functions bounded away from zero. Put $c_2=\xi''(0)$. At $\kappa=\beta=0$, direct differentiation gives
\begin{align}\label{eq:Hess-small-t}
 \left.\nabla_w^2\overline{\hat E}\right|_{w=0}
 =-\mathrm{Diag}\left(
 \frac{\tau[\tau+c_2(1-c_2)]}{(c_2+\tau)(2c_2+\tau)},\,1-c_2\right)+O(N^{-1}).
\end{align}
This is uniformly negative definite because $c_2<1$ and $\tau$ lies in a compact subset of $(0,\infty)$. Taking $\beta_0,t_0,\iota$ sufficiently small, using continuity and then \eqref{eq:Gamma-fluctuation} proves the following corollary on $\mathcal Y$.}

\rev{
\begin{cor}\label{cor:NegDef}
Let $t\in(0,T]$. There exist $\eps,\eta>0$ depending only on $(\xi,T,\iota)$ such that, under \eqref{eq:standing-scaled} on $\mathcal Y$, with $\beta_0,t_0,\iota$ sufficiently small,
for all $(a,b)$ such that $d_q(a-1, b-b_*) \le \eps$, we have
$\nabla^2_{(\tilde a,b)}\hat E(\tilde{a}/q,b)\preceq -\eta \mathrm{Diag}((1+q)/q, 1)$. 
\end{cor}
}

We will next prove several simple preliminary estimates before giving the proof of Proposition \ref{ppn:local-barycenter}.

Recall that on $D_N(a,b)$, we have defined the Hamiltonian $\thH(\tilde{\bsig})$,
which is a spin glass with mixture given by Eqs.~\eqref{eq:MixtureAB1} to \eqref{eq:MixtureAB3}. Let $\bA^{(p)}(a,b) = \nabla^{p} \thH(\bzero)/p!$ and $\bu(a,b) = \nabla \thH(\bzero)$. 

\rev{By Lemmas \ref{lem:Grad-E} and \ref{lem:anneal-cal}, Corollary \ref{cor:NegDef}, and the preceding remark, we take $a_*=1$. Then $E(a_*,b_*)=\widehat E(a_*,b_*)$ and $\nabla_w\widehat E(0)=O(N^{-1})$. Moreover, $\widehat E$ is strongly concave near $(a_*,b_*)$. In particular, there is $\eta>0$ such that, for sufficiently small $\eps$ and $(a,b)$ such that $|\sqrt{q}(a-a_*)|+|b-b_*|\le \eps$, we have
\begin{align}\label{eq:quad-upper-E}
    E(a,b) \le E(a_*,b_*) +\frac C N d_q(a-a_*,b-b_*) - \eta (q|a-a_*|^2+|b-b_*|^2).
\end{align}
For each $a,b$, let $\bm(a,b)$ be the unique point in $V := \spn(\bm,\bx)$ with
\[
\<\bm(a,b),\bm\>_N=qa,\qquad \<\bm(a,b),\bx\>_N=b.
\]
All nets below are taken in $w=(\sqrt q(a-a_*),b-b_*)$.}

The following lemma follows from standard control on suprema of Gaussian processes (see, e.g. \cite[Lemma A.3]{montanari2023solving}). 
\begin{lem}\label{lem:lip-grad}
    For $\eps>0$ sufficiently small there exist $c=c(\eps)$,
    $C=C(\eps) >0$ depending uniquely on $\eps$ such that the following holds with probability at least $1-e^{-cN}$ conditional on $\nabla \cF_\TAP(\bm)=0$, $\bx, \by-\by^\perp \in \mathcal{Y}$. For  $(a,b)$ 
    such that $d_q(a-a_*, b-b_*) \le \eps$, we have 
    \[
    \nabla H(\bm(a,b)) = \nabla H(\bm) + \nabla^2 H(\bm) (\bm(a,b)-\bm) + \Err,
    \]
    where $\|\Err\| \le CN^{-1/2} \|\bm(a,b)-\bm\|^2$. Furthermore, $\Proj_V^{\perp}(\nabla H(\bm)) = 0$, and for $\bv\in V$, $\|\nabla^2 H(\bm) \bv\|^2 \le C \|\bv\|^2$. (Here $V = \spn(\bm,\bx)$.)
\end{lem}

As a corollary of Lemma \ref{lem:lip-grad},  we obtain the following control on the effective fields \[\bu(a,b) =\sqrt{N(1-r(a,b))/n}\Proj_V^{\perp}(\nabla H(\bm(a,b)).\]  
\begin{lem}\label{lem:lip-field}
    For $\eps,\delta>0$ sufficiently small, the following holds with probability at least $1-e^{-cN}$. There exists $\bu_1,\bu_2 \in \bbR^{N-2}$ with $\|\bu_1\|,\|\bu_2\| = O(N^{1/2})$ such that, for any $(a,b)$ with \rev{$d_q(a-a_*,b-b_*) \le \eps$, we have $\|\bu(a,b) - \sqrt{q}(a-a_*)\bu_1 - (b-b_*)\bu_2\| \le CN^{1/2} (q|a-a_*|^2+|b-b_*|^2)$. }
    
    Furthermore, for $\gamma > \delta + \bbE\lambda_{\max}(\bA^{(2)}(a_*,b_*))$ and $i,j\in \{1,2\}$, \rev{for every fixed $\eta>0$, with probability at least $1-Ce^{-cN^{2\eta}}$, $\<\bu_i, (\gamma \bI - \bA^{(2)}(a_*,b_*))^{-1} \bu_j\>$ concentrates in a window of size $O(N^{1/2+\eta})$ around its expectation}.
\end{lem}
\begin{proof}
The first part follows from  Lemma \ref{lem:lip-grad}, using 
$\Proj_V^{\perp}(\nabla H(\bm)) = \bfzero$.

The second part holds by concentration of Lipschitz functions of Gaussian random variables. Indeed, \rev{note that the $\bu_i$ are centered Gaussian vectors, independent of the entire quadratic matrix family, with uniformly bounded covariance}. Under the
high-probability event  \rev{$\|\bA^{(2)}(a_*,b_*)\|_{\op}\le C$ and $\lambda_{\max}(\bA^{(2)}(a_*,b_*))\le\bbE\lambda_{\max}(\bA^{(2)}(a_*,b_*))+\delta/2$}, the quantity
$\<\bu_i, (\gamma \bI - \bA^{(2)}(a_*,b_*))^{-1} \bu_j\>$ is indeed Lipschitz in these
Gaussians and in $\bA^{(2)}(a_*,b_*)$.
\end{proof}

Let \rev{$c_*>0$ be sufficiently small, $r_N=N^{-1/2+c_*}$, and let}
\begin{align}
\rev{R := \big\{(a,b): d_q(a-a_*,b-b_*)\le r_N\big\}\, .}
\end{align}
Recall the random shifts $g_{a,b}$ in Eq.~(\ref{eq:(a,b)-reduction}). We have the following control on $g_{a,b}$.

\rev{
\begin{lem}\label{lem:lip-shift}
Fix $L<\infty$. There are constants $C,c>0$,
depending only on $(\xi,L,\iota)$, such that, uniformly in $t\in(0,L]$ (under \eqref{eq:standing-scaled}),
the centered Gaussian process $g_{a,b}$ satisfies
\begin{align}
    \EE\left[(g_{a,b}-g_{a',b'})^2\right]
    &\le
    C\left(d_q(a-a_*,b-b_*)+d_q(a'-a_*,b'-b_*)\right)^2
    d_q(a-a',b-b')^2 .
    \label{eq:g-shift-sharp-metric}
\end{align}
In particular,
\[
    \EE\left[(g_{a,b}-g_{a',b'})^2\right]
    \le C d_q((a,b),(a',b'))^2
    \qquad ((a,b),(a',b')\in R).
\]

Moreover, with probability at least
\[
    1-C\exp\{-c r_N^{-2}\}
    =
    1-C\exp\{-cN^{1-2c_*}\},
\]
we have simultaneously for all $(a,b)\in R$ that
\begin{equation}
    \label{eq:g-shift-pointwise}
    |g_{a,b}-g_{a_*,b_*}|
    \le C d_q(a-a_*,b-b_*).
\end{equation}
Furthermore, for every fixed $\delta>0$, with probability at least $1-Ce^{-cN^{2\delta}}$,
\begin{equation}\label{eq:shift-reflection}
 \sup_{|w|\le\eps}\|\nabla_w^j g_w\|\le N^\delta\quad(j=2,3),
 \qquad |g_w-g_{-w}|\le CN^\delta|w|^3.
\end{equation}
\end{lem}

\begin{proof}
Write $g_w:=g_{a,b}$ when
\[
    w=w(a,b)=\big(\sqrt q\,(a-a_*),b-b_*\big),
\]
and write $g_0:=g_{a_*,b_*}$.

Let $K_q(w,w')=\EE[g_wg_{w'}]$ denote the covariance kernel Eq.~\eqref{eq:shift-kernel}.
By Lemma~\ref{lem:shift-cov}(a,c),
\[
    \nabla_w(g_w-g_0)\big|_{w=0}=0
    \qquad\text{almost surely}
\]
and
\[
    \EE\|\nabla_w(g_w-g_0)\|^2\le C|w|^2
\]
for $|w|\le\eps$, uniformly in $t\in(0,L]$, since \eqref{eq:standing-scaled} gives $q\le\kappa_0^2$ and $b_*^2\le\beta_0^2q$. Here we use Eqs.~\eqref{eq:pi-scaled}--\eqref{eq:K-scaled}.

This implies, for $w,w'$ in this neighborhood
\begin{align*}
    \EE[(g_w-g_{w'})^2]
    &=
    \EE\left[
        \left(
        \int_0^1
        \left\langle
        \nabla g_{w'+s(w-w')},w-w'
        \right\rangle
        \de s
        \right)^2
    \right] \\
    &\le
    |w-w'|^2
    \int_0^1
        \EE\|\nabla g_{w'+s(w-w')}\|^2
    \,\de s \\
    &\le
    C(|w|+|w'|)^2|w-w'|^2,
\end{align*}
yielding \eqref{eq:g-shift-sharp-metric}.

We next derive \eqref{eq:g-shift-pointwise} by
Sudakov--Fernique. For $j\ge0$, let
\[
    r_j:=2^{-j}r_N,
    \qquad
    B_j:=\{w\in\bbR^2:|w|\le r_j\}.
\]
For $w,w'\in B_j$, \eqref{eq:g-shift-sharp-metric} gives
\[
    \EE[(g_w-g_{w'})^2]
    \le C_0 r_j^2|w-w'|^2
\]
for a constant $C_0$ independent of $j,q,t$, and $N$.

Let $\bgamma=(\gamma_1,\gamma_2)\sim\cN(0,\bI_2)$ and define the
linear Gaussian process
\begin{equation}
    \label{eq:SF-linear-comparison}
    X^{(j)}_w
    :=
    \sqrt{C_0}\,r_j\langle\bgamma,w\rangle,
    \qquad w\in B_j.
\end{equation}
Its associated metric is
\[
    \EE[(X^{(j)}_w-X^{(j)}_{w'})^2]
    =
    C_0r_j^2|w-w'|^2,
\]
and therefore dominates the metric of $g_w-g_0$ on $B_j$.
An application of the Sudakov--Fernique inequality yields
\begin{align*}
    \EE\sup_{w\in B_j}(g_w-g_0)
    &\le
    \EE\sup_{w\in B_j}X^{(j)}_w \\
    &=
    \sqrt{C_0}\,r_j^2\EE\|\bgamma\|
    \le Cr_j^2.
\end{align*}
Applying the same argument to $-(g_w-g_0)$ gives
\begin{equation}
    \label{eq:g-ball-expected-sup}
    \EE\sup_{w\in B_j}|g_w-g_0|
    \le Cr_j^2.
\end{equation}

Moreover,
\[
    \sup_{w\in B_j}\Var(g_w-g_0)\le Cr_j^4.
\]
Hence, by the Borell--TIS inequality, 
\[
    \PP\left\{
        \sup_{w\in B_j}|g_w-g_0|>Mr_j
    \right\}
    \le
    2\exp\left\{-\frac{c}{r_j^2}\right\}
\]
for a sufficiently large constant $M$ and all sufficiently large $N$.

Taking a union bound over $j\ge0$ gives
\begin{align*}
    \PP\left\{
        \exists j\ge0:
        \sup_{w\in B_j}|g_w-g_0|>Mr_j
    \right\}
    &\le
    2\sum_{j\ge0}
        \exp\left\{-\frac{c}{r_j^2}\right\} \\
    &=
    2\sum_{j\ge0}
        \exp\left\{-\frac{c4^j}{r_N^2}\right\} \\
    &\le
    C\exp\left\{-\frac{c}{r_N^2}\right\}.
\end{align*}

On the complementary event, fix $w\ne0$ with $|w|\le r_N$, and choose
$j$ so that
\[
    r_{j+1}<|w|\le r_j.
\]
Then
\[
    |g_w-g_0|
    \le Mr_j
    \le 2M|w|.
\]
The claim is trivial for $w=0$, and hence
\[
    |g_w-g_0|\le C|w|
\]
simultaneously for all $w\in B_0$, yielding \eqref{eq:g-shift-pointwise}.

The derivative bounds in \eqref{eq:shift-reflection} follow from Lemma~\ref{lem:shift-cov} and Borell--TIS. The reflection bound follows by Taylor expansion using $\nabla_wg_0=0$.
\end{proof}
}

\begin{lem}\label{lem:lip-quad}
    The scaled GOE matrices $\bA^{(2)}(a,b)$ for $(a,b)\in R$ form a Gaussian process with metric 
    \begin{align*}
    \bbE \Big\{\big\|\bA^{(2)}(a,b)-\bA^{(2)}(a',b')\big\|^2_{\mathrm{F}} \Big\} \le CN(q|a-a'|^2+|b-b'|^2)\, .
    \end{align*}
    \rev{
    Furthermore, for any $\eta>0$ there exist constants $c,C>0$ such that,
    with probability at least $1-e^{-cN}$, 
    \begin{align}
    & \|\bA^{(2)}(a,b)-\bA^{(2)}(a',b')\|^2_{\mathrm{F}} \le CN(q|a-a'|^2+|b-b'|^2)\, ,\;\;\;\;
    \forall (a,b), (a',b')\in R\, ,
    \label{eq:BoundSupA2_Frobenius}\\
    &\|\bA^{(2)}(a,b)-\bA^{(2)}(a',b')\|^2_{\op} \le C(q|a-a'|^2+|b-b'|^2)\, ,
    \;\;\;\;
    \forall (a,b), (a',b')\in R\, .
    \label{eq:BoundSupA2_Operator}
    \end{align}
      and
    \begin{align}
    & \sup_{(a,b)\in R}\|\bA^{(2)}(a,b)-\bA^{(2)}(a_*,b_*)\|^2_{\mathrm{F}} \le CN^{2c_*}\, ,
     \label{eq:BoundSupA2_Frobenius_R}\\
    &\sup_{(a,b)\in R} \|\bA^{(2)}(a,b)-\bA^{(2)}(a_*,b_*)\|^2_{\op} \le CN^{-1+2c_*}\, .
    \label{eq:BoundSupA2_Operator_R}
    \end{align}
    }
\end{lem}
\begin{proof}
The bound on the canonical distance of $\bA^{(2)}$ follows from a straightforward calculation. \rev{The same bounds hold for the derivatives in $w$. Apply chaining
to these derivatives and integrate along line segments to obtain the
Lipschitz estimates above.}
 
 The bounds \eqref{eq:BoundSupA2_Frobenius} and \eqref{eq:BoundSupA2_Frobenius_R}
 follow from chaining on $R$, together with the standard 
    bound on chi-squared random variables
    \begin{align*}
        \bbP \lt( \|\bA^{(2)}(a,b) - \bA^{(2)}(a',b')\|^2_{\mathrm{F}} > \kappa \, \bbE \|\bA^{(2)}(a,b) - \bA^{(2)}(a',b')\|^2_{\mathrm{F}}\rt) \le 
        2e^{-cN^2[(\kappa-1)\wedge (\kappa-1)^2]},
    \end{align*}
    for $\kappa > 1$. 
    
    The bounds \eqref{eq:BoundSupA2_Operator} and  \eqref{eq:BoundSupA2_Operator_R} follow from a similar chaining argument.
    Indeed, $\bA^{(2)}(a,b) - \bA^{(2)}(a',b')$ is a matrix with independent entries with variance 
    bounded by $C(q|a-a_*|^2+|b-b_*|^2)/N$, whence, by standard estimates on the norm of
    Gaussian random matrices, the following holds for all $\kappa>\kappa_0$
    \begin{align*}
        \bbP \lt( \|\bA^{(2)}(a,b) - \bA^{(2)}(a',b')\|^2_{\op} > \kappa \, N^{-1}\bbE \|\bA^{(2)}(a,b) - \bA^{(2)}(a',b')\|^2_{\mathrm{F}}\rt) \le 
        2e^{-cN\kappa}.
    \end{align*}
\end{proof}

Recall $G(\gamma; \bA, \bu)$ \rev{from Subsection~\ref{sec:FirstLocal}} and $G_{a,b}(\gamma)$ in (\ref{eq:Gab}). \rev{Let $G_w=G_{a,b}$ and $W(w)=\EE_{\ge3}Z(a,b)$. }
The next lemma gives control over $G$ and $G_{a,b}$ for $(a,b)\in R$. 
\begin{lem}\label{lem:lip-quad-G}
    Given a compact interval $I \subseteq [M,\infty)$,
    where $M:=\eps + \bbE \lambda_{\max} (\bA^{(2)}(a_*,b_*))$, the following holds with probability at least $1-e^{-N^\delta}$ for appropriate $C_0,\rho,\delta>0$ depending on $\eps>0$:
    \begin{enumerate}
        \item We have 
        \begin{align}\label{eq:G-fluc}
            &\sup_{\gamma\in I, (a,b)\in R} \lt|n(G(\gamma; \bA^{(2)}(a,b), \bu(a,b))-G_{a,b}(\gamma)) - n(G(\gamma; \bA^{(2)}(a_*,b_*), \bu(a_*,b_*))-G_{a_*,b_*}(\gamma))\rt| \nonumber\\
            &= O(N^{-1/2+Cc_*}). 
        \end{align}
        The same estimates hold after taking up to two derivatives in $\gamma$. 
        \item We have 
        \rev{\begin{align}\label{eq:G-sym}
            \sup_{(a,b) \in R}\left|\log\frac{W(w)}{W(-w)}\right|=O(N^{-\rho}).
        \end{align}}

        \item The event \rev{from \cite[Lemma~3.1]{huang2026approximating}} holds uniformly in $(a,b)\in R$.
        Namely, 
        \begin{align}\label{eq:G-laplace-unif}
            &Z_{\le 2}(a,b) = \int e^{\<\tilde\bsig,\bA^{(2)}(a,b)\tilde\bsig\>+\<\bu(a,b),\tilde\bsig\>} \mu_0(\de\tilde\bsig) \nonumber\\
            &= (1+\Err_{a,b}(N)) (2e)^{-(N-2)/2} \sqrt{\frac{2}{G''(\gamma_{a,b}; \bA^{(2)}(a,b), \bu(a,b))}} e^{nG(\gamma_{a,b}; \bA^{(2)}(a,b), \bu(a,b))}\, ,
        \end{align}
        where $\sup_{(a,b)\in R} |\Err_{a,b}(N)|\le C_0 N^{-\rho}$.
    \end{enumerate}
\end{lem}
\begin{proof}
    We can represent $\bA^{(2)}(a,b) = \bA^{(2)}(a_*,b_*) + \bDel(a,b)$ where \rev{the entries on and above the diagonal of $\bDel(a,b)$ are independent Gaussian processes} with metric $\bbE[(\Delta_{i,j}(a,b) - \Delta_{i,j}(a',b'))^2]^{1/2} \le CN^{-1/2}(\rev{\sqrt{q}|a-a'|+|b-b'|})$, and $\bDel(a_*,b_*)=0$. \rev{Write $w=(\sqrt q(a-a_*),b-b_*)$ and $r_N=N^{-1/2+c_*}$.
    Put $\bA_*:=\bA^{(2)}(a_*,b_*)$ and $\alpha_*:=\alpha_{a_*,b_*}$.
    Gaussian concentration and
    Lemma~\ref{lem:lip-quad} give, with probability at least $1-Ce^{-cN}$,
    \[
        \lambda_{\max}(\bA_*)\le\bbE\lambda_{\max}(\bA_*)+\eps/4,
        \qquad \sup_{(a,b)\in R}\|\bDel(a,b)\|_{\op}\le Cr_N\le\eps/4.
    \]
    Since $\inf I\ge\bbE\lambda_{\max}(\bA_*)+\eps$, this implies
    \[
        \inf_{\gamma\in I,(a,b)\in R}
        \bigl(\gamma-\lambda_{\max}(\bA^{(2)}(a,b))\bigr)\ge\eps/2.
    \]
    We henceforth work under this event, together with $\sup_R\|\bA^{(2)}(a,b)\|_{\op}\le C$, which also occurs with probability
    $1-Ce^{-cN}$.}

    \rev{
    Let $\bQ_*(\gamma)=\gamma\bI-\bA_*$.
    Since $\bu(a_*,b_*)=0$, the exact expansion is
    \begin{align}\label{eq:Gvariation}
      G(\gamma;\bA_*+\bDel(a,b),\bu(a,b))
      &=G(\gamma;\bA_*,0)
       -\frac1{2n}\log\det\bigl(\bI-
          \bQ_*^{-1/2}\bDel(a,b)\bQ_*^{-1/2}\bigr)\nonumber\\
      &\quad+\frac1{4n}\langle\bu(a,b),
                     (\bQ_*-\bDel(a,b))^{-1}\bu(a,b)\rangle.
    \end{align}
    In particular, the leading quadratic field term is not subtracted.
    For $k\ge1$, set
    \[
      X_k(a,b)=\Tr\bigl((\bQ_*^{-1/2}\bDel(a,b)
                                      \bQ_*^{-1/2})^k\bigr).
    \]
    On the common spectral-gap event, Lemma~\ref{lem:lip-quad} gives
    \begin{equation}\label{eq:DeltaNormBound}
      \sup_{(a,b)\in R}\left\{\|\bDel(a,b)\|_{\op}\vee
        \|\bQ_*^{-1/2}\bDel(a,b)\bQ_*^{-1/2}\|_{\op}\right\}
      \le Cr_N.
    \end{equation}
    Telescoping the $k$th powers and using the Frobenius norm yields
    \[
      |X_k(a,b)-X_k(a',b')|
      \le C_k\sqrt n\,r_N^{k-1}
              \|\bA^{(2)}(a,b)-\bA^{(2)}(a',b')\|_{\mathrm F}
      \le C_k n r_N^{k-1}|w-w'|.
    \]
    Also $|\partial_\gamma X_k|\le C_k n r_N^k$.
    Consequently a net of mesh $N^{-L_{\rm net}}$ in $(w,\gamma)$
    makes the interpolation errors smaller than any prescribed power
    of $N$ by taking $L_{\rm net}$ sufficiently large.
    }

    Let $\bbP_\bDel$ and $\bbE_\bDel$ denote probability and expectation with respect to $\bDel(a,b)$ only, i.e. conditional on $\bA^{(2)}(a_*,b_*)$.
    Also let
    \[
        \bM(a,b) = \bQ_*(\gamma)^{-1/2}\bDel(a,b)\bQ_*(\gamma)^{-1/2}
    \]
    be the matrix appearing in $X_k(a,b)$. \rev{Write $\bA_*=\alpha_*\bW$ with $\bW\sim\GOE(n)$, where
    $\alpha_*=\alpha_{a_*,b_*}$. Then
    \[
        \bDel(a,b)=d_w\bW+\bB_w,\qquad
        \bB_w\ \text{independent of }\bW,\qquad
        \bB_w\ed e_w\bW',\quad \bW'\sim\GOE(n),
    \]
    where, if $\alpha_*>0$,
    \[
        d_w=\frac{n\Cov(\Delta_{12}(a,b),(\bA_*)_{12})}{\alpha_*},
        \qquad e_w^2=n\Var(\Delta_{12}(a,b))-d_w^2.
    \]
    Cauchy--Schwarz and the covariance increment bound imply
    $|d_w|+e_w\le C|w|$. If $\alpha_*=0$, take $d_w=0$ and
    $\bB_w=\bDel(a,b)$, with an auxiliary independent $\bW$.
    In an eigenbasis of $\bA_*$ the conditional mean of $\bM$ is
    diagonal, of norm $O(|w|)$ on an event with probability $1-\exp(-cN)$, and
    its centered entries on and above the diagonal are independent
    Gaussians with variances at most $Cr_N^2/n$.}
    It is well known (cf. \cite[Chapter 2]{anderson2010introduction}) that tracial moments of $\bM(a,b)$ amount to certain (weighted) cycle counts, and from this a routine calculation gives 
    \rev{
    \[
        \bbE_\bDel X_k=O_k(nr_N^k),\qquad
        \Var_\bDel(X_k)=O_k(r_N^{2k}),\qquad k=1,2.
    \]
    In particular, with $\bR_*=\bQ_*^{-1}$, the conditional means are
    \begin{align*}
        \bbE_\bDel X_1&=d_w\Tr(\bR_*\bW),\\
        \bbE_\bDel X_2&=d_w^2\Tr(\bR_*^2\bW^2)
          +\frac{e_w^2}{n}\left((\Tr\bR_*)^2+\Tr\bR_*^2\right).
    \end{align*}
    }
    
    For any fixed $(a,b) \in R$, Gaussian hypercontractivity gives
    \[
        \bbP_\bDel \lt\{|X_k(a,b) - \bbE_\bDel X_k(a,b)|\ge sr_N^k \rt\} \le \exp(-c_ks^{2/k}),
    \]
    because $X_k(a,b)$ is a degree $k$-polynomial in the entries of $\bM(a,b)$.
    \rev{Combining this with concentration of the conditional means,
    taking $s=N^\eta$ for a fixed $0<\eta<c_*$, and using the continuity
    estimates above, a union bound over an $N^{-L}$-net in $(w,\gamma)$
    gives, for $L$ sufficiently large,
    \[
        \sup_{\gamma\in I,(a,b)\in R}|X_k(a,b)-\bbE X_k(a,b)|
        \le C_kN^\eta r_N^k,\qquad k=1,2,
    \]
    with probability at least $1-Ce^{-N^{c(\eta)}}$. 
    }

    Recall that $|\log(1-x)+x+x^2/2|\le |x|^3$ for all $|x|\le 1/4$. \rev{By \eqref{eq:DeltaNormBound}, the trace of the Taylor
    remainder is bounded by $Cnr_N^3=O(N^{-1/2+3c_*})$, uniformly
    on the same event, also after up to two $\gamma$ derivatives. }
    Therefore, uniformly in $(a,b)\in R$,
    for all $\delta$ and $c$ small enough,
    \begin{align}\label{eq:G-logdetest}
        &\log \det \lt(\bI - \bQ_*(\gamma)^{-1/2}\bDel(a,b)\bQ_*(\gamma)^{-1/2}\rt)\nonumber \\
        &= -\sum_{k=1}^{2}\frac{1}{k}\Tr\lt(\lt(\bQ_*(\gamma)^{-1/2}\bDel(a,b)\bQ_*(\gamma)^{-1/2}\rt)^k\rt) + O(N^{-1/2+3c_*})\nonumber \\
        &= -\EE X_1(a,b)-\tfrac12\EE X_2(a,b) + O(N^{-1/2+3c_*+\eta}). 
    \end{align}

    We next turn to the term $\<\bu(a,b), (\bQ_*(\gamma)-\bDel(a,b))^{-1}\bu(a,b)\>$ in Eq.~\eqref{eq:Gvariation}. From Lemma \ref{lem:lip-field}, there are $\bu_1,\bu_2$ with $\|\bu_1\|,\|\bu_2\| = O(N^{1/2})$ such that 
    letting $\bu_0(a,b):=\sqrt{q}(a-a_*) \bu_1 + (b-b_*)\bu_2$, we have
    $\|\bu(a,b) - \bu_0(a,b)\|\le CN^{-1/2+2c_*}$ 
    and $\|\bu_0(a,b)\|\le CN^{c_*}$ for any $(a,b)\in R$.
    \rev{
    On the common spectral-gap and field events, the resolvent
    identity and \eqref{eq:DeltaNormBound} imply
    \[
      \|(\bQ_*-\bDel(a,b))^{-1}-\bQ_*^{-1}\|_{\op}\le Cr_N.
    \]
    Since $\|\bu_0\|\le C\sqrt N\,r_N$ and
    $\|\bu-\bu_0\|\le C\sqrt N\,r_N^2$, both replacing the
    resolvent by $\bQ_*^{-1}$ and replacing $\bu$ by $\bu_0$
    cost at most $CNr_N^3=O(N^{-1/2+3c_*})$. Hence, uniformly
    for $(a,b)\in R$ and $\gamma\in I$,
    \begin{align}\label{eq:G-fieldest}
      &\langle\bu(a,b),(\bQ_*(\gamma)-\bDel(a,b))^{-1}\bu(a,b)\rangle
         \nonumber\\
      &\quad=\langle\sqrt q(a-a_*)\bu_1+(b-b_*)\bu_2,
        \bQ_*(\gamma)^{-1}
        (\sqrt q(a-a_*)\bu_1+(b-b_*)\bu_2)\rangle
        +O(N^{-1/2+3c_*}).
    \end{align}
    The same argument after up to two $\gamma$ derivatives uses only
    additional uniformly bounded resolvents and gives the same order.
    }
    \rev{It remains to identify the deterministic terms. Let
    $\alpha_w=\alpha_{a,b}=\sqrt{\lambda_N/2}\,\tilde\gamma_2(a,b)$.
    The covariance metric gives $|\alpha_w-\alpha_*|\le C|w|$.
    The expected GOE linear-statistic estimate, applied also to the
    derivative of $\log(\gamma-\alpha x)$ in $\alpha$, gives
    \begin{align*}
        &\bbE\log\det(\gamma\bI-\bA^{(2)}(a,b))
          -\bbE\log\det\bQ_*\\
        &\qquad=n\int\log\frac{\gamma-\alpha_wx}{\gamma-\alpha_*x}
                    \,\mu_{\semic}(\de x)+O(|w|).
    \end{align*}
   Taking expectations in the same log-determinant
    expansion identifies $-\bbE X_1-\frac12\bbE X_2$ with this
    deterministic increment up to $O(Nr_N^3)$.

    The field family is independent of the quadratic matrix family,
    with $\Cov(\bu(a,b))=\lambda_N\tilde\gamma_1(a,b)^2I_n$ and
    $\tilde\gamma_1(a,b)^2=O(|w|^2)$. Gaussian quadratic-form
    concentration, Lemma~\ref{lem:lip-field}, and the expected
    semicircle resolvent estimate thus give
    \begin{align*}
        \langle\bu_0,\bQ_*^{-1}\bu_0\rangle
        &=N\tilde\gamma_1(a,b)^2
          \int\frac{\mu_{\semic}(\de x)}{\gamma-\alpha_wx}\\
        &\qquad+O(N^{1/2+\eta}r_N^2+Nr_N^3+r_N^2),
    \end{align*}
    uniformly on the same net and hence throughout $R\times I$.
    Here the replacement of $\bu$ by $\bu_0$ and of the central
    resolvent by the slice resolvent costs $O(Nr_N^3)$.
    This agrees with the field term in Eq.~\eqref{eq:Gab}.}

    Combining the last conclusion in Lemma \ref{lem:lip-field} with Eqs.~(\ref{eq:G-logdetest}) and (\ref{eq:G-fieldest}) and applying a union bound over $\gamma$ in any compact interval, we obtain, after changing $\delta$, that with probability at least $1-e^{-N^{\delta}}$, 
    \begin{align*}
    &\sup_{\gamma, (a,b)\in R} \lt|n(G(\gamma; \bA^{(2)}(a,b), \bu(a,b)) - G_{a,b}(\gamma) - G(\gamma; \bA^{(2)}(a_*,b_*),\bu(a_*,b_*)) + G_{a_*,b_*}(\gamma))\rt| \\
    &\le O(N^{-1/2+Cc_*}). 
    \end{align*} 
    Thus, with probability at least $1-e^{-N^{\delta}}$, Eq.~(\ref{eq:G-fluc}) holds uniformly in $(a,b)\in R$ and $\gamma$. \rev{Differentiating up to twice in $\gamma$ only inserts bounded
    resolvents into the trace and field calculations. The same estimates
    and net argument prove the derivative claims.}

    \rev{For Eq.~\eqref{eq:G-sym}, we include all factors in $W(w)$.
    By Lemmas~\ref{lem:Grad-E} and \ref{lem:anneal-cal},
    $\nabla_w E(0)=O(N^{-1})$ and the third derivatives of $E$ are
    uniformly bounded on $\mathcal Y$. Hence
    \[
        N|E(w)-E(-w)|\le C|w|+CN|w|^3.
    \]
    Choose a compact gap interval containing the quadratic saddles.
    Write
    \[
        F_w(\gamma)=n\left(G(\gamma;\bA^{(2)}(a,b),\bu(a,b))-G_{a,b}(\gamma)\right)
    \]
    and let $\bar\gamma_w$ minimize $G_{a,b}$.
    The central GOE trace estimates give
    $\sup_{\gamma}|\partial_\gamma^jF_0(\gamma)|\le CN^\eta$ for
    $j=0,1,2$ with probability $1-\exp(-N^{\eta})$.
    Together with Eq.~\eqref{eq:G-fluc} and the uniform lower bound
    on the saddle curvature, this gives
    \[
        n\min_\gamma G(\gamma;\bA^{(2)}(a,b),\bu(a,b))
        =nG_{a,b}(\bar\gamma_w)+F_w(\bar\gamma_w)+O(N^{-1+2\eta}).
    \]
    Since $\bar\gamma_w-\bar\gamma_0=O(|w|)$,
    $F_w(\bar\gamma_w)-F_0(\bar\gamma_0)=O(N^{-1/2+C c_*})$.
    The logarithm of the Laplace prefactor changes by $O(|w|)$.
    Using the uniform Laplace estimate proved below and
    Eq.~\eqref{eq:shift-reflection}, the remaining errors in the
    log ratio are $O(N^{-\rho_0})$ for some $\rho_0>0$ and
    $\sqrt N|g_w-g_{-w}|\le CN^{1/2+\eta}|w|^3$. Choosing $c_*$ and then
    $\rho>0$ sufficiently small proves Eq.~\eqref{eq:G-sym}.}
    
    Finally, we recall that for each $(a,b)\in R$, the event \rev{from \cite[Lemma~3.1]{huang2026approximating}} holds with probability at least $1-e^{-cN}$. \rev{The common replica-symmetry margin gives
    $\alpha_{a,b}\le\frac12\sqrt{1-\kappa/2}<1/2$, and the field
    has norm $O(N^{c_*})$, so the constants are uniform.} On the other hand, for appropriate $C>C'>1$, using Lemma \ref{lem:lip-field} and Lemma \ref{lem:lip-quad}, with probability $1-e^{-cN}$, uniformly over $(a,b),(a',b')\in R$ with $|\sqrt{q}(a-a')|+|b-b'|\le N^{-C}$, we have $Z_{\le 2}(a,b) = (1+O(N^{-C'}))Z_{\le 2}(a',b')$. Similar continuity estimates hold for the right hand side of Eq.~(\ref{eq:G-laplace-unif}). Taking a net of radius $N^{-C}$ in $w$ and applying the union bound, we obtain Eq.~(\ref{eq:G-laplace-unif}) uniformly in $(a,b)\in R$.
\end{proof}

\rev{Fix $\by-\by^\perp\in\mathcal Y$ and $\bm\in\mathcal D_\iota(t,\bx)$.
All probabilities below are conditional on $\bx,\by-\by^\perp,
\nabla\cF_\TAP(\bm)=0$, until the final averaging step.
For the slice tensors use the coefficient convention
$\bA^{(p)}(a,b)=\nabla^p\thH(0)/p!$, and use dimension $n=N-2$ in $G$.
Eq.~\eqref{eq:amp-works} gives a common analytic bound and, for large $N$,
$\lambda_N(1-z)^2\tilde\xi_{\ge2}''(z)\le1-\kappa$ on $[0,1)$,
so the uniform-mixture companion estimates apply with a common $\kappa>0$.
Fix a sufficiently large integer $L$, let $k=3L$, and choose
$c=c_*>c'>0$ sufficiently small for the companion moment estimate
with parameter $2L$ and the magnetization estimate at order at most $6$. The companion partition-function estimate is used with
$(L_{\rm comp},k)=(2L,3L)$, so that $L_{\rm comp}<k$.
Expanding there through degree $3L-1$ leaves remainder
$O(N^{-3L/2+9cL})$, which is $O(N^{-L})$ for $c<1/18$.
All constants and thresholds for $N$ are uniform for $t\in(0,T]$.} 
Define
\rev{\begin{align}
D_N(\eps) &: =\Big\{\bsig \in S_N:\; N^{-1/2+c} \le
\frac{1}{\sqrt{q}}|\<\bsig,\bm\>_N-a_*q|+|\<\bsig,\bx\>_N-b_*| \le \eps
\Big\}\, ,\\
\hat{D}_N(\eps) &: =\Big\{(a,b) \in \bbR^2:\; N^{-1/2+c} \le
\sqrt{q}|a-a_*|+|b-b_*| \le \eps
\Big\}\, ,
\end{align}}
Using Markov's inequality and the fact that the annealing upper bound $\widehat{E}(a,b)$ is strongly concave for $(a,b)\in \hat{D}_N(\eps)$
(see Lemma \ref{lem:anneal-cal} and Corollary \ref{cor:NegDef}),
we obtain that with probability $1-\exp(-N^{c'})$, for all sufficiently small $\eps>0$, there is $\eta>0$ such that
\rev{\begin{align}
\int_{D_N(\eps)} e^{H(\bsig)} \mu_0(\de\bsig)
\le e^{-\eta N^{2c}}q^{-1/2}N^{-1}W(a_*,b_*),\label{eq:AnnealedBoundFinal}
\end{align}}

Let
\begin{align*}
    Z(a,b) := \int e^{H(\bsig)}\mu_0^{a,b}(\bsig),
\end{align*}
and \begin{align*}
    Z := \int_{\Band_*(2\iota)}
        \exp(H_{N,t}(\bsig)) ~\mu_0(\de \bsig).
\end{align*}
Recall the definitions
\rev{\begin{align*}
    R &= \{(a,b)\in \bbR^2: \sqrt{q}|a-a_*|+|b-b_*|\le N^{-1/2+c}\}\, , \\
    R_+& = \{(a,b)\in \bbR^2: \sqrt{q}|a-a_*|+|b-b_*|\le N^{-1/2}\}\, .
\end{align*}}
Let 
\[
    Z_T(a,b) = \int \ind\{\bsig \in T(a,b)\} e^{H(\bsig)} \mu_0^{a,b}(\de \bsig),
    \qquad
    Z_T = \int_R Z_T(a,b)\,\de(a,b),\quad Z_R=\int_RZ(a,b)\,\de(a,b),
\]
\rev{where $T(a,b)$ is the typical set from
\cite[Eq.~(4.3)]{huang2026approximating}. Write $\mathcal B_{a,b}$ for
the event, measurable with respect to the field and quadratic part,
on which the companion estimates below hold and $\|\bu(a,b)\|\le n^{c_0}$.
For each fixed slice $\PP(\mathcal B_{a,b}^c)\le Ce^{-cN}$.
Set $T(a,b)=\varnothing$ on $\mathcal B_{a,b}^c$ and put
$W_T(a,b)=\EE_{\ge3}Z_T(a,b)$, $M=\int_RW(a,b)\,\de(a,b)$.
Here $\EE_{\ge3}$ averages the entire family of components of degree
at least $3$. On $\mathcal B_{a,b}$,
\begin{equation}\label{eq:ZT-event}
 0\le W(a,b)-W_T(a,b)\le e^{-cN}W(a,b).
\end{equation}}
Applying \cite[Lemma~4.2, Eq.~(4.9)]{huang2026approximating} with its parameter $L$ replaced by $2L$,
the following holds on $\mathcal B_{a,b}$, and holds trivially on its complement under the convention above: 
\begin{align}\label{eq:estZ-0}
    & \bbE_{\ge3} \lt[(Z_T(a,b)-\bbE_{\ge3} Z_T(a,b))^{2k}\rt] \le C_LN^{-L} \lt(\bbE_{\ge3}Z_T(a,b)\rt)^{2k},
\end{align}
and therefore
\begin{align}\label{eq:estZ-1}
    & \bbE_{\ge3} \lt[Z_T(a,b)^{2k}\rt] \le C_L \lt(\bbE_{\ge3}Z_T(a,b)\rt)^{2k}.
\end{align}
By Eqs.~\eqref{eq:G-fluc}, \eqref{eq:G-laplace-unif} of Lemma \ref{lem:lip-quad-G}, together with Eq.~\eqref{eq:(a,b)-reduction} and Lemma \ref{lem:lip-shift} (which implies that $e^{\sqrt{N}(g_{a,b} - g_{a',b'})} = O(1)$ for all $(a,b), (a',b') \in R_+$), with probability at least $1-e^{-N^\delta}$, uniformly in $(a,b)\in R_+$,
\begin{align} \label{eq:estZ-3}
    \bbE_{\ge3}Z(a,b) \ge \Omega\lt(\bbE_{\ge3}Z(a_*,b_*)\rt).
\end{align}
\rev{From Eq.~\eqref{eq:quad-upper-E}, Lemma~\ref{lem:lip-shift}, and
\begin{align}\label{eq:concave-est}
 C\sqrt N|w|-\eta N|w|^2\le C_\eta-\frac\eta2N|w|^2,
\end{align},}
we obtain that, uniformly in $(a,b)\in R$, 
\begin{align} \label{eq:estZ-3'}
    \bbE_{\ge3}Z(a,b) &\le 
    O\lt({\bbE_{\ge3}Z(a_*,b_*)} \cdot e^{- \eta N (q|a-a_*|^2+|b-b_*|^2)/2}\rt). 
\end{align}

Since $R_+$ has volume \rev{$\Theta(q^{-1/2}N^{-1})$,} on the event in \eqref{eq:estZ-3} we have 
\[
    \int_R \bbE_{\ge3} Z(a,b) \de(a,b) 
    \ge \int_{R_+} \bbE_{\ge3} Z(a,b) ~\de(a,b) 
    = \Omega(q^{-1/2}N^{-1} \bbE_{\ge3} Z(a_*,b_*)).
\]
Furthermore, when \eqref{eq:estZ-3'} holds we have
\[
    \int_R \bbE_{\ge3} Z(a,b) ~\de(a,b) 
    \le O(q^{-1/2}N^{-1} \bbE_{\ge3} Z(a_*,b_*))),
\]
where the \rev{$q^{-1/2}N^{-1}$} comes from integrating the exponential in \eqref{eq:estZ-3'}.
Thus, with probability at least $1-e^{-N^\delta}$,
\begin{align}
  \int_R \bbE_{\ge3} Z(a,b) ~\de(a,b) = \rev{\Theta(q^{-1/2}N^{-1} \bbE_{\ge3}Z(a_*,b_*))}.\label{eq:E3Z_EZ}
\end{align}

\rev{In particular,
\begin{equation}\label{eq:annealed-density-envelope}
 \frac{W(a,b)}{\int_R \bbE_{\ge3} Z(a,b) ~\de(a,b)}\le C\sqrt q\,N\exp\{-cN(q|a-a_*|^2+|b-b_*|^2)\}.
\end{equation}
Let $p_N(a,b)$ denote the right-hand side above. Since
$\int_Rp_N\,\de(a,b)=O(1)$ and $\PP(\mathcal B_{a,b}^c)\le Ce^{-cN}$,
Markov gives $\int_Rp_N\ind_{\mathcal B_{a,b}^c}\,\de(a,b)
\le e^{-cN}$ outside an exponentially small event. Together with
Eq.~\eqref{eq:ZT-event}, this gives
\begin{equation}\label{eq:compare-Z-ZT}
 \EE_{\ge3}(Z_R-Z_T)\le e^{-cN}\int_R \bbE_{\ge3} Z(a,b) ~\de(a,b).
\end{equation}
A further application of Markov's inequality gives $Z_R-Z_T\le e^{-cN/2}\int_R \bbE_{\ge3} Z(a,b) ~\de(a,b) $ outside an exponentially small event. Minkowski's inequality and \eqref{eq:estZ-0} give
\[
 \left(\EE_{\ge3}\left|Z_T-\int_RW_T(a,b)\,\de(a,b)\right|^{2k}\right)^{1/(2k)}
 \le C_LN^{-L/(2k)}\int_R \bbE_{\ge3} Z(a,b) ~\de(a,b) .
\]
Consequently,
\begin{align}\label{eq:estZ-2}
 \PP\left\{Z_R<\frac{1}{2}\int_R \bbE_{\ge3} Z(a,b) ~\de(a,b)\right\}\le C_LN^{-L}+Ce^{-N^{c'}}.
\end{align}

We can now control the contribution from outside $|w|\le\eps$.

\begin{lem}\label{lem:outer-band-tail}
On $\mathcal Y$, for every $D>0$, with conditional probability at least $1-C_DN^{-D}$,
the fraction of the partition function on $\Band_*(2\iota)$
from $|w|>\epsilon$ is at most $e^{-cN}$.
\end{lem}
\begin{proof}
For $t\ge t_0$, take $\iota$ small enough that the band is contained
in $|w|\le\epsilon$. For $t<t_0$, let $X=\sqrt q+w_1$, $Y=(w_2-\beta w_1)/\sqrt{1-\beta^2}$, so that
$r=X^2+Y^2$. On the band, $|b|\le q_*+2\iota$.
At $q=\beta=b=0$, the difference between $\overline{\hat E}$ and its
value at the center is
\[
 f_\tau(r)=\tfrac12\log(1-r)-\tfrac12\xi(r)
       +\frac{c_2(1+c_2)}{2c_2+\tau}r
       -\frac{c_2^2}{2(c_2+\tau)}r(1-r).
\]
Using $\xi(r)\ge c_2r^2/2$ and
$\log(1-r)\le-r-r^2/2$, we obtain
\[
 f_\tau(r)\le
 -\frac{\tau[\tau+c_2(1-c_2)]}{2(c_2+\tau)(2c_2+\tau)}r
 -\frac{1-c_2}{4}r^2.
\]
Thus there is a common strictly negative gap when $r\ge\epsilon^2/4$.
Continuity preserves this gap for sufficiently small
$t_0,\beta_0,\iota$ away from $r=1$; near $r=1$ the logarithmic
term dominates all remaining bounded terms. On $\mathcal Y$, \eqref{eq:Gamma-fluctuation} preserves this gap for $\hat E$. Also $|w|>\epsilon$
implies $r\ge\epsilon^2/4$ after decreasing $t_0$.
The covariance bounds give $\sup|g|=o(\sqrt N)$ with probability
$1-e^{-N^{c}}$. In the $w$ coordinates the coarea factor satisfies
$J_N\,\de a\,\de b=O(N)\,\de w$, uniformly in $q$.
Thus Markov's inequality bounds the outer mass by $e^{-cN}\int_R \bbE_{\ge3}Z(a,b)\,\de(a,b)$.
Apply Eq.~\eqref{eq:estZ-2}, choosing its moment parameter sufficiently
large for the prescribed $D$, to conclude.
\end{proof}
}

\begin{proof}[Proof of Proposition \ref{ppn:local-barycenter}]
\rev{We prove the estimate on $\mathcal D_\iota(t,\bx)$ for $t\in(0,T]$,
first for the third moment control.}

Denote by $\<\, \cdot \, \>_{a,b}$ the average with respect to the Gibbs measure restricted to band $D_N(a,b)$, namely with respect to
$\mu^{a,b}(\de\bsig)\propto \exp\{H(\bsig)\}\, \mu^{a,b}_0(\de\bsig)$. Note that $\bu(a_*,b_*)=0$, and for $\sqrt{q}|a-a_*|+|b-b_*|\le N^{-1/2+c}$, we have $\|\bu(a,b)\| = O(N^{c})$.

Recall $\gamma_{a,b} = \arg \min_{z>\lambda_{\max}(\bA^{(2)}(a,b))} G(z; \bA^{(2)}(a,b), \bu(a,b))$. Let
\rev{\begin{align*}
    \corr(a,b) &:= \frac{s(a,b)}{2}(\gamma_{a,b}\bI - \bA^{(2)}(a,b))^{-1} \bu(a,b) + \corr_3(a,b),\\
    \corr_3(a,b)&:=\frac{3s(a,b)}{4}(\gamma_{a,b}\bI - \bA^{(2)}(a,b))^{-1}
    \langle \bA^{(3)}(a,b),(\gamma_{a,b}\bI - \bA^{(2)}(a,b))^{-1}\rangle,
    \\[-2pt] s(a,b)&:=\sqrt{\frac{N(1-r(a,b))}{N-2}}.
\end{align*}
We first compare the correction at $(a_*,b_*)$ with $\corr(\bm)$. Write $U=\{\bm,\bx\}^{\perp}$, $\bB=\bA^{(2)}(\bm)$, $\bT=\bA^{(3)}(\bm)$, $Q=\bQ(\bm)$, and $s_*=s(a_*,b_*)$. At the central slice,
\[
 \bA^{(2)}(a_*,b_*)=\frac{s_*^2}{2}\bB|_U,
 \qquad \bA^{(3)}(a_*,b_*)=\frac{s_*^3}{6}\bT|_{U^{\otimes3}}.
\]
Set $Q_*=(\gamma_{a_*,b_*}\bI-\bA^{(2)}(a_*,b_*))^{-1}$ and $R_U=s_*^2Q_*/2$, extended by zero on $U^\perp$. Then
\[
 \Tr R_U=\Tr Q=N(1-q),\qquad
 \corr_3(a_*,b_*)=\frac12R_U\langle\bT,R_U\rangle.
\]
On the spectral events for the full tangent Hessian and its compression to $U$, we have
$|\gamma_{*,N}(\bm)-2\gamma_{a_*,b_*}/s_*^2|=O(N^{-1})$ and $\|Q-R_U\|_{\mathrm F}=O(1)$.
The centered part $\bT_0$ of $\bT$ is independent of $\bB$ under our conditioning, and its covariance is $\xi'''(q)N^{-2}$ times the sum over the six index pairings. Consequently,
\[
 \EE_{\bT_0}\left\|\frac12Q\langle\bT_0,Q\rangle
       -\frac12R_U\langle\bT_0,R_U\rangle\right\|^2
 \le \frac{C}{N}\|Q-R_U\|_{\mathrm F}^2.
\]
The planted part $\xi'''(b_*)N^{-2}(P_{\bm}^{\perp}\bx)^{\otimes3}$ contributes $O(N^{-1/2})$ to either correction. Gaussian concentration therefore gives
$\corr_3(a_*,b_*)=\corr(\bm)+O(N^{-1/2+\eta})$ outside an event of probability $Ce^{-cN^{2\eta}}$.

The centered cubic slice family is independent of the field and quadratic families. Differentiating its covariance in $w$ preserves the same index-pairing form, with coefficients $O(N^{-2})$ for every fixed number of derivatives, uniformly in $t$. The same contraction estimate, Gaussian concentration, and a net in $w$ therefore give
$\sup_R\|\nabla_w\corr_3(a,b)\|\le CN^\eta$ outside an event of probability $Ce^{-cN^{\eta}}$. Here we include the bounds on the first two $w$-derivatives of the field and quadratic families, obtained by the same Gaussian-process estimates from their analytic covariances. Their common spectral gap then bounds the first two derivatives of the resolvents. Those lemmas also bound the change in the linear correction by $O(N^{-1/2+Cc})$. Taking $0<\eta<c$, we obtain, on an event of probability $1-Ce^{-cN^\eta}$,
\begin{equation}\label{eq:cubic-correction-continuity}
 \corr(a,b)=\frac{s(a_*,b_*)}{2}Q_*\bu(a,b)+\corr(\bm)
 +O(N^{-1/2+Cc}),\qquad (a,b)\in R.
\end{equation}}

%
\rev{Let $\cE$ be the intersection of $\mathcal Y$, the field event
in the first conclusion of Lemma~\ref{lem:lip-field}, and the event of
Lemma~\ref{lem:lip-quad-G} and the spectral event in its proof. We also include the profile bounds
\eqref{eq:estZ-3}, \eqref{eq:estZ-3'}, the correction estimate
\eqref{eq:cubic-correction-continuity}, the discarded-mass bounds
\eqref{eq:compare-Z-ZT}, \eqref{eq:AnnealedBoundFinal}, and
Lemma~\ref{lem:outer-band-tail} with $D=L$, and the lower-bound event
controlled by Eq.~\eqref{eq:estZ-2}. Finally, include the Markov bound
following Eq.~\eqref{eq:compare-Z-ZT}:
\[
 Z_R-Z_T\le e^{-cN/2}\int_R\bbE_{\ge3}Z(a,b)\,\de(a,b).
\]
All bounds hold on their stated domains, with constants independent of
$t\in(0,T]$ and $\bm\in\mathcal D_\iota(t,\bx)$. 
By the referenced estimates and a union bound, after decreasing $c'$ if
necessary,
\[
 \PP(\cE^c)\le C_LN^{-L},\qquad
 \frac{Z-Z_T}{Z}\le e^{-N^{c'}}\quad\text{on }\cE.
\]
Consequently, on $\cE$,
\[
 \tbm_{2\iota}(\bm)
 =\int_R\frac{Z_T(a,b)}Z\langle\bsig\rangle_{a,b}\,\de(a,b)
       +O(\sqrt N e^{-N^{c'}}).
\]
}
By the above, we have $\bbP(\cE)\ge 1-CN^{-L/2}$. 
%
Under the event $\cE$, \rev{from $Z\ge \frac12\int_R \bbE_{\ge3}Z(a,b)\,\de(a,b)$, Eqs.~\eqref{eq:estZ-1}, \eqref{eq:annealed-density-envelope}}, we thus obtain
\begin{align}\label{eq:estZ-4}
    \int_R \bbE\lt [\bfone_{\cE}\lt(\frac{Z_T(a,b)}{Z}\rt)^{2}\rt]^{1/2} \de(a,b) = O(1),
\end{align}

By Jensen's inequality and the Cauchy--Schwarz inequality,
\begin{align*}
    & \bbE\lt[\bfone_{\cE}\lt\|\int_{R} \frac{Z_T(a,b)}{Z} (\<\bsig\>_{a,b} - \bcorr(a,b) - \bm(a,b))\de(a,b)\rt\|^{3} \rt] \\
    &\le \int_{R}\bbE\lt [ \bfone_{\cE}\frac{Z_T(a,b)}{Z} \lt\|(\<\bsig\>_{a,b} - \bcorr(a,b) - \bm(a,b))\rt\|^{3} \rt] \de(a,b)\\
    &\le \int_{R} \bbE\lt [\bfone_{\cE}\lt(\frac{Z_T(a,b)}{Z}\rt)^2\rt]^{1/2} \bbE\lt[\bfone_{\mathcal B_{a,b}}\lt\|(\<\bsig\>_{a,b} - \bcorr(a,b) - \bm(a,b))\rt\|^{6}\rt]^{1/2} \de(a,b).
\end{align*}
\rev{Here we have used $Z_T(a,b)=0$ on $\mathcal B_{a,b}^c$.} By Theorem \ref{lem:correction-band}, we have
\begin{align}
    \bbE\lt[\rev{\ind_{\mathcal B_{a,b}}}\lt\|\<\bsig\>_{a,b} - \corr(a,b) - \bm(a,b)\rt\|^{6}\rt] \le N^{-c}.
\end{align}
Combining with Eq.~(\ref{eq:estZ-4}), we obtain that
\begin{align}
    \bbE\lt[\bfone_{\cE}\lt\|\int_{R} \frac{Z_T(a,b)}{Z} (\<\bsig\>_{a,b} - \corr(a,b) - \bm(a,b))\de(a,b)\rt\|^{2+\delta} \rt]\le O(N^{-c}).
\end{align}

On the other hand, \rev{on $\cE$, the vectors $\bu_1,\bu_2$ from Lemma~\ref{lem:lip-field} satisfy $\|\bu_1\|,\|\bu_2\| = O(N^{1/2})$ and, for $d_q(a-a_*,b-b_*)\le N^{-1/2+c}$,}
\begin{align}
    \|\bu(a,b) - (\sqrt{q}(a-a_*)\bu_1+(b-b_*)\bu_2)\| = O(N^{-1/2+2c}).  
\end{align}
Using this, letting $\oZ = \bbE_{\ge 3}Z(a_*,b_*)$ and 
defining $\oa:= a-a_*$, $\ob:=b-b_*$, we have
\rev{\begin{align*}
    &\int_R \frac{Z_T(a,b)}{Z} \bu(a,b)\de(a,b)\\
    &= \int_R \frac{Z_T(a,b) - \oZ}{Z} (\sqrt{q}\oa\bu_1+\ob\bu_2)\de(a,b) + \int_R \frac{\oZ}{Z} (\sqrt{q}\oa\bu_1+\ob\bu_2)\de(a,b) + O(N^{-1/2+2c})\\
    &= \int_R \frac{Z_T(a,b) - \oZ}{Z} (\sqrt{q}\oa\bu_1+\ob\bu_2)\de(a,b) + O(N^{-1/2+2c}) \\
    &= \int_R \frac{Z_T(a,b) - W_T(a,b)}{Z} (\sqrt{q}\oa\bu_1+\ob\bu_2)\de(a,b) \\
    &\qquad \qquad + \int_R \frac{\bbE_{\ge3}Z(a,b) - \bbE_{\ge3}Z(a_*,b_*)}{Z} (\sqrt{q}\oa\bu_1+\ob\bu_2)\de(a,b) + O(N^{-1/2+2c}). 
\end{align*}
Here we recall that $W_T(a,b) = \bbE_{\ge3} Z_T(a,b)$ and $W(a,b) = \bbE_{\ge 3}Z(a,b)$. \rev{Replacing $W_T$ by $W$ in the preceding
display costs $O(N^c e^{-cN})$ on $\cE$, by Eq.~\eqref{eq:compare-Z-ZT}. }
 Furthermore, by H\"older's inequality on the measure $\fr{\ind\{(a,b)\in R\} \de(a,b)}{\Vol(R)}$, 
\baln
    &\bbE\lt[\bfone_{\cE}\lt\|\int_R \frac{Z_T(a,b) - W_T(a,b)}{Z} (\sqrt{q}(a-a_*)\bu_1+(b-b_*)\bu_2) \de(a,b)\rt\|^{3}\rt]\\
    &\le \int_R \bbE \lt[\lt(\bfone_{\cE}\lt|\frac{Z_T(a,b) - W_T(a,b)}{Z}\rt| \lt\|\sqrt{q}(a-a_*)\bu_1+(b-b_*)\bu_2\rt\| \Vol(R) \rt)^{3}\rt] \fr{\de(a,b)}{\Vol(R)}
\ealn}
\rev{By $Z\ge \frac12\int_R \bbE_{\ge3}Z(a',b')\,\de(a',b')$ on $\cE$ and Eq.~\eqref{eq:annealed-density-envelope}, $Z=\Omega(q^{-1/2}N^{-1}W(a,b))$. 
Moreover, we have the estimates $\sqrt{q}|a-a_*|,|b-b_*| \le N^{-1/2 + c}$ by definition of $R$, $\Vol(R) = O(q^{-1/2}N^{-1+2c})$}, and $\norm{\bu_1},\norm{\bu_2} \le C\sqrt{N}$.
Combining these estimates, the last display is bounded by
\beq
    \label{eq:holder-intermediate-bd}
    O(N^{9c}) 
    \int_R \bbE \lt[\bfone_{\cE}\lt|\frac{Z_T(a,b) - W_T(a,b)}{W(a,b)}\rt|^{3}\rt] \fr{\de(a,b)}{\Vol(R)}
\eeq
\rev{By \eqref{eq:estZ-0}, the expectation in this display is at most $C_LN^{-3L/(2k)}$. Thus
Eq.~\eqref{eq:holder-intermediate-bd} is $O(N^{-1/2+9c})=O(N^{-c_1})$
for $c$ sufficiently small and some $c_1>0$.}

By Eqs.~(\ref{eq:G-laplace-unif}) and (\ref{eq:G-sym}) of Lemma \ref{lem:lip-quad-G}, \rev{using $R=-R$ in $w$ and
\[
 \int_R\frac{W(a,b)}{\int_R \bbE_{\ge3}Z(a',b')\,\de(a',b')}\sqrt N|w|\,\de(a,b)=O(1)
\]
from
Eq.~\eqref{eq:annealed-density-envelope}, on $\cE$,}
\begin{align*}
    \lt\|\int_R \frac{\bbE_{\ge3}Z(a,b) - \bbE_{\ge3}Z(a_*,b_*)}{Z} (\sqrt{q}\oa\bu_1+\ob\bu_2)\de(a,b)\rt\| = O(N^{-c}). 
\end{align*}
Similarly, we have
\begin{align*}
    \bbE \lt[\bfone_{\cE}\lt\|\int_R \frac{Z_T(a,b)-W_T(a,b)}{Z} (\bm(a,b)-\bm) \de(a,b)\de(a,b)\rt\|^{3}\rt] &= O(N^{-c}).
\end{align*}
Again by Eqs.~(\ref{eq:G-laplace-unif}) and (\ref{eq:G-sym}) of Lemma \ref{lem:lip-quad-G}, noting that $\bm(a,b)+\bm(a',b')=2\bm$ if $(a,b)+(a',b')=2(a_*,b_*)$, 
\begin{align*}
    \bbE \lt[\bfone_\cE\lt\|\int_R \frac{\bbE_{\ge3}Z(a,b)}{Z} (\bm(a,b)-\bm) \de(a,b)\rt\|^{3}\rt] &= O(N^{-c}).
\end{align*}
\rev{The constant term $\bm+\corr(\bm)$ has weight $Z_T/Z=1+O(e^{-N^{c'}})$ on $\cE$. Combining the preceding estimates with \eqref{eq:cubic-correction-continuity} gives
\begin{align}
 \EE\left[\bfone_\cE\|\tbm_{2\iota}(\bm)-\bm-\corr(\bm)\|^{3}\right]\le CN^{-c_1}.
\end{align}
Averaging over $\by-\by^\perp$, we now condition only on $\nabla\cF_\TAP(\bm)=0,\bx$, and recall $\PP(\cE^c)\le C_LN^{-L}$. By \eqref{eq:GammaStarCorr}, $\bQ\succ0$ and $\Tr\bQ=N(1-q)\le N$, so $\|\corr(\bm)\|\le N^2\|\bA^{(3)}(\bm)\|_{\mathrm F}/2$. The Gaussian tensor $\bA^{(3)}(\bm)$ in \eqref{eq:A2A3-Def} is independent of $\by-\by^\perp$ under this conditional law and has moments polynomial in $N$, uniformly on the local domain. Together with $\|\tbm_{2\iota}(\bm)\|\le\sqrt N$, H\"older's inequality therefore removes $\bfone_\cE$ from the last display for sufficiently large $L$. After possibly decreasing $c_1$, we have
$\EE\|\tbm_{2\iota}(\bm)-\bm-\corr(\bm)\|^3\le CN^{-c_1}$.
Choose $0<\delta<\min(1,c_1)$. By H\"older and normalization,
\[
 \EE\|\tbm_{2\iota}(\bm)-\bm-\corr(\bm)\|_N^{2+\delta}
 \le C N^{-1-\delta/2-c_1(2+\delta)/3}
 \le N^{-(1+\delta)}.
\]
The constants and the threshold for $N$ are uniform on
$\mathcal D_\iota(t,\bx)$, $t\in(0,T]$. }
\end{proof}

\section{Lognormal fluctuations of partition function}
\label{sec:partition-fn-fluctuations}

\begin{proof}[Proof of Lemma~\ref{lem:partition-fn-fluctuations}]
    Recall that $H_{N,2}$ denotes the degree-$2$ part of $H_N$, which is of the form
    \[
        H_{N,2}(\bsig) = \fr{\xi''(0)^{1/2}}{2} \la \bG \bsig, \bsig \ra,
    \]
    for $\bG \sim \GOE(N)$. Let
    \[
        Z_{N,2} = \int_{S_N} \exp H_{N,2}(\bsig) ~\de \mu_0(\bsig).
    \]
    It follows from \cite[Theorem 1.2]{baik2016fluctuations} (with $w_2=2, W_4=3$) that, with $\sigma^2 = -\fr12 \log (1-\xi''(0))$ and $W \sim \cN(-\fr12 \sigma^2, \sigma^2)$,
    \beq
        \label{eq:ZN2-fluctuations}
        \fr{Z_{N,2}}{\EE Z_{N,2}}
        = \fr{Z_{N,2}}{\exp(N\xi''(0)/2)}
        \stackrel{d}{\to} \exp(W).
    \eeq
    Recall that the  estimates recorded in
    Section~\ref{sec:FirstLocal} only require
    \eqref{eq:replica-symmetry}, rather than \eqref{eq:amp-works}, and
    therefore apply here.  Fix a small $\delta>0$ and let $T=T(\delta)$
    be the typical set defined in \cite[Eq.~(4.3)]{huang2026approximating}:
    points for which the Gibbs
    mass at overlap larger than $\delta$ is exponentially small.  Write
    \[
        Z_N(T) = \int_T \exp H_N(\bsig) ~\de \mu_0(\bsig).
    \]
    The truncation estimate 
\rev{\cite[Lemma~4.1, Eq.~(4.8)]{huang2026approximating}} gives
    \[
        \EE[Z_N - Z_N(T)] \le e^{-cN} \EE [Z_N].
    \]
    By Markov's inequality, applied respectively to the randomness of $H_N$ and $H_{N,2}$, with probability $1-e^{-cN}$,
    \[
        (Z_N - Z_N(T))\vee {\EE}_{\ge 3}[Z_N - Z_N(T)] \le e^{-cN} \EE [Z_N].
    \]
    Its conditional moment estimate \rev{\cite[Lemma~4.2, Eq.~(4.9)]{huang2026approximating}} (with $k=2$) also gives, with
    probability $1-o(1)$,
    \[
        {\EE}_{\ge 3}\lt[\lt(
            Z_N(T) - {\EE}_{\ge 3}[Z_N(T)]
        \rt)^2\rt] = o(1) \EE [Z_N]^2.
    \]
    Thus, with probability $1-o(1)$,
    \[
        |Z_N(T) - {\EE}_{\ge 3} Z_N(T)| \le o(1) \EE [Z_N].
    \]
    On the intersection of these events,
    \[
        \lt|\fr{Z_N}{\EE Z_N} - \fr{\EE_{\ge 3} Z_N}{\EE Z_N} \rt|
        \le
        \fr{|Z_N - Z_N(T)|}{\EE Z_N}
        + \fr{|Z_N(T) - \EE_{\ge 3}Z_N(T)|}{\EE Z_N}
        + \fr{|\EE_{\ge 3} Z_N - \EE_{\ge 3}Z_N(T)|}{\EE Z_N}
        = o(1).
    \]
    Since
    \[
        \fr{\EE_{\ge 3} Z_N}{\EE Z_N} = \fr{Z_{N,2}}{\EE Z_{N,2}},
    \]
    the result follows from \eqref{eq:ZN2-fluctuations}.
\end{proof}

\section{Proof of Proposition \ref{ppn:yT-close-in-TV} and Proposition~\ref{ppn:mala-accuracy}}
\label{sec:girsanov}

In this section we complete the missing 
details in the proofs of Proposition 
\ref{ppn:yT-close-in-TV} and 
Proposition~\ref{ppn:mala-accuracy},
which were sketched in Section~\ref{sec:proof-main-theorem}.

\subsection{TV-closeness of Euler discretization: Proof of Proposition \ref{ppn:yT-close-in-TV}}
\label{sec:Proof-Euler-Closeness}

We prove Proposition~\ref{ppn:yT-close-in-TV} by an application of
Girsanov's theorem, an approach introduced
in \cite{chen2022sampling} in a related context.
For all $0\le \ell \le L-1$, define
\rev{
\[
 \hbm(\by,0)=\bzero,\qquad
 \hbm(\by,\ell\delta)=\bm^\alg(H_N,\by,\ell\delta),
 \quad 1\le\ell\le L-1.
\]
}
to be the output of Algorithm~\ref{alg:mean} with these inputs.
Then, define the process $(\hby_t)_{t\in [0,T]}$ by $\hby_0 = \bfzero$ and, for $t\in [\ell \delta, (\ell+1) \delta)$,
\beq
    \label{eq:sl-sde-discretized-drift}
    \de \hby_t = \hbm(\hby_{\ell \delta},\ell \delta)~\de t + \de\bB_t.
\eeq
On each interval $[\ell \delta, (\ell+1) \delta)$, the drift in \eqref{eq:sl-sde-discretized-drift} is constant, so this SDE can be integrated directly: conditional on $H_N, \hby_{\ell \delta}$,
\[
    \hby_{(\ell+1) \delta}
    =
    \hby_{\ell \delta}
    + \delta \hbm(\hby_{\ell \delta},\ell\delta)
    + \bB_{(\ell+1)\delta}-\bB_{\ell\delta}.
\]
Note that $\bB_{(\ell+1)\delta}-\bB_{\ell\delta} =_d \sqrt{\delta} \bw^\ell$ for $\bw^\ell \sim \cN(0,\bI_N)$, so this is precisely the Euler discretization in Algorithm~\ref{alg:main}.
It follows that
\beq
    \label{eq:euler-disc-matches-discretized-sde}
    \cL(\hby_{T} | H_N) =\cL(\by^L | H_N).
\eeq
\begin{lem}
    \label{lem:KL-girsanov}
    Given $H_N$, let $(\by_t)_{t\in [0,T]}$ be generated by \eqref{eq:sl-sde} and $(\hby_t)_{t\in [0,T]}$ be generated by \eqref{eq:sl-sde-discretized-drift}.
    Then,
    \[
        \bbE_{H_N \sim \bbP}
        \KL(\cL(\by_{T} | H_N),\cL(\hby_{T} | H_N))
        \le
        \fr12
        \sum_{\ell = 0}^{L-1}
        \int_{\ell \delta}^{(\ell + 1)\delta}
        \bbE_{\bbP}
        \norm{\hbm(\by_{\ell \delta},\ell \delta) - \bm(\by_t,t)}^2
        ~\de t.
    \]
\end{lem}
\begin{proof}
    Fix any realization of $H_N$.
    For $0\le \ell \le L-1$ and $t\in [\ell \delta, (\ell+1) \delta)$, define the process
    \[
        \bb_t = \hbm(\by_{\ell \delta},\ell \delta) - \bm(\by_t,t).
    \]
    Let
    \[
        \cE_t = \exp \lt(
            \int_0^t \la \bb_s, \de \bB_s\ra
            - \fr12 \int_0^t \norm{\bb_s}^2 ~\de s
        \rt).
    \]
    Let $Q$ be the probability measure (conditional on $H_N$) under which $(\bB_t)_{t\in [0,T]}$ is a Brownian motion and let $P$ be the probability measure with $\fr{\de P}{\de Q} = \cE_T$.
    By Girsanov's theorem \cite[Theorem 5.22]{le2016brownian},
    \[
        \bbeta_t = \bB_t - \int_0^t \bb_s ~\de s
    \]
    is a Brownian motion under $P$.
    (Since $\norm{\hbm(\by_{\ell \delta},\ell \delta)}, \norm{\bm(\by_t,t)} \le \sqrt{N}$, $\bb_t$ is a.s. bounded, and thus the conditions of Girsanov's theorem are satisfied.)
    The SDE \eqref{eq:sl-sde} rearranges as
    \[
        \de \by_t = (\bm(\by_t,t) + \bb_t)~\de t + \de \bbeta_t
        = \hbm(\by_{\ell \delta},\ell \delta)~\de t + \de \bbeta_t, \qquad t\in [\ell \delta, (\ell + 1)\delta).
    \]
    Thus, under $P$, the law of $(\by_t)_{t\in [0,T]}$ is that of $(\hby_t)_{t\in [0,T]}$.
    \rev{By the data processing inequality for relative entropy,}
    \[
        \KL(\cL(\by_T | H_N), \cL(\hby_T | H_N))
        \le \KL(Q,P)
        = \bbE_Q \log \fr{\de Q}{\de P}
        = \fr12 \int_0^T \bbE_Q \norm{\bb_t}^2 ~\de t.
    \]
    The result follows by taking expectation over $H_N$.
\end{proof}

\begin{lem}
    \label{lem:girsanov-error}
    \rev{We have
     \[
      \sum_{\ell=0}^{L-1}\int_{\ell\delta}^{(\ell+1)\delta}
      \bbE_{\bbP}\|\hbm(\by_{\ell\delta},\ell\delta)
                           -\bm(\by_t,t)\|^2\,\de t=o_N(1).
    \]}
\end{lem}
\begin{proof}
    \rev{The error incurred at $\ell=0$ is bounded by $\delta N = N^{-3}$.} We first estimate
    \[
        \bbE_{\bbP} \norm{\hbm(\by_{\ell \delta},\ell \delta) - \bm(\by_t,t)}^2
        \le 2\bbE_{\bbP} \norm{\hbm(\by_{\ell \delta},\ell \delta) - \bm(\by_{\ell \delta},\ell \delta)}^2
        + 2\bbE_{\bbP} \norm{\bm(\by_{\ell \delta},\ell \delta) - \bm(\by_t,t)}^2.
    \]
    \rev{By the uniform estimate in Theorem~\ref{thm:mean}, the first term on
    the right-hand side satisfies
    \[
        \sup_{1\le \ell\le L-1}
        \bbE_{\bbP}\norm{\hbm(\by_{\ell\delta},\ell\delta)
        -\bm(\by_{\ell\delta},\ell\delta)}^2
        =N\sup_{1\le \ell\le L-1}
        \bbE_{\bbP}\tnorm{\hbm(\by_{\ell\delta},\ell\delta)
        -\bm(\by_{\ell\delta},\ell\delta)}_N^2
        =o_N(1).
    \]
    It therefore suffices to bound the second term uniformly in $\ell$ and
    $t$.}
    Recall that for \rev{$(\bx,H_N,(\by_t)_{t\in [0,T]}) \sim \bbP$}, conditional on $(H_N,\by_t)$ the posterior law on $\bx$ is $\mu_t(\bsig) \propto e^{H_{N,t}(\bsig)}$, for $H_{N,t}(\bsig)$ as in \eqref{eq:HNt-with-y}.
    Furthermore, for $s = t - \ell \delta$ and $\bg \sim \cN(0,I_N)$,
    \[
        H_{N,t}(\bsig) = H_{N,\ell \delta}(\bsig) + \la s \bx + \sqrt{s} \bg, \bsig \ra.
    \]
    Let $\bDel_{t,\ell\delta}(\bsig) = H_{N,t}(\bsig) - H_{N,\ell\delta}(\bsig)$.
    With probability $1-e^{-cN}$, $\norm{\bg} \le 2\sqrt{N}$.
    Let $\cE$ denote this event.
    \rev{Recall from Algorithm~\ref{alg:main} that $\delta=N^{-4}$.}
    On $\cE$,
    \beq
        \label{eq:girsanov-error-H-difference}
        \sup_{\bsig \in S_N} \norm{\bDel_{t,\ell\delta}(\bsig)}
        \le \delta \sqrt{N}\norm{\bx} + \sqrt{\delta N} \norm{\bg}
        \le 3\sqrt{\delta} N
        = 3/N.
    \eeq
    So,
    \baln
        \bm(\by_{\ell \delta},\ell \delta) - \bm(\by_t,t)
        &= \fr{
            \int_{S_N} \bsig e^{H_{N,\ell \delta}(\bsig)}~\rev{\mu_0(\de \bsig)}
        }{
            \int_{S_N} e^{H_{N,\ell \delta}(\bsig)}~\rev{\mu_0(\de \bsig)}
        }
        - \fr{
            \int_{S_N} \bsig e^{H_{N,t}(\bsig)}~\rev{\mu_0(\de \bsig)}
        }{
            \int_{S_N} e^{H_{N,t}(\bsig)}~\rev{\mu_0(\de \bsig)}
        } \\
        &= \fr{
            \iint \bsig^1 (
            e^{H_{N,\ell \delta}(\bsig^1) + H_{N,t}(\bsig^2)}
            - e^{H_{N,\ell \delta}(\bsig^2) + H_{N,t}(\bsig^1)}
            )~\mu_0^{\otimes 2}(\de \bsig)
        }{
            \iint
            e^{H_{N,\ell \delta}(\bsig^1) + H_{N,t}(\bsig^2)}
            ~\mu_0^{\otimes 2}(\de \bsig)
        } \\
        &= \fr{
            \iint \bsig^1 (
                e^{\bDel_{t,\ell\delta}(\bsig^1)}
                - e^{\bDel_{t,\ell\delta}(\bsig^2)}
            )
            e^{H_{N,\ell\delta}(\bsig^1) + H_{N,\ell\delta}(\bsig^2)}
            ~\mu_0^{\otimes 2}(\de \bsig)
        }{
            \iint
            e^{\bDel_{t,\ell\delta}(\bsig^2)}
            e^{H_{N,\ell\delta}(\bsig^1) + H_{N,\ell\delta}(\bsig^2)}
            ~~\mu_0^{\otimes 2}(\de \bsig)
        }.
    \ealn
    By \eqref{eq:girsanov-error-H-difference},
    \[
        \norm{\bsig^1} |
            e^{\bDel_{t,\ell\delta}(\bsig^1)}
            - e^{\bDel_{t,\ell\delta}(\bsig^2)}
        |
        = O(N^{-1/2})
    \]
    for all $\bsig^1,\bsig^2 \in S_N$, and thus $\norm{\bm(\by_{\ell \delta},\ell \delta) - \bm(\by_t,t)} = O(N^{-1/2})$.
    So,
    \baln
        \bbE_{\bbP} \norm{\bm(\by_{\ell \delta},\ell \delta) - \bm(\by_t,t)}^2
        &\le \bbE_{\bbP} \ind\{\cE\} \norm{\bm(\by_{\ell \delta},\ell \delta) - \bm(\by_t,t)}^2
        + \bbE_{\bbP} \ind\{\cE^c\} \norm{\bm(\by_{\ell \delta},\ell \delta) - \bm(\by_t,t)}^2 \\
        &\le O(N^{-1/2}) + e^{-cN} \cdot 4N
        = o_N(1).
    \ealn
    The conclusion follows. 
\end{proof}
\begin{proof}[Proof of Proposition~\ref{ppn:yT-close-in-TV}]
    By \eqref{eq:euler-disc-matches-discretized-sde} and Lemmas~\ref{lem:KL-girsanov} and \ref{lem:girsanov-error},
    \[
        \bbE_{H_N \sim \bbP}
        \KL(\cL(\by_T | H_N),\cL(\by^L | H_N))
        = o_N(1).
    \]
    The result follows from Pinsker's inequality and Jensen's inequality:
    \baln
        \bbE_{H_N \sim \bbP}
        \KL(\cL(\by_T | H_N),\cL(\by^L | H_N))
        &\ge 
        2\bbE_{H_N \sim \bbP}
        \lt[\TV(\cL(\by_T | H_N),\cL(\by^L | H_N))^2\rt] \\
        &\ge 2\lt[
            \bbE_{H_N \sim \bbP}
            \TV(\cL(\by_T | H_N),\cL(\by^L | H_N))
        \rt]^2.
    \ealn
\end{proof}

\subsection{Log-concavity of localized measures at large $T$: Proof of Proposition~\ref{ppn:mala-accuracy}}
\label{sec:MALA-Accuracy}
In this subsection, we prove Proposition~\ref{ppn:mala-accuracy}.
Let $\be_1,\ldots,\be_N$ be the standard basis.
By a change of coordinates, we may assume without loss of generality that $\hby = \by / \norm{\by}_N = \be_N \sqrt{N}$ and $\bU = (\be_1,\ldots,\be_{N-1})$.
\begin{lem}
    \label{lem:push-forward-cap-law}
    For any $\by \neq \bzero$, the push-forward of $\mu_{H_N,\by}(\, \cdot \, | \la \bsig,\by \ra > 0)$ under the stereographic projection $\bT_\by$ is $\nu_{H_N,\by}^{\proj}$, defined in \eqref{eq:ProjMeasureNoApprox}.
\end{lem}
\begin{proof}
    Note that (denoting by $\bD F$ the Jacobian of the map $F$):
    \[
        \bD \bsig_\by(\brho)^\top
        = \fr{[\bI_{N-1},\bzero]}{\sqrt{1 + \norm{\brho}_N^2}}
        - \fr{\brho \bsig_\by(\brho)^\top/N}{1 + \norm{\brho}_N^2}.
    \]
    Since $[\bI_{N-1},\bzero] \bsig_\by(\brho) = \fr{\brho}{\sqrt{1+\norm{\brho}_N^2}}$, we have
    \[
        \bD \bsig_\by(\brho)^{\top}
         \bD\bsig_\by(\brho)
        = \fr{\bI_{N-1}}{1 + \norm{\brho}_N^2}
        - \fr{\brho \brho^\top / N}{(1 + \norm{\brho}_N^2)^2}
        = \fr{\bI_{N-1}}{1 + \norm{\brho}_N^2} \lt(
            \bI_{N-1} - \fr{\brho \brho^\top / N}{1 + \norm{\brho}_N^2}
        \rt).
    \]
    The stereographic projection thus incurs a change of density factor of
    \[
        \det(\bD\bsig_\by(\brho)^{\top}\bD \bsig_\by(\brho))^{1/2}
        = (1 + \norm{\brho}_N^2)^{-N/2}.
    \]
    This precisely accounts for the term $-\fr{N}{2} \log(1 + \norm{\brho}_N^2)$ in \eqref{eq:def-Hproj}.
\end{proof}

\begin{lem}
    \label{lem:nu-conc-small-radius}
    For sufficiently large $T$, with probability $1-o_N(1)$ over $(H_N,\by_T)$ as in Proposition~\ref{ppn:mala-accuracy}, $\nu_{H_N,\by_T}^{\proj}(\norm{\brho}_N^2 \le \eps_0) = 1-o_N(1)$ and  $\mu_{H_N,\by_T}(\<\bsig,\by_T\>_N\le 0)=o_N(1)$.
\end{lem}
\begin{proof}
    Let \rev{$(\bx,H_N,\by_T)$} be a sample from $\bbP$, \rev{set $\hby_T=\by_T/\norm{\by_T}_N$}, and let $q_\ast = q_\ast(T)$ be as in Fact~\ref{fac:qt-unique}.
    Note that $q_\ast > 1 - \fr{1}{T}$, as
    \[
        \xi'_T\big(1 - 1/T\big) \ge T + \xi'_T\big(1 - 1/T\big) \ge T > T-1 = \fr{1-1/T}{1/T}.
    \]
    By Proposition~\ref{ppn:amp-dominate-gibbs-2}, with probability $1-o_N(1)$,
    \[
        \mu_{H_N,\by_T}(\<\bsig,\bx\>_N \ge 1 - 1/T) = 1-o_N(1).
    \]
    \rev{With probability $1-o_N(1)$, we have $\norm{\by_T}_N = \sqrt{T(T+1)}+o_N(1)$, so
    \[
        \la \bx, \hby_T \ra
        = \fr{\la \bx, \by_T \ra}{\norm{\by_T}_N}
        = \sqrt{1 - \fr{1}{T+1}} +o_N(1).
    \]}
    \rev{On this event, $\{\bsig \in S_N : \<\bsig,\bx\>_N \ge 1 - 1/T\} \subseteq \{\bsig \in S_N : \<\bsig,\hby_T\>_N \ge 1 - 2/T\}$.}
    So, with probability $1-o_N(1)$,
    \rev{\[
        \mu_{H_N,\by_T}(\<\bsig,\hby_T\>_N \ge 1 - 2/T) = 1-o_N(1).
    \]}
    (This, of course, implies $\mu_{H_N,\by_T}(\<\bsig,\by_T\>_N\le 0)=o_N(1)$.)
    \rev{For sufficiently large $T$, the stereographic projection $\bT_{\by_T}$ maps $\{\bsig \in S_N : \<\bsig,\hby_T\>_N \ge 1 - 2/T\}$ into $\{\brho \in \bbR^{N-1} : \norm{\brho}_N^2 \le \eps_0\}$.}
    The conclusion follows from Lemma~\ref{lem:push-forward-cap-law}.
\end{proof}

\begin{cor}
    \label{cor:tnu-nu-approx}
    \rev{Recall the definitions in Eqs.~\eqref{eq:ProjMeasureNoApprox} and \eqref{eq:ProjMeasure} of $\nu_{H_N,\by_T}^{\proj}$ and $\tnu_{H_N,\by_T}^{\proj}$, respectively.}
    For sufficiently large $T$, with probability $1-o_N(1)$ over $(H_N,\by_T)$, $\TV(\nu_{H_N,\by_T}^{\proj},\tnu_{H_N,\by_T}^{\proj}) = o_N(1)$.
\end{cor}
\begin{proof}
    Since $\varphi(x) = 0$ for $x\in [0,\eps_0]$ and $\varphi(x) \ge 0$ for $x > \eps_0$, we have
    \baln
        \int_{\norm{\brho}_N^2 \le \eps_0}
        \exp \tH_{N,\by_T}^{\proj}(\brho)
        ~\de \brho
        &= \int_{\norm{\brho}_N^2 \le \eps_0}
        \exp H_{N,\by_T}^{\proj}(\brho)
        ~\de \brho, \\
        \int_{\norm{\brho}_N^2 > \eps_0}
        \exp \tH_{N,\by_T}^{\proj}(\brho)
        ~\de \brho
        &\le \int_{\norm{\brho}_N^2 > \eps_0}
        \exp H_{N,\by_T}^{\proj}(\brho)
        ~\de \brho.
    \ealn
    Together with Lemma~\ref{lem:nu-conc-small-radius}, this implies that, with probability $1-o_N(1)$,
    \[
        \tnu_{H_N,\by_T}^{\proj}(\norm{\brho}_N^2 \le \eps_0)
        \ge \nu_{H_N,\by_T}^{\proj}(\norm{\brho}_N^2 \le \eps_0)
        \ge 1-o_N(1).
    \]
    Since $\tnu_{H_N,\by_T}^{\proj}$ and $\nu_{H_N,\by_T}^{\proj}$ are furthermore proportional on $\{\norm{\brho}_N^2 \le \eps_0\}$, the conclusion follows.
\end{proof}

\begin{ppn}
    \label{ppn:tHproj-concave}
    For sufficiently large $T$, there exist $C_{\min},C_{\max} > 0$ (depending on $T$) such that with probability $1-o_N(1)$, for all $\brho \in \bbR^{N-1}$,
    \[
        -C_{\max} \bI_{N-1}
        \preceq \nabla^2 \tH_{N,\by_T}^{\proj}(\brho)
        \preceq -C_{\min} \bI_{N-1}.
    \]
\end{ppn}
\begin{proof}
    \rev{Set $\by=\by_T$ and $\hby=\by/\|\by\|_N$, and assume without loss of generality that
    $\hby=\sqrt{N}\,\be_N$.}
    Let $\bU^\top = [\bI_{N-1}, \bzero] \in \bbR^{(N-1) \times N}$ be the projection onto the orthogonal complement of $\hby$.

    A direct calculation shows
    \baln
        \nabla^2 \tH_{N,\by}^{\proj}(\brho)
        &= \fr{\la \nabla H_{N,\by}(\bsig_\by(\brho)), \bsig_\by(\brho) \ra}{N(1+\norm{\brho}_N^2)}
        \lt(
            - \bI_{N-1} + \fr{3\brho\brho^\top}{N(1+\norm{\brho}_N^2)}
        \rt) \\
        &+ \fr{\la \nabla^2 H_{N,\by}(\bsig_\by(\brho)), \bsig_\by(\brho)^{\otimes 2} \ra}{N(1+\norm{\brho}_N^2)^2} \cdot \fr{\brho\brho^\top}{N}
        + \fr{\bU^\top \nabla^2 H_{N,\by}(\bsig_\by(\brho)) \bU}{1 + \norm{\brho}_N^2} \\
        &- \fr{
            \brho \bsig_\by(\brho)^\top \nabla^2 H_{N,\by}(\bsig_\by(\brho)) \bU
            + \bU^\top\nabla^2 H_{N,\by}(\bsig_\by(\brho)) \bsig_\by(\brho) \brho^\top
        }{N(1+\norm{\brho}_N^2)} \\
        &- \fr{
            \brho \nabla H_{N,\by}(\bsig_\by(\brho))^\top \bU
            + \bU^{\top} \nabla H_{N,\by}(\bsig_\by(\brho)) \brho^\top
        }{N(1+\norm{\brho}_N^2)^{3/2}} \\
        &- \lt(
            T\varphi'(\norm{\brho}_N^2) + \fr{1}{1+\norm{\brho}_N^2}
        \rt) \bI_{N-1}
        - \lt(
            T\varphi''(\norm{\brho}_N^2) - \fr{1}{(1+\norm{\brho}_N^2)^2}
        \rt) \fr{2\brho\brho^\top}{N}.
    \ealn
 
    By Proposition~\ref{ppn:gradients-bounded}, there exists $C>0$ (independent of $T$) such that with probability $1-o_N(1)$,
    \[
        \sup_{\bsig \in S_N}
        \norm{\nabla H_N(\bsig)}_N,
        \sup_{\bsig \in S_N}
        \norm{\nabla^2 H_N(\bsig)}_{\op}
        \le C.
    \]
    We will show that on this event,
    \balnn
        \notag
        \nabla^2 \tH_{N,\by}^{\proj}(\brho)
        &= \fr{\norm{\by}_N}{(1+\norm{\brho}_N^2)^{3/2}} \lt(
            - \bI_{N-1} + \fr{3\brho\brho^\top}{N(1+\norm{\brho}_N^2)}
        \rt) - T\varphi'(\norm{\brho}_N^2) \bI_{N-1} \\
        \label{eq:nabla2-tHproj-approximation}
        &- T\varphi''(\norm{\brho}_N^2) \cdot \fr{2\brho\brho^\top}{N}
        + O(1),
    \ealnn
    where $O(1)$ denotes a matrix of operator norm $O(1)$, independent of $T$.
    Note that
    \[
        \fr{\la \nabla H_{N,\by}(\bsig_\by(\brho)), \bsig_\by(\brho) \ra}{N(1+\norm{\brho}_N^2)}
        = \fr{\la \nabla H_N(\bsig_\by(\brho)) + \by, \bsig_\by(\brho) \ra}{N(1+\norm{\brho}_N^2)}
        = \fr{\la \nabla H_N(\bsig_\by(\brho)), \bsig_\by(\brho) \ra}{N(1+\norm{\brho}_N^2)}
        + \fr{\norm{\by}_N}{(1+\norm{\brho}_N^2)^{3/2}}.
    \]
    The first term on the right-hand side is bounded independently of $T$, as
    \[
        \fr{|\la \nabla H_N(\bsig_\by(\brho)), \bsig_\by(\brho) \ra|}{N}
        \le \norm{\nabla H_N(\bsig_\by(\brho))}_N \norm{\bsig_\by(\brho)}_N.
    \]
    Similarly, all other terms in the expansion of $\nabla^2 \tH_{N,\by}^{\proj}(\brho)$ above, aside from $T\varphi'(\norm{\brho}_N^2) \bI_{N-1}$ and $T\varphi''(\norm{\brho}_N^2) \cdot \fr{2\brho\brho^\top}{N}$, are bounded independently of $T$, due to the following inequalities:
    \baln
        \norm{\nabla^2 H_{N,\by}(\bsig_\by(\brho))}_{\op}
        &= \norm{\nabla^2 H_N(\bsig_\by(\brho))}_{\op}
        = O(1), \\
        \fr{\norm{\bU^\top \nabla H_{N,\by}(\bsig_\by(\brho)) \brho^\top}_{\op}}{N}
        &\le \norm{\bU^\top \nabla H_{N,\by}(\bsig_\by(\brho))}_N \norm{\brho}_N \\
        &= \norm{\bU^\top \nabla H_N(\bsig_\by(\brho))}_N \norm{\brho}_N
        \le O(1) \norm{\brho}_N,
    \ealn
    and $\norm{\brho\brho^\top}_{\op}/N = \norm{\brho}_N^2$, $\norm{\brho\bsig_\by(\brho)^\top}_{\op}/N = \norm{\brho}_N$.
    (Note that, in each of these terms, each copy of $\norm{\brho}_N^2$ in the resulting bound is compensated by at least one copy of $1+\norm{\brho}_N^2$ in the denominator.)
    This proves \eqref{eq:nabla2-tHproj-approximation}.

    With probability $1-o_N(1)$, we have $\norm{\by}_N = \sqrt{T(T+1)}+o_N(1)$.
    On this event, \eqref{eq:nabla2-tHproj-approximation} yields
    \[
        \nabla^2 \tH_{N,\by}^{\proj}(\brho)
        = T (-\bM(\brho) + o_T(1)),
    \]
    where $o_T(1)$ denotes a matrix with operator norm vanishing with $T$ and
    \[
        \bM(\brho) = \fr{\bI_{N-1}}{(1+\norm{\brho}_N^2)^{3/2}} - \fr{3\brho\brho^\top}{N(1+\norm{\brho}_N^2)^{5/2}}
        + \varphi'(\norm{\brho}_N^2) \bI_{N-1} + \varphi''(\norm{\brho}_N^2) \cdot \fr{2\brho\brho^\top}{N}.
    \]
    From this it is clear that \rev{$-C_{\max} \bI_{N-1} \preceq \nabla^2 \tH_{N,\by}^{\proj}(\brho)$} for suitable $C_{\max}$.
    For the other direction, note that $\bM(\brho)$ has eigenvalue $\fr{1}{(1+\norm{\brho}_N^2)^{3/2}} + \varphi'(\norm{\brho}_N^2)$ in all directions orthogonal to $\brho$, and
    \[
        \fr{1 - 2\norm{\brho}_N^2}{(1+\norm{\brho}_N^2)^{5/2}}
        + \varphi'(\norm{\brho}_N^2)
        + 2\norm{\brho}_N^2 \varphi''(\norm{\brho}_N^2)
    \]
    in the direction of $\brho$.
    By \eqref{eq:varphi-desiderata}, $\bM(\brho) \succeq \eps_0 \bI_{N-1}$, and thus \rev{$\nabla^2 \tH_{N,\by}^{\proj}(\brho) \preceq -C_{\min} \bI_{N-1}$} for $C_{\min} = T\eps_0 / 2$.
\end{proof}
Finally, we verify that $\varphi$ satisfying \eqref{eq:varphi-desiderata} exists.
\begin{fac}
    \label{fac:varphi-exists}
    For suitable $C > 0$, the function
    \[
        \varphi(x) = C \ind\{x>\eps_0\} \lt(
            x - \fr{\eps_0^2}{x} - 2 \eps_0 \log \fr{x}{\eps_0}
        \rt)
    \]
    is nonnegative, twice continuously differentiable, and satisfies \eqref{eq:varphi-desiderata}.
\end{fac}
\begin{proof}
    Note that for $x > \eps_0$,
    \baln
        \varphi'(x) &= C\lt(1 - \fr{\eps_0}{x}\rt)^2, &
        \varphi''(x) &= \fr{2C\eps}{x^2} \lt(1 - \fr{\eps_0}{x}\rt).
    \ealn
    Thus $\lim_{x\downarrow \eps_0} \varphi''(x) = 0$, so $\varphi$ is twice continuously differentiable.
    Note that $\varphi' \ge 0$, so integrating shows $\varphi \ge 0$.
    Let
    \[
        C_0 = \min_{x \ge 0} \fr{1-2x}{(1+x)^{5/2}}
    \]
    and set $C$ so that $C_0 + \varphi'(2\eps_0) \ge \eps_0$.
    Note $\varphi'' \ge 0$, and thus $\varphi'$ is increasing; thus \eqref{eq:varphi-desiderata} holds for all $x \ge 2\eps_0$.
    For all $x\in [0,2\eps_0]$, we verify that
    \[
        \fr{1}{(1+x)^{3/2}}
        \ge \fr{1-2x}{(1+x)^{5/2}}
        \ge \fr{1-4\eps_0}{(1+2\eps_0)^{5/2}}
        \ge \eps_0,
    \]
    so \eqref{eq:varphi-desiderata} holds.
\end{proof}

\begin{proof}[Proof of Proposition~\ref{ppn:mala-accuracy}]
    By Proposition~\ref{ppn:tHproj-concave}, $\tnu_{H_N,\by_T}^{\proj}$ is $O(1)$-smooth and strongly log-concave.
    By \cite[Theorem 3]{chewi2021optimal}, running MALA for time $\chi_{\mbox{\rm\tiny log-conc}} = \mathrm{poly}(N)$ produces $\brho^\sMALA \sim \nu^{\sMALA}$, where $\TV(\nu^{\sMALA},\tnu_{H_N,\by_T}^{\proj}) \le 1/N$.
    Combining this with Corollary~\ref{cor:tnu-nu-approx}, we find that, with probability $1-o_N(1)$,
    $\TV(\nu_{H_N,\by_T}^{\proj},\nu^\sMALA) = o_N(1)$.
    Lemma~\ref{lem:push-forward-cap-law} completes the proof.
\end{proof}

\section{Failure of stochastic localization in the complementary regime: Proof of Theorem~\ref{thm:sl-hardness}}
\label{sec:sl-hardness}

In this section, we prove Theorem~\ref{thm:sl-hardness}.
Similarly to Subsection~\ref{subsec:planted}, we may analyze the process \eqref{eq:gen-sl-sde} by passing to a planted model.
For any $T>0$, let $\cbbP, \cbbQ \in \cP(S_N \times \sH_N \times C([0,T],\bbR^N \times \cdots \times (\bbR^N)^{\otimes J}))$ be the laws of \rev{$(\bx,H_N,(\vby_t)_{t\in [0,T]})$}, generated as follows.
\begin{itemize}
    \item Under $\cbbQ$,
    \[
        H_N \sim \mu_\nul, \qquad
        \rev{\bx} \sim \mu_{H_N}, \qquad
        \by^j_t = \tau_j(t) \rev{\bx}^{\otimes j} + \bB^j_{\tau_j(t)}, \quad \forall j=1,\ldots,J,
    \]
    for $(\bB^1_t,\ldots,\bB^J_t)_{t\ge 0}$ independent of \rev{$(\bx,H_N)$}.
    Equivalently, $H_N \sim \mu_\nul$, $(\vby_t)_{t\ge 0}$ is given by the SDE \eqref{eq:gen-sl-sde}, and for any odd $j$ such that $\lim_{t\to\infty} \tau_j(t) = \infty$, \rev{$\bx$} is the unique solution to \rev{$\bx^{\otimes j}$} $= \lim_{t\to\infty} \by^j_t/\tau_j(t)$.
    \item Under $\cbbP$,
    \[
        \rev{(\bx,H_N)} \sim \mu_\pl, \qquad
        \by^j_t = \tau_j(t) \rev{\bx}^{\otimes j} + \bB^j_{\tau_j(t)}, \quad \forall j=1,\ldots,J,
    \]
    for $(\bB^1_t,\ldots,\bB^J_t)_{t\ge 0}$ independent of \rev{$(\bx,H_N)$}.
    Equivalently, we can generate first $H_N$, then $(\vby_t)_{t\ge 0}$ by \eqref{eq:gen-sl-sde}, and finally \rev{$\bx$} as above.
    Furthermore, the law of \rev{$(\bx,H_N) \sim \mu_\pl$} can be described by either \eqref{eq:planted-HN-first} or \eqref{eq:planted-bsig-first}.
\end{itemize}
Analogously to Proposition~\ref{ppn:planted-likelihood-ratio}, we have
\[
    \fr{\de \cbbP}{\de \cbbQ}\rev{(\bx,H_N,(\vby_t)_{t\in [0,T]})} = \fr{Z(H_N)}{\bbE Z(H_N)},
\]
and this ratio is tight by Lemma~\ref{lem:partition-fn-fluctuations}.
Thus $\cbbP$ and $\cbbQ$ are mutually contiguous.

Therefore, it suffices to analyze the AMP iteration \eqref{eq:gen-amp} under $\cbbP$.
\rev{As in \eqref{eq:HNt-with-y}, conditional on $(H_N,\vby_t)$, the posterior law of the planted signal $\bx$ under $\cbbP$ is the following measure, where $\bsig$ denotes its generic argument:}
\[
    \cmu_t(\de \bsig) = \fr{1}{Z} \exp \ccH_{N,t}(\bsig) \mu_0(\de \bsig),
\]
\rev{where we define the posterior Hamiltonian by}
\beq
    \label{eq:def-generalized-posterior-Hamiltonian}
    \rev{\ccH_{N,t}(\bsig)
    := H_N(\bsig) + \sum_{j=1}^J \fr{1}{N^{j-1}} \la \by^j_t, \bsig^{\otimes j} \ra.}
\eeq
\rev{Conditional on $\bx$, the planted representation \eqref{eq:planted-H-law} and the Gaussian observation channels give the following equality in distribution of random functions on $S_N$:}
\beq
    \label{eq:generalized-posterior-Hamiltonian-law}
    \rev{\big(\ccH_{N,t}(\bsig)\big)_{\bsig\in S_N}
    \stackrel{d}{=}
    \big(N \cxi_t(\<\bx,\bsig\>_N) + \tH_{N,t}(\bsig)\big)_{\bsig\in S_N},}
\eeq
\rev{where $\tH_{N,t}$ is a spin glass with mixture $\cxi_t$ defined in \eqref{eq:def-cxi-t}. The symbol $\stackrel{d}{=}$ in \eqref{eq:generalized-posterior-Hamiltonian-law} denotes only equality in distribution conditional on $\bx$, not an identity under the planted coupling.}
Let $q_\sAMP = q_\sAMP(t)$ be the smallest solution to $\cxi'_t(q) = \fr{q}{1-q}$.
Note that a solution exists because $\cxi'_t(0) \ge 0$ and $\lim_{q\uparrow 1} \fr{q}{1-q} = +\infty$.
\begin{ppn}
    \label{ppn:sl-hardness-amp}
    We have
    \[
        \lim_{k\to\infty} \plim_{N\to\infty} \<\bx,\cbm^k\>_N
        = \lim_{k\to\infty} \plim_{N\to\infty} \<\cbm^k,\cbm^k\>_N
        = q_\sAMP.
    \]
    Consequently, for all $1\le j\le J$,
    \[
        \lim_{k\to\infty}
        \lim_{N\to\infty}
        \bbE
        \fr{1}{N^j}
        \norm{\bx^{\otimes j} - (\cbm^k)^{\otimes j}}_2^2
        = 1 - q_\sAMP^j.
    \]
\end{ppn}
\begin{proof}
    Since $q \mapsto \fr{\cxi'_t(q)}{1 + \cxi'_t(q)}$ is increasing, the sequence $(\cq_k)_{k\ge 0}$ defined in \eqref{eq:def-cq} is increasing.
    Furthermore, if $\cq_k \le q_\sAMP$, then
    \[
        \cq_{k+1} = \fr{\cxi'_t(\cq_k)}{1 + \cxi'_t(\cq_k)}
        \le \fr{\cxi'_t(q_\sAMP)}{1 + \cxi'_t(q_\sAMP)}
        = q_\sAMP,
    \]
    and therefore by induction $(\cq_k)_{k\ge 0}$ is bounded above by $q_\sAMP$.
    As the limit of $(\cq_k)_{k\ge 0}$ must be a fixed point of $q \mapsto \fr{\cxi'_t(q)}{1 + \cxi'_t(q)}$, we have $\lim_{k\to\infty} \cq_k = q_\sAMP$.
    By state evolution, as in the proof of Proposition~\ref{ppn:amp-performance}, the first conclusion follows.
    Since
    \[
        \fr{1}{N^j}
        \norm{\bx^{\otimes j} - (\cbm^k)^{\otimes j}}_2^2
        = \<\bx,\bx\>_N^j
        - 2 \<\bx,\cbm^k\>_N^j
        + \<\cbm^k,\cbm^k\>_N^j,
    \]
    the second conclusion follows from the first.
\end{proof}
Let
\beq
    \label{eq:def-qbayes}
    Q_\bayes = Q_\bayes(t)
    = \argmax_{q\in [0,1)} \lt\{
        \cxi_t(q) + q + \log(1-q)
    \rt\}
    \subseteq [0,1)
\eeq
be the set of all maximizers of this quantity, and let
\[
    q_\bayes = q_\bayes(t) = \inf Q_\bayes(t).
\]
\begin{lem}
    \label{lem:finite-zeros}
    For any $t$, the equation $\cxi'_t(q) = \fr{q}{1-q}$ has finitely many solutions $q\in [0,1)$.
    Moreover, $Q_\bayes(t)$ is a finite set for all $t$.
    If $T_1 \subseteq [0,+\infty)$ is the set of $t_1$ such that $|Q_\bayes(t_1)| > 1$, then for each $t_1\in T_1$, there exists $\delta > 0$ such that $(t_1-\delta,t_1+\delta) \cap T_1 = \{t_1\}$.
\end{lem}
\begin{proof}
    Let $f_t(q) = (1-q)\cxi'_t(q) - q$, so any solution to $\cxi'_t(q) = \fr{q}{1-q}$ is a zero of $f_t$.
    Note that $f_t$ is not identically zero: if it were, then $\cxi'_t(q) = \fr{q}{1-q}$, contradicting the fact that the coefficients $\gamma_p^2$ of $\xi$ satisfy $\sum_{p\ge 2} 2^p \gamma_p^2 < \infty$.
    Since $f_t$ is complex analytic in the unit disc, its zero set has no limit point, and in particular it has finitely many zeros in $[0,1)$.
    This shows that there are finitely many solutions to $\cxi'_t(q) = \fr{q}{1-q}$.

    Note that $\fr{\de}{\de q}(\cxi_t(q) + q + \log(1-q)) = \cxi'_t(q) - \fr{q}{1-q}$.
    Any interior maximizer of \eqref{eq:def-qbayes} must therefore satisfy the stationarity condition $\cxi'_t(q) = \fr{q}{1-q}$.
    Since $\cxi'_t(0) \ge 0$, $0$ can be a maximizer only if it also solves this equation.
    Thus $Q_\bayes(t)$ is finite.

    Consider an arbitrary $t_1\in T_1$ and let $Q=Q_\bayes(t_1)$.
    For each $\tq \in Q$, let $I_\tq = [\tq-\eps,\tq+\eps]$, where $\eps>0$ is small enough that these intervals do not overlap.
    By continuity, for sufficiently small $\delta$ and all $t \in (t_1-\delta,t_1+\delta)$, all maximizers of $\cxi_t(q) + q + \log(1-q)$ lie in $\bigcup_{\tq\in Q} I_\tq$.
    Let
    \baln
        m(t,\tq) &= \max_{q\in I_\tq} \lt\{
            \cxi_t(q) + q + \log(1-q)
        \rt\}, &
        q(t,\tq) &= \argmax_{q\in I_\tq} \lt\{
            \cxi_t(q) + q + \log(1-q)
        \rt\}.
    \ealn
    Note that $q(t_1,\tq) = \tq$ for each $\tq \in Q$.
    Since the maximum of $\cxi_{t_1}(q) + q + \log(1-q)$ is attained over $I_\tq$ uniquely at $\tq$, by continuity $\lim_{t\to t_1} q(t,\tq) = \tq$.

    For $\tq \in Q$, $t \in (t_1,t_1+\delta)$, we have
    \baln
        \fr{m(t,\tq) - m(t_1,\tq)}{t-t_1}
        &\ge \fr{\cxi_t(q(t_1,\tq)) - \cxi_{t_1}(q(t_1,\tq))}{t-t_1}
        = \sum_{j=1}^J
        \tau'_j(t_1) q(t_1,\tq)^j
        + O(t-t_1), \\
        \fr{m(t,\tq) - m(t_1,\tq)}{t-t_1}
        &\le \fr{\cxi_t(q(t,\tq)) - \cxi_{t_1}(q(t,\tq))}{t-t_1}
        = \sum_{j=1}^J
        \tau'_j(t_1) q(t,\tq)^j
        + O(t-t_1).
    \ealn
    Taking the limit $t\downarrow t_1$ yields
    \[
        \lim_{t \downarrow t_1}
        \fr{m(t,\tq) - m(t_1,\tq)}{t-t_1}
        = \sum_{j=1}^J
        \tau'_j(t_1) \tq^j.
    \]
    A similar argument shows that the left derivative is also equal to this.
    Therefore
    \[
        \fr{\partial}{\partial t}
        m(t,\tq) \big|_{t=t_1}
        = \sum_{j=1}^J
        \tau'_j(t_1) \tq^j.
    \]
    This quantity is distinct for different $\tq \in Q$.
    Therefore, for all $t \in (t_1-\delta,t_1+\delta) \setminus \{t_1\}$, $|Q_\bayes(t)| = 1$.
\end{proof}
\begin{ppn}
    \label{ppn:sl-hardness-bayes}
    Suppose $t \not \in T_1$ satisfies $q_\bayes(t) > 0$.
    Let $\cxi_t(q) = \sum_{p\ge 1} \beta_p^2 q^p$ (where we suppress the dependence of the $\beta_p$ on $t$).
    For any $p$ such that $\beta_p > 0$, recalling the definition of $\bm_p(\vby_t,t)$ in Eq.~\eqref{eq:Mj_def}, we have
    \[
        \lim_{N\to\infty} \bbE \fr{1}{N^p}
        \norm{\bx^{\otimes p} - \bm_p(\vby_t,t)}_2^2
        = 1 - q_\bayes^p.
    \]
\end{ppn}
We first prove a preparatory lemma.
In what follows, we let $\tbeta_{p'}=\beta_{p'}$ be fixed for all $p'\neq p$ and treat $\tbeta_p$ as a variable.
Define $\txi(q) = \sum_{p'\ge 1} \tbeta_{p'}^2 q^{p'}$; we sometimes emphasize the dependence on $\tbeta_p$ by writing $\txi^{\tbeta_p}(q)$.
Let \rev{$\cP_{\rm Par}$} denote the Parisi functional for spherical spin glasses; see, e.g., \cite[Equation (1.12)]{talagrand2006spherical}.
(In the proof below we will only need the replica-symmetric case of this functional, which is given in Proposition~\ref{ppn:parisi-rs-ub}.)
Further, for $q\in (-1,1)$, let
\[
    \txi_q(s) = \txi(q^2 + (1-q^2)s) - \txi(q^2),
\]
and define
\[
    P(\tbeta_p)
    = \sup_{q\in [0,1)} \lt\{
        \txi(q)
        + \rev{\cP_{\rm Par}}(\txi_q)
        + \fr12 \log(1-q^2)
    \rt\}
\]
\begin{lem}
    \label{lem:P-differentiable}
    Assume the setting of Proposition~\ref{ppn:sl-hardness-bayes}.
    For all $\tbeta_p$ in a neighborhood of $\beta_p$,
    \beq
        \label{eq:P-differentiable-P}
        P(\tbeta_p)
        = \fr12 \sup_{q\in [0,1)} \lt\{
            \txi(1) + \txi(q)
            + q + \log(1-q)
        \rt\}.
    \eeq
    Furthermore, $P$ is differentiable at $\beta_p$, with
    \beq
        \label{eq:P-differentiable-Ppr}
        P'(\beta_p) = \beta_p (1 + q_\bayes^p).
    \eeq
\end{lem}
\begin{proof}
    By Proposition~\ref{ppn:parisi-rs-ub} with $u = \fr{q}{1+q}$, for all $q\in [0,1)$,
    \beq
        \label{eq:P-differentiable-P-ub}
        \rev{\cP_{\rm Par}}(\txi_q)
        \le \fr12 \lt\{
            \txi_q(1) - \txi_q(u) + \fr{u}{1-u} + \log(1-u)
        \rt\}
        = \fr12 \lt\{
            \txi(1) - \txi(q) + q - \log(1+q)
        \rt\},
    \eeq
    and thus
    \[
        P(\tbeta_p)
        \le \fr12 \sup_{q\in [0,1)} \lt\{
            \txi(1)
            + \txi(q)
            + q + \log(1-q)
        \rt\}.
    \]
    Since $\lim_{q\uparrow 1} \log(1-q) = -\infty$, the supremum is attained.
    Let $q(\tbeta_p)$ denote the maximizer.
    Arguing as in the proof of Proposition~\ref{ppn:band-model-rs-ub}, \eqref{eq:P-differentiable-P-ub} is an equality at $q = q(\tbeta_p)$.
    This proves \eqref{eq:P-differentiable-P}.

    Note that $q(\beta_p) = q_\bayes$ by definition.
    Since $t\not\in T_1$, the maximum in \eqref{eq:P-differentiable-P} at $\tbeta_p = \beta_p$ is attained uniquely at $q_\bayes$.
    \rev{The objective in \eqref{eq:P-differentiable-P} is jointly continuous in $(\tbeta_p,q)$ and tends to $-\infty$ as $q\uparrow1$, uniformly for $\tbeta_p$ in a neighborhood of $\beta_p$. Hence compactness and uniqueness of the maximizer at $\beta_p$ imply $\lim_{\tbeta_p\to\beta_p}q(\tbeta_p)=q_\bayes$.}
    Note that for any $\tbeta_p > \beta_p$,
    \baln
        \fr{P(\tbeta_p)-P(\beta_p)}{\tbeta_p-\beta_p}
        &\ge \rev{\fr12}\fr{\txi^{\tbeta_p}(1) + \txi^{\tbeta_p}(q(\beta_p)) - \txi^{\beta_p}(1) - \txi^{\beta_p}(q(\beta_p))}{\tbeta_p-\beta_p}
        = \rev{\beta_p (1 + q(\beta_p)^p)} + O(\tbeta_p - \beta_p), \\
        \fr{P(\tbeta_p)-P(\beta_p)}{\tbeta_p-\beta_p}
        &\le \rev{\fr12}\fr{\txi^{\tbeta_p}(1) + \txi^{\tbeta_p}(q(\tbeta_p)) - \txi^{\beta_p}(1) - \txi^{\beta_p}(q(\tbeta_p))}{\tbeta_p-\beta_p}
        = \rev{\beta_p (1 + q(\tbeta_p)^p)} + O(\tbeta_p - \beta_p).
    \ealn
    Taking the limit $\tbeta_p \downarrow \beta_p$ yields
    \[
        \lim_{\tbeta_p \downarrow \beta_p}
        \fr{P(\tbeta_p)-P(\beta_p)}{\tbeta_p-\beta_p}
        = \rev{\beta_p (1 + q_\bayes^p)}.
    \]
    A similar argument shows the left derivative also equals this, proving \eqref{eq:P-differentiable-Ppr}.
\end{proof}

\begin{proof}[Proof of Proposition~\ref{ppn:sl-hardness-bayes}]
    Let $\tH_N$ be a spin glass Hamiltonian with mixture $\txi$, and let
    \[
        F_N(\tbeta_p)
        = \fr1N \bbE \log \int_{S_N}
        \exp\lt\{
            N \txi(\<\bx,\bsig\>_N)
            + \tH_N(\bsig)
        \rt\}
        ~\de \mu_0(\bsig).
    \]
    \rev{We can decompose the integral over $S_N$ as an integral over contributions from each band $\<\bx,\bsig\>_N = q$, for $q\in [0,1]$. The band with overlap $q$ contributes volume $(1-q^2)^{(N-3)/2} = \exp(\frac{N}{2} \log (1-q^2) + o(N))$. Conditional on a fixed overlap $q$, and after rescaling the component orthogonal to $\bx$, the restriction of $\tH_N$ is a spherical spin glass with mixture $\txi_q$, up to an additive Gaussian constant whose expectation is zero. Thus the Parisi formula \cite[Theorem 1.1]{talagrand2006spherical} gives the limiting free energy on each fixed-overlap band. Consequently,}
    \[
        \lim_{N\to\infty} F_N(\tbeta_p)
        = \sup_{q\in (-1,1)} \lt\{
            \txi(q)
            + \rev{\cP_{\rm Par}}(\txi_q)
            + \fr12 \log(1-q^2)
        \rt\}.
    \]
    \rev{This equals $P(\tbeta_p)$: since $\txi$ has nonnegative coefficients, replacing $q<0$ by $|q|$ cannot decrease $\txi(q)$, while $\txi_q$ and $\log(1-q^2)$ depend only on $q^2$.}
    By H\"older's inequality, $F_N(\tbeta_p)$ is convex in $\tbeta_p$.
    So, for any $\delta > 0$,
    \[
        \fr{F_N(\beta_p)- F_N(\beta_p-\delta)}{\delta}
        \le F'_N(\beta_p)
        \le \fr{F_N(\beta_p+\delta)- F_N(\beta_p)}{\delta}.
    \]
    Differentiability of $P$ (by Lemma~\ref{lem:P-differentiable}) then implies
    \beq
        \label{eq:F-deriv-matches-P}
        \lim_{N\to\infty} F'_N(\beta_p)
        = P'(\beta_p)
        = \beta_p (1 + q_\bayes^p).
    \eeq
    \rev{Recall from \eqref{eq:def-generalized-posterior-Hamiltonian} that $\ccH_{N,t}(\bsig)$ is defined as $H_N(\bsig)+\sum_{j=1}^J N^{1-j}\la\by_t^j,\bsig^{\otimes j}\ra$. Let $\la \cdot \ra$ denote the average with respect to its Gibbs measure. Conditional on $\bx$, the distributional representation \eqref{eq:generalized-posterior-Hamiltonian-law} says that this random Hamiltonian has the same distribution as $N \txi(\<\bx,\bsig\>_N) + \tH_N(\bsig)$ when $\tbeta_p = \beta_p$. This is only an equality in distribution of the random Hamiltonians, not an identity under the planted coupling.}
    Note that $\bm_p(\vby_t,t) = \la \bsig^{\otimes p} \ra$.
    We calculate that
    \baln
        F'_N(\beta_p)
        &= 2 \beta_p \bbE \la \<\bx,\bsig\>_N^p\ra + \beta_p \lt(1 - \bbE \la \<\bsig,\bsig\>_N^p\ra\rt) \\
        &= \beta_p \lt(
            1 +
            2\bbE \fr{\la \bx^{\otimes p}, \bm_p(\vby_t,t) \ra}{N^p}
            - \bbE \fr{\la \bm_p(\vby_t,t), \bm_p(\vby_t,t) \ra}{N^p}
        \rt).
    \ealn
    Comparing with \eqref{eq:F-deriv-matches-P} shows
    \[
        \lim_{N\to\infty} \lt\{
            2\bbE \fr{\la \bx^{\otimes p}, \bm_p(\vby_t,t) \ra}{N^p}
            - \bbE \fr{\la \bm_p(\vby_t,t), \bm_p(\vby_t,t) \ra}{N^p}
        \rt\}
        = q_\bayes^p.
    \]
    Since
    \[
        \bbE \fr{1}{N^p} \norm{\bx^{\otimes p} - \bm_p(\vby_t,t)}_2^2
        =
        1 - 2\bbE \fr{\la \bx^{\otimes p}, \bm_p(\vby_t,t) \ra}{N^p}
        +\bbE \fr{\la \bm_p(\vby_t,t), \bm_p(\vby_t,t) \ra}{N^p},
    \]
    the result follows.
\end{proof}
\begin{lem}\label{lemma:MoreThanOne}
    If there exists $q\in [0,1)$ such that $\xi''(q) > \fr{1}{(1-q)^2}$, then there exists $t \ge 0$ such that $\cxi'_t(q) = \fr{q}{1-q}$ has more than one solution.
\end{lem}
\begin{proof}
    Let $g_t(q) = \cxi'_t(q) - \fr{q}{1-q}$, so solutions to $\cxi'_t(q) = \fr{q}{1-q}$ are zeros of $g_t$.
    Suppose for contradiction that for all $t\ge 0$, $g_t$ has a unique zero $q_\sAMP(t)$.
    Then, for all $t$, $g_t > 0$ on $[0,q_\sAMP(t))$ (this is vacuous if $q_\sAMP(t)=0$) and $g_t < 0$ on $(q_\sAMP(t),1)$.
    Note that for each $q$, $g_t(q)$ is continuous and increasing in $t$, and thus $q_\sAMP(t)$ is also continuous and increasing.

    Recall that $\tnorm{\tau(t)}_1 = t$ for all $t$.
    For each $q\in (0,1)$,
    \beq
        \label{eq:gt-lb}
        g_t(q) \ge \sum_{j=1}^J j\tau_j(t) q^{j-1} - \fr{q}{1-q}
        \ge \tnorm{\tau(t)}_1 q^{J-1} - \fr{q}{1-q}
        = tq^{J-1} - \fr{q}{1-q}.
    \eeq
    It follows that $g_t(q) > 0$ for sufficiently large $t$.
    Thus $\lim_{t\to+\infty} g_\sAMP(t) = 1$, so $q_\sAMP(t)$ ranges over all of $[0,1)$ as $t$ ranges over $[0,+\infty)$.

    Since $\xi''(q) > \fr{1}{(1-q)^2}$ for some $q\in [0,1)$, the function $g_0$ is not monotonically decreasing.
    Let $0\le q_1<q_2 < 1$ be such that $g_0(q_1) < g_0(q_2)$.
    Note that
    \[
        g_t(q_1) - g_0(q_1)
        = \sum_{j=1}^J j\tau_j(t) q_1^{j-1}
        \le \sum_{j=1}^J j\tau_j(t) q_2^{j-1}
        = g_t(q_2) - g_0(q_2).
    \]
    Thus $g_t(q_1) < g_t(q_2)$.
    Set $t$ such that $q_1 = q_\sAMP(t)$, so that $g_t(q_1) = 0$.
    This implies that $g_t(q_2) > 0$, and therefore $g_t$ has another zero in $[q_2,1)$.
\end{proof}

\begin{lem}
    If there exists $t \ge 0$ such that $\cxi'_t(q) = \fr{q}{1-q}$ has more than one solution, then there exists a nontrivial interval $I = [t_-,t_+] \subseteq [0,+\infty)$ such that for all $t'\in I$, $q_\sAMP(t') \neq q_\bayes(t')$.
\end{lem}
\begin{proof}
    Let $g_t$ be defined as in the proof of Lemma \ref{lemma:MoreThanOne} and
    $q_1 = q_\sAMP(t)$, so that $q_1$ is the smallest zero of $g_t$.
    Let $q_2 > q_1$ be the next smallest zero of $g_t$.
    Note that by Lemma~\ref{lem:finite-zeros}, either $g_t(q) > 0$ for all $q\in (q_1,q_2)$ or $g_t(q) < 0$ for all $q \in (q_1,q_2)$.

    Suppose the former case holds.
    We will show the conclusion holds with \rev{$I = [t-\delta,t]$} for small $\delta$.
    We first show that we must have $t>0$, so this is a valid interval.
    Suppose for contradiction that $t=0$; then $q_1 = 0$.
    So, $g_0(q) = \xi'(q) - \fr{q}{1-q}$ is positive on $(0,q_2)$.
    This implies that for $q\in (0,q_2]$,
    \[
        \xi(q) + q + \log(1-q) = \int_0^q g_0(s)~\de s > 0,
    \]
    contradicting \eqref{eq:replica-symmetry}.
    Note that
    \[
        \lt(\cxi_t(q_2) + q_2 + \log(1-q_2)\rt) - \lt(\cxi_t(q_1) + q_1 + \log(1-q_1)\rt)
        = \int_{q_1}^{q_2} g_t(q)~\de q > 0.
    \]
    We claim that $q_\sAMP(t')$ is continuous on $t'\in I$, for small enough $\delta$.
    If $q_1 = 0$, this is clear because $q_\sAMP(t)$ is increasing.
    Otherwise, since $g_t(0) \ge 0$ and $q_1$ is the smallest zero of $g_t$, we have $g_t(q) > 0$ for $q\in [0,q_1)$.
    Since the $g_t(q)$ are continuous and increasing in $t$, the claim follows.
    It follows that for sufficiently small $\delta$, for all $t'\in I$ and $q'_1 = q_\sAMP(t')$,
    \[
        \lt(\cxi_{t'}(q_2) + q_2 + \log(1-q_2)\rt) - \lt(\cxi_{t'}(q'_1) + q'_1 + \log(1-q'_1)\rt) > 0.
    \]
    Thus $q_\sAMP(t') \neq q_\bayes(t')$ for all $t'\in I$.

    Finally, we consider the case where $g_t(q) < 0$ for all $q \in (q_1,q_2)$.
    Then, $g_t > 0$ on $[0,q_1)$ (vacuously if $q_1=0$) and $g_t < 0$ on $(q_1,q_2)$.
    Let $t''$ be the smallest time such that $\inf_{q\in [q_1,q_2]} g_{t''}(q) \ge 0$; this is finite by the discussion surrounding \eqref{eq:gt-lb}.
    Since $g_t(q)$ is increasing in $t$, we have $g_{t''} \ge 0$ for $q\in [0,q_2]$, with equality attained at some $q \in [q_1,q_2]$.
    By definition, $q_\sAMP(t'')$ is the smallest such $q$.
    As \rev{$g_{t''}(q_2) > g_t(q_2) = 0$}, we have $q_\sAMP(t'') < q_2$.
    The result now follows from the first case.
\end{proof}

\begin{proof}[Proof of Theorem~\ref{thm:sl-hardness}]
    By the last two lemmas, there exists a nontrivial interval $I = [t_-,t_+] \subseteq [0,+\infty)$ such that $q_\sAMP(t) \neq q_\bayes(t)$ for all $t\in I$.
    Since $q_\bayes(t)$ is a maximizer of \eqref{eq:def-qbayes}, it satisfies the stationarity condition $\cxi'_t(q) = \fr{q}{1-q}$, and therefore $q_\sAMP(t) < q_\bayes(t)$.
    It also follows that $q_\bayes(t) > 0$.

    \rev{Let $U(t)=|\{1\le j\le J:\tau_j(t)>0\}|$. Since every $\tau_j$ is nonnegative and nondecreasing, $U$ is nondecreasing and can jump at most $J$ times;} let $T_0$ be the set of these discontinuities.

    We will show the theorem holds with $\cI = I \setminus (T_0 \cup T_1)$.
    (Recall the definition of $T_1$ in Lemma \ref{lem:finite-zeros}.)
    This is a positive-measure set by Lemma~\ref{lem:finite-zeros}.
    Consider any $t \in \cI$.
    Since $t\not\in T_0$, there exists $1\le j\le J$ such that the $q^j$ coefficient of $\cxi_t$ is positive and $\tau'_j(t) > 0$.
    By Propositions~\ref{ppn:sl-hardness-amp} and \ref{ppn:sl-hardness-bayes},
    \baln
        \lim_{k\to\infty}
        \lim_{N\to\infty}
        \bbE
        \fr{1}{N^j}
        \norm{\bx^{\otimes j} - (\cbm^k)^{\otimes j}}_2^2
        &= 1 - q_\sAMP(t)^j, \\
        \lim_{N\to\infty}
        \bbE
        \fr{1}{N^j}
        \tnorm{\bx^{\otimes j} - \bm_j(\vby_t,t)}^2
        &\le 1 - q_\bayes(t)^j.
    \ealn
    Since $q_\sAMP(t) < q_\bayes(t)$, the conclusion follows.
\end{proof}

\subsection*{Acknowledgements}

BH was supported by a Google PhD Fellowship, NSF CAREER grant DMS-1940092, and the 
Solomon Buchsbaum Research Fund at MIT.
AM was partially supported by the NSF Awards DMS MFAI-2501597
and DMS MFAI-2601871.
HTP was supported by a Clay Research Fellowship, a Sloan Research Fellowship, 
NSF Career grant DMS-2543870, and a Bell System Fellowship.

\bibliographystyle{alpha}
\newcommand{\etalchar}[1]{$^{#1}$}

\end{document}

%% file: commands.tex
\def\beq#1\eeq{%
    \begin{equation}%
    #1%
    \end{equation}%
}
\def\baln#1\ealn{%
    \begin{align*}%
    #1%
    \end{align*}%
}
\def\balnn#1\ealnn{%
    \begin{align}%
    #1%
    \end{align}%
}

\def\fr{\frac}
\def\lt{\left}
\def\rt{\right}
\def\la{\langle}
\def\ra{\rangle}
\def\<{\langle}
\def\>{\rangle}

\def\eps{{\varepsilon}}

\def\norm#1{\lt\|#1\rt\|}
\def\tnorm#1{\|#1\|}

\def\bbE{{\mathbb{E}}}

\def\bbN{{\mathbb{N}}}
\def\bbP{{\mathbb{P}}}
\def\bbQ{{\mathbb{Q}}}
\def\bbR{{\mathbb{R}}}
\def\bbS{{\mathbb{S}}}

\def\indep{{\perp\!\!\!\perp}}
\DeclareMathOperator*{\EE}{\bbE}
\DeclareMathOperator*{\PP}{\bbP}
\DeclareMathOperator*{\Cov}{Cov}
\DeclareMathOperator*{\Var}{Var}
\DeclareMathOperator*{\semic}{sc}

\DeclareMathOperator*{\plim}{p-lim}

\DeclareMathOperator*{\argmax}{arg\,max}

\def\cC{{\mathcal{C}}}
\def\cE{{\mathcal{E}}}
\def\cF{{\mathcal{F}}}
\def\cG{{\mathcal{G}}}
\def\cH{{\mathcal{H}}}
\def\cI{{\mathcal{I}}}
\def\cK{{\mathcal{K}}}
\def\cL{{\mathcal{L}}}
\def\cN{{\mathcal{N}}}
\def\cP{{\mathcal{P}}}
\def\cR{{\mathcal{R}}}
\def\cS{{\mathcal{S}}}
\def\cT{{\mathcal{T}}}
\def\cU{{\mathcal{U}}}

\def\sH{{\mathscr{H}}}

\def\bfzero{{\boldsymbol 0}}

\def\bb{{\boldsymbol b}}
\def\be{{\boldsymbol e}}
\def\bg{{\boldsymbol g}}
\def\bm{{\boldsymbol m}}

\def\bu{{\boldsymbol u}}
\def\bv{{\boldsymbol v}}
\def\bw{{\boldsymbol w}}
\def\bx{{\boldsymbol x}}
\def\by{{\boldsymbol y}}
\def\bz{{\boldsymbol z}}

\def\bA{{\boldsymbol A}}
\def\bE{{\boldsymbol E}}
\def\bG{{\boldsymbol G}}
\def\bM{{\boldsymbol M}}

\def\bW{{\boldsymbol W}}

\def\bU{{\boldsymbol U}}
\def\bDel{{\boldsymbol \Delta}}
\def\bDelM{{\boldsymbol E}}
\def\vby{{\vec \by}}
\def\tbeta{{\widetilde \beta}}
\def\tq{{\widetilde q}}

\def\bgamma{{\boldsymbol \gamma}}

\def\hH{\widehat{H}}
\def\thH{\underline{H}}
\def\uE{\underline{E}}
\def\oZ{\overline{Z}}

\def\oa{\bar{a}}
\def\ob{\bar{b}}
\def\Err{{\sf Err}}

\def\bfone{{\boldsymbol 1}}
\def\bT{{\boldsymbol T}}
\def\by{{\boldsymbol y}}
\def\bbeta{{\boldsymbol \beta}}
\def\bsig{{\boldsymbol \sigma}}
\def\brho{{\boldsymbol \rho}}
\def\bzero{{\boldsymbol 0}}
\def\btau{{\boldsymbol \tau}}
\def\bB{{\boldsymbol B}}
\def\bI{{\boldsymbol I}}
\def\ind{{\boldsymbol 1}}
\def\bLambda{{\boldsymbol \Lambda}}

\def\tH{{\widetilde H}}
\def\tV{{\widetilde V}}
\def\tZ{{\widetilde Z}}

\def\tnu{{\widetilde \nu}}
\def\txi{{\widetilde \xi}}
\def\tbm{{\widetilde \bm}}
\def\tbx{{\widetilde \bx}}
\def\tbz{{\widetilde \bz}}
\def\tpartial{{\widetilde \partial}}

\def\cxi{{\widecheck \xi}}
\def\cq{{\widecheck q}}
\def\cbm{{\widecheck \bm}}
\def\cbw{{\widecheck \bw}}
\def\sym{{\mathrm{sym}}}
\def\cbbP{{\widecheck \bbP}}
\def\cbbQ{{\widecheck \bbQ}}
\def\ccH{{\widecheck H}}
\def\cmu{{\widecheck \mu}}
\def\oH{{\overline H}}

\def\hq{{\widehat q}}
\def\hH{{\widehat H}}
\def\hV{{\widehat V}}
\def\hZ{{\widehat Z}}
\def\hmu{{\widehat \mu}}

\def\hxi{{\widehat \xi}}
\def\hbm{{\widehat \bm}}
\def\hbw{{\widehat \bw}}
\def\hby{{\widehat \by}}
\def\hbsig{{\widehat \bsig}}
\def\Proj{{\sf P}}
\def\Ts{{\sf T}}

\def\corr{{\Delta}}
\def\bcorr{{\boldsymbol \Delta}}

\def\bQ{{\boldsymbol Q}}
\def\bR{{\boldsymbol R}}
\def\bD{{\boldsymbol D}}

\def\Vol{{\mathsf{Vol}}}

\def\alg{{\mathsf{alg}}}
\def\bayes{{\mathsf{bayes}}}

\def\cond{{\mathsf{cond}}}
\def\de{{\mathsf{d}}}
\def\dist{{\mathsf{dist}}}

\def\nul{{\mathsf{null}}}
\def\pl{{\mathsf{pl}}}
\def\rd{{\mathsf{rad}}}

\def\smc{{\mathsf{sc}}}
\def\sp{{\mathsf{sp}}}
\def\spec{{\mathsf{spec}}}
\def\ssc{{\mathsf{sc}}}
\def\spn{{\mathsf{span}}}
\def\tn{{\mathsf{tan}}}
\def\unif{{\mathsf{unif}}}
\def\Ball{{\mathsf{B}}}
\def\Band{{\mathsf{Band}}}
\def\Crt{{\mathsf{Crt}}}
\def\AMP{{\mathsf{AMP}}}
\def\sAMP{{\mbox{\tiny\sf AMP}}}
\def\GD{{\mathsf{GD}}}
\def\GOE{{\mathsf{GOE}}}
\def\KL{{\mathsf{KL}}}
\def\Tr{{\mathsf{Tr}}}
\def\TAP{{\mathsf{TAP}}}
\def\sTAP{{\mbox{\tiny\sf TAP}}}
\def\sMALA{{\mbox{\tiny\sf MALA}}}
\def\TV{{\mathsf{TV}}}
\def\proj{{\mathsf{proj}}}
\def\ssh{{\mbox{\tiny\sf sh}}}
\def\sSL{{\mbox{\tiny\sf SL}}}

\def\Ts{{\sf T}}   
\def\Lip{{\rm Lip}}
\def\op{\mbox{{\tiny \rm op}}}
\def\inj{\mbox{{\tiny \rm inj}}}

\def\ed{\stackrel{\mbox{\tiny \rm d}}{=}}

\definecolor{darkblue}{RGB}{25,25,200}
\newcommand{\rev}[1]{{ #1}}

\definecolor{darkpurple}{RGB}{100,0,100}